\documentclass[11pt,reqno]{amsart}

\usepackage{geometry}
\usepackage{amsmath,amsthm,amssymb,amsfonts}

\usepackage[utf8]{inputenc} %
\usepackage[T1]{fontenc}    %
\usepackage[hidelinks]{hyperref} %
\hypersetup{colorlinks=true,linkcolor=red,citecolor=blue}
\usepackage[capitalise]{cleveref}

\usepackage{mathtools,nicematrix}
\usepackage{comment}
\usepackage{subcaption}
\usepackage{wrapfig}
\usepackage{bm}
\usepackage{url}            %
\usepackage{booktabs}       %
\usepackage{nicefrac}       %
\usepackage{microtype}      %

\usepackage{svg}

\usepackage{bbm}
\usepackage{cite}
\usepackage{psfrag}
\usepackage{cite}
\usepackage{color,soul}
\usepackage{breakcites}     %
\usepackage{bigints}

\usepackage{enumitem} \setlist{leftmargin=1.6em}
\usepackage{tikz}
\usetikzlibrary{graphs}
\usetikzlibrary{graphs.standard}

\newtheorem{theorem}{Theorem}
\newtheorem{proposition}{Proposition}
\newtheorem{assumption}{Assumption}
\newtheorem{lemma}{Lemma} 
\newtheorem{corollary}{Corollary}

\theoremstyle{definition}
\newtheorem{definition}{Definition}

\theoremstyle{remark}
\newtheorem{remark}{Remark}

\usepackage{graphicx}
\usepackage{xcolor}
\usepackage{svg}

\usepackage{algorithm}
\usepackage[noend]{algpseudocode}

\makeatletter
\def\subsubsection{\@startsection{subsubsection}{3}%
  \z@{.5\linespacing\@plus.7\linespacing}{-.5em}%
  {\normalfont\bfseries}}
\makeatother

\newcommand{\Cov}[0]{\mathrm{Cov}}
\newcommand{\Var}[0]{\mathrm{Var}}

\newcommand{\supp}{\mathrm{spt}}

\newcommand{\OT}{\mathsf{OT}}

\newcommand{\vertiii}[1]{{\left\vert\kern-0.25ex\left\vert\kern-0.25ex\left\vert #1 
    \right\vert\kern-0.25ex\right\vert\kern-0.25ex\right\vert}}

\newcommand{\NN}{\mathbb{N}}

\newcommand{\RR}{\mathbb{R}}

\newcommand{\vasti}{\bBigg@{3.5 }}
\newcommand{\vast}{\bBigg@{4}}
\newcommand{\Vast}{\bBigg@{5}}
\newcommand{\Vastt}{\bBigg@{7}}

\makeatother

\newcommand{\be}{\begin{equation}}
\newcommand{\ee}{\end{equation}}
\newcommand{\ba}{\begin{align}}
\newcommand{\ea}{\end{align}}
\newcommand{\baa}{\begin{align*}}
\newcommand{\eaa}{\end{align*}}

\newcommand{\bsigma}{\mathbf{\Sigma}}
\newcommand{\argmin}{\mathop{\mathrm{argmin}}}
\newcommand{\argmax}{\mathop{\mathrm{argmax}}}

\DeclareMathOperator{\interior}{int}

\DeclareMathOperator{\inte}{int}

\DeclareMathOperator{\conv}{conv}
\DeclareMathOperator{\Id}{Id}
\DeclareMathOperator{\Lag}{Lag}
\DeclareMathOperator{\vol}{vol}

\begin{document}

\title[Limit Laws for Gromov-Wasserstein Alignment]{Limit Laws for Gromov-Wasserstein Alignment with Applications to Testing Graph Isomorphisms}

\thanks{
Z. Goldfeld is partially supported by NSF grants CCF-1947801,  CCF-2046018, and DMS-2210368, and the 2020 IBM Academic Award.
K. Kato is partially supported by NSF grants  DMS-2210368 and DMS-2413405.
G. Rioux is partially supported by the NSERC postgraduate fellowship PGSD-567921-2022.}

\date{First version: April, 1 2024. This version: \today}

\author[G. Rioux]{Gabriel Rioux}

\address[G. Rioux]{
Center for Applied Mathematics, Cornell University.}
\email{ger84@cornell.edu}

\author[Z. Goldfeld]{Ziv Goldfeld}
\address[Z. Goldfeld]{
School of Electrical and Computer Engineering, Cornell University.
}
\email{goldfeld@cornell.edu}

\author[K. Kato]{Kengo Kato}
\address[K. Kato]{
Department of Statistics and Data Science, Cornell University.
}
\email{kk976@cornell.edu}

\begin{abstract}
The Gromov-Wasserstein (GW) distance enables comparing metric measure spaces based solely on their internal structure, making it invariant to isomorphic transformations. This property is particularly useful for comparing datasets that naturally admit isomorphic representations, such as unlabelled graphs or objects embedded in space. However, apart from the recently derived empirical convergence rates for the quadratic GW problem, a statistical theory for valid estimation and inference remains largely obscure. Pushing the frontier of statistical GW further, this work derives the first limit laws for the empirical GW distance across several settings of interest: (i)~discrete, (ii)~semi-discrete, and (iii)~general distributions under moment constraints under the entropically regularized GW distance. The derivations rely on a novel stability analysis of the GW functional in the marginal distributions. The limit laws then follow by an adaptation of the functional delta method. As asymptotic normality fails to hold in most cases, we establish the consistency of an efficient estimation procedure for the limiting law in the discrete case, bypassing the need for computationally intensive resampling methods. We apply these findings to testing whether collections of unlabelled graphs are generated from distributions that are isomorphic to each other. 
\end{abstract}


\maketitle

\section{Introduction}

The Gromov-Wasserstein (GW) distance provides a framework, rooted in optimal transport (OT) theory \cite{villani2008optimal,santambrogio15}, for comparing probability distributions on possibly distinct spaces by aligning them with one another. The $(p,q)$-GW distance between two metric measure (mm) spaces $(\mathcal X_0,\mathsf d_0,\mu_0)$ and $(\mathcal X_1,\mathsf d_1,\mu_1)$ is defined as \cite{Memoli11} 
\begin{equation}
\label{eq:GromovWassersteinDefinition}
\mathsf D_{p,q}(\mu_0,\mu_1)\coloneqq \inf_{\pi\in\Pi(\mu_0,\mu_1)}\left(\int|\mathsf d_0^q(x,x')-\mathsf d_1^q(y,y')|^pd\pi\otimes \pi(x,y,x',y')\right)^{\frac 1 p},
\end{equation}
where $\Pi(\mu_0,\mu_1)$ is the set of all couplings of $\mu_0$ and $\mu_1$. Evidently, $\mathsf D_{p,q}(\mu_0,\mu_1)$ seeks to minimize distortion of distances over all possible alignment plans (modeled by couplings) of the considered mm spaces. This distance serves as an OT-based $L^p$ relaxation of the classical Gromov-Hausdorff distance and defines a metric on the space of all mm spaces modulo the equivalence defined by isomorphism.\footnote{Two mm spaces $(\mathcal X_0,\mathsf d_0,\mu_0)$ and  $(\mathcal X_1,\mathsf d_1,\mu_1)$ are said to be isomorphic if there exists an isometry $T:\supp(\mu_0)\to \supp(\mu_1)$ for which $\mu_0\circ T^{-1}=\mu_1$.} %
The GW framework has seen recent success in tasks involving alignment of heterogeneous datasets, such as single-cell genomics \cite{blumberg2020mrec,cao2022manifold,
demetci2020gromov}, alignment of language models \cite{alvarez2018gromov}, shape matching \cite{koehl2023computing,memoli2009spectral}, graph matching \cite{chen2020graph,petric2019got,xu2019gromov,xu2019scalable}, heterogeneous domain adaptation \cite{sejourne2021unbalanced,yan2018semi}, and generative modeling \cite{bunne2019learning}.

Despite its appealing structure and broad applications, a principled statistical and computational framework for GW alignment remained elusive since its inception more than a decade ago \cite{Memoli11}. The difficulty in analyzing the GW problem lies in its quadratic dependence on $\pi$, leaving techniques developed for the (linear) OT problem not directly applicable. Recently, a variational form of the quadratic GW distance, i.e., with $p=q=2$, between Euclidean mm spaces $(\mathbb R^{d_0},\|\cdot\|,\mu_0)$ and $(\mathbb R^{d_1},\|\cdot\|,\mu_1)$ was derived, casting it as the infimum over a class of OT problems \cite{zhang2024gromov}:
\begin{equation}
\label{eq:variationalFormIntro}
     \mathsf D(\mu_0,\mu_1)^2\coloneqq \mathsf D_{2,2}(\mu_0,\mu_1)= \mathsf S_1(\mu_0,\mu_1)+\inf_{\mathbf A\in\mathbb R^{d_0\times d_1}}\left\{32\|\mathbf A\|_{\mathrm F}^2+\mathsf{OT}_{\mathbf A}(\mu_0,\mu_1) \right\}.
\end{equation}
Here $\mathsf S_1(\mu_0,\mu_1)$ is a constant that depends only on the moments of the marginals and $\mathsf{OT}_{\mathbf A}(\mu_0,\mu_1)$ denotes the OT problem with a cost function parametrized by an auxiliary matrix $\mathbf A$. This connection to the well-understood OT problem unlocked it as a tool for the study of GW, leading to new algorithms for approximate computation with formal guarantees \cite{rioux2023entropic}, sharp rates of empirical convergence \cite{zhang2024gromov,groppe2023lower}, and advances in GW gradient flows and differential geometry \cite{zhang2024gradient}. We seek to further progress the statistical GW theory by studying the asymptotic fluctuations of the empirical GW distance around the population value, providing first limit distribution results (upon suitable scaling) and efficient algorithms for simulating the limits. These advances are leveraged for applications to graph isomorphism testing.

\subsection{Contributions} 
Utilizing the so-called \emph{dual formulation} from \eqref{eq:variationalFormIntro}, we first explore stability properties of the GW distance under perturbations of the marginal distributions. Combining this with the functional delta method, we establish the first limit theorems for GW distances at the scale of $\sqrt{n}$, under several settings of interest. Specifically, we adopt the methodology for deriving distributional limits developed in \cite{goldfeld24statistical}, which requires showing that the GW functional is right differentiable along certain directions and locally Lipschitz continuous in the sup-norm over a certain function class. This strategy is applied in three cases: (i) discrete GW, when both population distributions are finitely discrete, (ii) semi-discrete GW, when one population is finitely discrete and the other is absolutely continuous and compactly supported, and (iii) entropically regularized GW when the populations are arbitrary 4-sub-Weibull distributions. In all cases, local Lipschitz continuity follows from similar arguments whereas proving right differentiability comprises a more challenging task that requires different approaches tailored to each setting.

In the finitely discrete setting, we employ stability results for the optimal value function in nonlinear optimization \cite{bonnans2013perturbation} to prove the desired right differentiability. The main challenge in doing so stems from the general result requiring a directional regularity condition that does not hold for the GW setting. We circumvent this by exploiting the variational form \eqref{eq:variationalFormIntro} to directly establish strong duality for a linearization of the discrete GW problem. As the obtained derivative is nonlinear in general, the resulting limit distribution is not normal and the na\"ive bootstrap fails to be consistent (cf. \cite{dumbgen1993,fang2019}). To address this, we propose an efficient method for directly estimating the limit distribution under the null hypothesis (i.e., the population distributions are isomorphic) and prove its consistency. Remarkably, sampling from the proposed estimator only requires solving a linear program, while computing the test statistic requires the resolution of a nonconvex (in fact, NP-complete) quadratic program. These results serve as the cornerstone of our proposed application: testing if two inhomogeneous random graph models are isomorphic.%

In the semi-discrete setting, a central component of our analysis is to derive a simple sufficient condition for the optimal coupling of $\mathsf{OT}_{\mathbf A}(\mu_0,\mu_1)$ to be unique and induced by a deterministic map, uniformly in the choice of  $\mathbf A\in\mathbb R^{d_0\times d_1}$. Consequently, this same condition implies that the GW problem itself admits an alignment plan that is induced by a map, a so-called \emph{Gromov-Monge map}. This adds to the limited number of cases where such alignment maps are known to exist (to wit \cite[Theorem 9.21]{sturm2012space}, \cite[Theorem 4.1.2 and Proposition 4.2.4]{vayer2020contribution}, and \cite[Theorem 3.2]{dumont2022existence}). %
Using this condition, the corresponding right derivative can be computed by appealing to the GW dual form, leading to a characterization of distributional limits for semi-discrete GW. As in the discrete case, the limiting law is generally not Gaussian, and the na\"ive bootstrap is inapplicable. 

To address the setting when both distributions are non-discrete, we consider the entropically regularized GW distance between 4-sub-Weibull distributions. The analysis leverages the fact that optimal couplings and dual potentials are unique (in a suitable sense) for each of the entropic OT problems featuring in the dual form. %
The required stability can then be demonstrated by following a similar argument to the semi-discrete case. Building on the results of \cite{rioux2023entropic}, we show that the derived limit distribution is normal once the regularization parameter surpasses a given threshold that depends on the marginal moments, and may fail to be normal otherwise. When asymptotic normality holds, the bootstrap enables consistent estimation of the distributional limit. 

We apply the derived limit distribution theory to the problem of two-sample testing for graph isomorphism. As our theory account only for Euclidean mm spaces, we propose an embedding for distributions on random graphs with $N$ vertices into the space of distributions on $\mathbb R^N$. Crucially, this embedding is constructed such that the embedded distributions are identified under the GW distance if and only if the associated distributions on graphs are isomorphic. This enables us to formulate a consistent test for the graph isomorphism question with the GW distance as the test statistic. To our knowledge, these are the first results in this setting (see literature review below). 
As the GW test statistic is hard to compute, we propose efficient procedures for estimating it numerically based on the linear programming-based sampler from the discrete GW limiting variable. This methodology is empirically validated in a number of numerical experiments. The experiments also showcase asymptotic normality of empirical entropic GW and consistency of the bootstrap under certain conditions.

\subsection{Literature Review}

The statistical properties of the empirical GW distance are largely unknown, beyond the following convergence rate result (see Theorem 3 \cite{zhang2024gromov}). If $(\mu_0,\mu_1)\in\mathcal P(\mathcal X_0)\times \mathcal P(\mathcal X_1)$ where $\mathcal X_0\subset \mathbb R^{d_0}$ and $\mathcal X_1\subset \mathbb R^{d_1}$ are compact, then, for $R=\max_{i\in\{0,1\}}\mathrm{diam}(\mathcal X_i)$ and $d=\min_{i\in\{0,1\}}d_i$,     
\[
    \mathbb E\left[\left|\mathsf D(\hat \mu_{0,n},\hat \mu_{1,n})^2-\mathsf D( \mu_{0},\mu_{1})^2\right|\right]\lesssim_{d_0,d_1}{R^4}n^{-\frac 12}+(1+R^4)n^{-\frac{2}{\max\left\{4,d\right\}}}(\log n)^{\mathbbm 1_{\{d=4\}}}, 
    \]
    where $\hat \mu_{0,n},\hat \mu_{1,n}$ are the empirical measures of $n$ i.i.d. observations from $\mu_0$ and $\mu_1$, respectively. These rates are sharp up to $\mathrm{polylog}$ factors and, notably, scale with the dimension of the \emph{lower-dimensional} space. As for the entropically regularized GW distance, Theorem 2 in \cite{zhang2024gromov} establishes an $n^{-1/2}$ parametric empirical convergence rate, provided that $\mu_0$ and $\mu_1$ satisfy the aforementioned 4-sub-Weibull condition. That work also showed that entropic GW with parameter $\varepsilon>0$ approximates the unregularized distance at the level of $O(\varepsilon\log(1/\varepsilon))$, thus approaching it as $\varepsilon\to0$. In \cite{groppe2023lower}, the above empirical convergence rates were shown to adapt to the intrinsic dimension of the population. Formal guarantees for a neural estimator of the GW distance were provided in \cite{wang2024neural}, accounting for empirical estimation and function approximation errors (but not the optimization error).

The existence of Gromov-Monge maps for the semi-discrete GW problem derived herein adds to a few other known cases for which deterministic schemes are optimal. %
In the Euclidean setting with $\mathsf D_{2,2}$, results are available when (i) the distributions are uniform on the same finite number of points \cite[Theorem 4.1.2]{vayer2020contribution}, (ii) $\mu_0$ and $\mu_1$ are absolutely continuous and rotationally invariant about their barycenter  \cite[Theorem 9.21]{sturm2012space}, and (iii) when the cross-correlation matrix of some optimal alignment plan for $\mathsf D(\mu_0,\mu_1)$ is of full rank among other conditions \cite[Proposition 4.2.4]{vayer2020contribution}. Beyond the quadratic setting, the only other result concerns the inner product GW distance (i.e., when the cost function in \eqref{eq:GromovWassersteinDefinition} is $\left|\langle x,x' \rangle-\langle y,y' \rangle\right|^2$, with $p=2$) between populations $\mu_0,\mu_1$ that are compactly supported in their respective, with the lower-dimensional one being absolutely continuous.

Our approach to proving limit theorems for GW distances is based on the generic framework put forth in \cite{goldfeld24statistical} which, in turn, relies on the functional delta method (see \cref{sec:deltaMethod} for details). The delta method \cite{romisch2004,shapiro1990} has seen great success for proving limit theorems for Wasserstein distances and regularized OT.
    Notably, \cite{Sommerfeld2018} first applied this technique to establish limit theorems for the $p$-Wasserstein distance on finite metric spaces. Next, \cite{Tameling2019} extended this approach to the case of countable metric spaces, and     \cite{hundrieser2022unifying} provided a generic stability result for the OT cost and provided sufficient conditions for the CLT to hold. This approach also proved effective in studying the distributional limits of regularized optimal transport functionals—under smoothing, slicing, or entropic regularization—and associated objects (e.g., maps and dual potentials) \cite{rioux2022smooth,goldfeld2024limit,goldfeld24statistical}. %

 Testing \emph{equality} of distributions on graphs was first addressed in \cite{gretton2012kernel} based on the maximum mean discrepancy. While the experiment in that work also considered equality up to isomorphisms, their theoretical results did not cover that setting. Another approach relies on testing if certain graph statistics, such as the mean graph Laplacian, are equal at the population level based on samples \cite{ginestet2017hypothesis}. Other lines of work include the setting where the number of samples grows with the graph size 
\cite{chatterjee2023two,chen2023hypothesis}  or when the graphs are large and only a small constant number of samples from each distribution is available \cite{ghoshdastidar2020two,ghoshdastidar2017two,ghoshdastidar2018practical,jin2024optimal, li2018two,tang2017semiparametric,Tang2017nonparametric}. The scopes and/or settings of those references differ from the statistical framework and accompanying theory  we provide for graph isomorphism testing.

\section{Background and Preliminary Results}

\subsection{Notation}

For a nonempty {Borel} set $S\subset \mathbb R^d$, $\mathcal P(S)$ is the set of {Borel} probability measures on $S$, $\mathcal C(S)$ denotes the set of continuous functions on $S$, and $\ell^{\infty}(S)$ is the Banach space of bounded real functions on $S$ equipped with the supremum norm $\|\ell\|_{\infty,S}=\sup_{s\in S}|\ell(s)|$. When $S\subset\mathbb R^d$ is a compact set, we write $\|S\|_{\infty}=\sup_{s\in S}\|s\|$, where the latter is the Euclidean norm.

For a probability measure $\rho\in\mathcal P(\mathbb R^{d})$, set $\mathcal P_{\rho} \coloneqq \left\{ \nu\in\mathcal P(\mathbb R^d):\supp(\nu)\subset \supp(\rho)\right\}$. If $T:\mathbb R^{d}\to \mathbb R^{d'}$ is a measurable map, $T_{\sharp}\rho\in\mathcal P(\mathbb R^{d'})$ is the pushforward measure defined by $T_{\sharp}\rho(A)=\rho(T^{-1}(A))$ for every Borel measurable set $A\subset \mathbb R^{d}$. The centered version of $\rho$ is denoted by $\bar \rho=(\Id-\mathbb E_{\rho}[X])_{\sharp}\rho$. The $p$-th moment of $\rho$ is denoted $M_p(\rho)=\int \|\cdot\|^pd\rho$, 
$\bsigma_{\rho}=\int zz^{\intercal}d\rho(z)$ is its cross-correlation matrix, and for a $\rho$-integrable function $f:\mathbb R^{d}\to \mathbb R$, we write $\rho(f)=\int fd\rho$. A probability distribution $\rho\in\mathcal P(\mathbb R^d)$ is said to be $\beta$-sub-Weibull with parameter $\sigma^2$ for $\sigma\geq 0$, provided that $\int e^{\|\cdot\|^{\beta}/(2\sigma^2)}d\rho\leq 2$. We denote by $\mathcal P_{\beta,\sigma}(\RR^d)$ the set of all $\beta$-sub-Weibull distributions with parameter $\sigma^2$. Distributions that are $2$-sub-Weibull are called sub-Gaussian, while $X\sim\mu$ being 4-sub-Weibull is equivalent to sub-Gaussianity of $\|X\|^2$.
For measures $(\rho_0,\rho_1)\in\mathcal P(\mathbb R^{d_0})\times \mathcal P(\mathbb R^{d_1})$, $\rho_0\otimes \rho_1$ denotes the product measure. \smash{Convergence in distribution and weak convergence are denoted, respectively, by $\stackrel{d}{\to}$ and $\stackrel{w}{\to}$.}

For matrices $\mathbf{A},\mathbf{ B}\in\mathbb R^{d_0\times d_1}$ and $(i,j)\in[N_0]\times[N_1]$, $\mathbf{A}_{ij}$ is the $ij$-th entry of $\mathbf{A}$, and $\langle \mathbf{A},\mathbf{ B}\rangle_{\mathrm F}=\sum_{(i,j)\in[d_0]\times [d_1]} \mathbf{A}_{ij}\mathbf{ B}_{ij}$ is the Frobenius inner product. The induced Frobenius norm is denoted by $\|\cdot\|_{\mathrm{F}}$, with $B_{\mathrm{F}}(r)\subset \mathbb R^{d_0\times d_1}$ designating the closed Frobenius ball with radius $r>0$. A function $f:\mathbb R^{d_0\times d_1}\to \mathbb R$ is said to be Fr{\'e}chet differentiable at a point $\mathbf{A}\in\mathbb R^{d_0\times d_1}$ if $\lim_{\|\mathbf{H}\|_{\mathrm{F}}\downarrow 0}\frac{|f(\mathbf{A}+\mathbf{H})-f(\mathbf{A})-D f_{[\mathbf{A}]}(\mathbf{H})|}{\|\mathbf{H}\|_{\mathrm{F}}}=0$ for some bounded linear functional $D f_{[\mathbf{A}]}:\mathbb R^{d_0\times d_1} \to \mathbb R$ which we call the Fr{\'e}chet derivative of~$f$~at~$\mathbf{A}$.

For functions $f_0:\mathbb R^{d_0}\to \mathbb R$ and $f_1:\mathbb R^{d_1}\to \mathbb R$, their direct sum is $f_0\oplus f_1:(x,y)\in\mathbb R^{d_0}\times \mathbb R^{d_1}\mapsto f_0(x)+f_1(y)$.  For any nonnegative integer $k$ and nonempty open set $S\subset \mathbb R^d$, $\mathcal C^k(S)$ is the set of $k$-times continuously differentiable functions on $S$. For $N\in\mathbb N$, $[N]$ denotes the set $\{1,\dots,N\}$, and $\mathbbm 1_N\in\mathbb R^N$ is the vector of all ones.

\subsection{Optimal Transport}
\label{sec:OT}
For $(\mu_0,\mu_1)\in\mathcal P(\mathbb R^{d_0})\times \mathcal P(\mathbb R^{d_1})$, and a continuous cost function $c:\mathcal X_0\times \mathcal X_1\to \mathbb R$ satisfying $c\geq a\oplus b$ for some continuous functions $(a,b)\in L^1(\mu_0)\times L^1(\mu_1)$, the optimal transport (OT) problem is given by 
\[
    \mathsf{OT}_c(\mu_0,\mu_1)=\inf_{\pi\in\Pi(\mu_0,\mu_1)}\int cd\pi,
\]
where $\Pi(\mu_0,\mu_1)$ is the set of all couplings of $(\mu_0,\mu_1)$. If, in addition, $\int c d\mu_0\otimes \nu_0<\infty$, the following duality result holds (cf. e.g. Theorem 5.10 and Remark 5.14 in \cite{villani2008optimal}) 
\begin{equation}
    \label{eq:OTDuality}
\OT_c(\mu_0,\mu_1)=\sup_{\substack{(\varphi_0,\varphi_1)\in\mathcal C(\mathcal X_0)\times \mathcal C(\mathcal X_1)\\\varphi_0\oplus\varphi_1\leq c}} \left\{\int \varphi_0 d\mu_0+\int \varphi_1 d\mu_1\right\}.
\end{equation}
Both the primal (minimization) problem and the dual (maximization) problems admit solutions, which we refer to, respectively, as an \textit{OT plan} and \textit{OT potentials} for $\OT_c(\mu_0,\mu_1)$. 

Of note is that if $(\varphi_0^{\star},\varphi_1^{\star})$ is a pair of OT potentials for $\OT_c(\mu_0,\mu_1)$, then so is $(\varphi_0^{\star}+a,\varphi_1^{\star}-a)$ for any $a\in\mathbb R$; we refer to potentials related in this fashion as \textit{versions} of one another. If $\pi=(\Id,T)_{\sharp}\mu_0$ for some map $T:\mathbb R^{d_0} \to  \mathbb R^{d_1}$, we say that $\pi$ is induced by the map $T$ and, if such a $\pi$ is optimal for $\OT_c(\mu_0,\mu_1)$, then $T$ is called an \textit{OT map} (also known as a Monge/Brenier map). 

\subsection{Entropic Optimal Transport}
\label{sec:EOT}
The entropic optimal transport (EOT) problem was introduced in \cite{cuturi2013lightspeed} as a convexification of the classical setting. It is defined by regularizing the transportation cost with the (strongly convex) Kullback-Leibler divergence: %
\begin{equation}
\label{eq:EOTPrimal}
    \mathsf{OT}_{c,\varepsilon}(\mu_0,\mu_1) =\inf_{\pi\in\Pi(\mu_0,\mu_1)} \int cd\pi+\varepsilon\mathsf{D}_{\mathrm{KL}}(\pi\|\mu_0\otimes \mu_1),
\end{equation}
where %
\[
    \mathsf{D}_{\mathrm{KL}}(\alpha\|\beta)=\begin{cases}
        \int \log\left(\frac{d\alpha}{d\beta}\right)d\alpha,&\text{ if }\alpha\ll\beta,
        \\
        +\infty,&\text{otherwise.}
    \end{cases}
\]
The strong convexity of the regularized objective endows \eqref{eq:EOTPrimal} with many useful properties. First, if $c\in L^1(\mu_0\otimes \mu_1)$, \eqref{eq:EOTPrimal} is paired in strong duality with the problem 
\begin{equation}
\label{eq:EOTDual}
    \sup_{(\varphi_0,\varphi_1)\in L^1(\mu_0)\times L^1(\mu_1)}\left\{ \int \varphi_0d\mu_0+\int\varphi_1d\mu_1-\varepsilon \int e^{\frac{\varphi_0\oplus \varphi_1-c}\varepsilon}d\mu_0\otimes \mu_1+\varepsilon\right\}.
\end{equation}
A solution, $(\varphi_0^{\star},\varphi_1^{\star})$, to \eqref{eq:EOTDual} is called a pair of \emph{EOT potentials} for $\mathsf{OT}_{c,\varepsilon}(\mu_0,\mu_1)$ and is known to be a.s. unique up to additive constants (i.e. any other solution $(\psi_0^{\star},\psi_1^{\star})$ satisfies $(\psi_0^{\star},\psi_1^{\star})=(\varphi_0^{\star}+a,\varphi_1^{\star}-a)$ $\mu_0\otimes \mu_1$-a.s. for some $a\in\mathbb R$). Moreover, \eqref{eq:EOTPrimal} has a unique solution $\pi^{\star}\in\Pi(\mu_0,\mu_1)$, given in terms of the marginal measures and the EOT potentials via
\[
    \frac{d\pi^{\star}}{d\mu_0\otimes \mu_1}=e^{\frac{\varphi_0^{\star}\oplus \varphi_1^{\star}-c}{\varepsilon}},\quad \mu_0\otimes \mu_1\text{-a.s.}
\]
This further implies that $(\varphi_0,\varphi_1)\in L^1(\mu_0)\times L^1(\mu_1)$ solves \eqref{eq:EOTDual} if and only if $(\varphi_0,\varphi_1)$ solves the Schr\"odinger system
\begin{equation}
\label{eq:SchrodingerSystem}
    \int e^{\frac{\varphi_0(x)+\varphi_1-c(x,\cdot)}\varepsilon}d\mu_1=1,\text{ for } \mu_0 \text{-a.e. $x\in\mathbb R^{d_0}$,}
    \quad
    \int e^{\frac{\varphi_0+\varphi_1(y)-c(\cdot,y)}\varepsilon}d\mu_0=1,\text{ for } \mu_1 \text{-a.e. $y\in\mathbb R^{d_1}$.}
\end{equation}
See \cite{nutz2021introduction} for a comprehensive introduction to EOT.

\subsection{Gromov-Wasserstein Distances}
\label{sec:GWDistance}

Throughout, we focus on the (squared) $(2,2)$-GW distance between the Euclidean mm spaces $(\RR^{d_0},\|\cdot\|,\mu_0)$ and $(\RR^{d_1},\|\cdot\|,\mu_1)$, henceforth identified with the probability measures themselves. The distance between  $\mu_0\in\mathcal P(\mathbb R^{d_0}),\mu_1\in\mathcal P(\mathbb R^{d_1})$ with finite fourth moments is 
\begin{equation}
    \label{eq:GWPrimal}
    \mathsf D(\mu_0,\mu_1)^2=\inf_{\pi\in\Pi(\mu_0,\mu_1)}\|\Delta\|_{L^2(\pi\otimes\pi)},\ \ \text{ for }\ \ \Delta(x,y,x',y')\coloneqq\left|\|x-x'\|^2-\|y-y'\|^2\right|. 
\end{equation}
A coupling $\pi\in\Pi(\mu_0,\mu_1)$ solving \eqref{eq:GWPrimal} always exists  \cite[Corollary 10.1]{Memoli11} and is called an \emph{GW plan}. If a GW plan is induced by a measurable map, $T:\mathbb R^{d_0}\to \mathbb R^{d_1}$
, we call $T$ a Gromov-Monge map.  
Of note is that $\mathsf D$ nullifies if and only if there exists an isometry $\iota:\supp(\mu_0)\to \supp(\mu_1)$ %
for which $\mu_1=\iota_{\sharp}\mu_0$. When such a relationship holds, we say that the mm spaces $(\mathbb R^{d_0},\|\cdot\|,\mu_0)$ and  $(\mathbb R^{d_1},\|\cdot\|,\mu_1)$ are isomorphic. Furthermore, $\mathsf D$ defines a metric on the quotient space of all Euclidean mm spaces whose measures have finite fourth moments modulo the aforementioned isomorphism relation.

Recalling that $\bar\mu_i=(\Id-\mathbb E_{\mu_i}[X])_{\sharp}\mu_i$, for $i=0,1$, denotes the centered measures, it was recently shown in \cite{zhang2024gromov} (see Corollary 1 in that work) that $\mathsf D$ admits the following variational form,
\begin{equation}
    \label{eq:GWVariational}
    \begin{gathered} 
    \mathsf D(\mu_0,\mu_1)^2=
    \mathsf D(\bar\mu_0,\bar\mu_1)^2=\mathsf S_{1}(\bar\mu_0,\bar\mu_1)+\mathsf S_{2}(\bar\mu_0,\bar\mu_1)
    \\
    \mathsf S_1(\mu_0,\mu_1)=\int \|x-x'\|^4d\mu_0\otimes \mu_0(x,x')+\int \|y-y'\|^4d\mu_1\otimes \mu_1(y,y')-4M_2(\mu_0)M_2(\mu_1)
    \\
    \mathsf S_{2}(\mu_0,\mu_1)=\inf_{\mathbf{A}\in\mathbb R^{d_0\times d_1}}\Big\{32\|\mathbf{A}\|_{\mathrm{F}}^2+\OT_{\mathbf{A}}(\mu_0,\mu_1)  \Big\},
\end{gathered}
\end{equation}
where $\OT_{\mathbf{A}}(\mu_0,\mu_1)\coloneqq \OT_{c_{\mathbf{A}}}(\mu_0,\mu_1)$ for $c_{\mathbf{A}}:(x,y)\mapsto -4\|x\|^2\|y\|^2-32x^{\intercal}\mathbf{A}y$ and the equality $\mathsf D(\mu_0,\mu_1)^2=\mathsf D(\bar\mu_0,\bar\mu_1)^2$ is due to translation invariance. In particular, $\mathsf S_1$ is an explicit constant whereas $\mathsf S_{2}$ is a minimization problem with objective function, 
\begin{equation}
    \label{eq:Objective}
    \Phi_{(\mu_0,\mu_1)}:\mathbf{A}\in\mathbb R^{d_0\times d_1}\mapsto 32\|\mathbf{A}\|^2_{\mathrm F}+\OT_{\mathbf{A}}(\mu_0,\mu_1).
\end{equation}
Further, if $\pi^{\star}$ solves \eqref{eq:GWPrimal}, then $\mathbf{A}^{\star}=\frac{1}{2}\int xy^{\intercal}d\pi^{\star}(x,y)$ solves \eqref{eq:Objective} provided that $\mu_0,\mu_1$ are centered (see the proof of Theorem 1 in \cite{zhang2024gromov}). Our first result further refines the characterization of minimizers of $\eqref{eq:Objective}$. 

\begin{theorem}[On minimizers of \eqref{eq:Objective}]
    \label{thm:VariationalMinimizers}
    Let $(\mu_0,\mu_1)\in\mathcal P(\mathbb R^{d_0})\times \mathcal P(\mathbb R^{d_1})$ be compactly supported. The following statements hold:
    \begin{enumerate}
        \item $\Phi_{(\mu_0,\mu_1)}$ is locally Lipschitz continuous and coercive. If all OT plans, $\pi_{\mathbf{A}}$, for $\OT_{\mathbf{A}}(\mu_0,\mu_1)$ admit the same cross-correlation matrix, $\int xy^{\intercal}d\pi_{\mathbf{A}}(x,y)$, then 
            $\Phi_{(\mu_0,\mu_1)}$ is Fr{\'e}chet differentiable at $\mathbf{A}\in\mathbb R^{d_0\times d_1}$ with $\left(D\Phi_{(\mu_0,\mu_1)}\right)_{[\mathbf{A}]}(\mathbf{B})=64\langle\mathbf{A}-\frac{1}{2}\int xy^{\intercal}d\pi_{\mathbf{A}}(x,y),\mathbf{B}\rangle_{\mathrm F}$. 
        \item  If $\mathbf{A}^{\star}$ minimizes \eqref{eq:Objective}, then $2\mathbf{A}^{\star}=\int xy^{\intercal}d\pi^{\star}(x,y)\in {B_{\mathrm{F}}(\sqrt{M_2(\mu_0)M_2(\mu_1)})}$ for some OT plan $\pi^{\star}$ for $\OT_{\mathbf{A}^{\star}}(\mu_0,\mu_1)$. If $\mu_0,\mu_1$ are centered, then $\pi^{\star}$ solves \eqref{eq:GWPrimal}.  
    \end{enumerate} 
\end{theorem}
\cref{thm:VariationalMinimizers}, whose proof is provided in \cref{app:proofOfthm:VariationalMinimizers}, shows that all minimizers of \eqref{eq:Objective} are contained in ${B_{\mathrm{F}}\big(\frac{1}{2}\sqrt{M_2(\mu_0)M_2(\mu_1)}\big)}$ and, given such a minimizer, a solution to \eqref{eq:GWPrimal} can be obtained. Although Point 1 appears to directly imply Point 2, since a global minimizer of a differentiable coercive function must be a critical point, we stress that if $\OT_{\mathbf{A}}(\mu_0,\mu_1)$ admits multiple OT plans, $\Phi_{(\mu_0,\mu_1)}$ may fail to be differentiable at $\mathbf{A}$. To formalize this fact, 
the proof of \cref{thm:VariationalMinimizers} characterizes the Clarke subdifferential \cite{clarke1975generalized} of $\Phi_{(\mu_0,\mu_1)}$.  

\begin{remark}[Uniqueness of cross-correlation matrix] 
\label{rmk:uniqueCrosscorrelation}
The condition that all OT plans for $\OT_{\mathbf{A}}(\mu_0,\mu_1)$ share the same cross-correlation matrix seems nontrivial to verify in general and, in fact, fails in some cases. For instance, example given in Figure 1 of \cite{zhang2024gromov} indicates that  $\Phi_{(\mu_0,\mu_1)}$ is nonsmooth for simple discrete measures on the line. With this, it appears reasonable to impose the stronger condition that, for every $\mathbf{A}\in \mathbb R^{d_0\times d_1}$, $\OT_{\mathbf{A}}(\mu_0,\mu_1)$ admits a unique optimal plan. However, standard results guaranteeing uniqueness of OT plans generally require injectivity of $\nabla_x c(x,\cdot)$, known as the twist condition, (see e.g. Theorem 10.28 in \cite{villani2008optimal}) among other assumptions. As $\nabla_xc_{\mathbf{A}}(x,y)=-8x\|y\|^2-32\mathbf{A}y$, this condition can fail even in anodyne situations (e.g.
$0\in\inte(\supp(\mu_0))$ and there exists $y,y'\in\supp(\mu_1)\cap \ker(\mathbf{A})$). The failure of this condition constitutes a roadblock to proving the existence of Gromov-Monge alignment maps for \eqref{eq:GWPrimal}. Nevertheless, we demonstrate in \cref{sec:semidiscreteGW} that, in the semi-discrete case, a simple condition guarantees uniqueness of OT plans and existence of OT maps
for $\OT_{\mathbf{A}}(\mu_0,\mu_1)$. This enables us to prove existence of Gromov-Monge maps for the semi-discrete GW problem, which adds to the few known cases for which optimal alignment maps exist; 
\cite[Theorem 9.21]{sturm2012space} and \cite[Proposition 4.2.4]{vayer2020contribution}.
\end{remark}

The entropic GW distance between the Euclidean mm spaces $(\RR^{d_0},\|\cdot\|,\mu_0)$ and $(\RR^{d_1},\|\cdot\|,\mu_1)$ is obtained by regularizing \eqref{eq:GWPrimal} using the Kullback-Leibler divergence, analogously to EOT. Fixing $\varepsilon>0$, the distance between  $\mu_0\in\mathcal P(\mathbb R^{d_0}),\mu_1\in\mathcal P(\mathbb R^{d_1})$ with finite fourth moments is given by 
\begin{equation}
    \label{eq:entropicGWPrimal}
    \mathsf D_{\varepsilon}(\mu_0,\mu_1)^2=\inf_{\pi\in\Pi(\mu_0,\mu_1)}\left\{\|\Delta\|_{L^2(\pi\otimes \pi)}^2+\varepsilon\mathsf D_{\mathrm{KL}}(\pi\|\mu_0\otimes \mu_1)\right\}. 
\end{equation}
This modification was introduced in \cite{peyre2016gromov,solomon2016entropic} due to computational considerations, resulting in a heuristic iterative algorithm which involves solving an EOT problem at each iterate. %
More recently, a gradient-based algorithm for solving the entropic GW problem subject to non-asymptotic convergence guarantees was provided in \cite{rioux2023entropic} by leveraging \cref{thm:entropicVariationalMinimizers} ahead.
We underscore that $\mathsf D_{\varepsilon}(\mu_0,\mu_0)\neq 0$ unless $\mu_0$ is a Dirac measure\footnote{$\mathsf D_{\mathrm{KL}}(\pi\|\mu_0\otimes \mu_0)\geq 0$ with equality if and only if $\pi=\mu_0\otimes \mu_0$ whereas $\int\Delta d\pi\otimes\pi\geq 0$ with equality if and only if $\pi$ is induced by an isomorphism.} so that $\mathsf D_{\varepsilon}$ does not define a metric. Nevertheless, $\mathsf D_{\varepsilon}$ is still invariant to isometric transformations of the marginal spaces and always admits a solution.\footnote{The proof of Corollary 10.1 in \cite{Memoli11} shows that the functional $\varpi\in\mathcal P(\mathbb R^{d_0}\times \mathbb R^{d_1})\mapsto \int \Delta d\varpi\otimes \varpi$ is continuous in the topology of weak convergence of probability measures. The conclusion follows by noting that $\varpi\mapsto \mathsf{D}_{\mathrm{KL}}(\varpi\|\mu_0\otimes \mu_1)$ is lower semicontinuous and that $\Pi(\mu_0,\mu_1)$ is compact in this topology.} By analogy with \eqref{eq:GWVariational}, the entropic GW problem also admits a variational form (obtained by replacing $\mathsf{OT}_{\mathbf A}$ by  $\mathsf{OT}_{{\mathbf A},\varepsilon}\coloneqq \mathsf{OT}_{c_{\mathbf A},\varepsilon}$ in \eqref{eq:GWVariational}) and \cref{thm:VariationalMinimizers} carries over to the entropic case with the advantage that the measures need only be $4$-sub-Weibull and that the objective $\Phi_{(\mu_0,\mu_1)}$ is everywhere Fr{\'e}chet differentiable; see \cref{sec:entropicGWVar} for details.

\section{Discrete Gromov-Wasserstein Distance}
\label{sec:DiscreteGW}
We first consider the GW distance between finitely discrete measures $\mu_0\in\mathcal P(\mathbb R^{d_0})$ and $\mu_1\in\mathcal P(\mathbb R^{d_1})$. Throughout, we let $\mathcal X_0=(x^{(i)})_{i=1}^{N_0}=\supp(\mu_0)$ and $\mathcal X_1=(y^{(j)})_{j=1}^{N_1}=\supp(\mu_1)$. Our analysis hinges on the variational formulation of the GW distance and the fact that $\mathsf D(\mu_0,\mu_1)^2$ can be identified with the finite-dimensional quadratic program 
\begin{equation}
    \label{eq:discreteQP}
    \begin{aligned}
        \inf_{\mathbf{P}\in\mathbb R^{N_0\times N_1}}\; &\left\langle \mathbf{P},\Delta(\mathbf{P}) \right\rangle_{\text{F}},
        \\
        \text{s.t. }\;&\mathbf{P}\mathbbm 1_{N_1}=m_0,
        \\
        &\mathbbm 1_{N_0}^{\intercal}\mathbf{P}=m_1,
        \\
        &\hspace{1.2em}\mathbf{P}_{ij}\geq 0,\;\forall (i,j)\in[N_0]\times[N_1].
    \end{aligned}
\end{equation}
where $\Delta(\mathbf{P})\in\mathbb R^{N_0\times N_1}$ is the matrix with $ij$-th entry 
\[
(\Delta(\mathbf{P}))_{ij}=\sum_{(k,l)\in[N_0]\times [N_1]}\Delta\left(x^{(i)},y^{(j)},x^{(k)},y^{(l)}\right) \mathbf{P}_{kl},
\]
and $(m_0,m_1)\in\mathbb R^{N_0}\times \mathbb R^{N_1}$ are the weight vectors for $\mu_0$ and $\mu_1$, respectively, i.e., $(m_0)_{i}=\mu_0(\{x^{(i)}\})$ and $(m_1)_{j}=\mu_1(\{y^{(j)}\})$ for $(i,j)\in[N_0]\times[N_1]$. Evidently, any $\mathbf{P}\in\mathbb R^{d_0\times d_1}$ satisfying the constraints in \eqref{eq:discreteQP} can be identified with a coupling $\pi_{\mathbf{P}}\in\Pi(\mu_0,\mu_1)$ via $\pi_{\mathbf{P}}\left(\left\{(x^{(i)},y^{(j)})\right\}\right)=\mathbf{P}_{ij}$ and it is clear that $\langle \mathbf{P},\Delta(\mathbf{P})\rangle_{\text{F}}=\int \Delta d\pi_{\mathbf{P}}\otimes \pi_{\mathbf{P}}$.

\subsection{Stability analysis}\label{subsec:discrete_stability}

We explore the stability of the discrete GW functional along perturbations of the marginal distributions. Combining this with the methodology from \cite{goldfeld24statistical} will then result in distributional limits for empirical GW in the discrete setting. In order to preserve the structure of \eqref{eq:discreteQP} along the perturbations \[\mu_{0,t}\coloneqq\mu_0+t(\nu_0-\mu_0),\quad \mu_{1,t}\coloneqq\mu_1+t(\nu_1-\mu_1), \quad t\in [0,1]\] appearing in \cref{prop:unified} ahead, we require that $\nu_i\in\mathcal P(\mathcal X_i)$, for $i\in\{0,1\}$, is such that $\mathsf{D}(\mu_{0,t},\mu_{1,t})^2$ can be written as \eqref{eq:discreteQP}  with the corresponding weight vectors in place of $m_0,m_1$. For any $(\nu_0,\nu_1)\in\mathcal P(\mathcal X_0)\times\mathcal P(\mathcal X_1)$, we have $\frac 12 \sqrt{M_2(\nu_0)M_2(\nu_1)}\leq \frac 12 \|\mathcal X_0\|_{\infty}\|\mathcal X_1\|_{\infty}$,   
so that all minimizers of $\Phi_{(\nu_0,\nu_1)}$ are contained in $B_{\mathrm{F}}(M)$ for $M\coloneqq \frac{1}{2}\|\mathcal X_0\|_{\infty}\|\mathcal X_1\|_{\infty}$.
Moreover, $M_2(\bar\nu_i)=M_2(\nu_i)-\|\mathbb E_{\nu_i}[X]\|^2\leq M_2(\nu_i)$, for $i\in\{0,1\}$, so that the  minimizers of $\Phi_{(\bar \nu_0,\bar \nu_1)}$ are also contained in $B_{\mathrm{F}}(M)$.  

Stability properties of $\mathsf D(\mu_0,\mu_1)^2$ along valid perturbations of the marginals are established by studying a type of linearization of this quadratic programs which is related to $\OT_{\mathbf{A}}(\mu_0,\mu_1)$, as defined in \cref{sec:GWDistance}, for some choice of $\mathbf{A}\in\mathbb R^{d_0\times d_1}$. %
 To state the subsequent stability property, set 
    \begin{align*}
        \mathcal F_i&\coloneqq\left\{f :\mathcal X_i \to
        \mathbb R:\|f\|_{\infty,\mathcal X_i}\leq 1\right\},\quad i\in\{0,1\},
        \\
        \mathcal F^{\oplus}&\coloneqq \left\{f_0\oplus f_1:f_0\in\mathcal F_0,f_1\in\mathcal F_1\right\}.
    \end{align*}

\begin{theorem}[Stability for discrete GW]
    \label{thm:discreteGWStability}
    Let  
    $\mathcal A=\argmin_{B_{\mathrm{F}}(M)}\Phi_{(\bar\mu_0,\bar\mu_1)}$.
    Then, for any pairs $(\nu_0, \nu_1),(\rho_0, \rho_1)\in\mathcal P_{\mu_0}\times \mathcal P_{\mu_1}$, we have $\left|\mathsf D(\nu_0,\nu_1)^2-\mathsf D(\rho_0,\rho_1)^2\right|\leq C\|\nu_0\otimes \nu_1-\rho_0\otimes\rho_1\|_{\infty,\mathcal F^{\oplus}}$, for some constant $C>0$ which is independent of $(\nu_0, \nu_1),(\rho_0, \rho_1)$.  Furthermore, %
       
    \[
    \begin{aligned}
        &\frac{d}{dt}\mathsf D(\mu_{0,t},\mu_{1,t})^2\big\vert_{t=0}
        \\
        &\hspace{5em}=  \inf_{\mathbf{A}\in\mathcal A}\inf_{\pi \in \bar \Pi^{\star}_{\mathbf{A}}}\sup_{(\varphi_0,\varphi_1)\in\bar{\mathcal D}_{\mathbf A}}\left\{
             \int (f_0+ g_{0,\pi}+\bar \varphi_0)d(\nu_0-\mu_0)+\int (f_1+ g_{1,\pi}+\bar \varphi_1)d(\nu_1-\mu_1)
            \right\},
    \end{aligned}
    \]
    where, for $x\in\mathbb R^{d_0}$, %
    \[
            f_0(x)=2\int \|x-x'\|^4d\mu_0(x')-4M_2(\bar \mu_1)\|x-\mathbb E_{\mu_0}[X]\|^2,
            \quad g_{0,\pi}(x)= 8x^\intercal\mathbb E_{(X,Y)\sim \pi}[\|Y\|^2X],
    \]
    with $f_1,g_{1,\pi}$ defined symmetrically by interchanging $x\in\mathbb R^{d_0}$ and $y\in\mathbb R^{d_1}$, $X$ and $Y$, and $\mu_0$ and~$\mu_1$. Further,
    $\bar\varphi_i=\varphi_i(\cdot-\mathbb E_{\mu_i}[X])$ for $i\in\{0,1\}$, and, for $\mathbf{A}\in B_{\mathrm{F}}(M)$, $\bar \Pi^{\star}_{\mathbf{A}}\subset\Pi(\bar \mu_0,\bar \mu_1)$ is the set of all OT plans, $\pi^{\star}$, for $\OT_{\mathbf{A}}(\bar \mu_0,\bar \mu_1)$ satisfying $\mathbf{A}=\frac{1}{2}\int xy^{\intercal}d\pi^{\star}(x,y)$, while $\bar{\mathcal D}_{\mathbf{A}}$ is the set of all corresponding OT potentials.
\end{theorem}

The proof of \cref{thm:discreteGWStability} (see \cref{sec:proof:thm:discreteGWStability}) employs a linearization of the GW quadratic program \eqref{eq:discreteQP} along certain perturbations of its marginals. %
The approach builds on directional differentiability results for the optimal value function in nonlinear programming (see Section 4.3.2 in \cite{bonnans2013perturbation}) while utilizing stability properties of solutions to perturbed (finite-dimensional) quadratic programs \cite{Klatte1985Lipschitz}. We note that the results from \cite{bonnans2013perturbation} do not apply directly to the problem at hand as the directional regularity condition, used to enforce strong duality for the linearized problem, does not hold due to the {structure of the GW coupling set} (see \cref{rmk:directionalRegularityFailure}). Nevertheless, we show that the dual variables correspond to OT potentials for $\mathsf{OT}_{\mathbf A}(\bar\mu_0,\bar\mu_1)$, where $\mathbf A=\frac{1}{2}\int xy^{\intercal}d\pi(x,y)$ for some GW plan $\pi$, and use the fact that dual OT potentials can always be replaced with a bounded version thereof without changing the objective to restrict optimization to a bounded constraint set. %

\begin{remark}[Comparison with OT stability]
To compare GW stability to that of OT, let $(\mathcal X,\mathsf d_{\mathcal X})$ be a finite metric space and set $\mu_0,\mu_1\in\mathcal P(\mathcal X)$ and  $(\nu_0,\nu_1)\in \mathcal P_{\mu_0}\times \mathcal P_{\mu_1}$.
Adopting the notation of \cref{thm:discreteGWStability}, Theorem 3.1 in \cite{gal2012advances} yields 
\[
\frac{d}{dt}\mathsf{OT}_{\mathsf d^p_{\mathcal X}}(\mu_{0,t},\mu_{1,t})\big\vert_{t=0}=\sup_{(\varphi_0,\varphi_1)\in \mathcal D_{p}}\left\{\int \varphi_0d(\nu_0-\mu_0)+\int \varphi_1d(\nu_1-\mu_1)\right\},
\]
where $\mathcal D_p$ is the set of OT potentials for $\mathsf{OT}_{\mathsf d^p_{\mathcal X}}(\mu_0,\mu_1)$. This observation was used in \cite{Sommerfeld2018} to derive limit distributions for the $p$-Wasserstein distance on a finite space via the delta method. A similar term appears in the GW stability result (\cref{thm:discreteGWStability}), along with (i) a minimization over $\mathcal A$, (ii) a minimization over corresponding OT plans, and (iii) the terms $f_0$ and $f_1$ which depend only on $\mu_0$ and $\mu_1$. The first occurs due to the minimization over matrices in the variational form, the second is due to the centering required to access the variational form, while the third accounts for the term $\mathsf S_1$ in the variational form \eqref{eq:GWVariational}.  
\end{remark}

\subsection{Distributional limits}\label{subsec:discrete_limits}

\medskip
The characterization of the stability profile of the discrete GW problem with respect to perturbations of the weights given in \cref{thm:discreteGWStability} is the driving force behind the subsequent limit theorem once the following assumption is made.\footnote{\cref{assn:weakConvergence} will be used when stating limit laws beyond the discrete case, the function class $\mathcal F^{\oplus}$ should be understood based on the context.}

\begin{assumption}
   \label{assn:weakConvergence}  $(\mu_{0,n},\mu_{1,n})_{n\in\mathbb N}\subset\mathcal P_{\mu_0}\times \mathcal P_{\mu_1}$ is such that  $\sqrt n(\mu_{0,n}\otimes \mu_{1,n}-\mu_0\otimes \mu_1)\stackrel{d}{\to}\chi_{\mu_0\otimes \mu_1}$ in $\ell^{\infty}(\mathcal F^{\oplus})$ where $\chi_{\mu_0\otimes \mu_1}$ is tight.
\end{assumption}

The canonical statistical setting that fulfills \cref{assn:weakConvergence}  is when
$\hat \mu_{0,n}$ and $\hat \mu_{1,n}$ are the empirical measures of i.i.d. observations $X_{0,1},\dots, X_{0,n}$ from $\mu_0$ and $X_{1,1},\dots, X_{1,n}$ from $\mu_1$, where the two collections of  samples are independent.  

\begin{theorem}[Limit distributions for discrete GW]
    \label{thm:discreteGWLimitDistribution} In the setting of \cref{thm:discreteGWStability}, if $(\mu_{0,n},\mu_{1,n})_{n\in\NN}$ satisfy \cref{assn:weakConvergence}, then
    \[
        \begin{aligned}
        \sqrt n\left( \mathsf D(\mu_{0,n},\mu_{1,n})^2-\mathsf D(\mu_{0},\mu_{1})^2 \right)&
        \\
        &\hspace{-1.8em}
        \stackrel{d}{\to}
\inf_{\mathbf{A}\in\mathcal A}
            \inf_{\pi \in \bar \Pi^{\star}_{\mathbf{A}}} \sup_{(\varphi_0,\varphi_1)\in\bar {\mathcal D}_{\mathbf{A}}}\left\{\chi_{\mu_0\otimes \mu_1} \big( (f_0+g_{0,\pi}+\bar \varphi_0)\oplus (f_1+g_{1,\pi}+\bar\varphi_1)\big)
        \right\}.
    \end{aligned} 
    \]
    In particular, if  $(\mu_{0,n},\mu_{1,n})=(\hat \mu_{0,n},\hat \mu_{1,n})$, we have %
    $\chi_{\mu_0\otimes \mu_1}(g_0\oplus g_1)= G_{\mu_0}(g_0)+G_{\mu_0}(g_1)$ for any $g_0\oplus g_1\in\mathcal F^{\oplus}$, where $G_{\mu_i}$ a tight $\mu_i$-Brownian bridge in $\ell^{\infty}(\mathcal F_i)$ for $i\in\{0,1\}$. 
\end{theorem}

Given \cref{assn:weakConvergence} and \cref{thm:discreteGWStability}, \cref{thm:discreteGWLimitDistribution} is a direct consequence of the delta method-based framework developed in  \cite[Proposition 1]{goldfeld24statistical}, included herein for completeness as \cref{prop:unified} in \cref{sec:deltaMethod}). Complete details are provided in \cref{sec:proof:thm:discreteGWLimitDistribution}. 

\begin{remark}[Extension of limits]
\label{rmk:extensionBrownianBridge}
In \cref{thm:discreteGWLimitDistribution}, the distributional limit $\chi_{\mu_0\otimes \mu_1}$ is evaluated at $\xi_0\oplus \xi_1$ for $\xi_0\coloneqq f_0+g_{0,\pi}+\bar \varphi_0$ and $\xi_1\coloneqq f_1+g_{1,\pi}+\bar \varphi_1$, respectively, for some choice of OT plan and pair of OT potentials. Though $\xi_0$ and $\xi_1$ are not necessarily elements of $\mathcal F_0$ and $\mathcal F_1$, it is easy to see that $\|\xi_0\|_{\infty,\mathcal X_0}$ and $\|\xi_1\|_{\infty,\mathcal X_1}$ can be bounded by a constant $C$ which depends only on $\|\mathcal X_0\|_{\infty}$ and $\|\mathcal X_1\|_{\infty}$ (up to choosing the versions of the OT potentials which are bounded and noting that the direct sum is the same for all versions, see \cref{lem:discretePotentialsNull}). It follows that $(C^{-1}\xi_0,C^{-1}\xi_1)\in\mathcal F_0\times \mathcal F_1$ whereby $\sqrt n( \mu_{0,n}\otimes \mu_{1,n}-\mu_0\otimes \mu_1)(\xi_0\oplus \xi_1)=C\sqrt n( \mu_{0,n}\otimes \mu_{1,n}-\mu_0\otimes \mu_1)(C^{-1}\xi_0\oplus C^{-1} \xi_1)\stackrel{d}{\to}C\chi_{\mu_0\otimes \mu_1}(C^{-1}\xi_0\oplus C^{-1} \xi_1).$ We thus identify $\chi_{\mu_0\otimes \mu_1}(\xi_0\oplus \xi_1)$ with $C\chi_{\mu_0\otimes \mu_1}(C^{-1}\xi_0\oplus C^{-1} \xi_1)$.%
\end{remark}

\begin{remark}[Removing the square] 
\label{rmk:removingSquare}
Although the limit distribution in \cref{thm:discreteGWLimitDistribution} is presented for the squared distance $\mathsf{D}^2$, the asymptotics for $\mathsf{D}$ can be obtained by a second application of the delta method with the function $\sqrt{(\cdot)}$ under the assumption that $\mathsf D(\mu_0,\mu_1)\neq 0$.  
\end{remark}

\begin{remark}[Estimating the limit] 
\label{rmk:estimatingLimit}
  When the derivative established in \cref{thm:discreteGWStability} fails to be linear as a function of $(\nu_0-\mu_0,\nu_1-\mu_1)$, the limit distribution in \cref{thm:discreteGWLimitDistribution} cannot be estimated via the na\"ive bootstrap \cite{dumbgen1993nondifferentiable}. To guarantee that the derivative be linear, it would suffice to show that $\mathcal A=\{\mathbf A^{\star}\}$ is a singleton, that  the OT plan for $\mathsf{OT}_{\mathbf A^{\star}}(\bar\mu_0,\bar\mu_1)$ is unique, and that the corresponding OT potentials are unique (up to versions). In this discrete setting, these conditions imply that $\mu_0$ and $\mu_1$ are Dirac delta distributions, rendering the limit trivial. We overcome this predicament by developing a direct and efficient method for estimating the limit under the null using linear programming, described in \cref{sec:directEstimation} ahead.
\end{remark}

Under the null hypothesis $\mathsf D(\mu_0,\mu_1)=0$ (i.e. $\mu_0=\mu_1$ up to isometries), the obtained limit distribution simplifies as follows.

\begin{corollary}[Limit distribution under the null]
\label{cor:discreteGWLimitDistributionNull}
   In the setting of \cref{thm:discreteGWLimitDistribution}, assume that $\mu_0$ is not a point mass and that $\mathsf D(\mu_0,\mu_1)=0$. If $\bar{\mathcal T}$ is the set of all GW plans between $\bar \mu_0$ and $\bar \mu_1$, then  %
   \[
   \begin{gathered}
        \sqrt n \mathsf D(\mu_{0,n},\mu_{1,n})^2
        \stackrel{d}{\to}
\inf_{T\in\bar {\mathcal T}}\sup_{h\in{\mathcal H}}\left\{\chi_{\mu_0\otimes \mu_1}\big( h(\cdot-\mathbb E_{\mu_0}[X])\oplus -h( T^{-1}(\cdot-\mathbb E_{\mu_0}[X]))\big)\right\}\eqqcolon L_{\mu_0},
\\
    \mathcal H=\left\{h:\supp(\bar\mu_0)\to \mathbb R\:|\: h(x)-h(x')\leq 2(\|x\|^2-\|x'\|^2)^2+8(x-x')^{\intercal}\bsigma_{\bar \mu_0}(x-x
    '),\forall x,x'\in \supp(\bar \mu_0)\right\}.
   \end{gathered} 
    \] 
    If $(\mu_{0,n},\mu_{1,n})=(\hat \mu_{0,n},\hat \mu_{1,n})$, the above result holds and $L_{\mu_0}= \sqrt 2\|G_{\mu_0}\|_{\infty,\bar{\mathcal H}}$, where $\bar {\mathcal H}=\{ h(\cdot-\mathbb E_{\mu_0}[X]):h\in\mathcal H\}$.%
\end{corollary}

\cref{cor:discreteGWLimitDistributionNull} follows from \cref{thm:discreteGWLimitDistribution} upon characterizing the OT potentials under the null and simplifying the resulting expressions. Note that if $h\in\mathcal H$, so too is $h-a$, and the value of the direct sum in the limit distribution is invariant to translation. Thus, the direct sum can always be chosen such that it lies in $\mathcal F^{\oplus}$ up to a dilation by e.g. setting $h(x)=0$ for some $x\in\supp(\bar\mu_0)$.  %
See \cref{sec:proof:cor:discreteGWLimitDistributionNull} for complete details.

\subsection{Direct estimation of the limit}
\label{sec:directEstimation}

As the distribution $L_{\mu_0}$ in \cref{cor:discreteGWLimitDistributionNull} admits a reasonably simple expression in the case of empirical measures, we may consider estimating it directly as an alternative to the bootstrap. To this end, let $(\bar x^{(i)})_{i=1}^{N_n}=\supp(\bar{\hat \mu}_{0,n})$ and set  
\[
{\mathcal H}_n=\big\{u\in\mathbb R^{N_n}\:|\: u_i-u_j\leq 2(\|\bar x^{(i)}\|^2-\|\bar x^{(j)}\|^2)^2
    +8(\bar x^{(i)}-\bar x^{(j)})^{\intercal}\bsigma_{\bar{\hat \mu}_{0,n}}(\bar x^{(i)}-\bar x^{(j)}
    ),%
    i\neq j\in[N_{n}]\big\}.
\]
Recalling that $m_0$ denotes the weight vector for $\mu_0$, it is convenient to set $u_f=(f(x^{(1)}),\dots, f(x^{(N_0)}))$ for any $f\in\mathcal F_0$ so that $\mu_0(f)=m_0^{\intercal}u_f$. With this,  
$G_{\mu_0}(f)=Z^{\intercal}u_f$ for every $f\in\mathcal F_0$ and $Z\sim N(0,\bsigma_{m_0})$, where, for a vector $r\in\mathbb R^{N}$, $\bsigma_{r}\in\mathbb R^{N\times N}$ is the covariance matrix of the multinomial distribution with $1$ trial and probability vector $r$, whose entries are given by
\begin{equation}
\label{eq:multinomialCovariance}
    \left(\bsigma_{r}\right)_{ij}=\begin{cases}
       r_i(1-r_i),&{\text{if } i=j},
       \\
       -r_ir_j,&{\text{if } i\neq j},
    \end{cases}
\end{equation}
Taking $Z_n\sim N(0,\bsigma_{m_n})$ and writing $m_n$ for the weight vector associated with  $\hat \mu_{0,n}$, the above suggest
estimating $L_{\mu_0}$ as 
\begin{equation}
\begin{gathered}
    L_n =\sqrt{2} \sup_{u\in\mathcal H_n} Z_n^{\intercal},%
\end{gathered}
\end{equation}
\cref{alg:directEstimator} summarizes the procedure for sampling from $L_n$.  It is worth mentioning that, as soon as $ \hat \mu_{0,n}$ is supported on more than one point, ${\mathcal H}_n$ has non-empty interior (see \cref{lem:Hnonempty}), but is unbounded. However, the samples $Z_n$ are almost surely orthogonal to the vector of ones,  %
so that each instance of the linear program admits a bounded solution.\footnote{In practice, it is recommended to add a box constraint on the first variable, $-\delta\leq h(x^{(1)})\leq \delta$ for any $\delta>0$, even if the entries of $Z_n$ sum to one with probability one  to account for  numerical error. This modification results in a bounded feasible region, but does not effect the optimal value obtained, see the proof of \cref{thm:directEstimator}.}

\begin{algorithm}[!t]
\caption{Sampling from $ L_n$}\label{alg:directEstimator}
\begin{algorithmic}[1]
\Statex Given samples $X_{1},\dots, X_n$ from $\mu_0$, construct $\hat \mu_{0,n}$,  
\State Let $(\bar x^{(i)})_{i=1}^{N_{n}}=\supp(\bar \mu_{0,n})$.   
\State Let  $b_n\in\mathbb R^{N_n(N_n-1)}$ and $\mathbf{A}\in\mathbb R^{N_n(N_n-1)\times N_n}$ be such that ${\mathcal H}_n=\{u\in\mathbb R^{N_n}:\,\mathbf{A}u\leq  b_n\} $.
\State Sample $Z_n\sim N(0,\bsigma_{m_n})$ and solve $\sqrt 2\sup_{u\in{\mathcal H}_n}Z_n^{\intercal}u\sim  L_n$; repeat Step $3$ to generate new~samples. 
\end{algorithmic}
\end{algorithm}

\medskip
We now establish that this procedure is consistent for estimating the distribution of $L_{\mu_0}$.

\begin{theorem}
\label{thm:directEstimator}
In the setting of \cref{cor:discreteGWLimitDistributionNull} and \cref{assn:simplifiedForm}, if $\mu_0$ is not a point mass, then %
\[
\lim_{n\to \infty}\sup_{t\geq 0}\left|\mathbb P(L_n\leq t|X_1,X_2,\dots,X_n)-\mathbb P( L_{\mu_0}\leq t)\right|= 0,
\]
for almost every realization of $X_1,X_2,\dots$  %
\end{theorem}

The proof of \cref{thm:directEstimator}, included in \cref{sec:proof:thm:directEstimator}, %
leverages the stability properties of the optimal value function for linear programs when viewed as a function of the cost and constraint vectors  \cite{Gisbert2019}. Using this result we show that, conditionally on $X_1,X_2,\dots$, we have \smash{$L_n\stackrel{d}{\to}L_{\mu_0}$} for almost every realization of the data. The result is established by noting that $L_{\mu_0}$ has a continuous distribution function. We highlight that this result holds beyond the setting of empirical measure, i.e., when other weakly convergent estimators of the population distributions are employed. Indeed, the proof is carried out under general conditions which encompass both the empirical measure and the estimator used for testing graph isomorphisms, introduced in \cref{sec:graphApplication}.  %

We find it noteworthy that while $L_{\mu_0}$ is the weak limit of a sequence of optimal values for random nonconvex quadratic programs (NP-hard), \cref{thm:directEstimator} shows that it can be approximated by solving linear programs. By contrast, the na\"ive bootstrap estimator may fail to be consistent as the directional derivative provided in \cref{thm:discreteGWLimitDistribution} is nonlinear (see Proposition 1 in \cite{dumbgen1993nondifferentiable} or Corollary 3.1 in \cite{fang2019}). The $m$-out-of-$n$ bootstrap \cite{politis1994large,bickel2012resampling} may prove consistent, but its computational overhead and the required tuning of $m$ render it computationally costly and unattractive compared to the direct estimator. 

\begin{remark}[Computational considerations]
Estimating the distribution of $L_n$ given $X_1,\dots, X_n$ requires solving multiple instances of a linear program with a fixed feasible set, but with varying cost vectors.  Software for numerical resolution of linear programs, both open-source and commercial, are plentiful (see e.g. \cite{sandia,highs,clarabel} for benchmarks of linear programming solvers) and are based on  variants of the interior-point method \cite{karmarkar1984new}, which has polynomial time complexity and the simplex method \cite{dantzig1951maximization}, which has exponential worst case complexity, but performs as well or better than the interior-point method on most problems.%

\end{remark}

\section{Semi-Discrete Gromov-Wasserstein Distance} 
\label{sec:semidiscreteGW}

In this section, we consider the semi-discrete setting for the GW distance, where $\mu_0\in\mathcal P(\mathbb R^{d_0})$ is supported in $\mathcal X_0$, an open ball centered at $0$ with finite radius, and $\mu_1\in\mathcal P(\mathbb R^{d_1})$ is supported on the points $\mathcal X_1\coloneqq\left(y^{(i)}\right)_{i=1}^{N}$. To establish limit distributions for the semi-discrete GW cost, we again require stability of the solution set. However, unlike the discrete case, we can no longer rely on an analysis of finite-dimensional quadratic programs. Rather, we prove that, under mild conditions, the semi-discrete $(2,2)$-GW problem is solved by a Gromov-Monge map---a result of independent interest. This, %
 in~turn, allows us to deduce the desired stability. %

To arrive at the main structural results concerning the existence of  Gromov-Monge maps, we start by specializing some  results for semi-discrete OT to the family of costs $c_{\mathbf{A}}$ figuring in the variational form of the GW distance \eqref{eq:GWVariational}. To that end, recall the so-called semi-dual form of the semi-discrete OT problem,\footnote{This representation is obtained by recasting the dual OT problem in \eqref{eq:OTDuality} as 
$\sup_{\substack{\varphi_1\in \mathcal C(\mathcal X_1)}} \left\{\int \varphi_1^c d\mu_0+\int \varphi_1 d\mu_1\right\}$, where $\varphi_1^c:x\in\mathcal X_0\mapsto \inf_{y\in\mathcal X_1}\left\{c(x,y)-\varphi_1(y)\right\}$ is the $c$-transform of $\varphi_1$. In the semi-discrete case, one then identifies the OT potential $\varphi_1$ with a vector $z\in\mathbb R^{N}$ to arrive at \eqref{eq:SemidiscreteDual}.}
\begin{equation}
    \label{eq:SemidiscreteDual}
    \OT_{c}(\mu_0,\mu_1)=\sup_{z\in\mathbb R^N}\left\{\sum_{i=1}^Nz_i\mu_1\big(\big\{y^{(i)}\big\}\big)+\int\min_{1\leq i\leq N}\left\{ c\big(\cdot,y^{(i)}\big)-z_i \right\} d\mu_0  \right\}.
\end{equation}
We call a solution to \eqref{eq:SemidiscreteDual} an \emph{optimal vector} for $\OT_{c}(\mu_0,\mu_1)$. %

\subsection{Existence of Gromov-Monge maps}
\label{sec:semidiscreteStructural}

A standard approach for establishing the existence of OT maps for $\mathsf{OT}_c(\mu_0,\mu_1)$ employs the so-called twist condition, which guarantees injectivity of $\nabla_x c(x,\cdot)$  \cite[Theorem 10.28]{villani2008optimal}. We leverage 
a version of the twist condition for semi-discrete OT with a generic cost from \cite{kitagawa2019convergence} to identify a primitive condition for $\mathsf{OT}_{\mathbf A}(\mu_0,\mu_1)$ to admit a Monge map, uniformly in $\mathbf A$. Once this proviso is met,  \cref{thm:VariationalMinimizers} guarantees the existence of a Gromov-Monge map for the semi-discrete GW problem.

\begin{theorem}[Existence of Gromov-Monge maps]
    \label{thm:existenceMongeMap} Assume without loss of generality that $(\mu_0,\mu_1)$ are centered and let $\mathbf{A}^{\star}$ minimize $\Phi_{(\mu_0,\mu_1)}$. If $\mathcal X_1$ is such that $y^{(i)}-y^{(j)}\not\in\ker(\mathbf{A}^{\star})$ for every $i\neq j\in[N]$ with $\|y^{(i)}\|=\|y^{(j)}\|$, %
     then there exists a GW plan for $\mathsf D(\mu_0,\mu_1)$ which is induced by the $\mu_0$-a.e. unique map 
     \[
       T_{\mathbf{A}^{\star}}:x\in\mathcal X\mapsto y^{(I_{z^{\mathbf{A}^{\star}}}(x))},\text{ where } I_{z^{\mathbf{A}^{\star}}}(x)\in\argmin_{1\leq i \leq N}\left( c_{\mathbf{A}^{\star}}(x,y^{(i)})-z^{\mathbf{A}^{\star}}_i \right),
    \]
    where $z^{\mathbf{A}^{\star}}_i$ is an optimal vector for $\mathsf{OT}_{\mathbf A^{\star}}(\mu_0,\mu_1)$.    %
     Furthermore, 
    $   \mathbf{A}^{\star}=\frac{1}{2}\sum_{i=1}^N\int_{\Lag_{z^{\mathbf{A}^{\star}},i}}xd\mu_0(x)\left(y^{(i)}\right)^{\intercal}$, where $\Lag_{z^{\mathbf{A}^{\star}},i}\coloneqq\{x\in\mathcal X_0:i\in \argmin_{1\leq i \leq N}\left( c_{\mathbf{A}^{\star}}(x,y^{(i)})-z^{\mathbf{A}^{\star}}_i\right)\}$. A sufficient condition for the  assumption on $\mathcal X_1$ to holds at every $\mathbf A\in\mathbb R^{d_0\times d_1}$
    is that $\|y^{(i)}\|\neq \|y^{(j)}\|$ for every $i\neq j\in[N]$. 
    \end{theorem}

The proof of \cref{thm:existenceMongeMap}, given in \cref{sec:proof:thm:existenceMongeMap}, shows that the condition on the support of $\mu_1$ is equivalent to a variant of the standard twist condition. This, in turn, guarantees that the corresponding OT problem admits a unique OT plan that is induced by a map and the result follows using the dual form \eqref{eq:GWVariational}. We emphasize that, for given $\mu_0,\mu_1$, it is easy to verify the sufficient condition that $\|y^{(i)}\|\neq \|y^{(j)}\|$ for every $i\neq j\in[N]$, whereas the general condition may be difficult to check as $\argmin_{\mathbb R^d}\Phi_{(\mu_0,\mu_1)}$ is \emph{a priori} unknown.%

\begin{remark}[Semi-discrete Gromov-Monge map]
Existence (and optimality) of Gromov-Monge maps for the GW problem \eqref{eq:GWPrimal} is known only in a few cases. Theorem 4.1.2 in \cite{vayer2020contribution} shows that it suffices to optimize over permutations when $\mu_0,\mu_1$ are both uniform probability measures supported on the same (finite) number of points, while Theorem 3.2 in
\cite{dumont2022existence} treats the inner product GW problem (i.e., when the cost function in \eqref{eq:GWPrimal} is set to $\Delta(x,y,x',y')=\langle x,x' \rangle-\langle y,y' \rangle$) between compactly supported and absolutely continuous distributions. The only results for the quadratic cost considered herein require additional high-level assumptions, such as that $\mu_0$ and $\mu_1$ be absolutely continuous and rotationally invariant about their barycenter \cite[Theorem 9.21]{sturm2012space} or \emph{a priori} knowledge that the cross-correlation matrix of some GW plan for $\mathsf D(\mu_0,\mu_1)$ is of full rank \cite[Proposition 4.2.4]{vayer2020contribution}. While such conditions are restrictive and hard to verify in practice, they may hold beyond the semi-discrete setting. \cref{thm:existenceMongeMap} presents a new instance of the existence of Gromov-Monge maps for the quadratic GW problem, assuming the semi-discrete setting with the primitive geometric condition on the magnitude of the discrete support points. 
\end{remark}

\subsection{Stability and limit distributions}

We now derive the asymptotic distribution of the empirical semi-discrete GW distance. To that end, we first study stability of $\mathsf D$ by leveraging the fact that the same condition guaranteeing existence of Gromov-Monge maps for the semi-discrete case, also imply the uniqueness of the OT plan for $\mathsf{OT}_{\mathbf A}(\bar\mu_0,\bar\mu_1)$ for every $\mathbf A\in \argmin_{\mathbb R^{d_0\times d_1}}\Phi_{(\mu_0,\mu_1)}$. As in the fully discrete case, we set $M:=\frac{1}{2}\|\mathcal X_0\|_{\infty}\|\mathcal X_1\|_{\infty}$ 
so that, for any $(\nu_0,\nu_1)\in\mathcal P(\mathcal X_0)\times\mathcal P(\mathcal X_1)$, all minimizers of $\Phi_{(\bar{\nu}_0,\bar{\nu}_1)}$ are contained in $B_{\mathrm{F}}(M)$

{To apply the unified approach described in} \cref{prop:unified}, we define the function class  
\[
\begin{aligned}
    \mathcal F^{\oplus} &\coloneqq \left\{f_0 \oplus f_1 : (f_0,f_1)\in\mathcal F_0\times \mathcal F_1\right\},
    \text{ for }
    \mathcal F_0 = B_{\mathcal C^k(\mathcal X_0)} + \mathcal G\cup\{0\},\\
    \mathcal F_1 &\coloneqq \left\{ f:\mathcal X_1^{\circ}\to \mathbb R : \|f\|_{\infty,\mathcal X_1}\leq 1
    \right\},
\end{aligned}
\]
where $k=\lfloor d/2\rfloor+1$, $B_{\mathcal C^k(\mathcal X_0)}$ denotes the unit ball in $\mathcal C^k(\mathcal X_0)$, and      
\[
\begin{aligned}
    &\mathcal G = \Big\{x\in\mathcal X_0\mapsto \min_{1\leq i\leq N}\left\{c_{\mathbf{A}}\big(x-\xi,y^{(i)}-\zeta\big)-z_i\right\}\left.\right.\\&\hspace{10em}\left.\right.:\mathbf{A}\in B_{\mathrm F}(M),z\in\mathbb R^N,\|z\|_{\infty}\leq K,\zeta\in\mathbb R^{d_1}, \|\zeta\|\leq \|\mathcal X_1\|_{\infty},\xi\in\mathcal X_0\Big\}, 
\end{aligned}
\]
These definitions are motivated by the regularity properties of the OT potentials (see \cref{prop:semidiscretePotentialProperties}).

\begin{theorem}[Stability for semi-discrete GW]
    \label{thm:semidiscreteGWStability}
    Let  
    $\mathcal A=\argmin_{B_{\mathrm{F}}(M)}\Phi_{(\bar\mu_0,\bar\mu_1)}$.
Then, for any pairs $(\nu_0, \nu_1),(\rho_0, \rho_1)\in\mathcal P_{\mu_0}\times \mathcal P_{\mu_1}$, $\left|\mathsf D(\nu_0,\nu_1)^2-\mathsf D(\rho_0,\rho_1)^2\right|\leq C\|\nu_0\otimes \nu_1-\rho_0\otimes\rho_1\|_{\infty,\mathcal F^{\oplus}}$ for some constant $C>0$ which is independent of $(\nu_0, \nu_1),(\rho_0, \rho_1)$, and, if $\bar \mu_0,\bar\mu_1$ verify the conditions of \cref{thm:existenceMongeMap} %
at every $\mathbf{A}\in\mathcal A$ and $\inte(\supp(\bar\mu_0))$ is connected with negligible boundary, we further have 
    \[
    \begin{aligned}
        &\frac{d}{dt}\mathsf D\big(\mu_{0,t},\mu_{1,t}\big)^2\big\vert_{t=0}
        = \inf_{\mathbf{A}\in\mathcal A}\left\{
             \int (f_0+ g_{0,\pi^{\mathbf{A}}}+\bar{\varphi}_0^{\mathbf{A}})d(\nu_0-\mu_0)+\int (f_1+g_{1,\pi^{\mathbf{A}}}+\bar{\varphi}_1^{\mathbf{A}})d(\nu_1-\mu_1)
            \right\},
    \end{aligned}
    \]  
       where $f_0$, $f_1$, $g_{0,\pi}$,  $g_{1,\pi}$, $\mu_{0,t}$, and $\mu_{1,t}$ are as defined in \cref{thm:discreteGWStability}, 
$\bar\varphi_i^{\mathbf{A}}=\varphi_i^{\mathbf{A}}(\cdot-\mathbb E_{\mu_i}[X])$ for $i=0,1$, and, for $\mathbf{A}\in \mathcal A$, $\pi^{\mathbf{A}}$ is the unique OT plan for $\OT_{\mathbf{A}}(\bar \mu_0,\bar \mu_1)$, and $(\varphi_0^{\mathbf{A}},\varphi_1^{\mathbf{A}})$ is any choice of corresponding OT potentials.
\end{theorem}

The proof of \cref{thm:semidiscreteGWStability} 
relies on %
the decomposition $\mathsf D(\bar\nu_0,\bar\nu_1)^2=\mathsf S_1(\bar\nu_0,\bar\nu_1)+\mathsf S_2(\bar\nu_0,\bar\nu_1)$. %
The desired right differentiability and Lipschitz continuity of $\mathsf S_1$ is relatively straightforward, whereas that of $\mathsf S_2$ requires a more careful analysis. %
This is since $\mathsf S_2$ includes an infimum, and the fact that $\bar{\mu}_{0,t}$ and $\bar{\mu}_{1,t}$ may not share the same support as $\bar \mu_0$ and $\bar \mu_1$, respectively. %
This complicates the comparison of OT potentials corresponding to different marginals. We address this issue by first computing the right derivative of $\mathsf{OT}_{\mathbf{A}}(\bar \nu_0,\bar \nu_1)$ at $(\bar \mu_0,\bar \mu_1)$ for a fixed $\mathbf{A}$, and then use the fact that $\mathbf{A}\in B_{\mathrm{F}}(M)\mapsto \mathsf{OT}_{\mathbf{A}}(\bar\mu_{0,t},\bar\mu_{1,t})$ converges uniformly to $\mathbf{A}\in B_{\mathrm{F}}(M)\mapsto \mathsf{OT}_{\mathbf{A}}(\bar\mu_0,\bar\mu_1)$, which enables evaluating the derivative of the infimum. 
Complete details are included in \cref{sec:proof:thm:semidiscreteGWStability}.

Compared to \cref{thm:discreteGWStability}, \cref{thm:semidiscreteGWStability} includes neither an infimum over OT couplings nor a supremum over OT potentials. This distinction is a consequence of the uniqueness of OT couplings for $\mathsf{OT}_{\mathbf A}(\bar \mu_0,\bar\mu_1)$ for $\mathbf A\in\mathcal A$ under the conditions of \cref{thm:existenceMongeMap} and the uniqueness (up to versions) of the OT potentials resulting from the assumption on $\supp(\bar \mu_0)$, see \cref{prop:semidiscretePotentialProperties}. This latter condition is standard in the literature \cite{delbarrio2021,del2024central}.

 \medskip
Armed with the stability result, we obtain the distributional limit for semi-discrete GW by invoking the unified approach from \cref{prop:unified}. 

\begin{theorem}[Limit distributions for semi-discrete GW]
    \label{thm:semidiscreteGWLimitDistribution} In the setting of \cref{thm:semidiscreteGWStability}, let $(\mu_{0,n},\mu_{1,n})$ satisfy \cref{assn:weakConvergence}. Then,
    \[
        \begin{aligned}
        \sqrt n\left( \mathsf D(\mu_{0,n},\mu_{1,n})^2-\mathsf D(\mu_{0},\mu_{1})^2 \right)            \stackrel{d}{\to}
\inf_{\mathbf{A}\in\mathcal A}
            \big\{\chi_{\mu_0\otimes \mu_1} \big(( f_0+g_{0,\pi}+\bar\varphi_0^{\mathbf{A}})\oplus (f_1+g_{1,\pi}+\bar\varphi_1^{\mathbf{A}})\big)\big\}.
    \end{aligned} 
    \]
    If $(\mu_{0,n},\mu_{1,n})=(\hat \mu_{0,n},\hat \mu_{1,n})$,  \cref{assn:weakConvergence} is met and $\chi_{\mu_0\otimes \mu_1}(g_0\oplus g_1)= G_{\mu_0}(g_0)+G_{\mu_0}(g_1)$ for any $g_0\oplus g_1\in\mathcal F^{\oplus}$, where $G_{\mu_i}$ a tight $\mu_i$-Brownian bridge in $\ell^{\infty}(\mathcal F_i)$ for $i\in\{0,1\}$.
\end{theorem}

As in the discrete case, we slightly abuse notation by identifying $G_{\mu_0}$ and $G_{\mu_1}$ with their extensions to a certain dilation of $\mathcal F_0$ and $\mathcal F_1$, see \cref{rmk:extensionBrownianBridge} for details. The approach outlined in \cref{rmk:removingSquare} can also be applied in this case to obtain a limit distribution for $\mathsf D$ (without the square).  Finally, as discussed in \cref{rmk:estimatingLimit}, the consistency of the na\"ive bootstrap for the limit in \cref{thm:semidiscreteGWLimitDistribution} can be established if $\mathcal A$ is a singleton. At present, sufficient conditions for such a result to hold are unknown.%

\section{Entropic Gromov-Wasserstein Distance}
\label{sec:entropicGW}

To go beyond the fully discrete and semi-discrete settings, we consider the entropic GW distance under the sole assumption that $\mu_0,\mu_1$ are $4$-sub-Weibull with parameter~$\sigma^2$. Note that if $(\mu_0,\mu_1)\in\mathcal P_{4,\sigma}(\mathbb R^{d_0})\times P_{4,\sigma}(\mathbb R^{d_1})$, then $\bar\mu_0$ and $\bar \mu_1$ are $4$-sub-Weibull with parameter $\bar\sigma^2\coloneqq 8\frac{2+\log 2}{\log 2}\sigma^2$ and $\frac{1}{2}\sqrt{M_2(\bar \mu_0)M_2(\bar \mu_1)}\leq \sigma$, these auxiliary results are proved in \cref{proof:thm:entropicGWStability}. We thus set $M=\sigma$ so that $B_{\mathrm F}(M)$ contains all minimizers of $\Phi_{(\bar\mu_0,\bar\mu_1)}$ (see \cref{thm:entropicVariationalMinimizers}). The derivation is similar to the semi-discrete case, as the EOT problem that appears in the entropic GW variational form (see \cref{sec:entropicGWVar}) enjoys uniqueness of EOT plans and potentials. %

\medskip
 The sub-Weibull assumption stems from the regularity theory for EOT potentials from Lemma 4 of \cite{zhang2024gromov} (restated in \cref{prop:entropicGWPotentials} ahead), which motivates the choice of the function class as $\mathcal F^{\oplus}\coloneqq\left\{f_0\oplus f_1:(f_0,f_1)\in\mathcal F_{0}\times \mathcal F_{1}\right\}$, where
 \[
\mathcal F_{i}=\left\{f\in\mathcal C^{\bar k}(\mathbb R^{d_i}):|f|\leq 1+\|\cdot\|^4,|D^{\alpha}f|\leq 1+\|\cdot\|^{3\bar k},\forall \alpha \in\mathbb N_0^d\text{ with }0<|\alpha|\leq  \bar k  \right\},
\]
for $\bar k=\lfloor (d_0\vee d_1)/2\rfloor +1$. We have the following stability theorem.

\begin{theorem}[Stability for entropic GW]
    \label{thm:entropicGWStability}
    Let  
    $\mathcal A=\argmin_{B_{\mathrm{F}}(M)}\Phi_{(\bar\mu_0,\bar\mu_1)}$.
   Then, for any pairs $(\nu_0, \nu_1),(\rho_0, \rho_1)\in\mathcal P_{4,\sigma}(\mathbb R^{d_0})\times \mathcal P_{4,\sigma}(\mathbb R^{d_1})$, $\left|\mathsf D_{\varepsilon}(\nu_0,\nu_1)^2-\mathsf D_{\varepsilon}(\rho_0,\rho_1)^2\right|\leq C\|\nu_0\otimes \nu_1-\rho_0\otimes\rho_1\|_{\infty,\mathcal F^{\oplus}}$ for some constant $C>0$ which is independent of $(\nu_0, \nu_1),(\rho_0, \rho_1)$, and,     
    \[
    \begin{aligned}
        &
        \frac{d}{dt}\mathsf D_{\varepsilon}\big(\mu_{0,t},\mu_{1,t}\big)^2\big\vert_{t=0}
        &= \inf_{\mathbf{A}\in\mathcal A}\left\{
             \int(f_0+ g_{0,\pi^{\mathbf{A}}}+\bar{\varphi}_0^{\mathbf{A}})d(\nu_0-\mu_0)+\int (f_1+g_{1,\pi^{\mathbf{A}}}+\bar{\varphi}_1^{\mathbf{A}})d(\nu_1-\mu_1)
            \right\},
    \end{aligned}
    \]  
           where $f_i$, $g_{i,\pi}$, $\mu_{i,t}$ $\bar\varphi_i^{\mathbf{A}}$, for $i=0,1$, as well as $\pi^{\mathbf{A}}$ are all defined as in \cref{thm:semidiscreteGWStability} with $\mathsf{OT}_{\mathbf A,\varepsilon}(\bar \mu_0,\bar\mu_1)$ in place of $\mathsf{OT}_{\mathbf A}(\bar \mu_0,\bar\mu_1)$. 
\end{theorem}

The proof of \cref{thm:entropicGWStability} follows similar lines as that of \cref{thm:semidiscreteGWStability} with the important distinction that the measures need only be $4$-sub-Weibull. The regularity of the EOT potentials %
and the strong structural properties of solutions to the EOT problem  enable treating measures with unbounded support. See \cref{proof:thm:entropicGWStability} for complete details. 

\medskip
Invoking \cref{prop:unified} now leads to the following characterization of distributional limits for the entropic GW  cost and bootstrap consistency when $\mathcal{A}$ is a singleton. For $X_{0,1}^B,\dots, X_{0,n}^B$ and $X_{1,1}^B,\dots, X_{1,n}^B$  independent collections of i.i.d. samples from $\hat \mu_{0,n}$ and $\hat \mu_{1,n}$, let $\hat\mu_{k,n}^B=\frac{1}{n}\sum_{i=1}^n\delta_{X_{k,i}^B}$ for $k\in\{0,1\}$ denote the bootstrap measures and  $\mathbb P^B$ the conditional probability given the data.

\begin{theorem}[Limit distributions for entropic GW and bootstrap consistency]
    \label{thm:entropicGWLimitDistribution} In the setting of \cref{thm:entropicGWStability}, let $(\mu_{0,n},\mu_{1,n})_{n\in\mathbb N}$ satisfy \cref{assn:weakConvergence} with $\mathcal P_{4,\sigma}(\mathbb R^{d_0})\times \mathcal P_{4,\sigma}(\mathbb R^{d_1})$ in place of $\mathcal P_{\mu_0}\times \mathcal P_{\mu_1}$, then
    \[
        \begin{aligned}
        \sqrt n\left( \mathsf D_{\varepsilon}(\mu_{0,n},\mu_{1,n})^2-\mathsf D_{\varepsilon}(\mu_{0},\mu_{1})^2 \right)            \stackrel{d}{\to}
\inf_{\mathbf{A}\in\mathcal A}
            \big\{\chi_{\mu_0\otimes \mu_1} \big((f_0+g_{0,\pi}+\bar\varphi_0^{\mathbf{A}})\oplus (f_1+g_{1,\pi}+\bar\varphi_1^{\mathbf{A}})\big)\big\}.
    \end{aligned} 
    \]
    Moreover, if $\varepsilon>16\sqrt{M_4(\bar\mu_0)M_4(\bar\mu_1)}$,  $\mathcal A$ is  a singleton and the limit is a point evaluation. 

    If $(\mu_{0,n},\mu_{1,n})=(\hat \mu_{0,n},\hat \mu_{1,n})$,  then \cref{assn:weakConvergence} is met and $\chi_{\mu_0\otimes \mu_1}(g_0\oplus g_1)= G_{\mu_0}(g_0)+G_{\mu_0}(g_1)$ for any $g_0\oplus g_1\in\mathcal F^{\oplus}$, where $G_{\mu_i}$ a tight $\mu_i$-Brownian bridge in $\ell^{\infty}(\mathcal F_i)$ for $i\in\{0,1\}$.  In this case, if $\mathcal A=\{\mathbf A\}$ is a singleton, then the limit is normal and the bootstrap is consistent 
    \[
       \sup_{t\in\mathbb R}\left|\mathbb P^B\left(\sqrt n\left(\mathsf D(\hat \mu_{0,n}^B,\hat \mu_{1,n}^B)^2-\mathsf D(\hat \mu_{0,n},\hat \mu_{1,n})^2\right)\leq t\right)-\mathbb P\left(N(0,v^2_0+v^2_1)\leq t\right)\right|\stackrel{\mathbb P}{\to}0,
    \]
    where  $v^2_0=\Var_{\mu_0}(f_0+g_{0,\pi}+\bar\varphi_0^{\mathbf A})$ and  $v^2_1=\Var_{\mu_1}(f_1+g_{1,\pi}+\bar\varphi_1^{\mathbf A})$ provided that $v_0^2+v_1^2>0$.
\end{theorem}

Again, we identify $G_{\mu_0}$ and $G_{\mu_1}$ with their extensions to a certain dilation of $\mathcal F_0$ and $\mathcal F_1$, see \cref{rmk:extensionBrownianBridge}. To obtain a limit distribution for $\mathsf D_{\varepsilon}$ itself, the technique from \cref{rmk:removingSquare} can be used. The condition on $\varepsilon$ that guarantees that  $\mathcal{A}$ is a singleton stems from Theorem 6 in \cite{rioux2023entropic}, which shows that in that case the objective of the entropic GW variational is strictly convex in $\mathbf A$. Beyond this condition, it is unclear when $\mathcal A$ is a singleton so that the consistency of the na\"ive bootstrap cannot be established in general (see  \cref{rmk:estimatingLimit}).

\section{Testing for Graph Isomorphism}
\label{sec:graphApplication}

A key property of the GW distance is that it identifies distributions that are related through an isometry. Armed with our limit theorems for GW distances, we consider an application to testing if two random graph models are isomorphic, based on samples from them. As the limit distribution theory accounts for Euclidean mm space, to employ it for graph isomorphism testing, we propose a framework for embedding distributions of random graphs on $N$ vertices into $\mathcal{P}(\mathbb R^N)$. Crucially, the embedding is constructed so that the random graph are isomorphic if and only if the GW distance between the embedded distributions nullifies. This allows casting the graph isomorphism testing question as a test for equality of the embedded distributions under the GW distance.

\subsection{Embedding of random graph distributions}

Consider the set $\mathcal G_N$ of all weighted undirected graphs on $N$ vertices with nonnegative weights and no self-loops. A graph $G\in\mathcal G_N$ is specified by its matrix of weights $\mathbf W^G\in\mathbb R^{N\times N}$, which is defined as $\mathbf W^G_{ij}=\mathbf W^G_{ji}=w$ if the nodes $i$ and $j$ are connected by an edge with  weight $w$ for $i\neq j\in[N]$. We adopt the convention that a weight of $0$ indicates the absence of an edge, which implies $\mathbf W_{ii}^G=0$ for every $i\in[N]$ since there are no self-loops. With this, a distribution $\nu\in\mathcal P(\mathcal G_N)$ can be specified by fixing $N(N-1)/2$ univariate \emph{edge distributions}, $(\rho_{ij})_{i,j\in[N]:\,i<j}\subset \mathcal P([0,\infty))$, and setting $\mathrm{law}(\mathbf W^{G}_{ij})=\rho_{ij}$ for $G\sim \nu$. %
Such a distribution treats each edge independently, and  we write $\mathcal P_{\text{ind}}(\mathcal G_N)$ to denote the set of all such distributions to emphasize this fact.  

To formulate the comparison of two distributions $\nu_0,\nu_1\in\mathcal P_{\text{ind}}(\mathcal G_N)$ with associated edge distributions $(\rho_{0,ij})_{i,j\in[N]:\,i<j}, (\rho_{1,ij})_{i,j\in[N]:\,i<j}$  in terms of the Euclidean GW distance considered herein, we introduce the following embedding 
 \begin{equation}
 \label{eq:embedding}
    \iota : (\rho_{ij})_{\substack{i,j\in[N]\\i<j}}\in \mathcal P([0,\infty))^{N(N-1)/2}\mapsto \frac{1}{N(N+1)(N-1)} 
    \sum_{\substack{i,j=1\\i\neq j}}^N\eta_{\rho_{ij}}+\frac{1}{N+1}\sum_{i=1}^N\delta_{-e_i}\in \mathcal P(\mathbb R^N),
 \end{equation}
 where $\eta_{\rho_{ij}}(B)= \int \mathbbm 1_{B}(-e_i+te_j)d\rho_{ij}(t)$, for every Borel set $B\subset \mathbb R^N$. To simplify notation,
 we identify $\nu\in\mathcal P_{\mathrm{ind}}(\mathcal G_N)$ with its edge distributions, $(\rho_{ij})_{i,j\in[N]:\,i<j}$, and write $\iota(\nu)$ in place of $\iota$ applied to the edge distributions.  

This construction enables a comparison of $\nu_0$ and $\nu_1$ up to a relabelling of the graph nodes. Precisely, $\nu_0$ and $\nu_1$ are 
 said to be isomorphic if there exists a permutation $\sigma:[N]\to [N]$ for which $\rho_{1,ij}=\rho_{0,\sigma(i)\sigma(j)}$ for every $i\neq j\in[N]$ under the convention that $\rho_{0,ij}= \rho_{0,ji}$ for every $i\neq j$. This notion of isomorphic graph distributions   coincides with the standard notion of isomorphic  graphs when $\nu_0$ and $\nu_1$ are Dirac distributions. We now establish that $\nu_0$ and $\nu_1$ are isomorphic if and only if $\iota(\nu_0)$ and $\iota(\nu_1)$ are identified under the GW distance.  
 \begin{proposition}[Identification of isomorphic graph distributions]
\label{prop:permContinuous}
For $\nu_0,\nu_1\in\mathcal P_{\mathrm{ind}}(\mathcal G_N)$ and $\mu_0=\iota(\nu_0)$ and $\mu_1=\iota(\nu_1)$, we have that $\nu_0$ and $\nu_1$ are isomorphic if and only if  $\mathsf D(\mu_0,\mu_1)=0$.
\end{proposition}
The proof of \cref{prop:permContinuous}, included in \cref{proof:prop:permContinuous}, uses the particular construction of the embedding in \eqref{eq:embedding}. Notably, the first term in \eqref{eq:embedding} serves to compare the individual edge distributions, whereas the second term ensures that the only possible isometry relating $\mu_0$ and $\mu_1$ is a permutation of coordinate axes. Without this final term, $\mathsf D(\mu_0,\mu_1)$ may nullify spuriously, e.g. if, for every $i,j\in[N]$ with $i<j$, $\rho_{0,ij}$ and $\rho_{1,ij}$ are related by a common translation.

\subsection{GW testing for graph isomorphism}
We are ready to set up the statistical GW-based test for graph isomorphism. Given mutually independent collections of sample graphs $G_{0,1},\dots,G_{0,n}\sim \nu_0\in\mathcal P_{\mathrm{ind}}(\mathcal G_N)$ and $G_{1,1},\dots,G_{1,n}\sim \nu_1\in\mathcal P_{\mathrm{ind}}(\mathcal G_N)$, we seek to test between the hypotheses:
\[
H_0:\,\nu_0\text{ and }\nu_1\text{ are isomorphic}\qquad \text{ vs. }\qquad H_1:\, \nu_0\text{ and }\nu_1\text{ are not isomorphic}.
\]
Given \cref{prop:permContinuous}, this amounts to testing $H_0:\mathsf D(\iota(\nu_0),\iota(\nu_1))= 0$ versus the alternative $H_1:\mathsf D(\iota(\nu_0),\iota(\nu_1))> 0$. To formulate the proposed test statistic, let $\hat\rho_{k,ij,n}=\frac{1}{n}\sum_{l=1}^n\delta_{\mathbf W_{ij}^{G_{k,l}}}$, for $k\in\{0,1\}$ and $i,j\in[N]$ with $i<j$, which corresponds to the empirical edge distribution, and let $\nu_{0,n}$ and $\nu_{1,n}$ denote the corresponding distributions on $\mathcal G_N$. With this, the test statistic is given by  $\sqrt n\mathsf D\big(\iota(\nu_{0,n}),\iota(\nu_{1,n})\big)^2$ and the critical value for the asymptotic level $\alpha \in(0,1)$ is given  by $q_{n,1-\alpha}$, the conditional $(1-\alpha)$ quantile of the direct estimator $L_n = \sqrt 2 \sup_{u\in\mathcal H_n}Z_n^{\intercal}u$ for $Z_n\sim N(0,\bsigma_n)$ as defined in \eqref{eq:directEstimatorGeneral},   where  $\mu_{0,n}=\iota(\nu_{0,n})$ and $\bsigma_n$ is the empirical analogue of $\bsigma_Z$ defined in \eqref{eq:covarianceStruct}. Using the machinery developed in our study of the empirical discrete GW distance from \cref{sec:DiscreteGW}, we prove in \cref{proof:thm:statsContinuousWeights} that the proposed test is consistent.
   
\begin{theorem}[Test consistency and asymptotic distribution]
\label{thm:statsContinuousWeights}
Let $\nu_0,\nu_1 \in\mathcal P_{\mathrm{ind}}(\mathcal G_{N})$ be such that their edge distributions, $(\rho_{0,ij})_{i,j=1:\,i<j}^N$ and $(\rho_{1,ij})_{i,j=1:\,i<j}^N$, are supported in $[0,K]$. Then, with $\mu_{k}= \iota(\nu_{k})$ and
 $\mu_{k,n}=\iota(\nu_{k,n})$ for $k\in\{0,1\}$, we have
 \[
    \mathbb E[|\mathsf D(\mu_{0,n},\mu_{1,n})^2-\mathsf D(\mu_{0},\mu_{1})^2|]\lesssim (1+2K^4)n^{-\frac 12}.
 \]
 Furthermore, if $(\rho_{0,ij})_{i,j=1:\,i<j}^N$ and $(\rho_{0,ij})_{i,j=1:\,i<j}^N$ are each finitely supported, then with the notation of \cref{thm:discreteGWLimitDistribution},  
    there exists Gaussian processes $\chi_{\mu_0}$ in $\ell^{\infty}(\mathcal F_0)$ and $\chi_{\mu_1}$ in $\ell^{\infty}(\mathcal F_1)$ for which  
    \[
        \begin{aligned}
        \sqrt n\left( \mathsf D(\mu_{0,n},\mu_{1,n})^2\mspace{-1mu}-\mspace{-1mu}\mathsf D(\mu_{0},\mu_{1})^2 \right)\stackrel{d}{\to}
\inf_{\mathbf{A}\in\mathcal A}
            \inf_{\pi \in \bar \Pi^{\star}_{\mathbf{A}}} \sup_{(\varphi_0,\varphi_1)\in\bar {\mathcal D}_{\mathbf{A}}}\left\{\chi_{\mu_0} ( f_0\mspace{-1mu}+\mspace{-1mu}g_{0,\pi}\mspace{-1mu}+\mspace{-1mu}\bar\varphi_0)\mspace{-1mu}+\mspace{-1mu}\chi_{\mu_1}(f_1\mspace{-1mu}+\mspace{-1mu}g_{1,\pi}\mspace{-1mu}+\mspace{-1mu}\bar\varphi_1)
        \right\}.
    \end{aligned} 
    \]
    If, in addition, $\mathsf D(\mu_0,\mu_1)=0$, then 
    $
        \sqrt n \mathsf D(\mu_{0,n},\mu_{1,n})^2
    \stackrel{d}{\to}
\sqrt{2}\|\chi_{\mu_0}\|_{\infty,\bar{\mathcal H}},
    $
    where $\bar{\mathcal H}$ is as defined in \cref{cor:discreteGWLimitDistributionNull} and the test which rejects the null hypothesis, $H_0:\nu_0$ and $\nu_1$ are isomorphic, if $\sqrt n \mathsf D(\mu_{0,n},\mu_{1,n})^2>q_{n,1-\alpha}$ is asymptotically consistent with level $\alpha$.
\end{theorem}

\subsection{Computational framework} 
\label{sec:computation}

For the purposes of testing if two random graph models are isomorphic by using the GW-based approach, the distance $\mathsf D(\mu_{0,n},\mu_{1,n})$ should be computed exactly to obtain statistically significant results. As computing this statistic is NP-complete in general, we propose to apply a local subgradient-based solver with a particular warm-start strategy. %

Recalling \eqref{eq:GWVariational}, $\mathsf D(\eta_0,\eta_1)^2=\mathsf S_1(\bar \eta_0,\bar \eta_1)+\inf_{\mathbf A\in\mathbb R^{d_0\times d_1}}\Phi_{(\eta_0,\eta_1)}(\mathbf A)$ for any $(\eta_0,\eta_1)\in\mathcal P(\mathbb R^{d_0})\times \mathcal P(\mathbb R^{d_1})$, whereby $\mathsf D(\eta_0,\eta_1)^2$ can be computed by  minimizing $\Phi_{(\eta_0,\eta_1)}$. This approach was explored in the context of the entropic GW problem in \cite{rioux2023entropic}, where the underlying minimization problem is smooth, but possibly nonconvex. We adapt Algorithm 2 in that work to the nonsmooth case by substituting the subgradient of $\Phi_{\bar\eta_0,\bar\eta_1}$, derived in  \cref{lem:ClarkeSubdifferential}, in place of the gradient (see \cref{alg:subgradient}). Though the convergence guarantees presented in \cite[Theorem 11]{rioux2023entropic} no longer apply in the current setting, in the experiments presented in \cref{sec:NumericalExperiments}, \cref{alg:subgradient} always returns a point at which a subgradient with  norm less than $5\times10^{-8}$ exists when $L$ is set to $128$. 

A warm-start for \cref{alg:subgradient} consists of a particular choice for the initialization $\mathbf C_0$. Our proposed warm-start methods are motivated by the fact that,  under the null, the set of optimizers for $\Phi_{(\bar \mu_0,\bar \mu_1)}$ is a subset of all matrices of the  form $\mathbf A_T=\frac{1}{2}\int xT(x)^{\intercal}d\bar \mu_0(x)$, where $T:\mathbb R^N\to \mathbb R^N$ is an isometry corresponding to a permutation of the coordinate axes. We write $\mathcal{T}$ for the set of all such isometries. The fact that any limit point of a  sequence of minimizers of $\Phi_{(\bar \mu_{0,n},\bar \mu_{1,n})}$ minimizes $\Phi_{(\bar \mu_{0},\bar \mu_{1})}$ with probability $1$ (see \cref{app:gammaConvergence}) suggests numerically estimating $\mathsf D(\mu_{0,n},\mu_{1,n})^2$ via an \emph{exhaustive search}. Namely, initialize \cref{alg:subgradient} at all such matrices $\mathbf A_T$, record the output of the subgradient method as $\Lambda_T$, and return $\min_{T\in\mathcal T}\Lambda_T$ as an estimate for the GW distance.
The exhaustive search is accurate under the null for large $n$, but onerous given that it requires solving $N!$ optimization problems, which limits its  usage to small scale problems.

\begin{algorithm}[!h]
\caption{Subgradient method for estimating $\mathsf D{(\eta_0,\eta_1)}^2$}\label{alg:subgradient}
\begin{algorithmic}[1]
\Statex Given the initialization $\mathbf C_0\in\mathcal D_M$,  tolerance $\delta>0$, and $L>0$, fix the step sequences $\beta_k=\frac{1}{2L}$, $\gamma_k=\frac{k}{4L}$, and $\tau_k=\frac{2}{k+2}$. 
\State $k\gets 1$
\State $\mathbf A_1\gets \mathbf C_0$
\State $\mathbf G_1\gets 64 \mathbf A_1+32\int xy^{\intercal} d\pi_{\mathbf A_1}(x,y)$ for any OT plan $\pi_{\mathbf A_1}$ of $\mathsf{OT}_{\mathbf A_1}(\bar\eta_0,\bar\eta_1)$
\While{$\|\mathbf G_k\|_{\mathrm F}\geq \delta$}
\State $\mathbf B_{k}\gets \min\left(1,\frac{M}{2\|\mathbf A_k-\beta_k\mathbf G_k\|_F}\right)(\mathbf A_k-\beta_k\mathbf G_k)$
\State $\mathbf C_{k}\gets 
\min\left(1,\frac{M}{2\|\mathbf C_{k-1}-\gamma_k\mathbf G_k\|_F}\right)(\mathbf C_{k-1}-\gamma_k\mathbf G_k)$
\State $\mathbf A_{k+1}\gets\tau_k\mathbf C_k+(1-\tau_k) \mathbf B_k$
\State $\mathbf G_{k+1}\gets 64 \mathbf A_{k+1}+32\int xy^{\intercal} d\pi_{\mathbf A_{k+1}}(x,y)$ for any OT plan $\pi_{\mathbf A_{k+1}}$ of $\mathsf{OT}_{\mathbf A_{k+1}}(\bar\eta_0,\bar\eta_1)$
\State $k\gets k+1$
\EndWhile
\State \textbf{return} $\mathsf S_1(\bar \eta_0,\bar \eta_1)+\Phi_{(\bar\eta_0,\bar\eta_1)}(\mathbf B_k)$
\end{algorithmic}
\end{algorithm}

\subsubsection{Relaxed graph matching} To scale up the computation of $\mathsf D(\mu_{0,n},\mu_{1,n})^2$ to larger problems, we proposed an alternative approach to the exhaustive search, applicable for unweighted graphs. Consider the related problem of directly aligning the empirical probabilities $\hat p_{0,ij}$ and $\hat p_{1,ij}$ of two unweighted graphs by defining  
\[
    \hat {\mathbf{P}}_0= \begin{pmatrix}
       0&\hat p_{0,12}&\hat p_{0,13}&\dots&\hat p_{0,1N}
       \\\hat  
       p_{0,21}&0&\hat p_{0,23}&\dots&\hat p_{0,2N}
       \\
       \vdots &\vdots & \vdots &\vdots&\vdots 
       \\\hat  
       p_{0,N1}&\hat p_{0,N2}&\hat p_{0,N3}&\dots&0
    \end{pmatrix} \quad \hat {\mathbf{P}}_0= \begin{pmatrix}
       0&\hat p_{1,12}&\hat p_{1,13}&\dots&\hat p_{1,1N}
       \\\hat  
       p_{1,21}&0&\hat p_{1,23}&\dots&\hat p_{1,2N}
       \\
       \vdots &\vdots & \vdots &\vdots&\vdots 
       \\\hat  
       p_{1,N1}&\hat p_{1,N2}&\hat p_{1,N3}&\dots&0
    \end{pmatrix},
\]
and solving the problem $\min_{\mathbf E\in \mathcal E}\|\mathbf E\hat{\mathbf P}_0-\hat{\mathbf P}_1\mathbf E\|_{\mathrm F}^2$, where $\mathcal E$ is the set of all permutation matrices on $\mathbb R^N$. This is known as the graph matching problem and, whose exact solution generally requires an exhaustive search, though the objective can be evaluated quickly.

A common relaxation of the graph matching problem replaces $\mathcal E$ by its convex hull, known as the \emph{Birkhoff polytope},
\[
\conv(\mathcal E)=\left\{ \mathbf E \in\mathbb R^{N\times N}: \mathbf E \geq 0, \mathbf E\mathbbm 1_{N_0}=\mathbf E^{\intercal}\mathbbm 1_{N_0}=\mathbbm 1_{N_0}\right\}, 
\]
which yields the convex relaxation $\min_{\mathbf R\in \conv(\mathcal E)}\|\mathbf R\hat{\mathbf P}_0-\hat{\mathbf P}_1\mathbf R\|_{\mathrm F}^2$.
 As a solution, ${\mathbf R^{\star}}$, to this relaxed problem is not necessarily a permutation matrix, a 
 projection onto the set of permutation matrices is performed by solving $\max_{\mathbf E\in\mathcal E}\langle \mathbf R^{\star},\mathbf E\rangle_{\text{F}}$, which is known as the linear assignment problem and can be solved efficiently. The permutation matrix $\mathbf E^{\star}$ obtained by performing these two steps is called a \emph{relaxed graph matching}. The approach is summarized in \cref{alg:relaxedGraph} below.

 \begin{algorithm}
\caption{Relaxed graph matching}\label{alg:relaxedGraph}
\begin{algorithmic}[1]
\Statex 

Given $\mu_{0,n},\mu_{1,n}$, construct $\hat {\mathbf{P}}_0$ and $\hat {\mathbf{P}}_1$,
\State Obtain $\mathbf R^{\star}\in\textstyle{\argmin_{\mathbf   R\in\conv(\mathcal E)}}\|\mathbf R{\hat{\mathbf{P}}}_0-\hat{{\mathbf{P}}}_1\mathbf R\|_{\mathrm F}^2$ using convex optimization software.
\State Obtain $\mathbf E^{\star}\in\argmax_{\mathbf E\in\mathcal E}\langle \mathbf R^{\star},\mathbf E\rangle_{\mathrm F}$ using a linear assignment problem solver.
\State \Return the output of \cref{alg:subgradient} initialized at $\mathbf A_{T^{\star}}$, where $T^{\star}$ is the permutation of axes corresponding to $\mathbf E^{\star}$. 
\end{algorithmic}
\end{algorithm}

 Remarkably, there are settings where this relaxed graph matching is known to coincide with the solution of the original graph matching problem. 

\begin{proposition}[Theorem 2 in \cite{aflalo2015convex}]
\label{prop:almostIsomorphism}
 Fix $\delta,\epsilon>0$ and suppose that the eigenvalues of $\hat{\mathbf P}_0$, $(\lambda_i)_{i=1}^N$, satisfy $\min_{i,j\in [N], \ i\neq j}|\lambda_i-\lambda_j|> \delta$, $\max_{i\in[N]}|\lambda_i|=1$, and its eigenvectors $(v_i)_{i=1}^N$ satisfy $\epsilon \leq v_i^{\intercal}\mathbbm 1_{N}\leq\frac 1 \epsilon$ for $i\in[N]$. If 
  $\|\mathbf E^{\star}\hat{\mathbf P}_0-\hat{\mathbf P}_1\mathbf E^{\star}\|_{\mathrm F}  < \frac{\delta^2\epsilon^4}{12 N^{1.5}}$, then the graph matching problem and  its relaxation are equivalent.     
\end{proposition}

The assumptions of \cref{prop:almostIsomorphism} effectively guarantee that both the relaxed and original problems admit a unique solution (in particular, the only isomorphism relating $\hat{\mathbf  P}_0$ to itself must be the identity). Observe that the conditions concerning  $\hat{\mathbf P}_0$ can be verified \emph{a priori}, and both $\hat{\mathbf P}_0$ and $\hat{\mathbf P}_1$ can be normalized such that the condition $\max_{i\in[N]}|\lambda_i|=1$ holds. Consequently,  upon obtaining $\mathbf E^{\star}$, it suffices to verify that $\|\mathbf E^{\star}\hat {\mathbf P}_0-\hat {\mathbf P}_1\mathbf E^{\star}\|_{\text{F}}$ satisfies the final condition to confirm if the relaxation $\mathbf E^{\star}$ solves the graph matching problem. 

Although \cref{prop:almostIsomorphism} is only applicable in certain settings, the underlying optimization problems can be solved efficiently thus enabling the matching of moderately sized graphs. Moreover, this approach may still yield good solutions beyond the theoretical guarantees. This graph matching approach is well-suited for unweighted graphs, but appears difficult to generalize to weighted case. By contrast, exhaustive search can easily be applied to weighted graphs, but is much more costly to implement and hence is limited to small-scale problems. A thorough numerical study of the performance of the exhaustive search procedure and \cref{alg:relaxedGraph} in provided in \cref{sec:extraExperiments}.

\section{Numerical Experiments}
\label{sec:NumericalExperiments}

The experiments in this section serve to  verify the different results presented in the paper. We examine settings for graph isomorphism testing and present results that numerically validate the developed limit distribution theory.

\subsection{Testing for graph isomorphism}

 The following experiments serve to   empirically validate \cref{thm:statsContinuousWeights} by studying the type 1 and type 2 errors of the proposed test for graph isomorphism.  As the problem is inherently discrete, we operate under the framework of \cref{sec:DiscreteGW} and focus on the consistency of the direct estimator $L_n$ of the limit, which employs linear programming (see \cref{sec:directEstimation}). We implement the test which rejects the null hypothesis ($\nu_0$ and $\nu_1$ are isomorphic) if $\sqrt n\mathsf D(\mu_{0,n},\mu_{1,n})^2>q_{n,1-\alpha}$, where $q_{n,1-\alpha}$ is the $(1-\alpha)$ quantile of $L_n$ and $\alpha$ is the significance level.   
 
\subsubsection{Unweighted graphs}
\label{sec:testingBinary}

 Figure \ref{fig:testingGraphs} presents the underlying  distributions $\nu_0,\nu_1,$ and $\nu_2$ on the set of binary graphs on $10$ vertices. Here, $\nu_0$ and $\nu_1$ are chosen to be isomorphic whereas $\nu_2$ is set to be equal to $\nu_1$ save for the fact that we set $p_{2,12}=p_{1,13}$, resulting in $\nu_0$ and $\nu_2$ not being isomorphic.   
 
 \begin{figure}[!t]
    \centering
    \includegraphics[width=.8\textwidth,height=.33\textwidth]{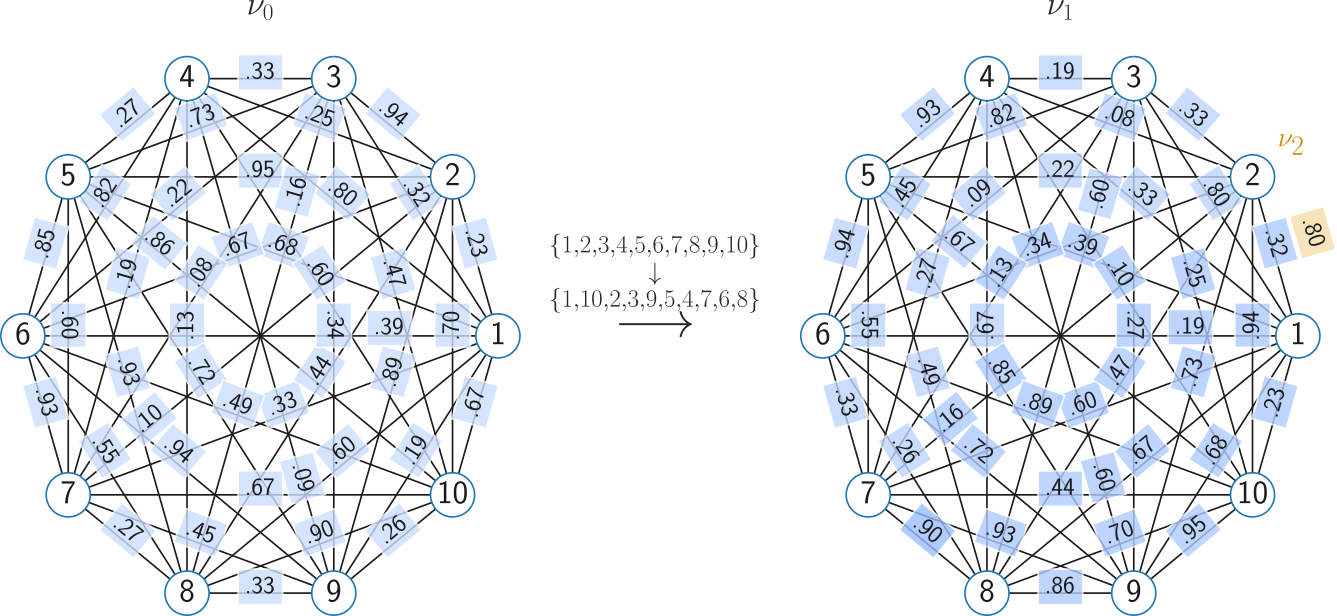}
    \caption{Distributions from \cref{sec:testingBinary}. $\nu_0$ is generated by   sampling $(p_{0,ij})_{\substack{i,j\in[N]\\i<j}}$ from the uniform distribution on $[0,1]$. $\nu_1$ is obtained from $\nu_0$ by a randomly chosen permutation. The modified probability for $\nu_2$ is highlighted in orange. Edge probabilities are presented with two significant figures for ease of reading.}
    \label{fig:testingGraphs}
\end{figure}

  We generate datasets of  $n\in\{10,25,50,100,500,1000,5000\}$ graphs from each of these distributions and construct the estimators $\mu_{0,n},\mu_{1,n},$ and $\mu_{2,n}$ as $\iota(\nu_{0,n}),\iota(\nu_{1,n}), $ and $\iota(\nu_{2,n})$. We then compute $\mathsf D(\mu_{0,n},\mu_{1,n})^2$ and $\mathsf D(\mu_{0,n},\mu_{2,n})^2$ approximately via \cref{alg:relaxedGraph}. 
 Next, we generate $200$ samples from $L_n$ using \cref{alg:directEstimatorGeneral} (an extension of \cref{alg:directEstimator} to estimators other than the empirical measures, as employed herein) and estimate $q_{n,1-\alpha}$ for $\alpha \in \{k/9\}_{k=0}^9$. With this,~we record if $\sqrt n\mathsf D(\mu_{0,n},\mu_{1,n})^2>q_{n,1-\alpha}$, i.e., the test spuriously rejects the null hypothesis, or if $\sqrt n\mathsf D(\mu_{0,n},\mu_{2,n})^2\leq q_{n,1-\alpha}$, whence the test fails to reject the null when the alternative hypothesis is~true.  To estimate the type 1 and type 2 errors, we perform $100$ repetitions of this procedure and report the  fraction of repetitions for which $\sqrt n\mathsf D(\mu_{0,n},\mu_{1,n})^2>q_{n,1-\alpha}$ as the estimated type 1 error and the fraction for which $\sqrt n\mathsf D(\mu_{0,n},\mu_{2,n})^2\leq q_{n,1-\alpha}$ as the estimated type 2 error. These results are compiled in  
  Figure \ref{fig:type12errorBinary}.

\begin{figure}[!t]
    \centering
    \includegraphics[width=0.8\textwidth]{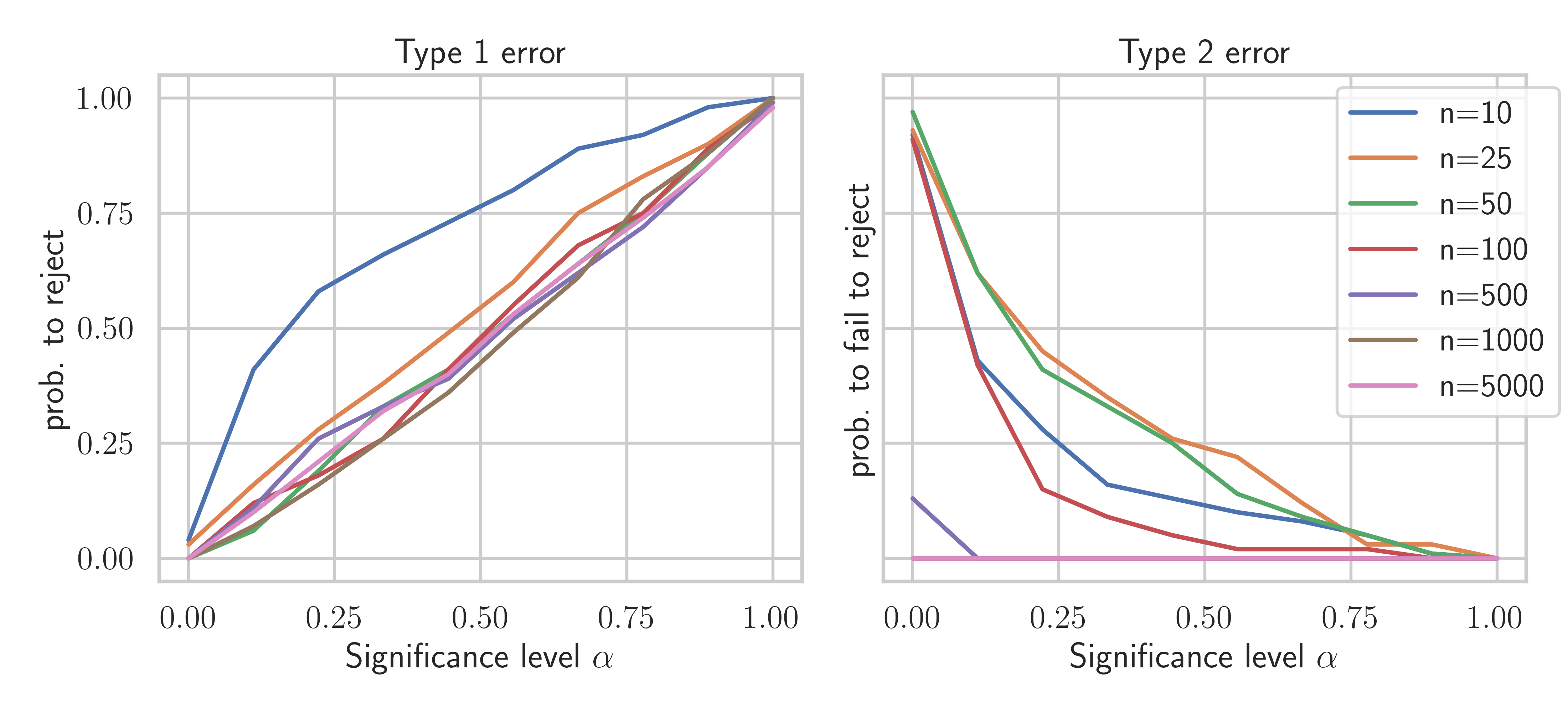}
    \caption{These plots compile the type 1 and type 2 error of the proposed test for varying numbers of samples and desired significance levels by following the methodology described in \cref{sec:testingBinary}.}
    \label{fig:type12errorBinary}
\end{figure}

Figure \ref{fig:type12errorBinary} validates \cref{thm:statsContinuousWeights}. Notably, even for a relatively small number of samples, the graph of the estimated rejection probability as a function of the significance level approaches the graph of the identity function, as expected. In the low sample regime, $n=10$, we note that the type $1$ error is significantly higher than expected whereas the type $2$ error is lower than that for $n=25$ and $n=50$. This can be explained by the fact that \cref{alg:relaxedGraph} is prone to overestimating the GW distance when few samples are available (see \cref{sec:comparisonWarmStart} for additional discussion).

\subsubsection{Graphs with finitely many weights}
\label{sec:testingWeighted}

We now consider the case of weighted graphs whose edges may take on the values $\{0,1,2,3\}$. Figure \ref{fig:testingGraphsWeighted} depicts isomorphic distributions $\nu_0,\nu_1$ on this space of graphs along with a distribution, $\nu_2$, which is identical to $\nu_1$ except for $\rho_{2,12}$ which is set to equal $\rho_{1,13}$ so that $\nu_0$ and $\nu_2$ are not isomorphic. The methodology for estimating the type $1$ and type $2$ error of this statistic is identical to \cref{sec:testingBinary}, up to the fact that we use the exhaustive search to estimate the GW distances (given that the relaxed graph matching approach is limited to unweighted graphs). The results are compiled in Figure \ref{fig:type12errorWeigthed}, presenting similar trends to those discussed above with the notable exception that the performance of the test in the case $n=10$ does not deviate heavily from the performance obtained with larger sample sizes.

\begin{figure}[!htb]
    \centering
    \includegraphics[width=.7\textwidth]{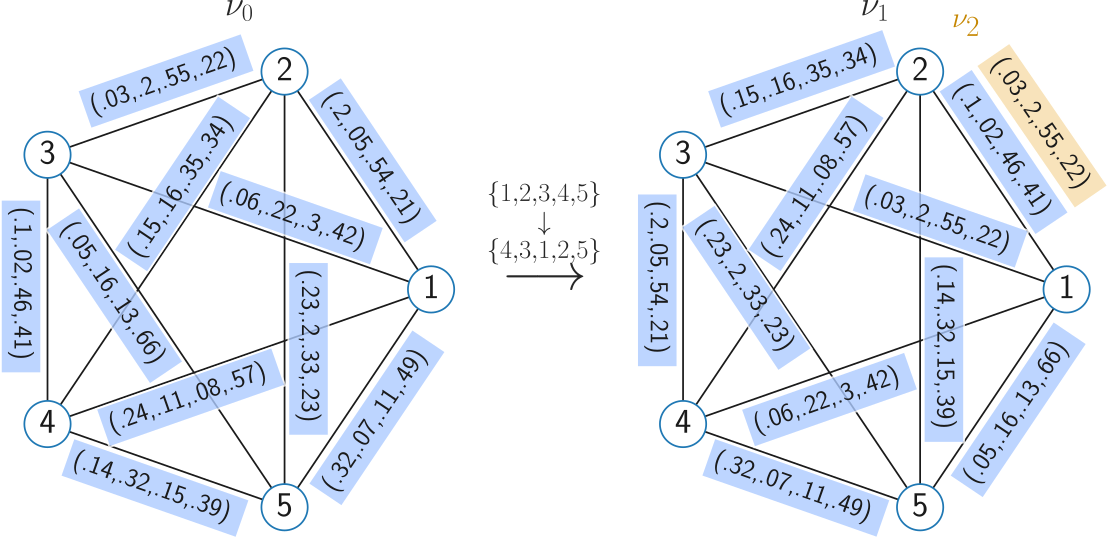}
    \caption{Distributions from \cref{sec:testingWeighted}. First, edge weight probabilities for $\nu_0$ were sampled from the Dirichlet distribution, then a random permutation was chosen to generate $\nu_1$ from $\nu_0$ as depicted. Edge weight probabilities are given as vectors whose $k$-th entry represents the probability (truncated to two significant figures) that the edge takes the weight $k-1$.}
    \label{fig:testingGraphsWeighted}
\end{figure}

\begin{figure}[!htb]
    \centering
    \includegraphics[width=0.75\textwidth]{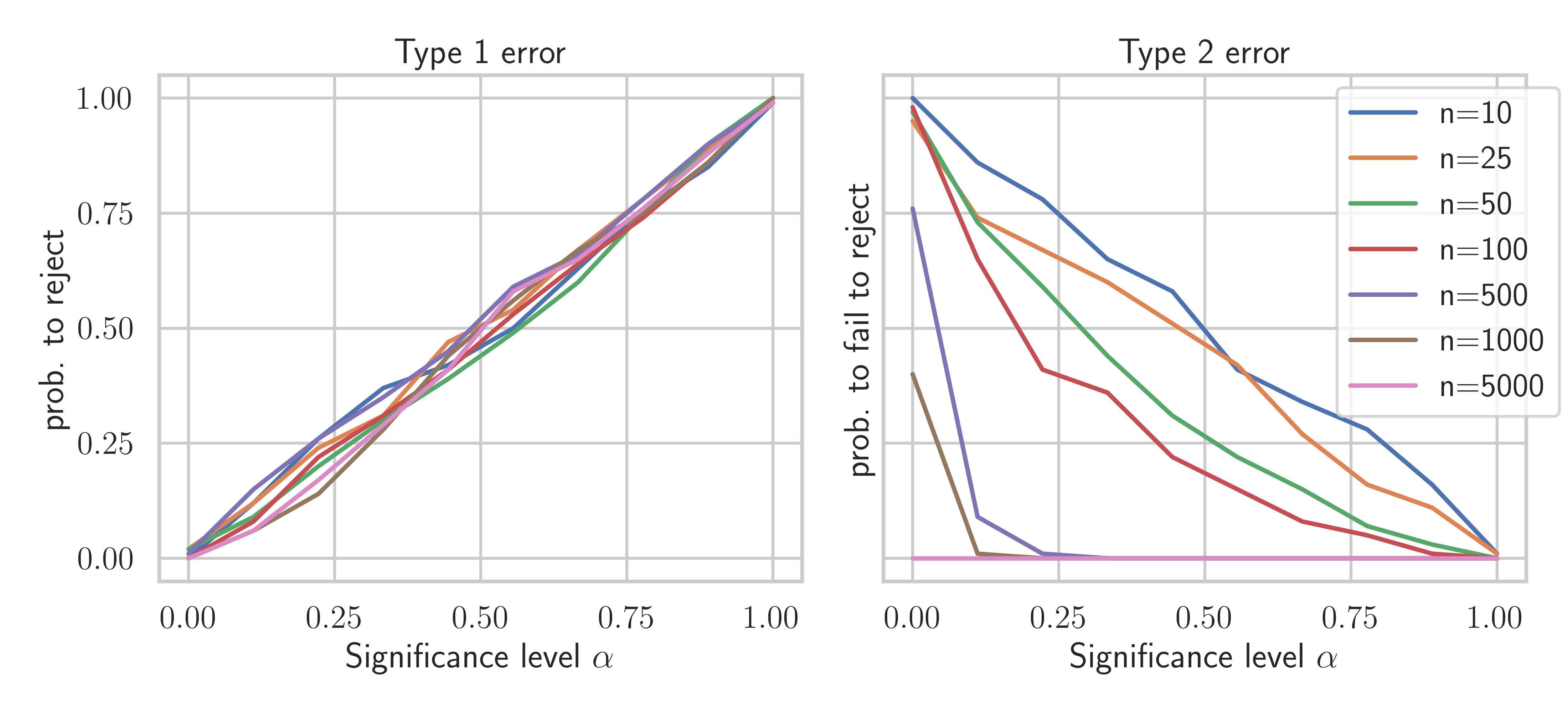}
    \caption{These plots compile the type 1 and type 2 error of the proposed test for varying numbers of samples and desired significance levels by following the methodology described in \cref{sec:testingWeighted}.}
    \label{fig:type12errorWeigthed}
\end{figure}

\subsection{Simulating limit distributions}

We next set to illustrate the structure of the limit distributions under the different settings studied in this paper.

\subsubsection{Graphs with compactly supported weight distributions}
\label{sec:compactDists}
    This experiment explores the law of $\sqrt n \mathsf D(\mu_{0,n},\mu_{1,n})^2$ when the underlying distributions on the edge weights are isomorphic. While \cref{thm:entropicGWLimitDistribution} only accounts for finitely supported distributions, we go beyond that setting and consider here continuous, compactly supported distributions, as shown in Figure \ref{fig:testingContinuousGraphs}. The ground truth permutation relating these distributions is taken to be the identity, and, given $n$ sampled graphs, we estimate $\mathsf D(\mu_{0,n},\mu_{1,n})^2$ using the subgradient method (\cref{alg:subgradient}) initialized at $\frac {1}{2}\int xx^{\intercal} d\bar{\mu}_{0,n}(x)$. To generate additional samples from $\mathsf D(\mu_{0,n},\mu_{1,n})^2$, we draw a new set of $n$ graphs and repeat this procedure. As the underlying distributions are isomorphic, $\mathsf D(\mu_0,\mu_1)=0$ and there is no need to center the statistic. 
  
\begin{figure}[!htb]
    \centering
    \includegraphics[width=.7\textwidth]{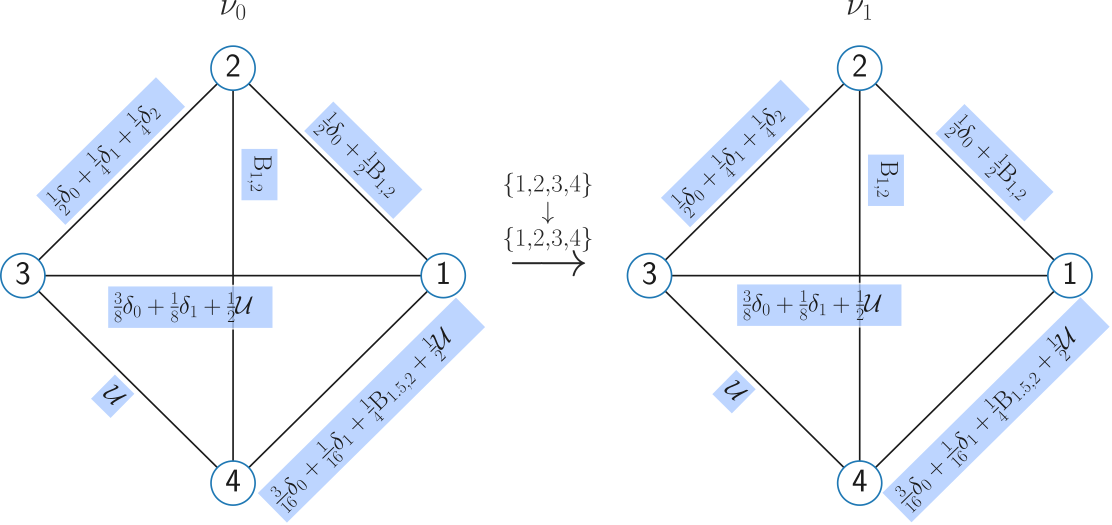}
    \caption{Distributions from \cref{sec:compactDists}. Here $\mathrm B_{a,b}$ denotes the beta distribution with parameters $(a,b)$ whereas $\mathcal U$ is the uniform distribution on $[0,1]$.}
    \label{fig:testingContinuousGraphs}
\end{figure}

\begin{figure}[!htb]
    \centering
    \includegraphics[width=\textwidth]{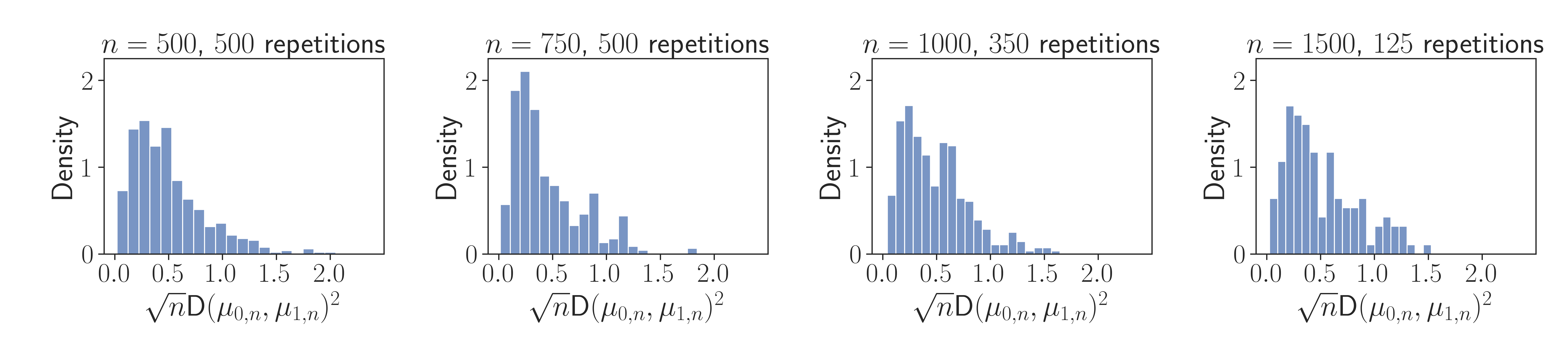}
    \caption{Estimated distribution of $\sqrt n\mathsf D(\mu_{0,n},\mu_{1,n})^2$ for different values of $n$.}
    \label{fig:type12errorContinuous}
\end{figure}

    The results of this experiment, presented in Figure \ref{fig:type12errorContinuous}, appear to indicate that the distribution of $\sqrt n\mathsf D(\mu_{0,n},\mu_{1,n})^2$ is stable as $n$ increases, this finding indicates that the derived limit distribution results from \cref{thm:statsContinuousWeights} may extend to the compact case. Moreover, the mean rate of convergence provided in that result is in line with these simulations.
\subsubsection{Entropic Gromov-Wasserstein distance}
\label{sec:entropicGWExperiment}

As noted in \cref{thm:entropicGWLimitDistribution}, $\mathsf D_{\varepsilon}(\hat \mu_{0,n},\hat \mu_{1,n})^2$ is asymptotically normal under proper scaling and centering  once $\varepsilon>16\sqrt{M_4(\bar\mu_0)M_4(\bar\mu_1)}$. With this, the na\"ive bootstrap can be seen to be consistent by following the proof of Theorem 7 in \cite{goldfeld24statistical}.

To illustrate this finding numerically, we fix $\mu_0$ to be the uniform distribution on $[-1,1]^3$ and $\mu_1$ to be the uniform distribution on the unit sphere in $\mathbb R^3$ and let $\varepsilon = 16.2\sqrt{\frac{19}{15}}>16\sqrt{M_4(\bar \mu_0)M_4(\bar \mu_1)}$, where $\frac{19}{15}$ is the value of the product of the fourth moments. Then, we generate $n=2500$ samples from both distributions to form the empirical measures $\hat \mu_{0,n},\hat\mu_{1,n}$ and compute $\mathsf D_{\varepsilon}(\hat \mu_{0,n},\hat \mu_{1,n})^2$ using Algorithm 1 in \cite{rioux2023entropic}. It is verified that this choice of $\varepsilon$ satisfies $\varepsilon>16\sqrt{M_4(\bar\mu_{0,n})M_4(\bar\mu_{1,n})}$ for each computation so that Theorem 11 in \cite{rioux2023entropic} guarantees that a nearly optimal value is obtained using this procedure. Finally, the boostrapped empirical measures are constructed by sampling $X_{0,1}^B,\dots, X_{0,n}^B$ from $\hat \mu_{0,n}$, $X_{1,1}^B,\dots, X_{1,n}^B$ from $\hat \mu_{1,n}$, and setting $\hat \mu_{0,n}^B\coloneqq\frac 1n \sum_{i=1}^n\delta_{X_{0,i}^B}$, $\hat \mu_{1,n}^B\coloneqq\frac 1n \sum_{i=1}^n\delta_{X_{1,i}^B}$. The distribution of $\sqrt n \left(\mathsf D_{\varepsilon}(\hat \mu_{0,n}^B,\hat \mu_{1,n}^B)^2-\mathsf D_{\varepsilon}(\hat \mu_{0,n},\hat \mu_{1,n})^2\right)$ for fixed $X_1,\dots,X_n$ is then estimated from $1500$ repetitions. The resulting histogram is presented in Figure \ref{fig:entropicHistogram} and is seen to be approximately normal.       

\begin{figure}[!htb]
    \centering
    \includegraphics[width=0.45\textwidth]{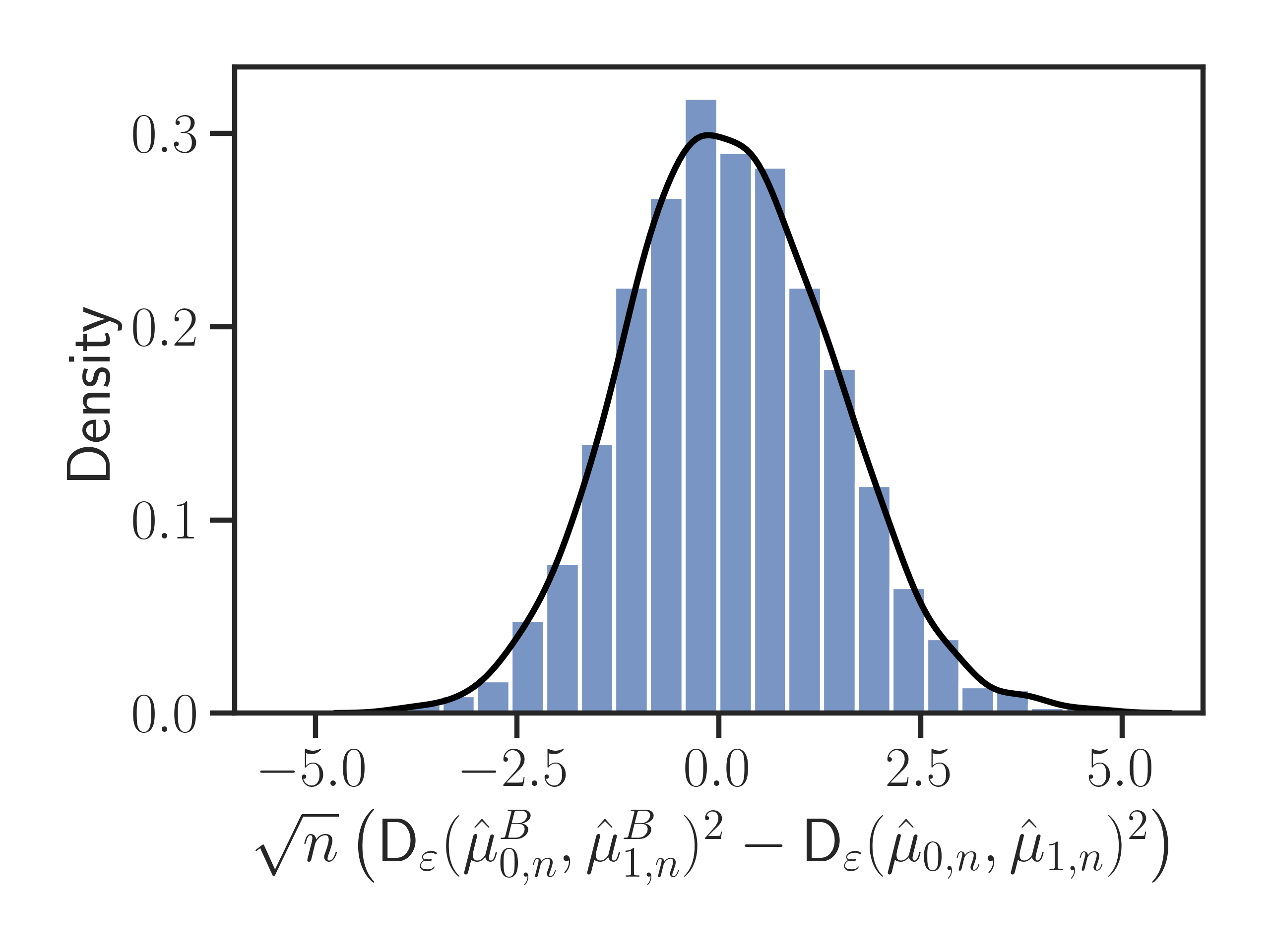}
    \caption{Histogram of the bootstrap estimator for the entropic GW distance limit distribution in the setting of \cref{sec:entropicGWExperiment} along with a kernel density estimator.}
    \label{fig:entropicHistogram}
\end{figure}

\section{Concluding Remarks}

This work initiated the study of distributional limits for the GW problem, covering both discrete and semi-discrete cases, as well as general 4-sub-Weibull distributions for the entropically regularized GW distance. To this end, we characterized the stability properties of the GW and entropic GW distances along perturbations of the marginal distributions, which may be of independent interest and offer utility beyond our statistical exploration. As the derived limits are generally non-Gaussian, we prove the consistency of a direct estimator of the limiting variable in the case of discrete distributions. Remarkably, we can efficiently sample from the direct estimator using linear programming, in spite of the fact that computing the empirical GW value is NP-complete in general. The developed limit distribution theory was leveraged for an application to graph isomorphism testing. To facilitate that, we proposed a framework for embedding distributions of random graphs into $\mathcal{P}(\RR^N)$, i.e., distributions over a Euclidean space. This allows casting the graph isomorphism testing question as a test for equality of the embedded distributions under the GW distance. We then leverage the derived limit laws to furnish an asymptotically consistent test, and propose efficient procedures for (approximate) evaluation of the GW test statistic. Numerical results validating the test performance on several random graph models were provided, along with simulations of the derived limiting laws.

Future research directions stemming from this work are abundant. Firstly, although we did not provide limit distributions for the empirical GW distance between compactly supported distributions in low dimensions, the numerical experiments in \cref{sec:compactDists} suggest that such a result may hold, at least under the null. Extending our current proof technique to the compactly supported case would require, among other results, proving that the OT potentials $(\varphi_{0,t}^{\mathbf A},\varphi_{1,t}^{\mathbf A})$ for $\mathsf{OT}_{\mathbf A}((\cdot-\mathbb E_{\mu_0}[X])_{\sharp}\mu_{0,t},(\cdot-\mathbb E_{\mu_1}[X])_{\sharp}\mu_{1,t})$ where $\mathbf A\in\argmin_{\mathbf A\in\mathbb R^{d_0\times d_1}}\Phi_{(\bar \mu_0,\bar\mu_1)}$ converge pointwise as $t\downarrow 0$ to a pair of OT potentials for $\mathsf{OT}_{\mathbf A}(\bar \mu_0,\bar\mu_1)$ which maximizes $\Upsilon:(\psi_0,\psi_1)\mapsto \int \bar\psi_0d(\nu_0-\mu_0)+\int \bar\psi_1d(\nu_1-\mu_1)$ over all such OR potentials. While it is straightforward to show that the sequence of OT potentials converges along a subsequence to a pair of OT potentials for $\mathsf{OT}_{\mathbf A}(\bar \mu_0,\bar\mu_1)$, the main challenge lies in proving that this pair is maximal for $\Upsilon$. %

Another promising direction pertains to isometry testing of heterogeneous datasets using the GW distance. Our approach to testing for graph isomorphisms, presented in \cref{sec:graphApplication}, relies on a particular embedding of random graph distributions into the space of distributions on a Euclidean space. The key property of this construction is that the embeddings are identified under the GW distance if and only if the distributions on graphs are isomorphic. This embedding, provided in  \eqref{eq:embedding}, consists of two terms: the first compares the individual edge distributions, while the second ensures that the embeddings can only be related by specific isometries. This hints at a broader strategy for testing whether two datasets are generated from distributions that are related by particular isometric transformations, but not others. We plan to explore this direction further, aiming a general isometry testing methodology between heterogenuous datasets via the GW distance.

\bibliographystyle{alpha}
\bibliography{ref}

\pagebreak

\appendix
\section{Proofs of Main Results}
\subsection{Proof of \texorpdfstring{\cref{thm:VariationalMinimizers}}{Theorem 1}}
\label{app:proofOfthm:VariationalMinimizers}

To simplify the proof of \cref{thm:VariationalMinimizers}, we separately prove each statement as its own lemma. 
\begin{lemma}
    \label{lem:FrechetDerivativeObjective}
    $\Phi_{(\mu_0,\mu_1)}$ is Fr{\'e}chet differentiable at $\mathbf{A}\in\mathbb R^{d_0\times d_1}$ with derivative $\left(D\Phi_{(\mu_0,\mu_1)}\right)_{[\mathbf{A}]}(\mathbf{B})=64\langle\mathbf{A}-\frac{1}{2}\int xy^{\intercal}d\pi_{\mathbf{A}}(x,y),\mathbf{B}\rangle_{\mathrm F}$ provided all optimal couplings for $\OT_{\mathbf{A}}(\mu_0,\mu_1)$ admit the same cross-correlation matrix $\frac{1}{2}\int xy^{\intercal}d\pi_{\mathbf{A}}(x,y)$.
\end{lemma}
\begin{proof}
    It is easy to see that $\|\cdot\|_{\mathrm{F}}^2$ is Fr{\'e}chet differentiable at $\mathbf{A}$ with derivative $2\langle \mathbf{A},\cdot\rangle_{\mathrm F}$.
    For any $\mathbf{H}\in\mathbb R^{d_0\times d_1}$, we have that 
    \begin{equation}
        \label{eq:upperBound}
        \OT_{\mathbf{A}+\mathbf{H}}(\mu_0,\mu_1)-\OT_{\mathbf{A}}(\mu_0,\mu_1)\leq \int c_{\mathbf{A}+\mathbf{H}}d\pi_{\mathbf{A}}-\int c_{\mathbf{A}}d\pi_{\mathbf{A}}=-32\int x^{\intercal}\mathbf{H}yd\pi_{\mathbf{A}}(x,y), 
    \end{equation}
    for any choice of OT plan $\pi_{\mathbf{A}}$ for $\OT_{\mathbf{A}}(\mu_0,\mu_1)$.
    Similarly, 
    \begin{equation}
        \label{eq:lowerBound}
        \OT_{\mathbf{A}+\mathbf{H}}(\mu_0,\mu_1)-\OT_{\mathbf{A}}(\mu_0,\mu_1)\geq -32\int x^{\intercal}\mathbf{H}yd\pi_{\mathbf{A}+\mathbf{H}}(x,y),
    \end{equation}
    for any choice of optimal coupling $\pi_{\mathbf{A}+\mathbf{H}}$ for $\OT_{\mathbf{A}+\mathbf{H}}(\mu_0,\mu_1)$.
    Now, consider an arbitrary sequence $\mathbf{H}_n$ converging to $0$. Note that  
    \[
        \sup_{\substack{x\in\supp(\mu_0)\\y\in\supp(\mu_1)}}\left|c_{\mathbf{A}+\mathbf{H}_n}(x,y)-c_{\mathbf{A}}(x,y)\right|=\sup_{\substack{x\in\supp(\mu_0)\\y\in\supp(\mu_1)}}\left|32x^{\intercal}\mathbf{H}_ny\right|\leq 32\sup_{\supp(\mu_0)}\|\cdot\|\sup_{\supp(\mu_1)}\|\cdot\|\|\mathbf{H}_n\|_{\mathrm{F}}\to 0,
    \]      
    hence $c_{\mathbf{A}+\mathbf{H}_n}\to c_{\mathbf{A}}$ uniformly on $\supp(\mu_0)\times \supp(\mu_1)$. It follows from Theorem 5.20 in \cite{villani2008optimal} that, for any subsequence $n'$ of $n$ there exists a further subsequence $n''$ along which $\pi_{\mathbf{A}+\mathbf{H}_{n''}}\stackrel{w}{\to} \pi$  for some optimal coupling $\pi$ for $\OT_{\mathbf{A}}(\mu_0,\mu_1)$. Thus $\int xy^{\intercal}d\pi_{\mathbf{A}+\mathbf{H}_{n''}}(x,y)\to\int xy^{\intercal}d\pi(x,y)=\int xy^{\intercal}d\pi_{\mathbf{A}}(x,y)$ by
    assumption. As the limit is the same regardless of the choice of subsequence, conclude that $\int xy^{\intercal}d\pi_{\mathbf{A}+\mathbf{H}}(x,y)\to\int xy^{\intercal}d\pi_{\mathbf{A}}(x,y)$ as $\mathbf{H}\to 0$, thus 
    \begin{align*}
        &\|\mathbf{H}\|^{-1}_{F}\left|\OT_{\mathbf{A}+\mathbf{H}}(\mu_0,\mu_1)-\OT_{\mathbf{A}}(\mu_0,\mu_1)+32\int x^{\intercal}\mathbf{H}yd\pi_{\mathbf{A}}(x,y)\right|
\\
&\leq 32\left\|\int xy^{\intercal}d\pi_{\mathbf{A}+\mathbf{H}}(x,y)-\int xy^{\intercal}d\pi_{\mathbf{A}}(x,y)\right\|_{\mathrm{F}}\to 0, 
    \end{align*}
    which proves the claim.
\end{proof}

\begin{lemma}
    \label{lem:couplingFrobeniusNorm}
    Let $\pi\in\Pi(\mu_0,\mu_1)$ be arbitrary, then $\int xy^{\intercal}d\pi(x,y)\in{B_{\mathrm{F}}(\sqrt{M_2(\mu_0)M_2(\mu_1)})}$.
\end{lemma}
\begin{proof}
    By Jensen's inequality, $\left\|\int xy^{\intercal}d\pi(x,y)\right\|_{\mathrm{F}}\leq \int\left\| xy^{\intercal}\right\|_{\mathrm{F}}d\pi(x,y)=\int\| x\|\|y\|d\pi(x,y)$. This final term is bounded above by $\sqrt{\int \|x\|^2d\pi(x,y)\int \|y\|^2d\pi(x,y)}=\sqrt{M_2(\mu_0)M_2(\mu_1)}$ by the Cauchy-Schwarz inequality.
\end{proof}

\begin{lemma}
    \label{prop:Lipschitz}
    $\Phi_{(\mu_0,\mu_1)}$ is locally Lipschitz continuous and coercive. 
\end{lemma}
\begin{proof}
    Fix a compact set $K\subset \mathbb R^{d_0\times d_1}$. For any $\mathbf{A},\mathbf{A}'\in K$, it follows from \eqref{eq:upperBound} and \eqref{eq:lowerBound} that, \begin{align*}
        \OT_{\mathbf{A}'}(\mu_0,\mu_1)-\OT_{\mathbf{A}}(\mu_0,\mu_1)&\leq -32\int x^{\intercal}(\mathbf{A}'-\mathbf{A})yd\pi_{\mathbf{A}}(x,y)
        \\&\leq 32\|\mathbf{A}'-\mathbf{A}\|_{\mathrm{F}}\left\|\int xy^{\intercal}d\pi_{\mathbf{A}}(x,y)\right\|_{\mathrm{F}} , 
   \\ 
        \OT_{\mathbf{A}'}(\mu_0,\mu_1)-\OT_{\mathbf{A}}(\mu_0,\mu_1)&\geq- 32\|\mathbf{A}'-\mathbf{A}\|_{\mathrm{F}}\left\|\int xy^{\intercal}d\pi_{\mathbf{A}'}(x,y)\right\|_{\mathrm{F}},
    \end{align*}
    that is, 
    \[
        \left|\OT_{\mathbf{A}'}(\mu_0,\mu_1)-\OT_{\mathbf{A}}(\mu_0,\mu_1)\right|\leq 32 \|\mathbf{A}'-\mathbf{A}\|_{\mathrm{F}}\left(\left\|\int xy^{\intercal}d\pi_{\mathbf{A}'}(x,y)\right\|_{\mathrm{F}}\bigvee \left\|\int xy^{\intercal}d\pi_{\mathbf{A}}(x,y)\right\|_{\mathrm{F}}\right).
    \]
    Applying \cref{lem:couplingFrobeniusNorm}, we obtain 
    \[
        \left|\OT_{\mathbf{A}'}(\mu_0,\mu_1)-\OT_{\mathbf{A}}(\mu_0,\mu_1)\right|\leq 32\sqrt{M_2(\mu_0)M_2(\mu_1)} \|\mathbf{A}'-\mathbf{A}\|_{\mathrm{F}}.
    \]
    Further, $\left|\|\mathbf{A}\|^2_{\mathrm F}-\|\mathbf{A}'\|^2_{\mathrm F}\right|=(\|\mathbf{A}\|_{\mathrm{F}}+\|\mathbf{A}'\|_{\mathrm{F}})\left|\|\mathbf{A}\|_{\mathrm{F}}-\|\mathbf{A}'\|_{\mathrm{F}}\right|\leq 2\sup_{K}\|\cdot\|_{\mathrm{F}}\|\mathbf{A}-\mathbf{A}'\|_{\mathrm{F}}$.

    To show coercivity, observe that, for any $\mathbf{A}\in\mathbb R^{d_0\times d_1}$ and $\pi\in \Pi(\mu_0,\mu_1)$, 
    \[
        \begin{aligned}
        \int -4\|x\|^2\|y\|^2-32x^{\intercal}\mathbf{A}yd\pi(x,y)\geq -4\sqrt{M_4(\mu_0)M_4(\mu_1)}-32\sqrt{M_2(\mu_0)M_2(\mu_1)}\|\mathbf{A}\|_{\mathrm{F}}  
        \end{aligned}    
    \]
    Hence $32\|\mathbf{A}\|_{\mathrm{F}}+\frac{\OT_{\mathbf{A}}(\mu_0,\mu_1)}{\|\mathbf{A}\|_{\mathrm{F}}}\to \infty$ as $\|\mathbf{A}\|_{\mathrm{F}}\to\infty$ proving coercivity.
\end{proof}

\begin{lemma}
\label{lem:ClarkeSubdifferential}
    The Clarke subdifferential of $\Phi_{(\mu_0,\mu_1)}$ at $\mathbf A\in\mathbb R^{d_0\times d_1}$ is given by 
    \begin{equation}
    \label{eq:subdiff}
    \partial \Phi_{(\mu_0,\mu_1)}(\mathbf A)=
       \left\{ 64\mathbf A-32\int xy^{\intercal}d\pi_{\mathbf A}(x,y): \pi_{\mathbf A} \text{ is an OT plan for }\mathsf{OT}_{\mathbf A}(\mu_0,\mu_1)\right\}. 
    \end{equation}
\end{lemma}
\begin{proof}
   We first establish that $\partial \Phi_{(\mu_0,\mu_1)}$ is a subset of the right hand side in \eqref{eq:subdiff}. By Rademacher's theorem, local Lipschitz continuity of $\Phi_{(\mu_0,\mu_1)}$ (\cref{prop:Lipschitz}) guarantees that it is differentiable on a set $\Lambda$ of full measure and, by Theorem 8.1 in \cite{clarke2008nonsmooth}, the Clarke subdifferential of $\Phi_{(\mu_0,\mu_1)}$ at $\mathbf{A}\in\mathbb
R^{d_0\times d_1}$ can be defined as
\[\partial \Phi_{(\mu_0,\mu_1)}(\mathbf{A})=\conv\left(\left\{\lim_{\Omega\ni\mathbf{A}_n\to \mathbf{A}} (D\Phi_{(\mu_0,\mu_1)})_{[\mathbf{A}_n]}\right\}  \right)
\]
where $\Omega$ is any subset of $\Lambda$ for which $\Lambda\backslash\Omega$ is negligible and it is presupposed that the limit converges (here $\conv$ denotes the convex hull operation on a set). From \cref{lem:FrechetDerivativeObjective}, $(D\Phi_{(\mu_0,\mu_1)})_{[\mathbf{A}_n]}$ can be identified with $64\left(\mathbf{A}_n-\frac{1}{2}\int xy^{\intercal}d\pi_{\mathbf{A}_n}(x,y)  \right)$ where $\pi_{\mathbf{A}_n}$ is
any OT plan for $\OT_{\mathbf{A}_n}(\mu_0,\mu_1)$ (by assumption all such cross-correlation matrices are identical). As $\mathbf{A}_n\to \mathbf{A}$, it follows from the proof of \cref{lem:FrechetDerivativeObjective} that $\pi_{\mathbf{A}_n}\stackrel{w}{\to} \pi_{\mathbf{A}}$ up to a subsequence, where $\pi_{\mathbf{A}}$ is some OT plan for $\OT_{\mathbf{A}}(\mu_0,\mu_1)$. It follows that $\partial \Phi_{(\mu_0,\mu_1)}(\mathbf{A})=64\mathbf{A}-32\conv\left( \left\{ \int xy^{\intercal}d\pi_{\mathbf{A}}(x,y):\exists\; \pi_{\mathbf{A}_n}\stackrel{w}{\to}\pi_{\mathbf{A}},\mathbf{A}_n\in\Omega\right\}\right)$ such that $\partial \Phi_{(\mu_0,\mu_1)}(\mathbf{A})=D\left(\Phi_{(\mu_0,\mu_1)}\right)_{[\mathbf{A}]}$ for $\mathbf{A}\in\Lambda$. If $\mathbf{A}\not\in\Lambda$, the previous convex hull is simply a subset of all cross-correlation matrices for some OT plan for $\OT_{\mathbf{A}}(\mu_0,\mu_1)$, proving the claim.

Now, we show that $\partial \Phi_{(\mu_0,\mu_1)}$ is a superset of the right hand side in \eqref{eq:subdiff}.
To this end, note that $\mathbf A\in\mathbb R^{d_0\times d_1}\mapsto -\mathsf{OT}_{\mathbf A}(\mu_0,\mu_1)$ is convex by Theorem 5.5 in \cite{rockafellar1997convex} as it can be expressed as the pointwise supremum of a collection of linear functions of $\mathbf A$. By \eqref{eq:upperBound}, it follows that 
\[
    -\mathsf{OT}_{\mathbf A'}(\mu_0,\mu_1)-(-\mathsf{OT}_{\mathbf A}(\mu_0,\mu_1))\geq \left\langle32\int xy^{\intercal}d\pi_{\mathbf A}(x,y),\mathbf A'-\mathbf A\right\rangle_{\mathrm{F}},
\]
for any $\mathbf A'\in\mathbb R^{d_0\times d_1}$ and any choice of OT plan, $\pi_{\mathbf A}$, for $\mathsf{OT}_{\mathbf A}(\mu_0,\mu_1)$ so that the subdifferential of $-\mathsf{OT}_{(\cdot)}(\mu_0,\mu_1)$ at $\mathbf A$ (in the sense of convex analysis) contains all matrices of the form $32\int xy^{\intercal}d\pi_{\mathbf A}(x,y)$. Note, however, that the Clarke subdifferential coincides with the (convex) subdifferential in this context (see Proposition 2.2.7 in \cite{clarke1990optimization}) and that $\partial \Phi_{(\mu_0,\mu_1)}(\mathbf A)= 64\mathbf A-\partial(-\mathsf{ OT}_{(\cdot)}(\mu_0,\nu_1))(\mathbf A)$ by Proposition 2.3.1 and Corollary 1 on p.39 of \cite{clarke1990optimization}, proving the complementary inclusion. 
\end{proof}

We now prove point $2$, concluding the proof of \cref{thm:VariationalMinimizers}.

\begin{proof}[Proof of \cref{thm:VariationalMinimizers} (2)]
     By Proposition 2.3.2 in \cite{clarke1990optimization}, if $\bar{\mathbf{A}}$ is a local minimizer for $\Phi_{(\mu_0,\mu_1)}$, then
$0\in\partial \Phi_{(\mu_0,\mu_1)}(\bar{\mathbf{A}})$, characterized in \cref{lem:ClarkeSubdifferential}. Thus, there exists an OT plan $\pi_{\bar{\mathbf{A}}}$ for $\OT_{\bar{\mathbf{A}}}(\mu_0,\mu_1)$ satisfying $2\bar{\mathbf{A}}=\int xy^{\intercal}d\pi_{\bar{\mathbf{A}}}(x,y)\in{B_{\mathrm{F}}(\sqrt{M_2(\mu_0)M_2(\mu_1)})}$ by \cref{lem:couplingFrobeniusNorm}. By coercivity, proved in \cref{prop:Lipschitz}, at least one local minimizer is globally optimal. 

Now, assume that $\mu_0,\mu_1$ are centered. It is straightforward to see that if $\mathbf{A}^{\star}$ solves \eqref{eq:Objective}, then the associated OT plan $\pi_{\mathbf{A}^{\star}}$ satisfies 
\[
    \mathsf S_1(\mu_0,\mu_1)+\mathsf S_2(\mu_0,\mu_1)=\iint \left|\|x-x'\|^2-\|y-y'\|^2\right|d\pi_{\mathbf{A}^{\star}}\otimes\pi_{\mathbf{A}^{\star}}(x,y,x',y'),
\]
such that $\pi_{\mathbf{A}^{\star}}$ is optimal for \eqref{eq:GWPrimal} (see Section 5.1 in \cite{zhang2024gromov} for details).
\end{proof}

\subsection{Proofs for \texorpdfstring{\cref{sec:DiscreteGW}}{Section 3}}

\subsubsection{Proof of \texorpdfstring{\cref{thm:discreteGWStability}}{Theorem 2}}
\label{sec:proof:thm:discreteGWStability}

We first show that the OT potentials for $\mathsf{OT}_{\mathrm A}(\nu_0,\nu_1)$ can be chosen as to satisfy uniform bounds for any choice of $(\nu_0,\nu_1)\in\mathcal P_{\mu_0}\times\mathcal P_{\mu_1}$ and $\mathbf{A}\in B_{\mathrm{F}}(M)$.

\begin{lemma}
    \label{prop:discreteGWPotentials}
    For any  $(\nu_0,\nu_1)\in\mathcal P_{\mu_0}\times\mathcal P_{\mu_1}$, $\mathbf{A}\in B_{\mathrm{F}}(M)$, and any choice of OT potentials $(\varphi_0,\varphi_1)$ for $\OT_{\mathbf{A}}(\nu_0,\nu_1)$, there exists a version, $(\tilde \varphi_0,\tilde \varphi_1)$, of $(\varphi_0,\varphi_1)$ satisfying 
    $\|\tilde\varphi_0\|_{\infty,\mathcal X_0}\vee\|\tilde\varphi_1\|_{\infty,\mathcal X_1}\leq K$ where $K$ depends only on $\|\mathcal X_0\|_{\infty}$ and $\|\mathcal X_1\|_{\infty}$.
\end{lemma}
\begin{proof}
    In the discrete setting, $\OT_{\mathbf{A}}(\nu_0,\nu_1)$ can be identified with a finite dimensional linear program.
By the complementary slackness conditions, any pair  OT potentials $(\varphi_0,\varphi_1)$ for $\OT_{\mathbf{A}}(\nu_0,\nu_1)$ satisfies $\varphi_0(x)+\varphi_1(y)=c_{\mathbf{A}}(x,y)$ at all points $(x,y)\in\mathcal X_0\times \mathcal X_1$ where $\pi(\{(x,y)\})>0$ for any choice of OT plan $\pi$ for $\OT_{\mathbf{A}}(\nu_0,\nu_1)$. As $\pi\in\Pi(\nu_0,\nu_1)$ and $\supp(\nu_i)=\mathcal X_i$ for $i=0,1$, for every $x\in\mathcal X_0$ there exists $y_x\in\mathcal X_1$ for which $\varphi_0(x)+\varphi_1(y_x)=c_{\mathbf{A}}(x,y_x)$ and, similarly, for every $y\in\mathcal X_1$ there exists $x_y\in\mathcal X_0$ for which $\varphi_0(x_y)+\varphi_1(y)=c_{\mathbf{A}}(x_y,y)$. Moreover, $y_x\in\argmin_{\mathcal X_1}\left\{c_{\mathbf{A}}(x,\cdot)-\varphi_1\right\}$ and $x_y\in\argmin_{\mathcal X_0}\left\{c_{\mathbf{A}}(\cdot,y)-\varphi_0\right\}$, as $\varphi_0\oplus\varphi_1\leq c_{\mathbf{A}}$.  Whence,
\[
\begin{gathered}
-\|c_{\mathbf{A}}\|_{\infty,\mathcal X_0\times \mathcal X_1}-\sup_{\mathcal X_1} \varphi_1
    \leq\varphi_0(x)=\inf_{\mathcal X_1}\left\{c_{\mathbf{A}}(x,\cdot)-\varphi_1\right\}\leq \|c_{\mathbf{A}}\|_{\infty,\mathcal X_0\times \mathcal X_1}-\sup_{\mathcal X_1} \varphi_1,
    \\
    -\|c_{\mathbf{A}}\|_{\infty,\mathcal X_0\times \mathcal X_1}-\sup_{\mathcal X_0} \varphi_0
    \leq\varphi_1(y)=\inf_{\mathcal X_0}\left\{c_{\mathbf{A}}(\cdot,y)-\varphi_0\right\}\leq \|c_{\mathbf{A}}\|_{\infty,\mathcal X_0\times \mathcal X_1}-\sup_{\mathcal X_0} \varphi_0.
\end{gathered}
\]
Let $(\tilde \varphi_0,\tilde \varphi_1)= (\varphi_0+C, \varphi_1-C)$ for $C=-\sup_{\mathcal X_0}\varphi_0+\|c_{\mathbf{A}}\|_{\infty,\mathcal X_0\times \mathcal X_1}$ such that $\sup_{\mathcal X_0}\tilde \varphi_0=\|c_{\mathbf{A}}\|_{\infty,\mathcal X_0\times \mathcal X_1}$. From the prior display, we have  $\|\tilde\varphi_1\|_{\infty,\mathcal X_1}\leq 2\|c_{\mathbf{A}}\|_{\infty,\mathcal X_0\times \mathcal X_1}$, $\|\tilde\varphi_0\|_{\infty,\mathcal X_0}\leq \|c_{\mathbf{A}}\|_{\infty,\mathcal X_0\times \mathcal X_1}$. 

It suffices, therefore, to bound $\|c_{\mathbf{A}}\|_{\infty,\mathcal X_0\times \mathcal X_1}$. To this end, for any $(x,y)\in\mathcal X_0\times \mathcal X_1$, 
\begin{equation}
\label{eq:uniformCostBound}
    |c_{\mathbf{A}}(x,y)|=|-4\|x\|^2\|y\|^2-32x^{\intercal} \mathbf{A}y|\leq 4\|\mathcal X_0\|^2_{\infty}\|\mathcal X_1\|^2_{\infty}+32\|\mathcal X_0\|_{\infty}\|\mathcal X_1\|_{\infty}M,
\end{equation}
where the inequality is due to the triangle inequality and the fact that $|x^{\intercal} \mathbf{A}y|=|\langle xy^{\intercal}, \mathbf{A}\rangle_{\mathrm F}|\leq \|xy^{\intercal}\|_{\mathrm{F}}\|\mathbf{A}\|_{\mathrm{F}}=\|x\|\|y\|\|\mathbf{A}\|_{\mathrm{F}}\leq\|\mathcal X_0\|_{\infty}\|\mathcal X_1\|_{\infty}M $ as  $\mathbf{A}\in B_{\mathrm{F}}(M)$.
\end{proof}

Next,  we address the failure of the directional regularity condition from \cite{bonnans2013perturbation}. In what follows, we assume without loss of generality that $\mu_0$ and $\mu_1$ are centered. In the case that $\mu_0,\mu_1$ are not centered, observe that the limit in \eqref{eq:rightDerivativeDiscrete} can be written as 
    \[
        \lim_{t\downarrow 0}\frac{\mathsf D(\mu_{0,t},\mu_{1,t})^2-\mathsf D(\mu_{0},\mu_{1})^2}{t}= \lim_{t\downarrow 0}\frac{\mathsf D((\Id-\mathbb E_{\mu_0}[X])_{\sharp}\mu_{0,t},(\Id-\mathbb E_{\mu_1}[X])_{\sharp}\mu_{1,t})^2-\mathsf D(\bar \mu_{0},\bar \mu_{1})^2}{t},
    \]
    by translation invariance of the Gromov-Wasserstein distance. As the subsequent arguments only require that the perturbed measures and the base measures are supported on the same points and that the
    base measures are centered, we may apply the same result to the limit on the right hand side to prove the general claim.

    Our proof technique is inspired by the arguments presented in Section 4.3.2 of \cite{bonnans2013perturbation}. However, we highlight that the results contained therein are not applicable to the current setting as the so-called directional regularity condition does not hold. 

\begin{remark}[Failure of directional regularity]
\label{rmk:directionalRegularityFailure}
 Consider the perturbed problem
\begin{equation}
    \label{eq:discreteQPPerturb}
    \begin{aligned}
        \inf_{\mathbf{P}\in\mathbb R^{N_0\times N_1}}\; &\left\langle \mathbf{P},\Delta(\mathbf{P}) \right\rangle_{\text{F}},
        \\
        \text{s.t. }\;&\mathbf{P}\mathbbm {1}_{N_1}=m_{0,t},
        \\
        &\mathbf{P}^{\intercal}\mathbbm{ 1}_{N_0}=m_{1,t},
        \\
        &\hspace{1.2em}\mathbf{P}_{ij}\geq 0,\;\forall (i,j)\in[N_0]\times[N_1],
    \end{aligned}
\end{equation}
where $(m_{0,t},m_{1,t})\in\mathbb R^{N_0}\times \mathbb R^{N_1}$ are the weight vectors for $\mu_{0,t}$ and $\mu_{1,t}$ for $t\in[0,1]$. To establish stability of the optimal value of \eqref{eq:discreteQPPerturb} at $t=0$, Theorem 4.24 in \cite{bonnans2013perturbation} can be applied provided that the directional regularity condition holds.  

According to Theorem 4.9 in \cite{bonnans2013perturbation}, the directional regularity condition for the problem \eqref{eq:discreteQPPerturb} holds at a point $\mathbf P\in\mathbb R^{N_0\times N_1}$ which is feasible for \eqref{eq:discreteQPPerturb} at $t=0$ in the direction $d=(n_0-m_0,n_1-m_1,0)$ where $n_0,n_1$ are the weight vectors for $\nu_0,\nu_1$ if and only if
\begin{equation}
\label{eq:directionalRegularity}
    0\in \interior\left( G(\mathbf P,(m_0,m_1))+ DG_{[(\mathbf P,(m_0,m_1))]}(\mathbb R^{N_0\times N_1},\left\{t(n_0-m_0,n_1-m_1) :t\geq 0\right\})-K\right),
\end{equation}
where $G(\mathbf Q,(r_0,r_1))=(\mathbf{Q}\mathbbm {1}_{N_1}-r_0,\mathbf{Q}^{\intercal}\mathbbm{ 1}_{N_0}-r_1,\mathbf Q)$ and $K=\left\{0\right\}\times\left\{0\right\}\times [0,\infty)^{N_0\times N_1}$. It is easy to see that 
\[
\begin{aligned}
DG_{[(\mathbf P,(m_0,m_1))]}(\mathbb R^{N_0\times N_1},\left\{t(n_0-m_0,n_1-m_1) :t\geq 0\right\})&\\&\hspace{-15em}=\left\{(\mathbf{Q}\mathbbm {1}_{N_1}-t(n_0-m_0),\mathbf{Q}^{\intercal}\mathbbm{ 1}_{N_0}-t(n_1-m_1),\mathbf Q):\mathbf Q\in\mathbb R^{N_0\times N_1},t\geq0\right\},
\end{aligned}
\]
so that condition \eqref{eq:directionalRegularity} reads 
\[
\begin{aligned}
0\in\interior\left(\left\{(\mathbf{Q}\mathbbm {1}_{N_1}-t(n_0-m_0),\mathbf{Q}^{\intercal}\mathbbm{ 1}_{N_0}-t(n_1-m_1),\mathbf P + \mathbf Q-\mathbf S)\right.\right.\\&\hspace{-5em}\left.\left.:\mathbf Q\in\mathbb R^{N_0\times N_1}, \mathbf S \in [0,\infty)^{N_0\times N_1}, t\geq0\right\}\right).
\end{aligned}
\]
We remark that the first and second coordinates of any triple in the above set satisfy 
\[
    \mathbbm {1}_{N_1}^{\intercal}\left(\mathbf{Q}\mathbbm {1}_{N_1}-t(n_0-m_0)\right)=\mathbbm {1}_{N_1}^{\intercal}\mathbf{Q}\mathbbm {1}_{N_0}=\mathbbm {1}_{N_1}^{\intercal}\mathbf{Q}^{\intercal}\mathbbm 1_{N_0}=\mathbbm {1}_{N_1}^{\intercal}\left(\mathbf{Q}^{\intercal}\mathbbm{ 1}_{N_0}-t(n_1-m_1)\right),
\]
as $n_0-m_0$ and $n_1-m_1$ have total mass zero. It follows that this set has empty interior and thus condition \eqref{eq:directionalRegularity} fails. 
\end{remark}

For ease of presentation, we separate the proof of \cref{thm:discreteGWStability} into a number of lemmas. 
\begin{lemma}
    \label{lem:GateauxDerivativeDiscrete}
    Fix $(\nu_0, \nu_1)\in\mathcal P_{\mu_0}\times \mathcal P_{\mu_1}$ and let $\rho_i=\nu_i-\mu_i$, $\mu_{i,t}=\mu_i+t\rho_i$ for $t\in[0,1]$ and $i\in\{0,1\}$. Then,
\begin{equation}
 \label{eq:rightDerivativeDiscrete}
    \begin{aligned}
        \frac{d}{dt}\mathsf D(\mu_{0,t},\mu_{1,t})^2\big\vert_{t=0}&= +\inf_{\mathbf{A}\in\mathcal A} \inf_{\pi \in \bar \Pi^{\star}_{\mathbf{A}}} 
\sup_{(\varphi_0,\varphi_1)\in\bar {\mathcal D}_{\mathbf{A}}}\left\{
            \int (f_0+ g_{0,\pi}+\bar\varphi_0)d\rho_0+\int (f_1+g_{1,\pi}+\bar\varphi_1)d\rho_1
        \right\},
    \end{aligned}
    \end{equation}  
    where 
    \[
        \begin{aligned}
            f_0&=2\int \|\cdot-x\|^4d\mu_0(x)-4M_2(\bar \mu_1)\|\cdot-\mathbb E_{\mu_0}[X]\|^2,
            \\
            f_1&=2\int \|\cdot-y\|^4d\mu_0(y)-4M_2(\bar \mu_0)\|\cdot-\mathbb E_{\mu_1}[X]\|^2,
            \\
            g_{0,\pi}&= 8\mathbb E_{(X,Y)\sim \pi}[\|Y\|^2X]^{\intercal}(\cdot),
            \\
            g_{1,\pi}&=8\mathbb E_{(X,Y)\sim \pi}[\|X\|^2Y]^{\intercal}(\cdot),
        \end{aligned}
    \]
    $\bar\varphi_i=\varphi_i(\cdot-\mathbb E_{\mu_i}[X])$ for $i=0,1$, and, for $\mathbf{A}\in B_{\mathrm{F}}(M)$, $\bar \Pi^{\star}_{\mathbf{A}}$ is the set of all OT plans for $\OT_{\mathbf{A}}(\bar \mu_0,\bar \mu_1)$, and $\bar{\mathcal D}_{\mathbf{A}}$ is the set of all corresponding OT potentials.
\end{lemma}

\begin{proof}
    Assume without loss of generality that $\mu_0$ and $\mu_1$ are centered.
    We begin by establishing an upper bound on the limit. To this end, let $\bar\pi\in\Pi(\mu_0,\mu_1)$ be an arbitrary GW plan for $\mathsf D(\mu_{0},\mu_{1})$, and let $\gamma$ be a finite signed measure on $\mathcal X_0\times \mathcal X_1$  with first marginal $\nu_0-\mu_0$, second marginal $\nu_1-\mu_1$, and  with $\gamma(\{(x^{(i)},y^{(j)})\})\geq 0$ for every $(i,j)\in \mathcal I(\bar \pi):=\left\{ (i,j)\in[N_0]\times [N_1]: \bar
    \pi(\{(x^{(i)},y^{(j)})\})=0 \right\}$. With this, $\bar \pi+t\gamma \in \Pi(\mu_{0,t},\mu_{1,t})$ for every $t\in[0,1]$ sufficiently small, whence
    \begin{equation}
    \label{eq:upperBoundDiff}
        \begin{aligned}
            \mathsf D(\mu_{0,t},\mu_{1,t})^2-\mathsf D(\mu_{0},\mu_{1})^2&\leq \int \Delta d\left( \bar \pi+t\gamma \right)\otimes \left( \bar \pi+t\gamma \right)-\int \Delta d\bar\pi\otimes\bar\pi
            \\
            &= 2t \int \Delta d\bar\pi\otimes \gamma +t^2 \iint \Delta d\gamma \otimes \gamma,
        \end{aligned}
    \end{equation}
    for all such $t$. 

    We now tighten the right hand side of \eqref{eq:upperBoundDiff} by optimizing over the choice of $\gamma$ in the first term of the last line. To this end, define $\mathbf{C}\in\mathbb R^{N_0\times N_1}$ via $\mathbf{C}_{ij}=2\int \Delta\left( x^{(i)},y^{(j)},x,y \right)d\bar \pi(x,y)$ for $(i,j)\in [N_0]\times [N_1]$ and consider the linear program  
\begin{equation}
    \label{eq:discreteLP}
    \begin{aligned}
    \inf_{\mathbf{G}\in \mathbb R^{N_0\times N_1}} \quad \langle \mathbf{C},\mathbf{G}\rangle_{\mathrm F}& \\
    \text{subject to} \hspace{0.95em} \quad
    \mathbf{G}\mathbbm 1_{N_1} &= n_0-m_0, \\
    \mathbf{G}^{\intercal}\mathbbm 1_{N_0} &= n_1-m_1, \\
    \mathbf{G}_{ij}&\geq 0, \;\forall\;(i,j)\in\mathcal I(\bar \pi),
\end{aligned}
\end{equation}
where $\mathbbm 1_{N_i}$ is the vector of all $1$'s of length $N_i$ for $i=0,1$, and the vectors $n_0,m_0\in\mathbb R^{N_0},n_1,m_1\in\mathbb R^{N_1}$ contain the weights of $\nu_0,\mu_0,\nu_1,$ and $\mu_1$ respectively. The corresponding dual problem is given by 
\begin{equation}
    \label{eq:discreteLPDual}
\begin{aligned}
    \sup_{\substack{y\in\mathbb R^{N_0},z\in\mathbb R^{N_1}\\ (S_{ij})_{ (i,j)\in \mathcal I(\bar \pi)}}} \quad y^{\intercal}(n_0-m_0)&+ z^{\intercal}(n_1-m_1) \\
    \text{subject to} \hspace{2.65em} \quad
    y_i+z_j&=\mathbf{C}_{ij},\;\forall\;(i,j)\not\in\mathcal I(\bar \pi), \\
    y_i+z_j+S_{ij}&=\mathbf{C}_{ij},\;\forall\;(i,j)\in\mathcal I(\bar \pi), \\
    S_{ij}&\geq 0.
\end{aligned}
\end{equation}
Since $\mu_0,\mu_1$ are centered we have, for $(i,j)\in[N_0]\times [N_1]$,
\[
    \begin{aligned}
        \mathbf{C}_{ij}&= 2\int \|x^{(i)}-x\|^4 -2\|x^{(i)}-x\|^2\|y^{(j)}-y\|^2+ \|y^{(j)}-y\|^4d\bar \pi(x,y)
        \\
        &= 2\int \|x^{(i)}-x\|^4d\mu_0(x)+2\int\|y^{(j)}-y\|^4d\mu_1(y)-4M_2(\mu_1)\|x^{(i)}\|^2-4M_2(\mu_0)\|y^{(j)}\|^2
\\
            &+8\int \langle x^{(i)},x\rangle\|y\|^2d\bar \pi(x,y)+8\int \|x\|^2\langle y^{(j)},y\rangle d\bar \pi(x,y)-4\int\|y\|^2\|x\|^2d\bar \pi(x,y)
\\        
&-4\|x^{(i)}\|^2\|y^{(j)}\|^2-16\int\langle x^{(i)},x\rangle\langle y^{(j)},y\rangle d\bar \pi(x,y),  
    \end{aligned}
\]
where only the final two terms in the display depend simultaneously on $x^{(i)}$ and $y^{(j)}$. Adopting the auxiliary variables $r\in\mathbb R^{N_0},s\in\mathbb R^{N_1}$ given by 
\begin{equation}
    \label{eq:potentialConnection}
    \begin{aligned}
        r_i&=y_i-2\int \|x^{(i)}-x\|^4d\mu_0(x)-8\int \langle x^{(i)},x\rangle\|y\|^2d\bar \pi(x,y)+4M_2(\mu_1)\|x^{(i)}\|^2
        \\
        s_{j}&= z_j-2\int\|y^{(j)}-y\|^4d\mu_1(y)-8\int \|x\|^2\langle y^{(j)},y\rangle d\bar \pi(x,y)+4M_2(\mu_0)\|y^{(j)}\|^2\\&+4\int\|y\|^2\|x\|^2d\bar \pi(x,y)
    \end{aligned}
\end{equation}
for every $(i,j)\in[N_0]\times [N_1]$, we obtain the equivalent dual problem 
\begin{equation}
    \label{eq:discreteLPDualAlt}
\begin{aligned}
    \sup_{\substack{r\in\mathbb R^{N_0},s\in\mathbb R^{N_1}\\ (S_{ij})_{ (i,j)\in \mathcal I(\bar \pi)}}} \quad r^{\intercal}(n_0-m_0)&+ s^{\intercal}(n_1-m_1)+\varsigma_{\bar\pi} \\
    \text{subject to} \hspace{2.65em} \quad
    r_i+s_j&=\mathbf{R}_{ij},\;\forall\;(i,j)\not\in\mathcal I(\bar \pi), \\
    r_i+s_j+S_{ij}&=\mathbf{R}_{ij},\;\forall\;(i,j)\in\mathcal I(\bar \pi), \\
    S_{ij}&\geq 0, 
\end{aligned}
\end{equation}
where $\mathbf{R}_{ij}=-4\|x^{(i)}\|^2\|y^{(j)}\|^2-16\int\langle x^{(i)},x\rangle\langle y^{(j)},y\rangle d\bar \pi(x,y)$ for $(i,j)\in[N_0]\times [N_1]$, and 
\[
    \begin{aligned}
        \varsigma_{\bar\pi}&= \sum_{i=1}^{N_0}\left(2\int \|x^{(i)}-x\|^4d\mu_0(x)+8\int \langle x^{(i)},x\rangle\|y\|^2d\bar \pi(x,y)-4M_2(\mu_1)\|x^{(i)}\|^2  \right)\left( n_{0}-m_{0} \right)_i
\\
&+\sum_{j=1}^{N_1}\left(2\int\|y^{(j)}-y\|^4d\mu_1(y)+8\int \|x\|^2\langle y^{(j)},y\rangle d\bar \pi(x,y)-4M_2(\mu_0)\|y^{(j)}\|^2  \right)\left( n_{1}-m_{1} \right)_j,
    \end{aligned}
\]
the term involving $4\int\|y\|^2\|x\|^2d\bar \pi(x,y)$ is omitted, as $n_1-m_1$ sums to $0$.

If $(r,s)$ is feasible for \eqref{eq:discreteLPDualAlt}, $r_i+s_j\leq \mathbf{R}_{ij}$ for every $(i,j)\in[N_0]\times [N_1]$ with equality when $(i,j)\in \mathcal I(\bar \pi)$ (i.e. when $(x^{(i)},y^{(j)})\in\supp(\bar \pi)$). 
Recall from \cref{sec:GWDistance} that,
as $\bar \pi$ is optimal for $\mathsf D(\mu_0,\mu_1)$, it is also optimal for $\OT_{\mathbf{A}_{\bar\pi}}(\mu_0,\mu_1)$, where $\mathbf{A}_{\bar
\pi}=\frac 12 \int xy^{\intercal} d\bar \pi(x,y)$. It follows that $r^{\intercal}m_0+s^{\intercal}m_1=\int c_{\mathbf{A}_{\bar \pi}}d\bar \pi$ such that $(r,s)$ can be identified with a pair of OT potentials for  $\OT_{\mathbf{A}_{\bar\pi}}(\mu_0,\mu_1)$. As the objective in \eqref{eq:discreteLPDualAlt} is invariant to the transformation $(r,s)\mapsto(r+C,s-C)$ and any choice of OT potential for $\OT_{\mathbf{A}_{\bar\pi}}(\mu_0,\mu_1)$ admits a version which is bounded (see \cref{prop:discreteGWPotentials}), the feasible set in \eqref{eq:discreteLPDualAlt} can be restricted to a bounded set. Conclude that \eqref{eq:discreteLPDualAlt} (and hence \eqref{eq:discreteLPDual}) admit optimal solutions so that strong duality holds between the problems \eqref{eq:discreteLPDualAlt} and \eqref{eq:discreteLP} holds (i.e. their optimal values coincide).

Given these deliberations, it follows from \eqref{eq:upperBoundDiff} that  
\[ 
            \limsup_{t\downarrow 0}
            \frac{\mathsf D(\mu_{0,t},\mu_{1,t})^2-\mathsf D(\mu_{0},\mu_{1})^2}{t}\leq \varsigma_{\bar \pi}+\sup_{(\varphi_0,\varphi_1)\in \mathcal D_{\mathbf{A}_{\bar\pi}}}\left\{ \int \varphi_0d(\nu_0-\mu_0)+\int \varphi_1d(\nu_1-\mu_1) \right\}, 
    \]
    which can be further tightened by minimizing the right hand side  over all optimal couplings for $\mathsf D(\mu_{0},\mu_{1})^2$.  Alternatively, if $\mathbf{A}\in \argmin_{B_{\mathrm{F}}(M)}\left( \Phi_{(\mu_0,\mu_1)} \right)$, then any optimal coupling, $\pi^{\star}$, for $\OT_{\mathbf{A}}(\mu_0,\mu_1)$ with $\mathbf {A}=\frac{1}{2}\int xy^{\intercal}d\pi^{\star}(x,y)$ is optimal for $\mathsf D(\mu_{0},\mu_{1})^2$ by \cref{thm:VariationalMinimizers} (as aforementioned, any solution $\pi$ to $\mathsf D(\mu_{0},\mu_{1})^2$ also solves $\OT_{\mathbf{A}}(\mu_0,\mu_1)$ for $\mathbf{A}=\frac 12 \int
    xy^{\intercal}d\pi(x,y)$) yielding
\begin{equation}
    \label{eq:limsupRatioDiscrete}
    \begin{aligned}
        \limsup_{t\downarrow 0}\frac{\mathsf D(\mu_{0,t},\mu_{1,t})^2-\mathsf D(\mu_{0},\mu_{1})^2}{t}&\leq \int f_0d\rho_0+ \int  f_1d\rho_1
        \\
        &\hspace{-4em}+\inf_{\mathbf{A}\in\mathcal A}\left\{
            \inf_{\pi \in \bar \Pi^{\star}_{\mathbf{A}}} \int g_{0,\pi}d\rho_0+\int g_{1,\pi}d\rho_1+\sup_{(\varphi_0,\varphi_1)\in\bar {\mathcal D}_{\mathbf{A}}}\int \bar\varphi_0d\rho_0+\int \bar\varphi_1d\rho_1
        \right\},
    \end{aligned}
    \end{equation}  
upon expanding the expression for $\varsigma_{\bar \pi}$.

We conclude by proving the complementary inequality. Let $t_n\downarrow 0$ be arbitrary, and let $\pi_{t_n}$ be any GW plan for $\mathsf D(\mu_{0,t_n},\mu_{1,t_n})$. In what follows, we write $\mathbf{P}_{t_n}$ for the matrix with entries $\left( \mathbf{P}_{t_n} \right)_{ij}=\pi_{t_n}\left( \left\{ (x^{(i)},y^{(j)}) \right\} \right)$ for $(i,j)\in[N_0]\times [N_1]$ and $m_0,n_0,m_1,n_1$ for the vectors of weights of $\mu_0,\nu_0,\mu_1,\nu_1$ as in the first half of the proof.
With this, $\mathsf D(\mu_{0,t},\mu_{1,t})^2$ can be viewed as a perturbation of the (finite dimensional) quadratic program $\mathsf D(\mu_{0},\mu_{1})^2$ given in \eqref{eq:discreteQP}, where the equality constraints are of the form $\left\{ \mathbf{P} \in\mathbb R^{N_0\times N_1}: \langle\mathbf{\Lambda}_i,\mathbf{P}\rangle_{\mathrm F}=b_i(t)\text{ for }i\in[N_0+N_1] \right\}$ for some fixed matrices $(\mathbf{\Lambda}_i)_{i=1}^{N_0+N_1}\subset\mathbb R^{N_0\times N_1}$ and a vector $b(t)=(m_0+t(n_0-m_0),m_1+t(n_1-m_1))\in\mathbb
R^{N_0+N_1}$. In this setting, the set of matrix solutions, $S_{t}$, of $\mathsf D(\mu_{0,t},\mu_{1,t})^2$ for $t\in[0,1]$ satisfies 
\begin{equation}
    \label{eq:solnSetLipschitz}
S_t \subset S_0 + \gamma \underbrace{\|(m_{0,t},-m_{0,t},n_{0,t},-n_{0,t})-(m_{0},-m_{0},n_{0},-n_{0})\|}_{=t\|(m_0-n_0,n_0-m_0,m_1-n_1,n_1-m_1)\|}B_{\mathrm{F}}(1),
\end{equation}
for all $t$ sufficiently small, 
where $\gamma>0$ is a constant and the addition is understood in the sense of Minkowski (see Theorem 3 in \cite{Klatte1985Lipschitz}).       

By the Heine-Borel theorem, there exists a subsequence $t_{n'}$ along which $\mathbf{P}_{t_{n'}}\to \mathbf{P}\in[0,1]^{N_0\times N_1}$ and, evidently, $\mathbf{P}\in S_0$ by the previous arguments.\footnote{Indeed, $d(\mathbf{P},S_0)\leq \|\mathbf{P}-\mathbf{P}_{t_{n'}}\|_{\mathrm{F}}+d(\mathbf{P}_{t_{n'}},S_0)\to 0$ as $t_{n'}\downarrow 0$, where $d(\cdot,S_0)=\inf_{\mathbf{Q}\in S_0}\left\{ d(\cdot,\mathbf{Q}) \right\}$.} It follows from \eqref{eq:solnSetLipschitz} that, for every $n'$ sufficiently large, there exists some
$\mathbf{P}_{n'}^{0}\in S_0$ with $\|\mathbf{P}_{n'}^{0}-\mathbf{P}_{t_{n'}}\|_{\mathrm{F}}=O(t_{n'})$.

Let $(y_{n'},z_{n'},(S^{n'}_{ij})_{(i,j)\in\mathcal I(\pi_{n'})})$ denote a solution to \eqref{eq:discreteLPDual}, where the set $\mathcal I(\bar \pi)$ is replaced by $\mathcal I(\pi_{n'})$ (where $\pi_{n'}$ is the coupling with weights given by $\mathbf{P}_{n'}^0$), and the cost $\mathbf{C}$ is replaced by $\mathbf{C}_{n'}$ with $(\mathbf{C}_{n'})_{ij}=2\int \Delta\left( x^{(i)},y^{(j)},x,y \right)d \pi_{n'}(x,y)$ for $(i,j)\in [N_0]\times
[N_1]$ (note that all of the implications regarding $(y_{n'},z_{n'},(S^{n'}_{ij})_{(i,j)\in\mathcal I(\pi_{n'})})$ derived in the first half of the proof still hold, as $\pi_{n'}$ is optimal for $\mathsf D(\mu_0,\mu_1)^{2}$). For notational convenience, let $\mathbf{S}_{n'}\in\mathbb R^{N_0\times N_1}$ be such that 
\[
(\mathbf{S}_{n'})_{ij}=
\begin{cases}
    S_{ij}^{n'}=(\mathbf{C}_{n'})_{ij}-(y_{n'})_i-(z_{n'})_j,&\text{ if } (i,j)\in\mathcal I(\pi_{n'}),\\
    0,&\text{ if } (i,j)\not\in\mathcal I(\pi_{n'}),  
\end{cases} 
\]
and observe that $0=\langle \mathbf{S}_{n'},\mathbf{P}_{n'}^0\rangle_{\mathrm F}=\langle \mathbf{C}_{n'},\mathbf{P}_{n'}^0\rangle_{\mathrm F}-\langle y_{n'},m_0\rangle - \langle z_{n'},n_0\rangle$ by definition, and $\langle \mathbf{S}_{n'},\mathbf{B}\rangle_{\mathrm F}\geq 0$ for any $\mathbf{B}\in[0,\infty)^{N_0\times N_1}$ as $S_{ij}^{n'}\geq 0$. As such, 
\begin{equation}
    \label{eq:CompSlack}
\begin{aligned}
    \langle \mathbf{S}_{n'},\mathbf{P}_{t_{n'}}\rangle_{\mathrm F}&= \langle \mathbf{C}_{n'},\mathbf{P}_{t_{n'}}\rangle_{\mathrm F}-\langle y_{n'},m_0+t_{n'}(n_0-m_0)\rangle -\langle z_{n'},m_1+t_{n'}(n_1-m_1)\rangle
    \\
    &= \langle \mathbf{C}_{n'},\mathbf{P}_{t_n'}-\mathbf{P}_{n'}^0\rangle_{\mathrm F}-t_{n'}\langle y_{n'},n_0-m_0\rangle -t_{n'}\langle z_{n'},n_1-m_1\rangle\geq 0.
\end{aligned}
\end{equation}
Let $a:\mathbf{P}\in\mathbb R^{N_0\times N_1}\mapsto \sum_{\substack{i,k\in[N_0]\\j,l\in[N_1]}} \mathbf{P}_{ij}\Delta (x^{(i)},y^{(j)},x^{(k)},y^{(l)})\mathbf{P}_{kl}$ denote the quadratic form whose minimization underlies the discrete Gromov-Wasserstein distance. As $a$ is smooth with (Fr{\'e}chet) derivative $Da_{[\mathbf{P}]}(\mathrm{ Q})=2\sum_{\substack{i,k\in[N_0]\\j,l\in[N_1]}} \mathbf{P}_{ij}\Delta (x^{(i)},y^{(j)},x^{(k)},y^{(l)})\mathbf{Q}_{kl}$, we have   
\[
\begin{aligned}
    a(\mathbf{P}_{t_{n'}},\mathbf{P}_{t_{n'}})-a(\mathbf{P}_{n'}^0,\mathbf{P}_{n'}^0)-\langle \mathbf{C}_{n'},\mathbf{P}_{t_{n'}}-\mathbf{P}_{n'}^0\rangle_{\mathrm F}= O\left(\|\mathbf{P}_{t_n'}-\mathbf{P}_{n'}^0\|^2\right)=O\left( t_{n'}^2   \right),
\end{aligned}
\]
so that, from \eqref{eq:CompSlack}, 
\begin{equation}
\label{eq:lowerBoundRatio}
\begin{aligned}
    \mathsf D(\mu_{0,t_{n'}},\mu_{1,t_{n'}})^2-\mathsf D(\mu_0,\mu_1)^2&\geq a(\mathbf{P}_{t_{n'}},\mathbf{P}_{t_{n'}})-a(\mathbf{P}_{n'}^0,\mathbf{P}_{n'}^0)-\langle \mathbf{S}_{n'},\mathbf{P}_{t_{n'}}\rangle_{\mathrm F}\\
&
    = a(\mathbf{P}_{t_{n'}},\mathbf{P}_{t_{n'}})-a(\mathbf{P}_{n'}^0,\mathbf{P}_{n'}^0)-\langle \mathbf{C}_{n'},\mathbf{P}_{t_{n'}}-\mathbf{P}_{n'}^0\rangle_{\mathrm F}\\& +t_{n'}\langle y_{n'},n_0-m_0\rangle +t_{n'}\langle z_{n'},n_1-m_1\rangle
    \\&= t_{n'}\langle y_{n'},n_0-m_0\rangle +t_{n'}\langle z_{n'},n_1-m_1\rangle+o(t_{n'}).
\end{aligned}
\end{equation}
Recall from \eqref{eq:potentialConnection} and \eqref{eq:discreteLPDualAlt} and the surrounding discussion that $(y_{n'},z_{n'})$ satisfy 
\begin{equation}
    \label{eq:innerProductsLowerBound}
    \begin{aligned}
    &\langle y_{n'},n_0-m_0\rangle +\langle z_{n'},n_1-m_1\rangle
    \\
    &= \sup_{(\varphi_0,\varphi_1)\in\bar{\mathcal D}_{\mathbf{A}(\pi_{n'})}}\left\{\int \varphi_{0}d(\nu_0-\mu_0)+\int \varphi_{1}d(\nu_1-\mu_1)\right\}+\varsigma_{\pi_{n'}}
    \\
    &\geq \inf_{\pi\in\bar\Pi^{\star}}\left\{ \sup_{(\varphi_0,\varphi_1)\in\bar{\mathcal D}_{\mathbf{A}(\pi)}}\left\{\int \varphi_{0}d(\nu_0-\mu_0)+\int \varphi_{1}d(\nu_1-\mu_1)\right\}+\varsigma_{\pi}\right\},
    \end{aligned}
\end{equation}
where  $\mathbf{A}(\pi)=\frac 12 \int xy^{\intercal}d\pi$, and $\bar\Pi^{\star}$ denotes the set of optimal couplings for $\mathsf D(\mu_0,\mu_1)^2$. It follows from \eqref{eq:lowerBoundRatio} and \eqref{eq:innerProductsLowerBound} that 
\[
    \begin{aligned}
        &\liminf_{t_{n'}\downarrow 0}\frac{\mathsf D(\mu_{0,t_{n'}},\mu_{1,t_{n'}})^2-\mathsf D(\mu_0,\mu_1)^2}{t_{n'}}
    \\
    &\geq \inf_{\pi\in\bar\Pi^{\star}}\left\{ \sup_{(\varphi_0,\varphi_1)\in\bar{\mathcal D}_{\mathbf{A}(\pi)}}\left\{\int \varphi_{0}d(\nu_0-\mu_0)+\int \varphi_{1}d(\nu_1-\mu_1)\right\}+\varsigma_{\pi}\right\}. 
    \end{aligned}
\]
As this final lower bound is independent of the choice of subsequence, and the initial choice of sequence $t_n\downarrow 0$ was arbitrary, 
\[
    \begin{aligned}
    &\liminf_{t\downarrow 0}\frac{\mathsf D(\mu_{0,t},\mu_{1,t})^2-\mathsf D(\mu_0,\mu_1)^2}{t}
    \\
    &\geq \inf_{\pi\in\bar\Pi^{\star}}\left\{ \sup_{(\varphi_0,\varphi_1)\in\bar{\mathcal D}_{\mathbf{A}(\pi)}}\left\{\int \varphi_{0}d(\nu_0-\mu_0)+\int \varphi_{1}d(\nu_1-\mu_1)\right\}+\varsigma_{\pi}\right\},
\end{aligned}
\]
the right hand side can be rewritten in the same form as the right hand side of \eqref{eq:limsupRatioDiscrete} by expanding $\varsigma_{\pi}$ and applying the same reasoning as the first part of the proof, proving the claim. 
\end{proof}

\begin{lemma}
    \label{lem:LipschitzContinuityDiscrete}

    For any pairs $(\nu_0, \nu_1),(\rho_0, \rho_1)\in\mathcal P_{\mu_0}\times \mathcal P_{\mu_1}$, $\left|\mathsf D(\nu_0,\nu_1)^2-\mathsf D(\nu_0,\nu_1)^2\right|\leq C\|\nu_0\otimes \nu_1-\rho_0\otimes\rho_1\|_{\infty,\mathcal F^{\oplus}}$ for some constant $C>0$ which is independent of $(\nu_0, \nu_1),(\rho_0, \rho_1)$.  
\end{lemma}
\begin{proof}
    Analogously to the proof of \cref{lem:GateauxDerivativeDiscrete},  $\mathsf D(\nu_0, \nu_1)^2$ and $\mathsf D(\nu_0, \nu_1)^2$ represent the optimal values of quadratic programs on $\mathbb R^{N_0\times N_1}$. Let $n_0,p_0,n_1,p_1$ denote the weight vectors for $\nu_0,\rho_0,\nu_1,\rho_1$ respectively. By Theorem $3$ in \cite{Klatte1985Lipschitz}, 
    \begin{equation}
    \label{eq:dicsreteLipschitz}
    \begin{aligned}
    |\mathsf D(\nu_0, \nu_1)^2-\mathsf D(\nu_0, \nu_1)^2|&\leq
    C'\|(n_0,-n_0,n_1,-n_1)-(p_0,-p_0,p_1,-p_1)\|
    \\
    &\leq \sqrt 2 C'\left(\|n_0-p_0\|+\|n_1-p_1\|\right),
    \end{aligned}
    \end{equation}
    for some constant $C'$ which is independent of $n_0,n_1,p_0,p_1$, as all vectors involved lie in a bounded set and result in nonempty feasible sets for the relevant quadratic programs. As $\mathcal F_0,\mathcal F_1$ are symmetric in the sense that if $f\in\mathcal F_0$, then $-f\in\mathcal F_0$, and similarly for $\mathcal F_1$, 
\[
    \|\nu_0\otimes \nu_1-\rho_0\otimes \rho_1\|_{\infty,\mathcal F^{\oplus}}=\sup_{f\in\mathcal F_0}\left|(\nu_0-\rho_0)(f)\right|+\sup_{f\in\mathcal F_1}\left|(\nu_1-\rho_1)(f)\right|=\|n_0-p_0\|_1+\|n_1-p_1\|_1.
\]As all norms on a finite dimensional space are equivalent, this display and \eqref{eq:dicsreteLipschitz} prove the claim.
\end{proof}

\begin{proof}[Proof of \cref{thm:discreteGWStability}]
    \cref{thm:discreteGWStability} follows directly from Lemmas \ref{lem:GateauxDerivativeDiscrete} and \ref{lem:LipschitzContinuityDiscrete}.   
\end{proof}

\subsubsection{Proof of \texorpdfstring{\cref{thm:discreteGWLimitDistribution}}{Theorem 3}}
\label{sec:proof:thm:discreteGWLimitDistribution}

The generic limit theorems follow from \cref{thm:discreteGWStability} by applying the delta method, see \cref{prop:unified}.

  To see that the empirical measure satisfies the desired condition, note that, in this setting, $\hat\mu_{0,n}$ and $\mu_0$ can be identified with vectors of weights $\hat m_{0,n}$ and $m_0\in\mathbb R^{N_0}$. By the central limit theorem,  
    \[\sqrt n \left(\hat m_{0,n}-m_0\right)\stackrel{d}{\to}N(0,\bsigma_{m_0}), 
    \]
    where $\bsigma_{m_0}$ is the multinomial covariance matrix
    \[
    (\bsigma_{\mu_0})_{ij}=\begin{cases}(m_0)_i(1-(m_0)_i),&\text{ if }i=j,
            \\
            -(m_0)_i(m_0)_j,&\text{otherwise.}
        \end{cases}
    \]
    
    Thus, $\sqrt{n}(\hat\mu_{0,n}-\mu_0)\stackrel{d}\to G_{\mu_0}$ in $\ell^{\infty}(\mathcal F_0)$ where $G_{\mu_0}$ is a $\mu_0$-Brownian bridge. An analogous result evidently holds for $\sqrt n  (\hat \mu_{1,n}-\mu_1)$. 

    As the $X_{0,i}$'s and $X_{1,i}$'s are independent, Example 1.4.6 in \cite{van1996weak} and Lemma 3.2.4 \cite{dudley2014uniform} imply that 
    \[
        \sqrt{n}\left(\hat\mu_{0,n}-\mu_0,\hat\mu_{1,n}-\mu_1\right)\stackrel{d}{\to} (G_{\mu_0},G_{\mu_1}),\text{ in }\ell^{\infty}(\mathcal F_0)\times\ell^{\infty}(\mathcal F_1). 
    \] 
As the map $(l_0,l_1)\in\ell^{\infty}(\mathcal F_0)\times\ell^{\infty}(\mathcal F_1)\mapsto l_0\otimes l_1\in\ell^{\infty}(\mathcal F^{\oplus})$ is continuous it follows from the continuous mapping theorem that 
    \[
        \sqrt{n}\left(\hat\mu_{0,n}\otimes \hat\mu_{1,n}-\mu_0\otimes \mu_1\right)\stackrel{d}{\to} G_{\mu_0\otimes \mu_1},\text{ in }\ell^{\infty}(\mathcal F^{\oplus}), 
    \]
    where $G_{\mu_0\otimes \mu_1}(f_0\oplus f_1)=G_{\mu_0}(f_0)+G_{\mu_1}(f_1)$ for every $f_0\oplus f_1\in\mathcal F^{\oplus}$.
\qed

\subsubsection{Proof of \texorpdfstring{\cref{cor:discreteGWLimitDistributionNull}}{Corollary 1}}
\label{sec:proof:cor:discreteGWLimitDistributionNull}

Throughout we assume without loss of generality that $d_0\leq d_1$ and let $\iota:x\in\mathbb R^{d_0}\to (x,0,\dots,0)\in\mathbb R^{d_1}$ denote an isometric embedding of $\mathbb R^{d_0}$ into $\mathbb R^{d_1}$.
\begin{lemma}
    \label{lem:innerProductPreserving}
   If $T_{\sharp} \bar\mu_0=\bar\mu_1$ for some isometry $T:\supp(\bar \mu_0)\to \supp(\bar\mu_1)$ then $T$ is inner product preserving. 
\end{lemma}
\begin{proof}
    Let $S:\iota(\supp(\bar\mu_0))\to\supp(\bar\mu_1)$ be an isometry defined via $S=T\circ \iota^{-1}$ and extend $S$ to an isometric self-map on $\mathbb R^{d_1}$ which we denote by the same symbol. By the Mazur-Ulam theorem, $S(y)=Uy+b$ for some orthogonal matrix $U\in\mathbb R^{d_1}\times \mathbb R^{d_1}$ and $b\in\mathbb R^{d_1}$. Observe that 
    \[
        0=\mathbb E_{\bar \mu_1}[X]=\mathbb E_{\bar \mu_0}[T(X)]=\sum_{i=1}^N\bar \mu_0(\{x^{(i)}\})\left(U\iota(x^{(i)})+b\right)= b+U\sum_{i=1}^N\bar \mu_0(\{x^{(i)}\})\iota(x^{(i)}).
    \]
    As $\mathbb E_{\bar\mu_0}[X]=0$, $\sum_{i=1}^N\bar \mu_0(\{x^{(i)}\})\iota(x^{(i)})=\iota(\mathbb E_{\bar\mu_0}[X])=0$ so that $b=0$. Conclude that, for any $x,x'\in\supp(\bar\mu_0)$, $T(x)^{\intercal}T(x')=\iota(x)^{\intercal}U^{\intercal}U\iota(x')=x^{\intercal}x'$.
\end{proof}

\begin{lemma}
\label{lem:discretePotentialsNull}
   Suppose that $T_{\sharp}\bar \mu_0=\bar \mu_1$ for some isometry $T:\supp(\bar \mu_0)\to \supp(\bar \mu_1)$. Then, for any $\mathbf A\in\mathcal A$ and $\pi\in\bar\Pi^{\star}_{\mathbf A}$, $\pi$ is induced by an isometric alignment map $T_{\pi}$, $g_{0,\pi}(x)=8x^{\intercal}\mathbb E_{\bar \mu_0}[\|X\|^2X],g_{1,\pi}(y)=8y^{\intercal}\mathbb E_{\bar \mu_1}[\|Y\|^2Y]$, and every pair of OT potentials for $\mathsf{OT}_{\mathbf A}(\bar \mu_0,\bar\mu_1)$ is of the form 
   \[
        \varphi_0(x)=-2\|x\|^4-8x^{\intercal}\bsigma_{\bar \mu_0}x+h(x),\quad \varphi_1(y)=-2\|y\|^4-8y^{\intercal}\bsigma_{\bar \mu_1}y-h(T_{\pi}^{-1}(y))
   \]
   for some $h\in\mathcal H$.
\end{lemma}
\begin{proof}
    Given that $T_{\sharp}\bar \mu_0=\bar \mu_1$, we have that $\mathsf D(\bar \mu_0,\bar \mu_1)=0$ whereby $\pi$ is a GW plan if and only if it is induced by a Gromov-Monge map $T_{\pi}:\supp(\bar \mu_0)\to \supp(\bar \mu_1)$ which may differ from $T$. In any case, $g_{0,\pi}(x)=8x^{\intercal}\mathbb E_{(X,Y)\sim \pi}\left[\|Y\|^2X\right]=8x^{\intercal}\mathbb E_{\bar \mu_0}\left[\|T_{\pi}(X)\|^2X\right]=8x^{\intercal}\mathbb E_{\bar \mu_0}\left[\|X\|^2X\right]$ where we have used the
    fact that $T_{\pi}$ is inner product preserving and hence norm preserving (see \cref{lem:innerProductPreserving}). The same arguments apply to $g_{1,\pi}$.

    From \cref{thm:VariationalMinimizers}, $\mathbf A\in\mathcal A$ if and only if  $\mathbf{A}=\frac 12 \int xy^{\intercal}d\pi(x,y)=\frac{1}{2}\int zT_{\pi}(z)^{\intercal}d\bar\mu_0(z)$ for some GW plan $\pi$. The corresponding OT  potentials $(\varphi_{0,\pi},\varphi_{1,\pi})$ for $\mathsf{OT}_{\mathbf A}(\bar \mu_0,\bar\mu_1)$ satisfy 
    \[
        \varphi_{0,\pi}(x)+\varphi_{1,\pi}(y)\leq -4\|x\|^2\|y\|^2-32x^{\intercal}\mathbf Ay
    \]
    for every $(x,y)\in\supp(\bar \mu_0)\times \supp(\bar \mu_1)$ with equality $\pi$-a.e. Observe that 
    \[
        -4\|x\|^2\|y\|^2= 2\left(\|x\|^2-\|y\|^2\right)^2-2\|x\|^4-2\|y\|^4,
    \]
    whereas
    \[
        (x-T^{-1}_{\pi}(y))^{\intercal}\mathbf A(T_{\pi}(x)-y)= -x^{\intercal}\mathbf Ay+\frac{1}{2}x^{\intercal} \bsigma_{\bar \mu_0}x+\frac{1}{2}y^{\intercal}\bsigma_{\bar \mu_1}y-x^{\intercal}\mathbf{A}y, 
    \]
    noting that $x^{\intercal}\mathbf AT_{\pi}(x)=\frac 12 \int x^{\intercal}zT_{\pi}(z)^{\intercal}T_{\pi}(x)d\bar \mu_0(z)=\frac 12 \int x^{\intercal}zz^{\intercal}xd\bar \mu_0(z)=\frac{1}{2}x^{\intercal}\bsigma_{\bar \mu_0}x$ and that 
    \[
T^{-1}_{\pi}(y)^{\intercal}\mathbf AT_{\pi}(x)=\frac 12 \int y^{\intercal}T_{\pi}(z)z^{\intercal}xd\bar \mu_0(z)=\frac 12 x^{\intercal} \int zT_{\pi}(z)^{\intercal}d\bar \mu_0(z) y=x^{\intercal}\mathbf Ay.
    \]
    In sum, 
    \[
         \begin{aligned}
         {\varphi_{0,\pi}(x)+2\|x\|^4+8x^{\intercal} \bsigma_{\bar \mu_0}x}+{\varphi_{1,\pi}(y)+2\|y\|^4+8y^{\intercal}\bsigma_{\bar \mu_1}y}&\\&\hspace{-5em}\leq 2\left(\|x\|^2-\|y\|^2\right)^2+16 (x-T^{-1}_{\pi}(y))^{\intercal}\mathbf A(T_{\pi}(x)-y),
        \end{aligned}
    \] 
    and, setting $\varphi_{0,\pi}(x)=-2\|x\|^4-8x^{\intercal}\bsigma_{\bar \mu_0}x+h(x)$ and $\varphi_{1,\pi}(y)=-2\|y\|^4-8y^{\intercal}\bsigma_{\bar \mu_1}y+g(y)$ for some functions $h:\supp(\bar \mu_0)\to \mathbb R$ and $g:\supp(\bar \mu_1)\to \mathbb R$,      
   since $\pi=(\Id,T_{\pi})_{\sharp}{\bar \mu_0}$ we must have that $h(x)+g(T_{\pi}(x))=0$ for every $x\in\supp(\bar \mu_0)$ hence $g(y)=-h(T^{-1}_{\pi}(y))$. In order for $(\varphi_{0,\pi},\varphi_{1,\pi})$ to be OT potentials we further enforce that 
   \[
        h(x)-h(T^{-1}_{\pi}(y))\leq 2\left(\|x\|^2-\|y\|^2\right)^2+16 (x-T^{-1}_{\pi}(y))^{\intercal}\mathbf A(T_{\pi}(x)-y),
   \]
   for every $(x,y)\in\supp(\bar \mu_0)\times \supp(\bar \mu_1)$. As $T^{-1}_{\pi}(y)=\supp(\bar \mu_0)$ conclude that 
   \[
        h(x)-h(x')\leq 2\left(\|x\|^2-\|x'\|^2\right)^2+16 \underbrace{(x-x')^{\intercal}\mathbf A(T_{\pi}(x)-T_{\pi}(x'))}_{=\frac{1}{2}(x-x')^{\intercal}\bsigma_{\bar \mu_0}(x-x')},
   \]
   for every $x,x'\in\supp(\bar \mu_0)$ whereby $h\in\mathcal H$.
\end{proof}

\begin{proof}[Proof of \cref{cor:discreteGWLimitDistributionNull}]
   Applying \cref{thm:discreteGWLimitDistribution} and the result of \cref{lem:discretePotentialsNull},   
    \[
        \begin{aligned}
        \sqrt n \mathsf D(\hat\mu_{0,n},\hat\mu_{1,n})^2&
        \stackrel{d}{\to}
\inf_{\mathbf{A}\in\mathcal A}
            \inf_{\pi \in \bar \Pi^{\star}_{\mathbf{A}}} \sup_{(\varphi_0,\varphi_1)\in\bar {\mathcal D}_{\mathbf{A}}}\left\{\chi_{\mu_0\otimes \mu_1} \big(( f_0+g_{0,\pi}+\bar\varphi_0)\oplus (f_1+g_{1,\pi}+\bar\varphi_1)
        \right\},
    \end{aligned} 
    \]
    where $g_{0,\pi}(x)=8x^{\intercal}\mathbb E_{\bar \mu_0}[\|X\|^2X],g_{1,\pi}(y)=8y^{\intercal}\mathbb E_{\bar \mu_1}[\|Y\|^2Y]$ which are independent of $\pi$ and $\mathbf{A}$. Furthermore, for any $\mathbf A\in\mathcal A$ and $\pi\in\bar \Pi_{\mathbf A}$, $\pi$ is induced by an (isometric) alignment map $T_{\pi}:\supp(\bar \mu_0)\to \supp(\bar \mu_1)$  and every pair of OT potentials for $\mathsf{OT}_{\mathbf A}(\bar \mu_0,\bar\mu_1)$ is of the form 
   \[
        \varphi_0(x)=-2\|x\|^4-8x^{\intercal}\bsigma_{\bar \mu_0}x+h(x),\quad \varphi_1(y)=-2\|y\|^4-8y^{\intercal}\bsigma_{\bar \mu_1}y-h(T_{\pi}^{-1}(y))
   \]
   for some $h\in\mathcal H$.

   We now set $\psi_0(x)=-2\|x\|^4-8x^{\intercal}\bsigma_{\bar \mu_0}x,\psi_1(y)=-2\|y\|^4-8y^{\intercal}\bsigma_{\bar \mu_1}x$ and show that $f_0+g_0+\bar\psi_0$ is constant so that $f_1+g_1+\bar\psi_1$ is constant by symmetry. Then, $\sqrt n (\mu_{0,n}\otimes \mu_{1,n}-\mu_0\otimes \mu_1)(f_0+g_0+\bar\varphi_0\oplus f_1+g_1+\bar\varphi_1)=\sqrt n (\mu_{0,n}\otimes \mu_{1,n}-\mu_0\otimes \mu_1)(h(\cdot-\mathbb E_{\mu_0}[X])\oplus -h(T^{-1}(\cdot-\mathbb E_{\mu_1}[X])))$ so that 
   \[
        \chi_{\mu_0\otimes \mu_1} ( f_0+g_{0,\pi}+\bar\varphi_0\oplus f_1+g_{1,\pi}+\bar\varphi_1)=\chi_{\mu_0\otimes \mu_1}\big(h(\cdot-\mathbb E_{\mu_0}[X])\oplus -h(T^{-1}(\cdot-\mathbb E_{\mu_1}[X]))\big).
   \]

   To this end, we first expand $f_0(x)=2\int\|x-x'\|^4d\mu_0(x')-4\int \|x'-\mathbb E_{\mu_0}[X]\|^2\|x-\mathbb E_{\mu_0}[X]\|^2d\mu_0(x')$ (recall that $M_2(\bar \mu_0)=M_2(\bar \mu_1)$ in this case) using the formula 
   \[
   \begin{aligned}
        &2\int \|x-x'\|^4d\mu_0(x')\\&=2\int \left(\|x-\mathbb E_{\mu_0}[X]\|^2+2(x-\mathbb E_{\mu_0}[X])^{\intercal}(\mathbb E_{\mu_0}[X]-x')+\|\mathbb E_{\mu_0}[X]-x'\|^2\right)^2d\mu_0(x'),
        \\
        &=2\int \|x-\mathbb E_{\mu_0}[X]\|^4+4\left|(x-\mathbb E_{\mu_0}[X])^{\intercal}(\mathbb E_{\mu_0}[X]-x')\right|^2+\|\mathbb E_{\mu_0}[X]-x'\|^4d\mu_0(x')
        \\
        &+4\int \|x-\mathbb E_{\mu_0}[X]\|^2\|\mathbb E_{\mu_0}[X]-x'\|^2+2(x-\mathbb E_{\mu_0}[X])^{\intercal}(\mathbb E_{\mu_0}[X]-x')\|\mathbb E_{\mu_0}[X]-x'\|^2\\&\hspace{19em}+ 2(x-\mathbb E_{\mu_0}[X])^{\intercal}(\mathbb E_{\mu_0}[X]-x') \|x-\mathbb E_{\mu_0}[X]\|^2d\mu_0(x'),
        \end{aligned}
   \]
   and note that $\bar\psi_0(x)=-2\|x-\mathbb E_{\mu_0}[X]\|^4-8\int \left|(x-\mathbb E_{\mu_0}[X])^{\intercal}(\mathbb E_{\mu_0}[X]-x')\right|^2d\mu_0(x')$, and $g_0(x)=8(x-\mathbb E_{\mu_0}[X])^{\intercal}\mathbb E_{\mu_0}[\|X-\mathbb E_{\mu_0}[X]\|^2(X-\mathbb E_{\mu_0}[X])]+8\mathbb E_{\mu_0}[X]^{\intercal}\mathbb E_{\mu_0}[\|X-\mathbb E_{\mu_0}[X]\|^2(X-\mathbb E_{\mu_0}[X])]$, whence
   \begin{equation}
   \label{eq:expansionSumFcns}
   \begin{aligned}
        f_0+g_0+\bar \psi_0&=\int \|\mathbb E_{\mu_0}[X]-x'\|^4d\mu_0(x')+8\mathbb E_{\mu_0}[X]^{\intercal}\mathbb E_{\mu_0}[\|X-\mathbb E_{\mu_0}[X]\|^2(X-\mathbb E_{\mu_0}[X])]
        \\
        &+8\underbrace{\int\|x-\mathbb E_{\mu_0}[X]\|^2(x-\mathbb E_{\mu_0}[X])^{\intercal}(\mathbb E_{\mu_0}[X]-x')d\mu_0(x')}_{=0},
   \end{aligned}
   \end{equation}
   proving the claim.

   Conclude that
   \[
   \begin{aligned}
    \sqrt n\mathsf D(\mu_{0,n}, \mu_{1,n})^2&\stackrel{d}{\to} \inf_{\mathbf A\in\mathcal A}\inf_{\pi\in \bar{\Pi}_{\mathbf A}^{\star}}\sup_{h\in\mathcal H}\left\{\chi_{\mu_0\otimes \mu_1}\big(h(\cdot-\mathbb E_{\mu_0}[X])\oplus -h(T_{\pi}^{-1}(\cdot-\mathbb E_{\mu_1}[X]))\big)\right\}
   \\ 
    &=\inf_{T\in\bar{\mathcal T}}\sup_{h\in\mathcal H}\left\{\chi_{\mu_0\otimes \mu_1}\big(h(\cdot-\mathbb E_{\mu_0}[X])\oplus -h(T^{-1}(\cdot-\mathbb E_{\mu_1}[X]))\big)\right\},%
    \end{aligned} 
   \]
   proving the first claim.

   Specifying this result for the empirical measures, we have from  \cref{thm:discreteGWLimitDistribution} that $\chi_{\mu_0\otimes \mu_1}(f_0\oplus f_1)=G_{\mu_0}(f_0)+G_{\mu_1}(f_1)$ for every $f_0\oplus f_1\in\mathcal F^{\oplus}$. It remains to show that $G_{\mu_1}\big(-{h( T^{-1}(\cdot-\mathbb E_{\mu_1}[X]))}\big)=G_{\mu_0}\big(-h(\cdot-\mathbb E_{\mu_0}[X])\big)$ for any $T\in\bar{\mathcal T}$. To this end, note that if $f_1(\cdot-\mathbb E_{\mu_1}[X]),f_1'(\cdot-\mathbb E_{\mu_1}[X])\in\mathcal F_1$, 
   \[\begin{aligned}\Cov(G_{\mu_1}(f_1(\cdot-\mathbb E_{\mu_1}[X])),G_{\mu_1}(f_1'(\cdot-\mathbb E_{\mu_1}[X])))&
  \\ 
   &\hspace{-5em}=\Cov_{\bar \mu_0}(f_1\circ T,f_1'\circ T)
    \\ 
   &\hspace{-5em}=\Cov(G_{\mu_0}\big({f_1( T(\cdot-\mathbb E_{\mu_0}[X] ))}\big),G_{\mu_0}({f_1'( T(\cdot-\mathbb E_{\mu_0}[X] ))}\big)),
  \end{aligned} 
   \]  
   so that $\sqrt n\mathsf D(\hat \mu_{0,n}, \hat \mu_{1,n})^2\stackrel{d}{\to} \sup_{h\in\mathcal H}\left\{G_{\mu_0}(h(\cdot-\mathbb E_{\mu_0}[X]))+G_{\mu_0}'(-h(\cdot-\mathbb E_{\mu_0}[X]))\right\},$ where $G_{\mu_0}'$ is an independent copy of $G_{\mu_0}$. As $\mathcal H$ is symmetric in the sense that if $h\in\mathcal H$, then $-h\in\mathcal H$, this limit can be expressed as  $\sqrt 2\|G_{\mu_0}\|_{\infty,\bar{\mathcal H}}$. 
\end{proof}

\subsubsection{Proof of \texorpdfstring{\cref{thm:directEstimator}}{Theorem 5}}
\label{sec:proof:thm:directEstimator}

  We will prove that the direct estimation approach proposed in 
\cref{sec:directEstimation}  for estimating the limit $L_{\mu_0}$ in \cref{thm:discreteGWLimitDistribution} is consistent under the following, general, conditions. This enables us to treat both, by the same token, the empirical measures and the measures defined in \cref{sec:graphApplication} for testing graph isomorphisms.   

\begin{assumption}
   \label{assn:simplifiedForm}  
   The first set of assumptions regard the structure of  $\chi_{\mu_0\otimes \mu_1}$ from \cref{assn:weakConvergence}.
   \begin{itemize}
       \item  $\chi_{\mu_{0}\otimes \mu_1}(f_0\oplus f_1)=\chi_{\mu_0}(f_0)+\chi_{\mu_1}(f_1)$ for $f_0\oplus f_1\in\mathcal F^{\oplus}$, where $\chi_{\mu_0}$ and $\chi_{\mu_1}$ are tight 
 mean-zero Gaussian processes on $\mathcal F_0$ and $\mathcal F_1$,
    \item For any Gromov-Monge map $T$ between $\bar \mu_0$ and $\bar \mu_1$, $\chi_{\mu_1}(g_1(\cdot-\mathbb E_{\mu_0}[X]))=\chi_{\mu_0}(g_1(T(\cdot-\mathbb E_{\mu_1}[X])))$ whenever  $g_1(\cdot-\mathbb E_{\mu_0}[X])\oplus g_1(T(\cdot-\mathbb E_{\mu_1}[X]))\in\mathcal F^{\oplus}$,
    \item  There exists a  mean-zero Gaussian vector, $Z\in\mathbb R^{N_0}$, with  covariance $\bsigma_Z$ for which  $\chi_{\mu_0}(h) = Z^{\intercal}u_h$ for every $h\in \bar{\mathcal H}$, where $u_h=(h(\bar x^{(1)}),\dots,h(\bar x^{(N_0)}))\in\mathbb R^{N_0}$ for $(\bar x^{(i)})_{i=1}^{N_0}=\supp(\bar\mu_0)$.
   \end{itemize}
    The remaining assumptions pertain to the construction of the direct estimator. 
    \begin{itemize}
        \item $\mu_{0,n}$ and $\bsigma_{n}$ are estimators of $\mu_0$ and $\bsigma_Z$ constructed from samples $Y_1,\dots, Y_n$ from an auxiliary distribution $\nu$ on some (perhaps distinct) space,
        \item Conditionally on $Y_1,Y_2,\dots,$  $\mathbb E_{\mu_{0,n}}[X]\to \mathbb E_{\mu_0}[X]$, ${\bsigma}_{\bar \mu_{0,n}}\to \bsigma_{\bar \mu_0}$, and $\bsigma_{n}\to \bsigma_Z$ deterministically given almost every realization of $Y_1,Y_2,\dots$
    \end{itemize}
\end{assumption}

Under \cref{assn:simplifiedForm},  $L_{\mu_0}=\sqrt 2\|\chi_{\mu_0}\|_{\infty,\bar{\mathcal H}}=\sqrt 2\sup_{h\in\bar{\mathcal H}}Z^{\intercal} u_h$ as follows from the proof of \cref{cor:discreteGWLimitDistributionNull}. This suggests setting $(\bar x^{(i)})_{i=1}^{N_n}=\supp(\bar \mu_{0,n})$ and estimating $L_{\mu_0}$ as 
\begin{equation}
\begin{gathered}
\label{eq:directEstimatorGeneral}
    L_n =\sqrt{2} \sup_{u\in\mathcal H_n} Z_n^{\intercal} u\text{ for }Z_n\sim N(0,\Sigma_n)\text{ and } \\  
   {\mathcal H}_n=\big\{u\in\mathbb R^{N_n}\:|\: u_i-u_j\leq 2(\|\bar x^{(i)}\|^2-\|\bar x^{(j)}\|^2)^2
    +8(\bar x^{(i)}-\bar x^{(j)})^{\intercal}\bsigma_{\bar{\mu}_{0,n}}(\bar x^{(i)}-\bar x^{(j)}
    ),%
    i\neq j\in[N_{n}]\big\},
\end{gathered}
\end{equation}
\cref{alg:directEstimatorGeneral} summarizes the procedure for sampling from $L_n$.

\begin{algorithm}
\caption{Generate samples from $ L_n$}\label{alg:directEstimatorGeneral}
\begin{algorithmic}[1]
\Statex Given samples $Y_{1},\dots, Y_n$ from $\nu$,  
\State Let $(\bar x^{(i)})_{i=1}^{N_{n}}=\supp(\bar \mu_{0,n})$.   
\State Let  $b_n\in\mathbb R^{N_n(N_n-1)}$ and $\mathbf A\in\mathbb R^{N_n(N_n-1)\times N_n}$ be such that ${\mathcal H}_n=\{u\in\mathbb R^{N_n}:\mathbf Au\leq  b_n\} $.
\State Sample $Z_n\sim N(0,\bsigma_n)$ and solve $\sqrt 2\sup_{u\in{\mathcal H}_n}Z_n^{\intercal}u\sim  L_n$, padding $u$ by zeros if necessary. Repeat step $3$ to generate more samples. 
\end{algorithmic}
\end{algorithm}

We now establish the following generalization of \cref{thm:directEstimator}
\begin{theorem}
\label{thm:directEstimatorGeneralized}
In the setting of \cref{cor:discreteGWLimitDistributionNull} and \cref{assn:simplifiedForm}, if $\mathbb P(Z_n^{\intercal}\mathbbm 1_{N_0}= 0)=\mathbb P(Z^{\intercal}\mathbbm 1_{N_0}= 0)=1$ for all $n$ sufficiently large and $Z$ is not a point mass, then     
\[
\lim_{n\to \infty}\sup_{t\geq 0}\left|\mathbb P(L_n\leq t|Y_1,Y_2,\dots,Y_n)-\mathbb P( L_{\mu_0}\leq t)\right|= 0,
\]
for almost every realization of $Y_1,Y_2,\dots$   
\end{theorem}

The proof of \cref{thm:directEstimatorGeneralized} depends on the following technical lemmas.

\begin{lemma}
\label{lem:Hnonempty}
      $\mathcal H$, treated as a subset of $\mathbb R^{N_0}$, has empty interior if and only if $\mu_0$ is a point mass.
\end{lemma}
\begin{proof}
     For any $x,x'\in\supp(\bar \mu_0)$,
   \[
        (x-x')^{\intercal}\bsigma_{\bar \mu_0}(x-x')=\sum_{z\in\supp(\bar \mu_0)}\bar \mu_0(\{z\})(x-x')^{\intercal}zz^{\intercal}(x-x')= \sum_{z\in\supp(\bar \mu_0)}\bar \mu_0(\{z\})(z^{\intercal}x-z^{\intercal}x')^2,
   \]
   so that $(x-x')^{\intercal}\bsigma_{\bar \mu_0}(x-x')=0$ if and only if $x=x'$. To see this, observe that, {for distinct points $x,x'\in\supp(\bar\mu_0)$, one of $\|x\|^2-x^{\intercal}x'$ or $\|x'\|^2-x^{\intercal}x'$ is nonzero as otherwise $\|x\|^2=\|x'\|^2=x^{\intercal}x'$ which holds if and only if $x=x'$.} Let $C=8\min_{\substack{x,x'\in \supp(\bar \mu_0)\\x\neq x'}}(x-x')^{\intercal}\bsigma_{\bar\mu_0}(x-x')$, then any function, $h$, on $\supp(\bar\mu_0)$ taking values in $[a,a+C]$ for any $a\in\mathbb R$ is such that $h\in\mathcal H$.
\end{proof}

We proceed with the proof of \cref{thm:directEstimatorGeneralized} by setting some notation. For any $c\in\mathbb R^{N_0}$ and 
 any strictly positive vector $b=(b_{ij})_{\substack{i,j\in[N_0]\\i\neq j}} \in\mathbb R^{N_0(N_0-1)}_{>0}$, let $\mathrm{LP}(c,b)$ denote the linear program 
    \[
    \begin{aligned}
    \sup_{u\in \mathbb R^{N_0}} \quad \langle c,u\rangle\hspace{1em}& \\
    \text{subject to} \hspace{-0.95em} \qquad
    u_i-u_j &\leq b_{ij}, \forall i,j\in[N_0], i\neq j,
\end{aligned}
\]
$v(c,b)$ denote the corresponding optimal value, and $D= \mathbb R^{N_0}\times \mathbb R^{N_0(N_0-1)}_{>0}$ denote the set of all such pairs $(c,b)$. Remark that if $u$ satisfies the constraints of $\mathrm{LP}(c,b)$ then so too does $u+a$ for any $a\in\mathbb R^{N_0}$. As such,  $\mathrm{LP}(c,b)$ is unbounded and, unless $c^{\intercal}\mathbbm 1_{N_0}=0$, $v(c,b)=\infty$. On the other hand, if $c^{\intercal}\mathbbm 1_{N_0}=0$ we have that, for any feasible $u$, $\tilde u = u-u_1+a$ achieves the same objective value as $u$ for any fixed $a\in\mathbb R$. Observe that $\tilde u_1=a$ and $\tilde u_i\in [a-b_{1i},b_{i1}+a]$ for every $i\neq 1$. We thus consider the auxiliary problem $\mathrm{LP}_{\delta}(c,b)$ given by 
\[
    \begin{aligned}
    \sup_{u\in \mathbb R^{N_0}} \quad \langle c,u\rangle\hspace{1em}& \\
    \text{subject to} \hspace{-0.95em} \qquad
    u_i-u_j &\leq b_{ij}, \forall i,j\in[N_0], i\neq j,
    \\
    u_1&\leq \delta, -u_1\leq \delta, 
\end{aligned}
\] for any $\delta>0$. For any strictly positive $b$, this problem is feasible and bounded hence solvable and, if $c^{\intercal}\mathbbm 1_{N_0}=0$, its optimal value function $v_{\delta}(c,b)$ coincides with $v(c,b)$.

It is useful to note that this orthogonality condition is satisfied if $c\sim N(0,\bsigma_p)$, where $\bsigma_p$ is the covariance matrix in \eqref{eq:multinomialCovariance} and $p$ is a vector of probabilities.

\begin{lemma}
\label{lem:orthogonality}
Fix a nonnegative vector $p\in\mathbb R^{N_0}$ whose entries sum to $1$. If $Z\sim N(0, \bsigma_p)$, then $Z^{\intercal}\mathbbm 1_{N_0}=0$ almost surely.
\end{lemma}
\begin{proof}
   The $k$-th entry of the vector $\bsigma_p\mathbbm 1_{N_0}$ is given by $p_k(1-p_k)-\sum_{j\neq k}p_kp_j=p_k(1-1)=0$. Conclude that $Z^{\intercal}\mathbbm 1_{N_0}$ has mean and variance zero so that it takes the value $0$ with probability one. 
\end{proof}

By the previous discussion and the assumption that $\mathbb P(Z_n^{\intercal}\mathbbm 1_{N_0}\neq 0)=0$ for all $n$ sufficiently large, $v(Z_n, b_n)=v_{\delta}(Z_n,b_n)$ is almost surely finite for large $n$, where 
 $b_{n}=( b_{ij,n})_{\substack{i,j\in[N_0]\\i\neq j}}$ is the vector with 
\[
b_{ij,n}=2(\|\bar x^{(i)}\|^2-\|\bar x^{(j)}\|^2)^2+8(\bar x^{(i)}-\bar x^{(j)})^{\intercal}\bsigma_{\bar {\mu}_{0,n}}(\bar x^{(i)}-\bar x^{(j)}),i,j\in[N_0],i\neq j,
\]
which is known to be positive once $\mu_{0,n}$ is not a point mass (see \cref{lem:Hnonempty}), here $\bar x^{(i)} = x^{(i)}-\mathbb E_{\mu_{0,n}}[X]$ for ${i=1,\dots, N_0}$. We define $\bar b$ by analogy with $\mu_0$ in place of $\mu_{0,n}$ and remark that the same implications hold for $\mathrm{LP}(Z,\bar b)$ with $Z\sim N(0, \bsigma_{Z})$.

We now establish some technical lemmas which will be  useful in the sequel.

\begin{lemma}
\label{lem:primalDualSoln}
Fix $\delta>0$, $c\in\mathbb R^{N_0}$ and $\beta \in\mathbb R^{N_0(N_0-1)}_{>0}$ satisfying $\beta_{ij}\in (\underline{\beta},\overline{\beta})$ for some $0<\underline{\beta}< \overline{\beta}$ for $i,j\in[N_0],i\neq j$. Then any dual solution, $p^{\star}$, to  $\mathrm{LP}_{\delta}(c,\beta)$ satisfies 
$\|p^{\star}\|_1\leq \frac{\|c\|_1(\delta+\overline{\beta})}{\min\{\underline{\beta},\delta\}}$ and any primal solution $u^{\star}$ satisfies $\|u^{\star}\|_{\infty}\leq \delta +\overline{\beta}$. 
\end{lemma}
\begin{proof}
    The dual problem to $\mathrm{LP}_{\delta}(c,\beta)$ is given by 
 \[
    \begin{aligned}
    \inf_{p\in \mathbb R^{N_0(N_0-1)+2}} \quad \langle (\beta,\delta,\delta),p\rangle\hspace{-2.7em}& \\
    \text{subject to}\hspace{1.2em}  \qquad
    \mathbf A^{\intercal}p&=c,\;p\geq 0, 
\end{aligned}
\]   
where $\mathbf A$ is the constraint matrix from the primal problem. By definition, $p^{\star}$ satisfies $v_{\delta}(c,\beta) = \langle (\beta,\delta,\delta),p^{\star}\rangle\geq \min (\beta,\delta,\delta)\| p^{\star}\|_1$ so that $\|p^{\star}\|_1\leq \frac{ v_{\delta}(c,\beta)}{\min\{\underline {\beta},\delta\}}$.

By the discussion at the beginning of this section, a solution to the primal problem, $u^{\star}$, satisfies $|u_1^{\star}|\leq \delta$ and $u_i^{\star}\in [-\delta-\underline {\beta},\delta +\overline {\beta}]$ so that $\|u^{\star}\|_{\infty}\leq \delta +\overline {\beta}$. Conclude from H\"older's inequality that $v_{\delta}(c,{\beta})\leq \|c\|_1(\delta+\overline {\beta})$. 
\end{proof}

\begin{lemma}
\label{lem:uniformConvergence}
Fix $\delta>0$, $0<\underline {\beta}<\overline {\beta}$,  and let $D_{(\underline {\beta},\overline {\beta} )}=\{ {\beta}\in\mathbb R^{N_0(N_0-1)}:{\beta}_{ij}\in (\underline {\beta},\overline {\beta}),i,j\in[N_0],i\neq j\}$. Then, if $\beta_n\to {\beta} \in D_{(\underline {\beta},\overline {\beta})}$, $v_{\delta}(\cdot,\beta_n)\to v_{\delta}(\cdot,{\beta})$ uniformly on compact sets. 
\end{lemma}
\begin{proof}
Applying Theorem 5.2 in \cite{Gisbert2019} along with \cref{lem:primalDualSoln}, we have that the map $(c,{\beta})\in \mathbb R^{N_0}\times D_{[\underline {\beta},\overline {\beta}]}\mapsto v_{\delta}(c,{\beta})$ is locally Lipschitz. As such, for any compact set $K\subset \mathbb R^{N_0}$, there exists a constant $L_K$ for which 
\[
    \left|v_{\delta}(c,\beta_n)-v_{\delta}(c,{\beta})\right|\leq L_K\| \beta_n-{\beta}\|
\]
for every $c\in K$ and $n$ sufficiently large that $\beta_n\in D_{(\underline {\beta} ,\overline {\beta})}$, proving the claim. 
\end{proof}

\begin{proof}[Proof of \cref{thm:directEstimatorGeneralized}]
By assumption, we have that, conditionally on $Y_1,Y_2,\dots,$ $\mathbb E_{\mu_{0,n}}[X]\to \mathbb E_{\mu_0}[X]$, $\bsigma_{\bar \mu_{0,n}}\to \bsigma_{\bar \mu_0},$ and $\bsigma_n\to \bsigma_{Z}$ deterministically given almost every realization of $Y_1,Y_2,\dots$ so that $b_n\to b$ and $Z_n\sim N(0,\Sigma_n)\stackrel{d}{\to} Z\sim N(0,\Sigma_Z) $ conditionally on the data. It follows from the proof of \cref{lem:Hnonempty} that $b\in D_{(\underline b,\overline b)}$ for some $0<\underline b<\overline b$.  By \cref{lem:uniformConvergence}, $v_{\delta}(\cdot,b_n)\to v_{\delta}(\cdot,b)$ uniformly on compact sets conditionally on the data so that, applying the continuous mapping theorem, conditionally on $Y_1,Y_2,\dots,$ $v_{\delta}(Z_n,b_n)\stackrel{d}{\to} v_{\delta}(Z,b)$ given almost every realization of $Y_1,Y_2,\dots$ As  $Z^{\intercal}\mathbbm 1_{N_0}=0$ and 
$Z_n^{\intercal}\mathbbm 1_{N_0}=0$ with probability $1$ for $n$ sufficiently large, $v_{\delta}$ can be replaced by $v$ so that,  conditionally on $Y_1,Y_2,\dots,$ $L_n \stackrel{d}{\to} L_{\mu_0}$ given almost every realization of $Y_1,Y_2,\dots$ The first claim follows by applying Lemma 2.11 in \cite{vanderVaart1998asymptotic}, noting that the distribution function of $L_{\mu_0}$ is continuous as follows from Proposition 12.1 in \cite{Davydov1998}.
 
\end{proof}

\begin{proof}[Proof of \cref{thm:directEstimator}]
    If $\mu_{0,n}=\hat \mu_{0,n}$ and $\bsigma_{n}=\bsigma_{m_n}$, it follows from the law of large numbers that, conditionally on $Y_1,Y_2,\dots,$ $\mathbb E_{\mu_{0,n}}[X]\to \mathbb E_{\mu_0}[X]$, $\bsigma_{\bar \mu_{0,n}}\to \bsigma_{\bar \mu_0}$, and $\bsigma_{m_n}\to \bsigma_m$, given almost every realization of $Y_1,Y_2,\dots,$ where $m$ is the weight vector for $\mu_0$. By applying \cref{lem:orthogonality}, the orthogonality condition for $Z_n$ and $Z$ always hold. The remaining assumptions regarding the structure of $\chi_{\mu_0\otimes \mu_1}$ are verified in the proof of \cref{cor:discreteGWLimitDistributionNull} save for the fact that $G_{\mu_0}(h)=Z^{\intercal}u_h$ for every $h\in\bar{\mathcal H}$ where $Z\sim N(0,\bsigma_m)$, but this follows directly from the fact that $G_{\mu_0}$ is a $\mu_0$-Brownian bridge process.
\end{proof}

\subsection{Proofs for \texorpdfstring{\cref{sec:semidiscreteGW}}{Section 4}}

\label{sec:proof:thm:existenceMongeMap}

Throughout, we will assume that the following variant of the  twist condition holds.

\begin{assumption}
    \label{assn:semidiscreteSuffCond}
    $\mu_0\in\mathcal P(\mathcal X_0)$ is absolutely continuous with respect to the Lebesgue measure and,
    for every $i\neq j$ with $i,j\in\{1,\dots,N\}$ and $t\in\mathbb R$, we have that
    \[
        \mu_0\left( \left(c%
        (\cdot,y^{(i)})-c%
        (\cdot,y^{(j)})  \right)^{-1}(t)  \right)=0.%
    \]
\end{assumption}

Specializing to $c_\mathbf{A}(x,y)=-4\|x\|^2\|y\|^2-32x^{\intercal}\mathbf{A}y$ from \eqref{eq:GWVariational}, the condition above is related to $\mathcal X_1$ and holds under the following primitive condition.  

\begin{lemma}[Necessary and sufficient condition on $\mathcal X_1$]
    \label{prop:semidiscreteSuffCond}
    Let $\mu_0\in\mathcal P(\mathcal X_0)$ be absolutely continuous with respect to the Lebesgue measure. Then, \cref{assn:semidiscreteSuffCond} is satisfied with $c_\mathbf{A}$ at $\mathbf{A}\in\RR^{d_0\times d_1}$ if and only if $\mathcal X_1$ is such that $y^{(i)}-y^{(j)}\not\in\ker(\mathbf{A})$ for every $i\neq j\in\{1,\dots,N\}$ with $\|y^{(i)}\|=\|y^{(j)}\|$. In particular, if $\|y^{(i)}\|\neq\|y^{(j)}\|$ for every $i\neq j$, then \cref{assn:semidiscreteSuffCond} holds for any $\mathbf{A}\in\mathbb R^{d_0\times d_1}$.   
\end{lemma}

\begin{proof}
    Fix $i\neq j\in\{1,\dots, N\}$ and $\mathbf{A}\in\mathbb R^{d_0}\times \mathbb R^{d_1}$. Then, $c_{\mathbf{A}}(x,y^{(i)})-c_{\mathbf{A}}(x,y^{(j)})=-4\|x\|^2(\|y^{(i)}\|^2-\|y^{(j)}\|^2)-32 x^{\intercal}\mathbf{A}(y^{(i)}-y^{(j)})$ so that $c_{\mathbf{A}}(x,y^{(i)})-c_{\mathbf{A}}(x,y^{(j)})$ is constant as a function of $x$ on an open set $U\subset \mathbb R^{d_0}$ if and only if $-8x(\|y^{(i)}\|^2-\|y^{(j)}\|^2)=32 \mathbf{A}(y^{(i)}-y^{(j)})$ for every $x\in U$. Observe that the prior equality holds on a nonnegligible set with respect to the Lebesgue measure if and only $y^{(i)}-y^{(j)}\in\ker(\mathbf{A})$ and $\|y^{(i)}\|=\|y^{(j)}\|$. Given that $\mu_0$ is absolutely continuous with respect to the Lebesgue measure \cref{assn:semidiscreteSuffCond} holds at $\mathbf{A}$ if and only if $y^{(i)}-y^{(j)}\not\in\ker(\mathbf{A})$ whenever $\|y^{(i)}\|=\|y^{(j)}\|$.
\end{proof}

\cref{prop:semidiscreteSuffCond} provides a necessary and sufficient condition for \cref{assn:semidiscreteSuffCond} to hold at every $\mathbf{A}\in\mathbb R^{d_0\times d_1}$ simultaneously. This is crucial for our study of limit distributions of the empirical semi-discrete GW distance to guarantee, e.g., uniqueness of OT plans for $\OT_{\mathbf{A}}(\mu_0,\mu_1)$ uniformly in $\mathbf{A}$. This uniqueness will lead to a substantial simplification for the resulting limit, as compared to the discrete case (\cref{thm:discreteGWLimitDistribution}).   

Under \cref{assn:semidiscreteSuffCond}, solutions of $\OT_{\mathbf{A}}(\mu_0,\mu_1)$ admit the following characterization.

\begin{lemma}[On solutions of $\OT_{\mathbf{A}}(\mu_0,\mu_1)$] 
    \label{prop:semidiscreteOTSolns} Fix $\mathbf{A}\in\mathbb R^{d_0\times d_1}$ and let $z^{\mathbf{A}}$ be an optimal vector for $\OT_{\mathbf{A}}(\mu_0,\mu_1)$.
    Under \cref{assn:semidiscreteSuffCond},    the OT plan is unique and is induced by the map
    \begin{equation}
        \label{eq:optimalMap}
        T_{\mathbf{A}}:x\in\mathcal X\mapsto y^{(I_{z^{\mathbf{A}}}(x))},\text{ where } I_{z^{\mathbf{A}}}(x)\in\argmin_{1\leq i \leq N}\left( c_{\mathbf{A}}(x,y^{(i)})-z^{\mathbf{A}}_i \right),
    \end{equation}
    which is uniquely defined $\mu_0$-a.e.
\end{lemma}

\begin{proof}
      We first address uniqueness of the OT plan. For any $z\in\mathbb R^N$, let $\phi_z:x\in\mathcal X_0\mapsto\min_{1\leq i\leq N}\left\{ c_{\bm A}(x,y^{(i)})-z_i \right\}$.  If $x\in\mathcal X_0$ is such that $\phi_z(x)+\phi^c_z(y^{(i)})=c_{\bm A}(x,y^{(i)})$ and $\phi_z(x)+\phi^c_z(y^{(j)})=c_{\bm A}(x,y^{(j)})$ for some $i\neq j\in[N]$, then $c_{\bm A}(x,y^{(j)})-c_{\bm A}(x,y^{(i)})=\phi^c_z(y^{(j)})-\phi^c_z(y^{(i)})$. \cref{assn:semidiscreteSuffCond} guarantees that the previous equality occurs on a $\mu_0$-negligible
    set so that we can apply Theorem 5.30 in \cite{villani2008optimal} whereby the OT plan for $\OT_{\bm A}(\mu_0,\mu_1)$ is unique and induced by a map.  

    We now derive the expression for the optimal map. Let $z^{\bm A}$ be an optimal vector for $\mathsf{OT}_{\mathbf{A}}(\mu_0,\mu_1)$ and $\varphi_0^{\bm A}$ be the corresponding OT potential. For convenience, set $\zeta(y^{(i)})=z^{\bm A}_i$ for $i=1,\dots,N$. By strong duality, 
    \[
        \int \varphi_0^{\bm A}(x)+\zeta(y)-c_{\bm A}(x,y)d\pi_{\bm A}(x,y)=0. 
    \]
    As $\varphi_0^{\bm A}(x)+\zeta(y)\leq c_{\bm A}(x,y)$ for every  $(x,y)\in\mathcal X_0\times \mathcal X_1$ by definition, it follows that $\varphi_0^{\bm A}(x)+\zeta(y)= c_{\bm A}(x,y)$ $\pi_{\bm A}$-a.e. As $\pi_{\bm A}$ is induced by a map $T_{\bm A}:\mathcal X_0\to \mathcal X_1$, $\varphi_0^{\bm A}(x)= c_{\bm A}(x,T_{\bm A}(x))-\zeta(T_{\bm A}(x))$ $\mu_0$-a.e. Since $\varphi_0^{\bm A}(x)= \min_{1\leq i \leq N}\left\{c_{\bm A}(x,y^{(i)})-\zeta(y^{(i)})\right\}$, it holds that $T_{\bm A}(x)=y^{I_{z^{\bm A}}(x)}$ (recalling that $I_{z^{\bm A}}(x)$ is an arbitrary element of $\argmin_{1\leq i \leq N}\left\{c_{\bm A}(x,y^{(i)})-\zeta(y^{(i)})\right\}$). 
    \cref{assn:semidiscreteSuffCond} further guarantees that 
    $I_{z^{\bm A}}(x)$ is a singleton for $\mu_0$-a.e. $x\in\mathcal X_0$.

\end{proof}

 \cref{thm:existenceMongeMap} is a direct consequence of Lemmas
\ref{prop:semidiscreteSuffCond} and
\ref{prop:semidiscreteOTSolns}.

\subsubsection{Proof of \texorpdfstring{\cref{thm:semidiscreteGWStability}}{Theorem 5}}
\label{sec:proof:thm:semidiscreteGWStability}
 In this section, the centering required to access the variational form of the GW distance \eqref{eq:GWVariational} will play a prominent role. Of note is that if $(\nu_0, \nu_1)\in\mathcal P(\mathcal X_0)\times\mathcal P(\mathcal X_1)$, we generally do not have that   $\supp(\bar \nu_0)$ or $\supp(\bar \nu_1)$ are  contained in $\mathcal X_0$ and $\mathcal X_1$ respectively. In order to compare OT potentials arising in the relevant variational forms, we thus extend them to sets containing $\mathcal X_0-\mathbb E_{\nu_0}[X]$ and $\mathcal X_1-\mathbb E_{\nu_1}[X]$. 

\begin{definition}[Extended OT potentials]
    \label{def:extendedPotentials} For $i=0,1$, let $\mathcal X^{\circ}_i$ be open balls centered at $0$ with radius  $r>2\|\mathcal X_i\|_{\infty}$. 
    Fix arbitrary $(\nu_0,\nu_1)\in\mathcal P(\mathcal X_0)\times\mathcal P(\mathcal X_1)$ and $\mathbf{A}\in\mathbb R^{d_0\times d_1}$, let $z^{\mathbf{A}}$ be an optimal vector for $\OT_{\mathbf{A}}(\bar\nu_0,\bar\nu_1)$. Then, the extended OT potentials for $\OT_{\mathbf{A}}(\bar \nu_0,\bar \nu_1)$ are given by 
    \[
        \begin{aligned}
            \varphi^{\mathbf{A}}_0:x\in\mathcal X^{\circ}_0&\mapsto \min_{1\leq i\leq N}\left\{ c_{\mathbf{A}}\big(x, y^{(i)}-\mathbb E_{\nu_1}[X]\big)-z^{\mathbf{A}}_i \right\},
            \\
            \varphi^{\mathbf{A}}_1:y\in\mathcal X^{\circ}_1&\mapsto \inf_{x\in\mathcal X^{\circ}_0}\left\{ c_{\mathbf{A}}(x,y)-\varphi^{\mathbf{A}}_0(x) \right\}.
        \end{aligned}
    \]
\end{definition}
As $\varphi_1^{\mathbf{A}}$ is simply the $c_{\mathbf{A}}$-transform of 
 $\varphi_0^{\mathbf{A}}$, $(\varphi_0^{\mathbf{A}},\varphi_1^{\mathbf{A}})$ is still a pair of OT potentials for $\OT_{\mathbf{A}}(\bar \nu_0,\bar \nu_1)$. The definitions of $\mathcal X_0^{\circ}$ and $\mathcal X_1^{\circ}$ are motivated by the fact that if $\nu_i\in\mathcal P(\mathcal X_i)$, then $\|\mathbb E_{\nu_i}[X]\|\leq \mathbb E_{\nu_i}[\|X\|]\leq \|\mathcal X_i\|_{\infty}$ for $i\in\{0,1\}$, so that $\mathcal X_i-\mathbb E_{\nu_i}[X]\subset \mathcal X_i^{\circ}$ and thus $\bar \nu_i\in\mathcal P(\mathcal X_i^{\circ})$.

 We first establish some useful properties for the extended OT potentials.
\begin{lemma}[Properties of extended OT potentials] 
    \label{prop:semidiscretePotentialProperties}
    Fix $(\nu_0,\nu_1)\in\mathcal P(\mathcal X_0)\times\mathcal P(\mathcal X_1)$ and $\mathbf{A}\in{B_{\mathrm{F}}(M)}$. Let $z^{\mathbf{A}}$ be an optimal vector for $\OT_{\mathbf{A}}(\bar \nu_0,\bar \nu_1)$ and $(\varphi^{\mathbf{A}}_0,\varphi_1^{\mathbf{A}})$ be the extended OT potentials. Then, 
    \begin{enumerate}
        \item $\varphi^{\mathbf{A}}_0$ and $\varphi^{\mathbf{A}}_1$ are Lipschitz continuous with a shared  constant, $L$, depending only on $M,\|\mathcal X_0^{\circ}\|_{\infty},\|\mathcal X_1^{\circ}\|_{\infty}$.
        \item the triple $(z^{\mathbf{A}},\varphi_0^{\mathbf{A}},\varphi_1^{\mathbf{A}})$ can be chosen such that $\|z^{\mathbf{A}}\|_{\infty}\vee \|\varphi^{\mathbf{A}}_0\|_{\infty,\mathcal X_0^{\circ}}\vee \|\varphi^{\mathbf{A}}_1\|_{\infty,\mathcal X_1^{\circ}}\leq K$, where $K$ depends only on $M,\|\mathcal X^{\circ}_0\|_{\infty},\|\mathcal X^{\circ}_1\|_{\infty}$.  
        \item if $\inte(\supp(\nu_0))$ is connected with (Lebesgue) negligible boundary and $\nu_0$ is absolutely continuous with respect to the Lebesgue measure, then the pair $(\varphi^{\mathbf{A}}_0,\varphi^{\mathbf{A}}_1)$ is unique up to additive constants on $\supp(\nu_0)\times \supp(\nu_1)$. 
    \end{enumerate}
\end{lemma}
\begin{proof}
        For (1), fix $x,x'\in\mathcal X^{\circ}_0$ and observe that  
    \[
    \begin{aligned}
        \left|\varphi^{\mathbf{A}}_0(x)-\varphi^{\mathbf{A}}_0(x')\right| &\leq \max_{1\leq i\leq N}\left| c_{\mathbf{A}}\left(x,y^{(i)}-\mathbb E_{\nu_1}[X]\right)-c_{\mathbf{A}}\left(x',y^{(i)}-\mathbb E_{\nu_1}[X]\right)\right|
        \\&=  \max_{1\leq i\leq N}\left| -4\|x\|^2\|y^{(i)}-\mathbb E_{\nu_1}[X]\|^2 -32 x^{\intercal} \mathbf{A} (y^{(i)}-\mathbb E_{\nu_1}[X])\right.
        \\
        &\hspace{10em}\left.
        +4\|x'\|^2\|y^{(i)}-\mathbb E_{\nu_1}[X]\|^2 +32 {x'}^{\intercal} \mathbf{A} (y^{(i)}-\mathbb E_{\nu_1}[X])
        \right|
        \\
        &\leq 4|\|x\|^2-\|x'\|^2| \max_{i=1}^N\|y^{(i)}-\mathbb E_{\nu_1}[X]\|^2+32\|x-x'\|M\max_{i=1}^n\|y^{(i)}-\mathbb E_{\nu_1}[X]\|
        \\
        &\leq  L_0\|x-x'\|,
    \end{aligned}
    \] 
    as $|\|x\|^2-\|x'\|^2|=\left|(\|x\|+\|x'\|)(\|x\|-\|x'\|)\right|\leq 2\|\mathcal X_0^{\circ}\|_{\infty} \|x-x'\|$ due to the  reverse triangle inequality. Since $\|y^{(i)}-\mathbb E_{\nu_1}[X]\|\leq \|\mathcal X_1^{\circ}\|_{\infty}$,  $L_0$ may be chosen independently of $\mathbf{A},\nu_0,\nu_1$.

    The same approach can be used to derive a Lipschitz constant for $\varphi_1^{\mathbf{A}}$; the maximum of the two Lipschitz constants serves as a shared Lipschitz constant for both OT potentials.

    For (2), fix a solution $(z^{\mathbf{A}},\varphi_0^{\mathbf{A}},\varphi_1^{\mathbf{A}})$ for $\mathsf{OT}_{\mathbf{A}}(\bar \nu_0,\bar \nu_1)$ and assume that $\supp(\nu_1)=\mathcal X_1$ such that \[ 
        \int \varphi_0^{\mathbf{A}} d\bar \mu_0+\sum_{i=1}^Nz_i^{\mathbf{A}}\bar\mu_1(\{y^{(i)}\})=\int \varphi_0^{\mathbf{A}} d\bar \mu_0+\int \varphi_1^{\mathbf{A}}d\bar\mu_1=\mathsf{OT}_{\mathbf{A}}(\bar \mu_0,\bar \mu_1) 
    \]
    which, coupled with the condition $\varphi_0^{\mathbf{A}}\oplus \varphi_1^{\mathbf{A}}\leq c_{\mathbf{A}}$, implies that $\varphi_0^{\mathbf{A}}\oplus \varphi_1^{\mathbf{A}}=c_{\mathbf{A}}$ $\pi^{\mathbf{A}}$-a.e. for any choice of optimal coupling  $\pi^{\mathbf{A}}$ for $\mathsf{OT}_{\mathbf{A}}(\bar \mu_0,\bar \mu_1)$ and similarly for $(z^{\mathbf{A}},\varphi_0^{\mathbf{A}})$. In particular, if $\supp(\nu_1)=\mathcal X_1$ there exists some $x^{(i)}\in\supp(\bar \mu_0)$ for which $\varphi_0^{\mathbf{A}}(x^{(i)})+z_i^{\mathbf{A}}=c_{\mathbf{A}}(x^{(i)},y^{(i)}-\mathbb E_{\nu_1}[X])$ for every $i\in [N]$ and  $\varphi_0^{\mathbf{A}}(x)\leq c_{\mathbf{A}}(x,y^{(i)}-\mathbb E_{\nu_1}[X])-z_i^{\mathbf{A}}$ for any other $x\in\mathcal X_0^{\circ}$. Thus, 
    \[
    \varphi_1^{\mathbf{A}}(y^{(i)}-\mathbb E_{\nu_1}[X])=
        \inf_{x\in\mathcal X_0^{\circ}}\left\{c_{\mathbf{A}}(x,y^{(i)}-\mathbb E_{\nu_1}[X])-\varphi_0^{\mathbf{A}}(x)\right\}\geq z_{i}^{\mathbf{A}}
    \]
    with equality when $x=x^{(i)}$ i.e. $\varphi_1^{\mathbf{A}}(y^{(i)}-\mathbb E_{\nu_1}[X])=z_i^{\mathbf{A}}$ for every $i \in [N]$. Now, let $C=-\max_{1\leq i\leq N}z_i^{\mathbf{A}}$ and set $(\tilde z^{\mathbf{A}},\tilde \varphi_0^{\mathbf{A}},\tilde \varphi_1^{\mathbf{A}})=(z^{\mathbf{A}}+C, \varphi_0^{\mathbf{A}}-C, \varphi_1^{\mathbf{A}}+C)$ such that $\max_{1\leq i\leq N}\tilde z^{\mathbf{A}}_i=0$. With this normalization, for any $(x,y)\in\mathcal X_0^{\circ}\times\mathcal X_1^{\circ}$, 
   \[
   \begin{aligned}
   -\|c_{\mathbf{A}}\|_{\infty,\mathcal X_0^{\circ}\times\mathcal X_1^{\circ} }\leq \tilde \varphi_0^{\mathbf{A}}(x)&=\min_{1\leq i\leq N}\left\{c_{\mathbf{A}}(x,y^{(i)}-\mathbb E_{\nu_1}[X])-\tilde z_i^{\mathbf{A}}\right\}\leq \|c_{\mathbf{A}}\|_{\infty,\mathcal X_0^{\circ}\times\mathcal X_1^{\circ} },
            \\-\|c_{\mathbf{A}}\|_{\infty,\mathcal X_0^{\circ}\times\mathcal X_1^{\circ} }-\sup_{x\in\mathcal X_0^{\circ}} \tilde \varphi_0^{\mathbf{A}}(x)\leq 
            \tilde \varphi_1^{\mathbf{A}}(y)&=\inf_{x\in\mathcal X_0^{\circ}}
            \left\{c_{\mathbf{A}}(x,y)-\tilde \varphi_0^{\mathbf{A}}(x)\right\}\leq \|c_{\mathbf{A}}\|_{\infty,\mathcal X_0^{\circ}\times\mathcal X_1^{\circ} }-\sup_{x\in\mathcal X_0^{\circ}} \tilde \varphi_0^{\mathbf{A}}(x),
   \end{aligned}
   \]
   so that $\|\tilde z^{\mathbf{A}}\|_{\infty}\vee\|\tilde\varphi_0^{\mathbf{A}}\|_{\infty,\mathcal X_0^{\circ}}\vee\|\tilde \varphi_1^{\mathbf{A}}\|_{\infty,\mathcal X_1^{\circ}}\leq 2\|c_{\mathbf{A}}\|_{\infty,\mathcal X_0^{\circ}\times\mathcal X_1^{\circ}}$ where, by the previous arguments, $\|\tilde z^{\mathbf{A}}\|_{\infty}\leq \|\tilde \varphi_1^{\mathbf{A}}\|_{\infty,\mathcal X_1^{\circ}}$.  $\|c_{\mathbf{A}}\|_{\infty,\mathcal X_0^{\circ}\times\mathcal X_1^{\circ}}$ can be bounded as $4\|\mathcal X_0^{\circ}\|^2_{\infty}\|\mathcal X_1^{\circ}\|^2_{\infty}+32M\|\mathcal X_0^{\circ}\|_{\infty}\|\mathcal X_1^{\circ}\|_{\infty}$ via the Cauchy-Schwarz inequality, proving the claim.

    If $\supp(\nu_1)\neq \mathcal X_1$, one may set $z^{\mathbf{A}}_i=0$ if $y^{(i)}\not \in \supp(\nu_1)$ and the previous arguments can be applied to achieve the same bound, though the correspondence between $\tilde \varphi_1^{\mathbf{A}}$ and $z^{\mathbf{A}}$ will only hold on $\supp(\bar \nu_1)$.

    For (3), suppose that $(\tilde \varphi_0^{\bm A},\tilde \varphi_1^{\bm A})$ is another pair of extended OT potentials.%
    By Rademacher's theorem, $\varphi_0^{\bm A}$ and $\tilde \varphi_0^{\bm A}$ are differentiable almost everywhere on $\mathcal X_0^{\circ}$ (in particular on $\inte(\supp(\nu_0))$) so that  Proposition 1.15 in \cite{santambrogio15} implies that $\nabla \varphi_0^{\mathbf A}=\nabla \tilde \varphi_0^{\mathbf A}$ almost everywhere on $\interior(\supp(\nu_0))$. As $\varphi_0^{\mathbf A}$ and $\tilde{\varphi}_0^{\mathbf A}$  are Lipschitz continuous, Theorem 2.6 in \cite{delbarrio2021} implies that $\varphi_0^{\mathbf A}-\tilde \varphi_0^{\mathbf A}=a$ on $\interior(\supp(\nu_0))$ for some constant $a\in\mathbb R$. This equality extends to $\supp(\nu_0)$ again due to Lipschitz continuity. 

    As for $\varphi_1^{\mathbf A}$ and $\tilde \varphi_1^{\mathbf A}$, it holds that $\varphi_0^{\mathbf A}(x)+\varphi_1^{\mathbf A}(y)=c_{\mathbf A}(x,y)$ $\pi^{\mathbf A}$-almost everywhere, where  $\pi^{\mathbf A}$ is any OT plan for $\mathsf{OT}_{\mathbf A}(\nu_0,\nu_1)$. $(\tilde \varphi_0^{\mathbf A},\tilde\varphi_1^{\mathbf A})$ satisfy the same equality so that 
    \[
        \varphi_0^{\mathbf A}(x)+\varphi_1^{\mathbf A}(y)=\tilde\varphi_0^{\mathbf A}(x)+a+\varphi_1^{\mathbf A}(y)=\tilde \varphi_0^{\mathbf A}(x)+\tilde \varphi_1^{\mathbf A}(y)
    \]
   $\pi^{\mathbf A}$- almost everywhere i.e. $\tilde \varphi_1^{\mathbf A}(y)=\varphi_1^{\mathbf A}(y)+a$ $\nu_1$-almost everywhere, proving the claim.   
\end{proof}

To compute the relevant directional derivative, we fix $\nu_0\otimes \nu_1\in\mathcal P_{\mu_0}\times\mathcal P_{\mu_1} $, set $\mu_{i,t}=\mu_i+t(\nu_i-\mu_i)$ for $i=0,1$ and $t\in[0,1]$, and decompose the limit as follows 
\[
\begin{aligned}
    \lim_{t\downarrow0}\frac{\mathsf D(\mu_{0,t},\mu_{1,t})^2-\mathsf D(\mu_{0},\mu_{1})^2}{t}&= \lim_{t\downarrow0}\frac{\mathsf D(\bar \mu_{0,t},\bar \mu_{1,t})^2-\mathsf D(\bar \mu_{0},\bar \mu_{1})^2}{t}
    \\
    &=\lim_{t\downarrow0}\frac{\mathsf S_1(\bar \mu_{0,t},\bar \mu_{1,t})-\mathsf S_1(\bar \mu_{0},\bar \mu_{1})}{t}+\frac{\mathsf S_2(\bar \mu_{0,t},\bar \mu_{1,t})-\mathsf S_2(\bar \mu_{0},\bar \mu_{1})}{t}
\end{aligned}
\]
as follows from the variational form. The following lemmas compute each limit separately. 

Throughout, all extended OT potentials are chosen as to satisfy the estimates from \cref{prop:semidiscretePotentialProperties}.

\begin{lemma}
\label{lem:semidiscreteGateauxD1}
As $t\downarrow 0$, we have that 
    \[
    \begin{aligned}
        \frac{\mathsf S_1(\bar\mu_{0,t},\bar\mu_{1,t})-\mathsf S_1(\bar\mu_0,\bar\mu_1)}{t}&\to 2\iint \|x-x'\|^4 d\mu_{0}(x)d(\nu_0-\mu_{0})(x')
        \\
        &+2\iint \|y-y'\|^4 d\mu_{1}(y)d(\nu_1-\mu_{1})(y')
        \\
        &-4 \iint \|x-\mathbb E_{\mu_0}[X]\|^2\|y-\mathbb E_{\mu_1}[X]\|^2d(\nu_0-\mu_0)(x)d\mu_1(y)
        \\
        &-4\iint \|x-\mathbb E_{\mu_0}[X]\|^2\|y-\mathbb E_{\mu_1}[X]\|^2d\mu_0(x)d(\nu_1-\mu_1)(y).
    \end{aligned} 
\]
\end{lemma}

\begin{proof}
    For any probability measure $\eta$ on $\mathbb R^{d_i}$ ($i=0,1$) with finite fourth moment, 
\[
    \iint \|x-x'\|^4d\bar\eta(x)d\bar\eta(x)=\iint \|x-x'\|^4d\eta(x)d\eta(x),
\]
hence, for $i=0,1$, 
\begin{equation}
    \label{eq:S1Limit1}
    \begin{aligned}
    &\iint \|x-x'\|^4 d\bar\mu_{i,t}(x)d\bar\mu_{i,t}(x')-\iint \|x-x'\|^4 d\bar\mu_{i}(x)d\bar\mu_{i}(x')
    \\
    &=2t\iint \|x-x'\|^4 d\mu_{i}(x)d(\nu_i-\mu_{i})(x')+t^2\iint \|x-x'\|^4 d(\nu_i-\mu_{i})(x)d(\nu_i-\mu_{i})(x').
    \end{aligned}
\end{equation}
On the other hand, letting $f_i(t)=\mathbb E_{\mu_i}[X]+t\left( \mathbb E_{\nu_i}[X]-\mathbb E_{\mu_i}[X]\right)$ for $i=0,1$ and $t\in[0,1]$,
\[
    \begin{aligned}
        \iint \|x\|^2\|y\|^2 d\bar\mu_{0,t}(x)d\bar\mu_{1,t}(y)&=\iint \|x-f_0(t)\|^2\|y-f_1(t)\|^2 d\mu_{0,t}(x)d\mu_{1,t}(y)
        \\
        &= \iint \|x-f_0(t)\|^2\|y-f_1(t)\|^2 d\mu_{0}(x)d\mu_{1}(y)
        \\
        &+ t\iint \|x-f_0(t)\|^2\|y-f_1(t)\|^2 d(\nu_0-\mu_{0})(x)d\mu_{1}(y)
        \\
        &+ t\iint \|x-f_0(t)\|^2\|y-f_1(t)\|^2 d\mu_{0}(x)d(\nu_1-\mu_{1})(y)
        \\
        &+ t^2\iint \|x-f_0(t)\|^2\|y-f_1(t)\|^2 d(\nu_0-\mu_{0})(x)d(\nu_1-\mu_{1})(y).
    \end{aligned}
\]
As $\|x-f_i(t)\|^2=\|x-\mathbb E_{\mu_i}[X]\|^2-2t\langle x-\mathbb E_{\mu_i}[X],\mathbb E_{\nu_i}[X]-\mathbb E_{\mu_i}[X]\rangle+t^2\|\mathbb E_{\nu_i}[X]-\mathbb E_{\mu_i}[X]\|^2$ for $i=0,1$ and the term involving inner products integrates to $0$ w.r.t. $\mu_i$, 
\[
    \begin{aligned}
        \iint \|x\|^2\|y\|^2 d\bar\mu_{0,t}(x)d\bar\mu_{1,t}(y)&=\iint\|x\|^2\|y\|^2d\bar\mu_0(x)d\bar\mu_1(y) 
        \\
        &+t\iint \|x-\mathbb E_{\mu_0}[X]\|^2\|y-\mathbb E_{\mu_1}[X]\|^2d(\nu_0-\mu_0)(x)d\mu_1(y)
        \\
        &+t\iint \|x-\mathbb E_{\mu_0}[X]\|^2\|y-\mathbb E_{\mu_1}[X]\|^2d\mu_0(x)d(\nu_1-\mu_1)(y)
        \\
        &+o(t).
    \end{aligned}
\]
The above display and \eqref{eq:S1Limit1} prove the claim.
\end{proof}

\begin{lemma}
    \label{lem:semidiscreteS1Lipschitz}

    For any pairs $(\nu_0, \nu_1),(\rho_0, \rho_1)\in\mathcal P_{\mu_0}\times \mathcal P_{\mu_1}$, $\left|\mathsf S_{1}(\bar\nu_0,\bar\nu_1)-\mathsf S_{1}(\bar\rho_0,\bar\rho_1)\right|\leq C\|\nu_0\otimes \nu_1-\rho_0\otimes\rho_1\|_{\infty,\mathcal F^{\oplus}}$ for some constant $C>0$ which is independent of $(\nu_0, \nu_1),(\rho_0, \rho_1)$. 
\end{lemma}

\begin{proof}
      For $i=0,1$, we have that
    \[
    \begin{aligned}
        \iint \|x-x'\|^4d\bar \nu_i(x)d\bar \nu_i(x') &=\iint \|x-x'\|^4d \nu_i(x)d \nu_i(x')
        \\
        &=\iint \|x-x'\|^4d (\nu_i-\rho_i)(x)d \nu_i(x')
        +\iint \|x-x'\|^4d (\nu_i-\rho_i)(x)d \rho_i(x')
        \\
        &+\iint \|x-x'\|^4d \bar \rho_i(x)d \bar \rho_i(x'),
    \end{aligned}
    \]
    such that 
    \[
    \begin{aligned}
        &\left|\sum_{i=0}^1\iint \|x-x'\|^4d\bar \nu_i(x)d\bar \nu_i(x')-\iint \|x-x'\|^4d \bar \rho_i(x)d \bar \rho_i(x')\right|
        \\
        &\hspace{18em}=\sum_{i=0}^1 (\nu_i-\rho_i)\left(\int \|\cdot-x'\|^4d\nu_i(x')+\int \|\cdot-x'\|^4d\rho_i(x')\right).
    \end{aligned}
    \]
    Observe that $\xi_0:x\in\mathcal X_0\mapsto \int \|x-x'\|^4d\nu_0(x')+\int \|x-x'\|^4d\rho_0(x')$ is smooth with uniformly bounded derivatives of all orders so that there exists a positive constant $C_0$ which is independent of $\nu_0,\rho_0$ for which $C_0\xi_0\in\mathcal F_0$. Similarly, $\xi_1:x\in\mathcal X_1\mapsto \int \|x-x'\|^4d\nu_1(x')+\int \|x-x'\|^4d\rho_1(x')$ is uniformly bounded so that there exists a positive constant $C_1$ which is independent of $\nu_1,\rho_1$ for which $C_1\xi_1\in\mathcal F_1$. Conclude that 
    \begin{equation} 
    \label{eq:semidiscreteS1LipschitzBound}
    \begin{aligned}
        \left|\sum_{i=0}^1\iint \|x-x'\|^4d\bar \nu_i(x)d\bar \nu_i(x')-\iint \|x-x'\|^4d \bar \rho_i(x)d \bar \rho_i(x')\right|&\\
        &\hspace{-10em}=\left|C_0^{-1}(\nu_0-\rho_0)(C_0\xi_0)+C_1^{-1}(\nu_1-\rho_1)(C_1\xi_1)\right|
        \\
        &\hspace{-10em}\leq C_0^{-1}\left|(\nu_0-\rho_0)(C_0\xi_0)\right|+C_1^{-1}\left|(\nu_1-\rho_1)(C_1\xi_1)\right|
        \\
        &\hspace{-10em}\leq C_0^{-1}\sup_{f_0\in\mathcal F_0}\left|(\nu_0-\rho_0)(f_0)\right|+C_1^{-1}\sup_{f_1\in\mathcal F_1}\left|(\nu_1-\rho_1)(f_1)\right|
        \\
         &\hspace{-10em}\leq (C_0^{-1}+C_1^{-1})\|\nu_0\otimes \nu_1-\rho_0\otimes \rho_1\|_{\infty,\mathcal F^{\oplus}},
    \end{aligned}
    \end{equation}
    where we have used the fact that $0\in\mathcal F_1$ so that \begin{equation}
    \label{eq:classBound}
    \sup_{f_0\in\mathcal F_0}\left|(\nu_0-\rho_0)(f_0)\right|= \sup_{f_0\in\mathcal F_0}\left|(\nu_0-\rho_0)(f_0)+(\nu_1-\rho_1)(0)\right|\leq \|\nu_0\otimes \nu_1-\rho_0\otimes \rho_1\|_{\infty,\mathcal F^{\oplus}},
    \end{equation}
    the same implications hold for $\sup_{f_1\in\mathcal F_1}\left|(\nu_1-\rho_1)(f_1)\right|$.

    As for the other term,
    \[
    \begin{aligned}
        \int \|x\|^2d\bar\nu_0(x)\int \|x\|^2d\bar\nu_1(x)&= \int \|x-\mathbb E_{\nu_0}[X]\|^2d\nu_0(x)\int \|x-\mathbb E_{\nu_1}[X]\|^2d\nu_1(x)
        \\
        &= \int \|x-\mathbb E_{\nu_0}[X]\|^2d(\nu_0-\rho_0)(x)\int \|x-\mathbb E_{\nu_1}[X]\|^2d\nu_1(x)
        \\
        &+\int \|x-\mathbb E_{\nu_0}[X]\|^2d\rho_0(x)\int \|x-\mathbb E_{\nu_1}[X]\|^2d(\nu_1-\rho_1)(x)
        \\
        &+\int \|x-\mathbb E_{\nu_0}[X]\|^2d\rho_0(x)\int \|x-\mathbb E_{\nu_1}[X]\|^2d\rho_1(x),
    \end{aligned}
    \]
    where  
$\zeta_0:x\in\mathcal X_0\mapsto \|x-\mathbb E_{\mu_0}[X]\|^2$ is smooth with uniformly bounded derivatives of all orders so that there exists a positive constant $C_0'$ which is independent of $\nu_0,\rho_0$ for which $C_0'\zeta_0\in\mathcal F_0$. Similarly, $\zeta_1:x\in\mathcal X_1\mapsto \int \|x-\mathbb E_{\nu_1}[X]\|^2$ is uniformly bounded so that there exists a positive constant $C_1'$ which is independent of $\nu_1,\rho_1$ for which $C_1'\zeta_1\in\mathcal F_1$. Conclude that
\[
    \begin{aligned}
        &\left|\int \|x\|^2d\bar\nu_0(x)\int \|x\|^2d\bar\nu_1(x)-\int \|x\|^2d\bar\rho_0(x)\int \|x\|^2d\bar\rho_1(x)\right |
        \\
        &=\left|M_2(\bar \nu_1) \int \|x-\mathbb E_{\nu_0}[X]\|^2d(\nu_0-\rho_0)+\int \|x-\mathbb E_{\nu_0}[X]\|^2d\rho_0(x)\int \|x-\mathbb E_{\nu_1}[X]\|^2d(\nu_1-\rho_1)(x)\right|
        \\
        &\leq M_2(\bar \nu_1)\left|(\nu_0-\rho_0)(\zeta_0)\right|+\int \|x-\mathbb E_{\nu_0}[X]\|^2d\rho_0(x)\left|(\nu_1-\rho_1)(\zeta_1)\right| 
        \\
        &\leq \left(M_2(\bar \nu_1)(C_0')^{-1}+\int \|x-\mathbb E_{\nu_0}[X]\|^2d\rho_0(x)(C_1')^{-1}\right)\left\|\nu_0\otimes \nu_1-\rho_0\otimes \rho_1\right\|_{\infty,\mathcal F^{\oplus}},
    \end{aligned}
    \]
    where the final inequality is due to \eqref{eq:classBound}.
    This final constant can be bounded independently of $\rho_0,\rho_1,\nu_0,\nu_1$ so that the above display and \eqref{eq:semidiscreteS1LipschitzBound} yield the desired result. 
\end{proof}

\begin{lemma}
    \label{lem:semidiscreteOTGateaux}
  If \cref{assn:semidiscreteSuffCond} holds at $\mathbf{A}\in\mathbb R^{d_0\times d_1}$ and $\inte(\supp(\bar\mu_0))$ is connected with negligible boundary, then
    \[ 
    \begin{aligned}
    \lim_{t\downarrow 0} \frac{\mathsf{OT}_{\mathbf{A}}(\bar \mu_{0,t},\bar \mu_{1,t})-\mathsf{OT}_{\mathbf{A}}(\bar \mu_{0},\bar \mu_{1})}{t}&=\int \bar\varphi_{0}^{{\mathbf{A}}}+8\mathbb E_{(X,Y)\sim \pi^{\mathbf{A}}}[\|Y\|^2X]^{\intercal}(\cdot)d(\nu_0-\mu_0)
    \\
    &+\int \bar\varphi_{1}^{{\mathbf{A}}}+8\mathbb E_{(X,Y)\sim \pi^{\mathbf{A}}}[\|X\|^2Y]^{\intercal}(\cdot)d(\nu_1-\mu_1). 
    \end{aligned}
    \]
\end{lemma}

\begin{proof}
    Fix $\mathbf{A}\in\mathbb R^{d_0\times d_1}$ and consider the limit 
    \[\begin{aligned}\lim_{t\downarrow 0}\frac{\mathsf{OT}_{\mathbf{A}}(\bar \mu_{0,t},\bar \mu_{1,t})-\mathsf{OT}_{\mathbf{A}}(\bar \mu_{0},\bar \mu_{1})}{t}&=\lim_{t\downarrow 0}\left(\frac{\mathsf{OT}_{\mathbf{A}}(\bar \mu_{0,t},\bar \mu_{1,t})-\mathsf{OT}_{\mathbf{A}}(\tilde \mu_{0,t},\tilde \mu_{1,t})}{t}\right.
    \\&\hspace{13em}\left.+\frac{\mathsf{OT}_{\mathbf{A}}(\tilde \mu_{0,t},\tilde \mu_{1,t})-\mathsf{OT}_{\mathbf{A}}(\bar \mu_{0},\bar \mu_{1})}{t}\right)
    \end{aligned}
    \]
    where $(\tilde \mu_{0,t},\tilde \mu_{1,t})=\left((\Id-\mathbb E_{\mu_0}[X])_{\sharp}\mu_{0,t},(\Id-\mathbb E_{\mu_1}[X])_{\sharp}\mu_{1,t}\right)$. We compute each limit on the right hand side separately.

    First, let $(\varphi_0^{\mathbf{A}},\varphi_1^{\mathbf{A}})$ and $(\varphi_{0,t}^{\mathbf{A}},\varphi_{1,t}^{\mathbf{A}})$ be OT potentials for $\mathsf{OT}_{\mathbf{A}}(\bar\mu_0,\bar\mu_1)$ and  $\mathsf{OT}_{\mathbf{A}}(\tilde\mu_{0,t},\tilde\mu_{1,t})$ satisfying the bounds from \cref{prop:semidiscretePotentialProperties} (though $(\tilde \mu_{0,t},\tilde \mu_{1,t})$ are not centered versions of measures in $\mathcal P(\mathcal X_0)\times \mathcal P(\mathcal X_1)$, they are supported in $\mathcal P(\mathcal X_0^{\circ})\times\mathcal P(\mathcal X_1^{\circ})$ so that all relevant results still hold for the OT potentials) and observe that  
    \begin{equation}
    \label{eq:semidiscreteGateauxCentredUpper}
    \begin{aligned}
       \mathsf{OT}_{\mathbf{A}}(\tilde\mu_{0,t},\tilde\mu_{1,t})-\mathsf{OT}_{\mathbf{A}}(\bar\mu_{0},\bar\mu_{1})&\leq \int \varphi_{0,t}^{\mathbf{A}}d \tilde\mu_{0,t}+\int\varphi_{1,t}^{\mathbf{A}}d\tilde \mu_{1,t}-\int \varphi_{0,t}^{\mathbf{A}}d \bar\mu_{0}-\int\varphi_{1,t}^{\mathbf{A}}d\bar\mu_{1}
       \\
       &=t\left(\int \bar \varphi_{0,t}^{\mathbf{A}}d (\nu_0-\mu_0)+\int\bar\varphi_{1,t}^{\mathbf{A}}d(\nu_1-\mu_1)\right),
    \end{aligned}
    \end{equation}
    where, for a pair of functions $(g_0,g_1)$, $(\bar g_0,\bar g_1)=(g_0(\cdot-\mathbb E_{\mu_0}[X]), g_1(\cdot-\mathbb E_{\mu_1}[X]))$. 

    Now, fix an arbitrary sequence $t_n\downarrow 0$ and a subsequence $t_{n'}\downarrow 0$. 
    By the Arz{\`e}la-Ascoli theorem, there is a further subsequence $t_{n''}\downarrow 0$ along which $(\varphi_{0,t_{n''}}^{{\mathbf{A}}},\varphi_{1,t_{n''}}^{{\mathbf{A}}})$ converges uniformly on $\mathcal X_0^{\circ}\times \mathcal X_1^{\circ}$ to a pair of continuous functions $(\varphi_0, \varphi_1)$\footnote{\label{footnote:AA} Indeed, these OT potentials extend uniquely to the closure of their domains whilst preserving the uniform equicontinuity and equiboundedness properties.} and, as $(\tilde\mu_{0,t},\tilde\mu_{1,t})$ converge weakly to $(\bar \mu_0,\bar\mu_1)$, Proposition 5.20 in \cite{villani2008optimal} and the dominated convergence theorem imply that
    \[
    \begin{aligned}
    \int \bar \varphi_{0,t_{n''}}^{{\mathbf{A}}} d\mu_{0,t_{n''}}+\int \bar \varphi_{1,t_{n''}}^{{\mathbf{A}}} d\mu_{1,t_{n''}}&\to \int \bar \varphi_0 d\mu_0 +\int\bar \varphi_1 d\mu_1
    \\ 
       \OT_{{\mathbf{A}}}(\tilde \mu_{0,t_{n''}},\tilde \mu_{1,t_{n''}})&\to \mathsf{OT}_{{\mathbf{A}}}(\bar \mu_0,\bar\mu_1).
        \end{aligned}
    \]
    As both terms on the left are equal, conclude that $\int \varphi_0 d\bar\mu_0 +\int \varphi_1 d\bar\mu_1=\mathsf{OT}_{{\mathbf{A}}}(\bar \mu_0,\bar\mu_1)$. As the pointwise inequality $ \varphi_{0,t_{n''}}^{{\mathbf{A}}} \oplus  \varphi_{1,t_{n''}}^{{\mathbf{A}}}\leq c_{\mathbf{A}}$ is preserved in the limit, conclude that 
    $(\varphi_0,\varphi_1)$ is optimal for $\mathsf{OT}_{{\mathbf{A}}}(\bar \mu_0,\bar \mu_1)$ {and that  $(\varphi_0,\varphi_1)=(\varphi_0^{{\mathbf{A}}}+a,\varphi_1^{{\mathbf{A}}}-a)$ for some $a\in\mathbb R$ by \cref{prop:semidiscretePotentialProperties}} . It follows that 
    \[
       \lim_{t_{n''}\downarrow 0}\int \bar\varphi_{0,t_{n''}}^{{\mathbf{A}}}d(\nu_0-\mu_0)+\int \bar\varphi_{1,t_{n''}}^{{\mathbf{A}}}d(\nu_1-\mu_1)=\int \bar\varphi_{0}^{{\mathbf{A}}}d(\nu_0-\mu_0)+\int \bar\varphi_{1}^{{\mathbf{A}}}d(\nu_1-\mu_1), 
    \]
    as $\nu_0-\mu_0$ and $\nu_1-\mu_1$ have total mass zero. As this final limit is independent of the choice of original subsequence, \eqref{eq:semidiscreteGateauxCentredUpper} reads 
    \[
    \limsup_{t\downarrow 0}\frac{\mathsf{OT}_{\mathbf{A}}(\tilde \mu_{0,t},\tilde \mu_{1,t})-\mathsf{OT}_{\mathbf{A}}(\bar \mu_{0},\bar \mu_{1})}{t}\leq \int \bar\varphi_{0}^{{\mathbf{A}}}d(\nu_0-\mu_0)+\int \bar\varphi_{1}^{{\mathbf{A}}}d(\nu_1-\mu_1). 
    \]
On the other hand,
    \begin{equation}
    \label{eq:semidiscreteGateauxCentredLower}
\frac{\mathsf{OT}_{\mathbf{A}}(\tilde \mu_{0,t},\tilde \mu_{1,t})-\mathsf{OT}_{\mathbf{A}}(\bar \mu_{0},\bar \mu_{1})}{t}\geq \int \bar\varphi_{0}^{{\mathbf{A}}}d(\nu_0-\mu_0)+\int \bar\varphi_{1}^{{\mathbf{A}}}d(\nu_1-\mu_1),
    \end{equation}
    which readily implies that 
\begin{equation}
\label{eq:semidiscreteGateauxCentred}
    \lim_{t\downarrow 0}\frac{\mathsf{OT}_{\mathbf{A}}(\tilde \mu_{0,t},\tilde \mu_{1,t})-\mathsf{OT}_{\mathbf{A}}(\bar \mu_{0},\bar \mu_{1})}{t}= \int \bar\varphi_{0}^{{\mathbf{A}}}d(\nu_0-\mu_0)+\int \bar\varphi_{1}^{{\mathbf{A}}}d(\nu_1-\mu_1).  
\end{equation}

As for the other limit, note that if $\tilde \pi\in \Pi(\tilde \mu_{0,t},\tilde \mu_{1,t})$, then $\tau^{-t}_{\sharp}\tilde \pi\in\Pi(\bar \mu_{0,t},\bar \mu_{1,t})$, and if $\bar \pi\in \Pi(\bar \mu_{0,t},\bar \mu_{1,t})$, then $\tau^{t}_{\sharp}\bar \pi\in\Pi(\tilde \mu_{0,t},\tilde \mu_{1,t})$, where $\tau^t=(\tau_0^{t},\tau_1^{t})=(\Id+t(\mathbb E_{\nu_0}[X]-\mathbb E_{\mu_0}[X]),\Id+t(\mathbb E_{\nu_1}[X]-\mathbb E_{\mu_1}[X]))$ for $t\in[-1,1]$. Hence, letting $\tilde \pi_t$ and $\bar \pi_t$ be OT plans for $\mathsf{OT}_{\mathbf{A}}(\tilde \mu_{0,t},\tilde \mu_{1,t})$ and $\mathsf{OT}_{\mathbf{A}}(\bar \mu_{0,t},\bar \mu_{1,t})$, we have    
\begin{equation}
   \label{eq:semidiscreteGateauxUncentredUpperLower} 
\begin{aligned}    
   \mathsf{OT}_{\mathbf{A}}(\bar \mu_{0,t},\bar \mu_{1,t})-\mathsf{OT}_{\mathbf{A}}(\tilde \mu_{0,t},\tilde \mu_{1,t})&\leq \int c_{\mathbf{A}}\circ\tau^{-t}-c_{\mathbf{A}} d\tilde \pi_t,
   \\\mathsf{OT}_{\mathbf{A}}(\bar \mu_{0,t},\bar \mu_{1,t})-\mathsf{OT}_{\mathbf{A}}(\tilde \mu_{0,t},\tilde \mu_{1,t})&\geq \int c_{\mathbf{A}}-c_{\mathbf{A}}\circ\tau^{t} d\bar \pi_t ,
\end{aligned} 
\end{equation}
by expanding the expressions for the cost functions, we obtain
    \begin{equation}
    \label{eq:costTranslationExpansion}
    \begin{aligned}
        c_{{\mathbf{A}}}\circ \tau^t(x,y)&= -4(\|x\|^2+2\langle x,t(\mathbb E_{\nu_0}[X]-\mathbb E_{\mu_0}[X])\rangle+t^2\|\mathbb E_{\nu_0}[X]-\mathbb E_{\mu_0}[X]\|^2)\\&\times(\|y\|^2+2\langle y,t(\mathbb E_{\nu_1}[X]-\mathbb E_{\mu_1}[X])\rangle+t^2\|\mathbb E_{\nu_1}[X]-\mathbb E_{\mu_1}[X]\|^2)
        \\
        &
        -32(x+t(\mathbb E_{\nu_0}[X]-\mathbb E_{\mu_0}[X]))^{\intercal}{\mathbf{A}}(y+t(\mathbb E_{\nu_1}[X]-\mathbb E_{\mu_1}[X]))
        \\
        &=c_{{\mathbf{A}}}(x,y)-8t\langle x,\mathbb E_{\nu_0}[X]-\mathbb E_{\mu_0}[X]\rangle\|y\|^2-8t\langle y,\mathbb E_{\nu_1}[X]-\mathbb E_{\mu_1}[X]\rangle\|x\|^2
        \\
        &-32t\left(\mathbb E_{\nu_0}[X]-\mathbb E_{\mu_0}[X]\right)^{\intercal}{\mathbf{A}}y-32tx^{\intercal}{\mathbf{A}}\left(\mathbb E_{\nu_1}[X]-\mathbb E_{\mu_1}[X]\right)+O(t^2), 
    \end{aligned}
    \end{equation}
    such that 
\[
\begin{aligned}    
   \limsup_{t\downarrow 0}\frac{\mathsf{OT}_{\mathbf{A}}(\bar \mu_{0,t},\bar \mu_{1,t})-\mathsf{OT}_{\mathbf{A}}(\tilde \mu_{0,t},\tilde \mu_{1,t})}{t}&\leq \int8\langle x,\mathbb E_{\nu_0}[X]-\mathbb E_{\mu_0}[X]\rangle\|y\|^2
   \\&\hspace{8em}+8\langle y,\mathbb E_{\nu_1}[X]-\mathbb E_{\mu_1}[X]\rangle\|x\|^2d\pi^{\mathbf{A}}(x,y),
   \\
   \end{aligned}
   \]
   \[   
   \begin{aligned}
   \liminf_{t\downarrow 0}\frac{\mathsf{OT}_{\mathbf{A}}(\bar \mu_{0,t},\bar \mu_{1,t})-\mathsf{OT}_{\mathbf{A}}(\tilde \mu_{0,t},\tilde \mu_{1,t})}{t}
   &\geq
   \int8\langle x,\mathbb E_{\nu_0}[X]-\mathbb E_{\mu_0}[X]\rangle\|y\|^2
   \\&\hspace{8em}+8\langle y,\mathbb E_{\nu_1}[X]-\mathbb E_{\mu_1}[X]\rangle\|x\|^2d\pi^{\mathbf{A}}(x,y), 
\end{aligned} 
\]
noting that $\tilde\pi_t$ and $\bar \pi_t$ both converge weakly to the unique (under \cref{assn:semidiscreteSuffCond}) OT plan $\pi^{\mathbf{A}}$ for $\mathsf{OT}_{\mathbf{A}}(\bar \mu_0,\bar \mu_1)$ as follows from Theorem 5.20 in \cite{villani2008optimal}, Lemma 5.2.1 in \cite{ambrosio2005}, and the fact that $\bar\mu_0$ and $\bar \mu_1$ are mean-zero. This result, along with \eqref{eq:semidiscreteGateauxCentred}, proves the claim. 
\end{proof}

\begin{lemma}
    \label{lem:semidiscreteGateauxD2}   Let  
    $\mathcal A=\argmin_{B_{\mathrm{F}}(M)}\Phi_{(\bar\mu_0,\bar\mu_1)}$.
    If $\bar \mu_0,\bar\mu_1$ satisfy \cref{assn:semidiscreteSuffCond} at every $\mathbf{A}\in\mathcal A$ and $\inte(\supp(\bar\mu_0))$ is connected with negligible boundary,      
    \[
    \begin{aligned}
        &\lim_{t\downarrow 0 }\frac{\mathsf S_2(\bar\mu_{0,t},\bar\mu_{1,t})-\mathsf S_2(\bar \mu_0,\bar \mu_1)}{t}= \inf_{\mathbf{A}\in\mathcal A}\left\{
             \int g_{0,\pi^{\mathbf{A}}}+\bar{\varphi}_0^{\mathbf{A}}d(\nu_0-\mu_0)+\int g_{1,\pi^{\mathbf{A}}}+\bar{\varphi}_1^{\mathbf{A}}d(\nu_1-\mu_1)
            \right\},
    \end{aligned}
    \]  
      where, for $(x,y)\in\mathbb R^{d_0}\times\mathbb R^{d_1} $, 
    \[
        \begin{aligned}
           g_{0,\pi}(x)&= 8\mathbb E_{(X,Y)\sim \pi}[\|Y\|^2X]^{\intercal}x,
            \\
            g_{1,\pi}(y)&=8\mathbb E_{(X,Y)\sim \pi}[\|X\|^2Y]^{\intercal}y,
        \end{aligned}
    \] 
$\bar\varphi_i^{\mathbf{A}}=\varphi_i^{\mathbf{A}}(\cdot-\mathbb E_{\mu_i}[X])$ for $i=0,1$, and, for $\mathbf{A}\in \mathcal A$, $\pi^{\mathbf{A}}$ is the unique OT plan for $\OT_{\mathbf{A}}(\bar \mu_0,\bar \mu_1)$, and $(\varphi_0^{\mathbf{A}},\varphi_1^{\mathbf{A}})$ is any choice of corresponding OT potentials.
\end{lemma}

\begin{proof}
    Let $\mathbf{A}^{\star}\in\mathcal A$ be arbitrary. From \cref{lem:semidiscreteOTGateaux}, we have that 
\[
\begin{aligned}
\mathsf{S}_2(\bar \mu_{0,t},\bar\mu_{1,t})-\mathsf{S}_2(\bar \mu_{0},\bar\mu_{1})&\leq 
    \mathsf{OT}_{\mathbf{A}^{\star}}(\bar\mu_{0,t},\bar\mu_{1,t})-\mathsf{OT}_{\mathbf{A}^{\star}}(\bar\mu_{0},\bar\mu_{1})
    \\&=t\left(\int g_{0,\pi^{\mathbf{A}^{\star}}}+\bar\varphi_0^{\mathbf{A}^{\star}}d(\nu_0-\mu_0)+\int g_{1,\pi^{\mathbf{A}^{\star}}}+\bar\varphi_1^{\mathbf{A}^{\star}}d(\nu_1-\mu_1)\right)+o(t),
\end{aligned}
\]    
such that 
\begin{equation}
\label{eq:semidiscreteD2Upper}
\limsup_{t\downarrow 0}\frac{\mathsf{S}_2(\bar \mu_{0,t},\bar\mu_{1,t})-\mathsf{S}_2(\bar \mu_{0},\bar\mu_{1})}{t}\leq 
    \inf_{\mathbf{A}\in\mathcal A}\left\{\int g_{0,\pi^{\mathbf{A}}}+\bar\varphi_0^{\mathbf{A}}d(\nu_0-\mu_0)+\int g_{1,\pi^{\mathbf{A}}}+\bar\varphi_1^{\mathbf{A}}d(\nu_1-\mu_1)\right\}.
\end{equation}

For the opposite limit,  observe that the solutions of $\mathsf S_2(\bar \mu_{0,t},\bar \mu_{1,t})$ lie in {$B_{\mathrm{F}}(M)$} for every $t\in[0,1]$ as follows from \cref{thm:VariationalMinimizers} and let $\mathbf{A}_t$ be optimal for each corresponding problem.  By  \cref{prop:semidiscretePotentialProperties} and  equations \eqref{eq:semidiscreteGateauxCentredUpper}, \eqref{eq:semidiscreteGateauxCentredLower}, \eqref{eq:semidiscreteGateauxUncentredUpperLower}, and \eqref{eq:costTranslationExpansion}, 
\[
    \left|\mathsf{OT}_{\mathbf{A}}(\bar \mu_{0,t},\bar \mu_{1,t})-\mathsf{OT}_{\mathbf{A}}(\bar \mu_{0},\bar \mu_{1})\right|=O(t),
\] 
where the $O(t)$ term can be chosen to be independent of $\mathbf{A}$. Whence, $\mathbf{A}\in B_{\mathrm{F}}(M)\mapsto \|\mathbf{A}\|^2_{\mathrm F}+\mathsf{OT}_{\mathbf{A}}(\bar \mu_{0,t},\bar \mu_{1,t})$ converges uniformly to $\mathbf{A}\in B_{\mathrm{F}}(M) \mapsto \|\mathbf{A}\|^2_{\mathrm F}+\mathsf{OT}_{\mathbf{A}}(\bar \mu_{0},\bar \mu_{1})$ as $t\downarrow 0$. Extending these functions by $+\infty$ outside of $B_{\mathrm{F}}(M)$, we may apply Proposition 7.15  and Theorem 7.33 from \cite{rockafellar2009variational} to obtain that any cluster point of $(\mathbf{A}_{t})_{t\downarrow 0}$ is an element of $\mathcal A$  keeping in mind that these functions are continuous on $B_{\mathrm{F}}(M)$ (see \cref{thm:VariationalMinimizers}). 

Now, let $t_n\downarrow 0$ be arbitrary and fix a subsequence $t_{n'}$.
By the Bolzano-Weierstrass theorem, there exists a  further subsequence $t_{n''}\downarrow 0$ along which $\mathbf{A}_{t_{n''}}\to \mathbf{A}\in\mathcal A$. From equations \eqref{eq:semidiscreteGateauxCentredLower} and  \eqref{eq:semidiscreteGateauxUncentredUpperLower},
\[
\begin{aligned}
    \frac{\mathsf S_2(\bar \mu_{0,t_{n''}},\bar \mu_{1,t_{n''}})-\mathsf S_2(\bar \mu_{0},\bar \mu_{1})}{t_{n''}}&\geq \frac{\mathsf{OT}_{\mathbf{A}_{t_{n''}}}(\bar \mu_{0,t_{n''}},\bar \mu_{1,t_{n''}})-\mathsf{OT}_{\mathbf{A}_{t_{n''}}}(\bar \mu_{0},\bar \mu_{1})}{t_{n''}}
    \\
    &\geq \int\bar\varphi_0^{\mathbf{A}_{t_{n''}}}d(\nu_0-\mu_0)+\int\bar\varphi_1^{\mathbf{A}_{t_{n''}}}d(\nu_1-\mu_1)
    \\
    &+\int c_{\mathbf{A}_{t_{n''}}}- c_{\mathbf{A}_{t_{n''}}}\circ \tau^{t_{{n''}}}d \bar \pi_{t_{n''}},
\end{aligned}
\]
where $(\varphi_0^{\mathbf{A}_{t_{n''}}},\varphi_1^{\mathbf{A}_{t_{n''}}})$ are OT potentials for $\mathsf{OT}_{\mathbf{A}_{t_{n''}}}(\bar \mu_0,\bar \mu_1)$ and $\bar \pi_{t_{n''}}$ is an OT plan for $\mathsf{OT}_{\mathbf{A}_{t_{n''}}}(\bar \mu_{0,t_{n''}},\bar \mu_{1,t_{n''}})$. It remains to show that $\bar \pi_{t_{n''}}$ converges weakly to $\pi^{\mathbf{A}}$ weakly and that $(\bar \varphi_0^{\mathbf{A}_{t_{n''}}},\bar \varphi_1^{\mathbf{A}_{t_{n''}}})$ converges to $(\bar\varphi_0^{\mathbf{A}}+a,\bar\varphi_1^{\mathbf{A}}-a)$ uniformly on $\mathcal X_0\times \mathcal X_1$ along a further subsequence and some choice of $a\in\mathbb R$ to obtain that 
\[
    \liminf_{t\downarrow 0}\frac{\mathsf S_2(\bar \mu_{0,t},\bar \mu_{1,t})-\mathsf S_2(\bar \mu_{0},\bar \mu_{1})}{t}\geq \int g_{0,\pi^{\mathbf{A}}}+\bar\varphi_0^{\mathbf{A}}d(\nu_0-\mu_0)+\int g_{1,\pi^{\mathbf{A}}}+\bar\varphi_1^{\mathbf{A}}d(\nu_1-\mu_1), 
\]
recalling the expansion of the translated cost function from \eqref{eq:costTranslationExpansion}.
The above display, together with \eqref{eq:semidiscreteD2Upper}, yields the desired result.

Note that  $c_{\mathbf{A}_{t_{n''}}}$ converges uniformly to $c_{\mathbf{A}}$ on $\mathcal X_0\times \mathcal X_1$ and $(\bar \mu_{0,t_{n''}},\bar \mu_{1,t_{n''}})$ converges weakly to $(\bar \mu_{0},\bar \mu_{1})$ whereby $\pi^{\mathbf{A}_{t_{n'''}}}$ converges to $\pi^{\mathbf{A}}$ weakly along a further subsequence, $t_{n'''}$ 
(see Theorem 5.20 in \cite{villani2008optimal}).

 Applying the Arzel{\`a}-Ascoli theorem, we may also assume that $(\varphi_0^{\mathbf{A}_{t_{n'''}}},\varphi_1^{\mathbf{A}_{t_{n'''}}})$ converges uniformly to some pair of continuous functions $(\bar\varphi_0,\bar\varphi_1)$ on $\mathcal X_0^{\circ}\times \mathcal X_1^{\circ}$ (see also \cref{footnote:AA}). Applying the dominated convergence theorem and the weak convergence of the couplings, 
\[
\begin{aligned}
    \int \varphi_0^{\mathbf{A}_{t_{n'''}}}d\bar \mu_{0,t_{n'''}}+\int \varphi_1^{\mathbf{A}_{t_{n'''}}}d\bar \mu_{1,t_{n'''}}&\to\int \varphi_0d\bar{\mu}_0+\int \varphi_1d\bar{\mu}_1,
    \\
    \mathsf{OT}_{\mathbf{A}_{t_{n'''}}}(\bar \mu_{0,t_{n'''}},\bar\mu_{1,t_{n'''}})&\to \mathsf{OT}_{\mathbf{A}}(\bar \mu_{0},\bar \mu_{1}).
    \end{aligned}
\]
As the two terms on the left hand side are equal, $\int \varphi_0d\bar{\mu}_0+\int \varphi_1d\bar{\mu}_1=\mathsf{OT}_{\mathbf{A}}(\bar \mu_{0},\bar \mu_{1})$
and, 
as the inequalities $\varphi_0^{\mathbf{A}_{t_{n'''}}}\oplus\varphi_1^{\mathbf{A}_{t_{n'''}}}\leq c_{\mathbf{A}_{t_{n'''}}}$ are preserved in the limit,  $(\varphi_0,\varphi_1)$ is a pair of OT potentials for $\mathsf{OT}_{\mathbf{A}}(\bar \mu_0,\bar \mu_1)$ and hence coincide with $(\varphi_0^{\mathbf{A}},\varphi_1^{\mathbf{A}})$ up to constants, concluding the proof.  
\end{proof}

\begin{lemma}
    \label{lem:semidiscreteS2Lipschitz}
    For any pairs $(\nu_0, \nu_1),(\rho_0, \rho_1)\in\mathcal P_{\mu_0}\times \mathcal P_{\mu_1}$, $\left|\mathsf S_{2}(\bar\nu_0,\bar\nu_1)-\mathsf S_{2}(\bar\rho_0,\bar\rho_1)\right|\leq C\|\nu_0\otimes \nu_1-\rho_0\otimes\rho_1\|_{\infty,\mathcal F^{\oplus}}$ for some constant $C>0$ which is independent of $(\nu_0, \nu_1),(\rho_0, \rho_1)$.
\end{lemma}

\begin{proof}
 Observe that 
    \begin{equation}
    \label{eq:semidiscreteLipschitzD2}
        \left|\mathsf S_2(\bar \nu_0,\bar \nu_1)-\mathsf S_2(\bar \rho_0,\bar \rho_1)\right|\leq \sup_{\mathbf{A}\in B_{\mathrm{F}}(M)}\left|\mathsf{OT}_{\mathbf{A}}(\bar \nu_0,\bar \nu_1)-\mathsf{OT}_{\mathbf{A}}(\bar \rho_0,\bar \rho_1)\right|. 
    \end{equation}
    Let $(\varphi^{\mathbf{A}}_{0,\bar \nu},\varphi^{\mathbf{A}}_{1,\bar \nu})$ and $(\varphi^{\mathbf{A}}_{0,\bar \rho},\varphi^{\mathbf{A}}_{1,\bar \rho})$ be extended OT potentials for $\mathsf{OT}_{\mathbf{A}}(\bar \nu_0,\bar \nu_1)$ and $\mathsf{OT}_{\mathbf{A}}(\bar \rho_0,\bar \rho_1)$ respectively satisfying the estimates from \cref{prop:semidiscretePotentialProperties}. We have that 
    \[ 
    \begin{aligned}
\mathsf{OT}_{\mathbf{A}}(\bar \nu_0,\bar \nu_1)-\mathsf{OT}_{\mathbf{A}}(\bar \rho_0,\bar \rho_1)&\leq \int \varphi_{0,\bar \nu}^{\mathbf{A}} d(\bar \nu_0-\bar \rho_0)+\int \varphi_{1,\bar \nu}^{\mathbf{A}} d(\bar \nu_1-\bar \rho_1),\\ \mathsf{OT}_{\mathbf{A}}(\bar \nu_0,\bar \nu_1)-\mathsf{OT}_{\mathbf{A}}(\bar \rho_0,\bar \rho_1)
&\geq \int \varphi_{0,\bar \rho}^{\mathbf{A}} d(\bar \nu_0-\bar \rho_0)+\int \varphi_{1,\bar \rho}^{\mathbf{A}} d(\bar \nu_1-\bar \rho_1). 
    \end{aligned}
    \]
    Fix $i\in\{0,1\}$ and observe that, for any $L$-Lipschitz function $f_i:\mathcal X_i^{\circ}\to \mathbb R$, 
    \[
    \begin{aligned}
        \int f_i d(\bar \nu_i-\bar \rho_i)&=\int f_i(\cdot -\mathbb E_{\nu_i}[X]) d\nu_i-\int f_i(\cdot -\mathbb E_{\rho_i}[X]) d \rho_i
        \\
        &
        \leq \int f_i(\cdot -\mathbb E_{\rho_i}[X]) d(\nu_i- \rho_i)+L\|\mathbb E_{\nu_i}[X]-\mathbb E_{\rho_i}[X]\|,
        \\
        \int f_i d(\bar \nu_i-\bar \rho_i)&\geq \int f_i(\cdot -\mathbb E_{\rho_i}[X]) d(\nu_i- \rho_i)-L\|\mathbb E_{\nu_i}[X]-\mathbb E_{\rho_i}[X]\|. 
    \end{aligned}
    \]
    As the coordinate projections are smooth with uniformly bounded derivatives of all orders, there exists  a positive constant $C_1$ which is independent of $\nu_0,\nu_1,\rho_0,\rho_1$ such that $L\|\mathbb E_{\nu_i}[X]-\mathbb E_{\rho_i}[X]\|\leq L\sqrt{d_i}\max_{j=1}^{d_i}\left|\int x_j d(\nu_i-\rho_i)(x)\right|\leq C_1\|\nu_i-\rho_i\|_{\infty,\mathcal F_i}\leq C_1\|\nu_0\otimes\nu_1-\rho_0\otimes \rho_1\|_{\infty,\mathcal F^{\oplus}}$, where the final inequality follows from the fact that both $\mathcal F_0$ and $\mathcal F_1$ contain the zero function. As the extended OT potentials are $L$-Lipschitz (see \cref{prop:semidiscretePotentialProperties}), 
    \[
    \begin{aligned}
 \left |\mathsf{OT}_{\mathbf{A}}(\bar \nu_0,\bar \nu_1)-\mathsf{OT}_{\mathbf{A}}(\bar \rho_0,\bar \rho_1)\right|&\leq \left(\left|\int \varphi_{0,\bar \nu}^{\mathbf{A}}(\cdot -\mathbb E_{\rho_0}[X]) d(\nu_0- \rho_0)+\int \varphi_{1,\bar \nu}^{\mathbf{A}}(\cdot -\mathbb E_{\rho_1}[X]) d(\nu_1- \rho_1)\right|\right.
 \\&\hspace{1.4em}\left.\bigvee\left|\int \varphi_{0,\bar \rho}^{\mathbf{A}}(\cdot -\mathbb E_{\rho_0}[X]) d(\nu_0- \rho_0)+\int \varphi_{1,\bar \rho}^{\mathbf{A}}(\cdot -\mathbb E_{\rho_1}[X]) d(\nu_1- \rho_1)\right| \right)
 \\
 &+ 2C_1\|\nu_0\otimes\nu_1-\rho_0\otimes \rho_1\|_{\infty,\mathcal F^{\oplus}}.    
    \end{aligned}
    \]
    Observe that $\varphi_{0,\bar \nu}^{\mathbf{A}}(\cdot -\mathbb E_{\rho_0}[X]),\varphi_{0,\bar \rho}^{\mathbf{A}}(\cdot -\mathbb E_{\rho_0}[X])\in\mathcal F_0$ by construction and that $\varphi_{1,\bar \rho}^{\mathbf{A}}(\cdot -\mathbb E_{\rho_1}[X]) ,\varphi_{1,\bar \nu}^{\mathbf{A}}(\cdot -\mathbb E_{\rho_1}[X])$ are bounded on $\mathcal X_1^{\circ}$ and hence are elements of $\mathcal F_1$ up to some normalizing constant. It follows that the the term in brackets on the right hand side of the display can be bounded as $C_2\|\nu_0\otimes \nu_1-\rho_0\otimes \rho_1\|_{\infty,\mathcal F^{\oplus}}$ for some positive constant $C_2$ which is independent of $\mathbf{A},\nu_0,\nu_1,\rho_0,\rho_1$. Applying these bounds in \eqref{eq:semidiscreteLipschitzD2} proves the desired Lipschitz continuity of $\mathsf S_2$.
\end{proof}

\begin{proof}[Proof of \cref{thm:semidiscreteGWStability}] By applying the variational form \eqref{eq:GWVariational}, the desired G\^ateaux differentiability follows from 
    Lemmas \ref{lem:semidiscreteGateauxD1} and \ref{lem:semidiscreteGateauxD2} whereas the Lipschitz continuity follows from Lemmas \ref{lem:semidiscreteS1Lipschitz} and \ref{lem:semidiscreteS2Lipschitz}.   
    \end{proof}

\subsection{Proof of \texorpdfstring{\cref{thm:semidiscreteGWLimitDistribution}}{Theorem 5}}
\label{proof:thm:semidiscreteGWLimitDistribution}

Given \cref{thm:semidiscreteGWStability},
the first assertion  of \cref{thm:semidiscreteGWLimitDistribution} follows from the unified approach (\cref{prop:unified}).

Specializing to the empirical measures, the result will follow
upon showing that the function classes $\mathcal F_0$ and $\mathcal F_1$ are, respectively, $\mu_0$- and $\mu_1$-Donsker. Indeed, by applying the same argument as in the end of \cref{thm:discreteGWLimitDistribution}, Donskerness of each class and the independence assumption on the samples imply that $\sqrt n(\hat \mu_{0,n}-\mu_0)\otimes(\hat \mu_{1,n}-\mu_1)\stackrel{d}{\to}G_{\mu_0\otimes \mu_1}$ in $\ell^{\infty}(\mathcal F^{\oplus})$ where $G_{\mu_0\otimes \mu_1}(f_0\oplus f_1)=G_{\mu_0}(f_0)+G_{\mu_1}(f_1)$ and $G_{\mu_0}$, $G_{\mu_1}$ are, respectively, tight $\mu_0$- and $\mu_1$-Brownian bridge processes in $\ell^{\infty}(\mathcal F_0)$ and $\ell^{\infty}(\mathcal F_1)$. We begin with $\mathcal F_0$.

By Corollary 2.7.2 and the surrounding discussion in \cite{van1996weak}, the unit ball in $C^k(\mathcal X_0)$ is universally Donsker provided that  $k=\lfloor d/2\rfloor+1$. As unions and pairwise sums of Donsker classes are Donsker (see Example 2.10.7 in \cite{van1996weak}), it suffices to show that $\mathcal G$ is $\mu_0$-Donsker.

Let $x,\xi,\xi'\in\mathcal X_0$, $\zeta,\zeta'\in\mathbb R^{d_1}$ with $\|\zeta\|,\|\zeta'\|\leq \|\mathcal X_1\|_{\infty}$, $z,z'\in\mathbb R^N$ with $\|z\|_{\infty},\|z'\|_{\infty}\leq K$, and $\mathbf{A},\mathbf{A}'\in B_{\mathrm{F}}(M)$ be arbitrary. Observe that 
\[
\begin{aligned}
    &\left| \min_{1\leq i\leq N}\left\{c_{\mathbf{A}}\left(x-\xi,y^{(i)}-\zeta\right)-z_i\right\}-\min_{1\leq i\leq N}\left\{c_{\mathbf{A}'}\left(x-\xi',y^{(i)}-\zeta'\right)-z_i'\right\}\right|
    \\
    &\leq \max_{1\leq i\leq N}\left| -4\|x-\xi\|^2\|y^{(i)}-\zeta\|^2-32 (x-\xi)^{\intercal}\mathbf{A}(y^{(i)}-\zeta)-z_i\right.\\&\hspace{15em}\left. +4\|x-\xi'\|^2\|y^{(i)}-\zeta'\|^2+32 (x-\xi')^{\intercal}\mathbf{A}'(y^{(i)}-\zeta')+z_i' \right|. 
\end{aligned}
\]
By applying the reverse triangle inequality, we see that
\[
\begin{aligned}
    &\left|\|x-\xi\|^2\|y^{(i)}-\zeta\|^2-\|x-\xi'\|^2\|y^{(i)}-\zeta'\|^2\right|
    \\
    &\leq \left|\|x-\xi\|^2-\|x-\xi'\|^2\right|\|y^{(i)}-\zeta\|^2+\left|\|y^{(i)}-\zeta\|^2-\|y^{(i)}-\zeta'\|^2\right|\|x-\xi'\|^2
    \\
    &\leq (\|x-\xi\|+\|x-\xi'\|)\|\xi-\xi'\|\|y^{(i)}-\zeta\|^2+(\|y^{(i)}-\zeta\|+\|y^{(i)}-\zeta'\|)\|\zeta-\zeta'\|\|x-\xi'\|^2
    \\
    &\lesssim_{\mathcal X_0,\mathcal X_1}\|\xi-\xi'\|+\|\zeta-\zeta'\|,
\end{aligned}
\]
whereas
\[
\begin{aligned}
    &\left|(x-\xi)^{\intercal}\mathbf{A}(y^{(i)}-\zeta)-(x-\xi')^{\intercal}\mathbf{A}'(y^{(i)}-\zeta')\right|
    \\
    &\leq \left|(x-\xi)^{\intercal}(\mathbf{A}-\mathbf{A}')(y^{(i)}-\zeta)\right|+\left|(\xi'-\xi)^{\intercal}\mathbf{A}'(y^{(i)}-\zeta)\right|+\left|(x-\xi')^{\intercal}\mathbf{A}'(\zeta'-\zeta)\right|
    \\
    &\lesssim_{M,\mathcal X_0,\mathcal X_1} \left\|\mathbf{A}-\mathbf{A}'\right\|_{\mathrm{F}}+\left\|\xi'-\xi\right\|+\left\|\zeta'-\zeta\right\|,
\end{aligned}
\]
so that there exists a constant $C_{M,\mathcal X_0,\mathcal X_1}<\infty$ satisfying   
\[
\begin{aligned}
    &\left| \min_{1\leq i\leq N}\left\{c_{\mathbf{A}}\left(x-\xi,y^{(i)}-\zeta\right)-z_i\right\}-\min_{1\leq i\leq N}\left\{c_{\mathbf{A}'}\left(x-\xi',y^{(i)}-\zeta'\right)-z_i'\right\}\right|
    \\
    &\leq C_{M,\mathcal X_0,\mathcal X_1}\left(\left\|\mathbf{A}-\mathbf{A}'\right\|_{\mathrm{F}}+\left\|\xi'-\xi\right\|+\left\|\zeta'-\zeta\right\|+\|z-z'\|_{\infty}\right)
\end{aligned}
\]
whereby the function class $\mathcal G$ is Lipschitz with respect to the parameter $(\mathbf{A},\xi,\zeta,z)\in B_{\mathrm{F}}(M)\times \mathcal X_0\times B_{(\mathbb R^{d_1},\|\cdot\|)}(\|\mathcal X_1\|_{\infty})\times B_{(\mathbb R^N,\|\cdot\|_{\infty})}(K)\eqqcolon \mathcal Q$, where $B_{(\mathbb R^{d_1},\|\cdot\|)}(\|\mathcal X_1\|_{\infty})$ is the closed ball of radius $\|\mathcal X_1\|_{\infty}$ in $\mathbb R^{d_1}$ with respect to $\|\cdot\|$, and $B_{(\mathbb R^N,\|\cdot\|_{\infty})}(K)$ is defined analogously. Let $\tau:(\mathbf{A},\xi,\zeta,z)\in\mathcal Q\mapsto (A_{(\cdot),1},A_{(\cdot),2},\dots,A_{(\cdot),d_1},\xi,\zeta,z)\in\mathbb R^{d_0d_1+d_0+d_1+N}$, where $A_{(\cdot),i}$ is the $i$-th column of $\mathbf{A}$ for $i\in[d_1]$; we endow this latter Euclidean space with the norm $\|(u,v,w,z)\|_{\tau \mathcal Q}=\|u\|+\|v\|+\|w\|+\|z\|_{\infty}$ where $u\in\mathbb R^{d_0d_1}$, $v\in\mathbb R^{d_0}$, $w\in\mathbb R^{d_1}$, and $z\in\mathbb R^N$. 

By Theorem 2.7.11 in \cite{van1996weak}, $N_{[\;]}(2\epsilon C_{M,\mathcal X_0,\mathcal X_1},\mathcal G, L^2(\mu_0))\leq N(\epsilon,\tau \mathcal Q,\|\cdot\|_{\tau\mathcal Q})$. It suffices, therefore, to bound $N(\epsilon,\tau \mathcal Q,\|\cdot\|_{\tau\mathcal Q})$. To this end, observe that $\tau\mathcal Q\subset B_{\|\cdot\|_{\tau Q}}(M+\|\mathcal X_0\|_{\infty}+\|\mathcal X_1\|_{\infty}+K)$, the ball of radius $M+\|\mathcal X_0\|_{\infty}+\|\mathcal X_1\|_{\infty}+K$ in $\mathbb R^{d_0d_1+d_0+d_1+N}$ with respect to $\|\cdot\|_{\tau\mathcal Q}$.  
Let $S_{\epsilon}$ be an $\epsilon$-net for this ball with respect to  $\|\cdot\|_{\tau\mathcal Q}$ and note that 
$
    \cup_{x\in S_{\varepsilon}}(x+\frac \epsilon 2B_{\tau\mathcal Q}(1))\subset (M+\|\mathcal X_0\|_{\infty}+\|\mathcal X_1\|_{\infty}+K+\frac \epsilon 2)B_{\tau\mathcal Q}(1). 
$
As the sets on the left hand side of the previous expression are disjoint, \[|S_{\epsilon}|\left(\frac \epsilon 2\right)^{d_0d_1+d_0+d_1+N}\vol\left(B_{\tau \mathcal Q}(1)\right)\left(M+\|\mathcal X_0\|_{\infty}+\|\mathcal X_1\|_{\infty}+K+\frac \epsilon 2\right)^{d_0d_1+d_0+d_1+N}\vol\left(B_{\tau \mathcal Q}(1)\right)
\]
whereby $|S_{\epsilon}|\leq \left(\frac{2(M+\|\mathcal X_0\|_{\infty}+\|\mathcal X_1\|_{\infty}+K)}{\epsilon}+1\right)^{d_0d_1+d_0+d_1+N}$. 

In sum, $N(\epsilon,\tau \mathcal Q,\|\cdot\|_{\tau\mathcal Q})\leq\left(\frac{2(M+\|\mathcal X_0\|_{\infty}+\|\mathcal X_1\|_{\infty}+K)}{\epsilon}\bigvee 2+1\right)^{d_0d_1+d_0+d_1+N}$ so that $\mathcal G$ is $\mu_0$-Donsker provided that it admits a cover which is square integrable with respect to $\mu_0$ (see Theorem 3.7.38 in \cite{gine2021mathematical}). This last condition is evidently satisfied, as  $
    c_{\mathbf{A}}(x-\xi,y^{(i)}-\zeta)-z_i$ can be bounded by a constant independently of the choice of $(\mathbf{A},\xi,\zeta,z)\in\mathcal Q$, $x\in\mathcal X_0$, and $i\in[N]$.

As for $\mathcal F_1$, note that if $f,g\in\mathcal F_1$, then $\|f-g\|_{L^2(\mu_1)}^2=\sum_{i=1}^N(f(y^{(i)})-g(y^{(i)}))^2\mu_1(\{y^{(i)}\})\leq \|z_f-z_g\|_{\infty}^2$ where $z_f=(f(y^{(1)}),\dots,f(y^{(N)}))$ and $z_g$ is defined analogously. Thus, $N(\epsilon,\mathcal F_1,L^2(\mu_1))\leq N(\epsilon,\{z\in\mathbb R^N:\|z\|_{\infty}\leq K\},\|\cdot\|_{\infty})\leq \left(\frac{2K}{\epsilon}\bigvee 2+1\right)^N$, where the second inequality follows from the first part of the proof. Conclude from Theorem 2.5.2 in \cite{van1996weak} that $\mathcal F_1$ is $\mu_1$-Donsker.
\qed

\subsection{Proofs for \texorpdfstring{ \cref{sec:entropicGW}}{Section 5}}

\subsubsection{Proof of \texorpdfstring{Theorem 9}{\cref{thm:entropicGWStability}}}
\label{proof:thm:entropicGWStability}

We first show that the centered measures are $4$-sub-Weibull with a particular constant and establish a bound on the second moment of a $4$-sub-Weibull distribution. 

\begin{lemma} 
\label{lem:4suWeibullAssn}
Fix $\nu\in\mathcal P_{4,\sigma}(\mathbb R^{d})$. Then, $\bar \nu\in\mathcal P_{4,\bar \sigma}(\mathbb R^{d})$ for $\bar\sigma^2=\frac{2+\log 2}{\log 2}8\sigma^2$ and $M_2(\nu),M_2(\bar \nu)\leq 2\sigma$. 
\end{lemma}
\begin{proof}
    As $\nu\in\mathcal P_{4,\sigma}(\mathbb R^d)$,
    \[
        \int e^{\frac{\|\cdot\|^4}{2\bar\sigma^2}}d\bar{\nu}=\int e^{\frac{\|\cdot-\mathbb E_{\nu}[X]\|^4}{2\bar\sigma^2}}d{\nu}\leq e^{\frac{8\|\mathbb E_{\nu}[X]\|^4}{2\bar\sigma^2}}\int e^{\frac{8\|\cdot\|^4}{2\bar\sigma^2}}d{\nu}\leq e^{\frac{8\|\mathbb E_{\nu}[X]\|^4}{2\bar\sigma^2}}\left(\int e^{\frac{\|\cdot\|^4}{2\sigma^2}}d{\nu}\right)^{\frac{8\sigma^2}{\bar\sigma^2}}\leq e^{\frac{32\sigma^2}{2\bar\sigma^2}}2^{\frac{8\sigma^2}{\bar\sigma^2}}= 2,            
    \]
    where the first inequality follows from the triangle inequality and the inequality $(|a|+|b|)^p\leq 2^{p-1}(|a|^p+|b|^p)$ for $a,b\in\mathbb R$ and $p\geq 1$, the second is due to Jensen's inequality, noting that $8\sigma^2\leq \bar \sigma^2$, and the final inequality follows from the $4$-sub-Weibull assumption and the fact that 
    \[
        2\geq \int e^{\frac{\|\cdot\|^4}{2\sigma^2}}d\nu\geq \int \frac{\|\cdot\|^4}{2\sigma^2}d\nu=M_4(\nu)/(2\sigma^2) 
    \]
    so that $     \|\mathbb E_{\nu}[X]\|^4\leq \mathbb E_{\nu}[\|X\|^4]\leq 4\sigma^2$ by Jensen's inequality. This proves the first claim. For the second claim, we have by Jensen's inequality that $M_2^2(\nu)\leq M_4(\nu)\leq 4\sigma^2$ as above. Conclude by noting that $M_2(\bar \nu)=M_2(\nu)-\|\mathbb E_{\nu}[X]\|^2\leq M_2(\nu)$.
\end{proof}

Consequently, we set $M=\sigma$ and analyze the  regularity of EOT potentials for $\mathsf{OT}_{\mathbf A,\varepsilon}(\nu_0,\nu_1)$ for $4$-sub-Weibull distributions with $\mathbf A\in B_{\mathrm F}(M)$. 
\begin{lemma}[Regularity of EOT potentials \cite{zhang2024gromov}]
    \label{prop:entropicGWPotentials}
    Fix $\sigma>0$. Then, for any choice of $(\nu_0,\nu_1)\in\mathcal P_{4,\sigma}(\mathbb R^{d_0})\times\mathcal P_{4,\sigma}(\mathbb R^{d_1})$ and $\mathbf{A}\in B_{\mathrm{F}}(M)$, there exists a pair of EOT potentials $(\varphi_0^{\mathbf{A},\varepsilon},\varphi_1^{\mathbf{A},\varepsilon})$ for $\mathsf{OT}_{\mathbf{A},\varepsilon}(\nu_0,\nu_1)$ satisfying the Schr{\"o}dinger system \eqref{eq:SchrodingerSystem} on $\mathbb R^{d_0}\times \mathbb R^{d_1}$. Furthermore, this pair is unique up to additive constants on $\mathbb R^{d_0}\times \mathbb R^{d_1}$ and, for every $k\in\mathbb N$, there exists a constant $K$ depending only on ${k,d_0,d_1,\sigma},$ and $\varepsilon$ for which
    \[   \varphi_i^{\mathbf{A},\varepsilon}\leq  K(1+\|\cdot\|^{2}),\quad -\varphi_i^{\mathbf{A},\varepsilon} \leq  K(1+\|\cdot\|^{4}), \quad
        |D^{\alpha}\varphi_i^{\mathbf{A},\varepsilon}|\leq  K(1+\|\cdot\|^{3k}),
    \]
    for every nonzero multi-index $\alpha\in\mathbb N_0^{d_i}$ with $|\alpha|\leq k$ and $i\in\{0,1\}$. 
\end{lemma}
\begin{proof}
    Lemma 4 in \cite{zhang2024gromov} establishes that there exists a pair of EOT potentials $(\varphi_0^{\mathbf{A},1},\varphi_1^{\mathbf{A},1})$ for $\mathsf{OT}_{\mathbf{A},1}(\nu_0,\nu_1)$ satisfying the Schr\"odinger system on $\mathbb R^{d_0}\times \mathbb R^{d_1}$ and, for any fixed $k\in\mathbb N$, there exists a constant $K'$ depending only on ${k,d_0,d_1,\sigma}$ for which 
    \begin{equation}
    \label{eq:entropicOT1bound}
   \varphi_i^{\mathbf{A},1}\leq K'(1+\|\cdot\|^{2}),\quad -\varphi_i^{\mathbf{A},1} \leq K'(1+\|\cdot\|^{4}), \quad
        |D^{\alpha}\varphi_i^{\mathbf{A},1}|\leq  K'(1+\|\cdot\|^{3k})
    \end{equation}
    for every nonzero multi-index $\alpha\in\mathbb N_0^{d_i}$ with $|\alpha|\leq k$ and $i\in\{0,1\}$. To extend this result to the general case $\varepsilon >0$, observe that if $(\eta_0,\eta_1)\in\mathcal P(\mathbb R^{d_0})\times\mathcal P(\mathbb R^{d_1}) $ have finite fourth moments and  $\pi\in\Pi(\eta_0,\eta_1)$ with $\pi\ll\eta_0\otimes \eta_1$, then
        \[
            \int c_{\mathbf{A}}d\pi+\varepsilon\mathsf{D}_{\mathrm{KL}}(\pi\|\eta_0\otimes\eta_1)=\varepsilon\left(\int c_{\mathbf{A}}d\pi^{\varepsilon}+\mathsf{D}_{\mathrm{KL}}(\pi^{\varepsilon}\|\eta_0^{\varepsilon}\otimes\eta_1^{\varepsilon})\right),
        \]
        where $\pi^{\varepsilon}=(\varepsilon^{-1/4}\Id,\varepsilon^{-1/4}\Id)_{\sharp}\pi$ and similarly for $\eta_0^{\varepsilon}$ and $\eta_1^{\varepsilon}$. The above display implies that $\mathsf{OT}_{\mathbf{A},\varepsilon}(\eta_0,\eta_1)=\varepsilon\mathsf{OT}_{\mathbf{A},1}(\eta_0^{\varepsilon},\eta_1^{\varepsilon})$, that if $\pi_{\star}$ is optimal for $\mathsf{OT}_{\mathbf{A},\varepsilon}(\eta_0,\eta_1)$, then $\pi_{\star}^{\varepsilon}$ is optimal for $\mathsf{OT}_{\mathbf{A},1}(\eta_0^{\varepsilon},\eta_1^{\varepsilon})$, and $\frac{d\pi_{\star}}{d\eta_0\otimes \eta_1}=e^{\frac{\varphi_0\oplus\varphi_1-c_{\mathbf{A}}}{\varepsilon}}$ where $(\varphi_0,\varphi_1)$ are the associated EOT potentials whereas $\frac{d\pi_{\star}^{\varepsilon}}{d\eta_0^{\varepsilon}\otimes \eta_1^{\varepsilon}}=\frac{d\pi_{\star}}{d\eta_0\otimes \eta_1}\circ (\varepsilon^{1/4}\Id)=e^{{\varepsilon^{-1}(\varphi_0(\varepsilon^{1/4}\cdot)\oplus\varphi_1(\varepsilon^{1/4}\cdot))-c_{\mathbf{A}}}}$ so that $\varepsilon^{-1}(\varphi_0(\varepsilon^{1/4}\cdot),\varphi_1(\varepsilon^{1/4}\cdot))$ is optimal for $\mathsf{OT}_{\mathbf{A},1}(\eta_0^{\varepsilon},\eta_1^{\varepsilon})$. We also point out that if $\eta_0,\eta_1$ are $4$-sub-Weibull with parameter $\sigma^2$, then $\eta_0^{\varepsilon},\eta_1^{\varepsilon}$ are $4$-sub-Weibull with parameter $\sigma^2/\varepsilon$. The claimed estimates then follow readily from \eqref{eq:entropicOT1bound}.

        Finally, assume that $(\varphi_0,\varphi_1)$ is another pair of EOT potentials for $\mathsf{OT}_{\mathbf{A},\varepsilon}(\nu_0,\nu_1)$, then $(\varphi_0,\varphi_1)$ coincides with $(\varphi_0^{\mathbf{A},\varepsilon},\varphi_1^{\mathbf{A},\varepsilon})$ up to additive constants $\nu_0\otimes \nu_1$-a.s. Thus, there exists some $a\in\mathbb R$ for which 
        \[
            e^{-\frac{\varphi_0(x)}{\varepsilon}}=\int e^{\frac{\varphi_1-c_{\mathbf{A}}(x,\cdot)}\varepsilon}d\nu_1=\int e^{\frac{\varphi_1^{\mathbf{A},\varepsilon}-a-c_{\mathbf{A}}(x,\cdot)}\varepsilon}d\nu_1=e^{-\frac{\varphi_0^{\mathbf{A},\varepsilon}(x)+a}\varepsilon},
        \]
        for every $x\in\mathbb R^{d_0}$ so that $\varphi_0$ coincides with $\varphi_0^{\mathbf{A},\varepsilon}+a$ on $\mathbb R^{d_0}$. The same argument applies to $\varphi_1$.
\end{proof}

To prove \cref{thm:entropicGWStability}, we adopt a similar strategy to the semi-discrete case (\cref{thm:semidiscreteGWStability}). Namely, for $i\in\{0,1\}$, we set $\mu_{i,t}=\mu_i+t(\nu_i-\mu_i)$ for $t\in [0,1]$ and some pair $(\nu_0,\nu_1)\in\mathcal P_{4,\sigma}(\mathbb R^{d_0})\times\mathcal P_{4,\sigma}(\mathbb R^{d_1})$ and utilize the variational form $\mathsf D_{\varepsilon}(\bar \mu_0,\bar \mu_1)^2=\mathsf S_1(\bar \mu_1,\bar \mu_1)+\mathsf S_{2,\varepsilon}(\bar \mu_0,\bar \mu_1)$ from \eqref{eq:GWVariational}. We make precise that, since $\mu_0,\mu_1,\nu_0,\nu_1$ are all $4$-sub-Weibull with the same parameter $\sigma$, so too are $\mu_{0,t}$ and $\mu_{1,t}$ with the same parameter.

\begin{lemma}
\label{lem:entropicGateauxD1}
As $t\downarrow 0$, we have that 
    \[
    \begin{aligned}
        \frac{\mathsf S_1(\bar\mu_{0,t},\bar\mu_{1,t})-\mathsf S_1(\bar\mu_0,\bar\mu_1)}{t}&\to 2\iint \|x-x'\|^4 d\mu_{0}(x)d(\nu_0-\mu_{0})(x')
        \\
        &+2\iint \|y-y'\|^4 d\mu_{1}(y)d(\nu_1-\mu_{1})(y')
        \\
        &-4 \iint \|x-\mathbb E_{\mu_0}[X]\|^2\|y-\mathbb E_{\mu_1}[X]\|^2d(\nu_0-\mu_0)(x)d\mu_1(y)
        \\
        &-4\iint \|x-\mathbb E_{\mu_0}[X]\|^2\|y-\mathbb E_{\mu_1}[X]\|^2d\mu_0(x)d(\nu_1-\mu_1)(y).
    \end{aligned} 
\]
\end{lemma}

The proof of \cref{lem:entropicGateauxD1} follows the same lines as the proof of \cref{lem:semidiscreteGateauxD1}. The only distinction is that boundedness of all relevant integrals is justified using the $4$-sub-Weibull assumption rather than compact support. The proof is thus ommitted for brevity.

\begin{lemma}
    \label{lem:entropicS1Lipschitz}

    For any pairs $(\nu_0, \nu_1),(\rho_0, \rho_1)\in\mathcal P_{\mu_0}\times \mathcal P_{\mu_1}$, $\left|\mathsf S_{1}(\bar\nu_0,\bar\nu_1)-\mathsf S_{1}(\bar\rho_0,\bar\rho_1)\right|\leq C\|\nu_0\otimes \nu_1-\rho_0\otimes\rho_1\|_{\infty,\mathcal F^{\oplus}}$ for some constant $C>0$ which is independent of $(\nu_0, \nu_1),(\rho_0, \rho_1)$. 
\end{lemma}  
\begin{proof}
The same arguments from the proof of \cref{lem:semidiscreteS1Lipschitz} apply in this setting with the distinction that the relevant integrals can be bounded (uniformly in the choice of $\nu_0,\nu_1,\rho_0,\rho_1$) by appealing to the $4$-sub-Weibull assumption.

For completeness, we underscore that the functions \[
\begin{aligned}
\xi_0:x\in\mathbb R^{d_0}&\mapsto \int \|x-x'\|^4d\nu_0(x')+\int \|x-x'\|^4d\rho_0(x'),
\\
\xi_1:x\in\mathbb R^{d_1}&\mapsto \int \|x-x'\|^4d\nu_1(x')+\int \|x-x'\|^4d\rho_1(x'),
\\
\zeta_0:x\in\mathbb R^{d_0}&\mapsto \|x-\mathbb E_{\mu_0}[X]\|^2,
\\
\zeta_1:x\in\mathbb R^{d_1}&\mapsto \|x-\mathbb E_{\nu_1}[X]\|^2,
\end{aligned}
\]
are smooth and have derivatives (of all orders) which can be bounded as $1+\|\cdot\|^4$ up to a multiplicative constant which is independent of $\nu_0,\nu_1,\rho_0,\rho_1$. Conclude that $\xi_0,\zeta_0\in\mathcal F_0$ and $\xi_1,\zeta_1\in\mathcal F_1$ so that the remainder of the proof follows analogously to the proof of \cref{lem:semidiscreteS1Lipschitz}.   
\end{proof}

\begin{lemma}
    \label{lem:entropicOTGateaux}
If $(\mu_0,\mu_1)\in\mathcal P_{4,\sigma}(\mathbb R^{d_0})\times \mathcal P_{4,\sigma}(\mathbb R^{d_1})$ for some $\sigma>0$, then, for any $\mathbf{A}\in B_{\mathrm{F}}(M)$,
    \[ 
    \begin{aligned}
    \lim_{t\downarrow 0} \frac{\mathsf{OT}_{\mathbf{A},\varepsilon}(\bar \mu_{0,t},\bar \mu_{1,t})-\mathsf{OT}_{\mathbf{A},\varepsilon}(\bar \mu_{0},\bar \mu_{1})}{t}&=\int \bar\varphi_{0}^{{\mathbf{A},\varepsilon}}+8\mathbb E_{(X,Y)\sim \pi^{\mathbf{A},\varepsilon}}[\|Y\|^2X]^{\intercal}(\cdot)d(\nu_0-\mu_0)
    \\
    &+\int \bar\varphi_{1}^{{\mathbf{A},\varepsilon}}+8\mathbb E_{(X,Y)\sim \pi^{\mathbf{A},\varepsilon}}[\|X\|^2Y]^{\intercal}(\cdot)d(\nu_1-\mu_1). 
    \end{aligned}
    \]
\end{lemma}

\begin{proof}
    Fix $\mathbf{A}\in B_{\mathrm{F}}(M)$. Proceeding as in the semi-discrete case (\cref{lem:semidiscreteOTGateaux}), consider 
    \[\begin{aligned}\lim_{t\downarrow 0}\frac{\mathsf{OT}_{\mathbf{A},\varepsilon}(\bar \mu_{0,t},\bar \mu_{1,t})-\mathsf{OT}_{\mathbf{A},\varepsilon}(\bar \mu_{0},\bar \mu_{1})}{t}&=\lim_{t\downarrow 0}\left(\frac{\mathsf{OT}_{\mathbf{A},\varepsilon}(\bar \mu_{0,t},\bar \mu_{1,t})-\mathsf{OT}_{\mathbf{A},\varepsilon}(\tilde \mu_{0,t},\tilde \mu_{1,t})}{t}\right.
    \\&\hspace{11em}\left.+\frac{\mathsf{OT}_{\mathbf{A},\varepsilon}(\tilde \mu_{0,t},\tilde \mu_{1,t})-\mathsf{OT}_{\mathbf{A},\varepsilon}(\bar \mu_{0},\bar \mu_{1})}{t}\right)
    \end{aligned}
    \]
    where $(\tilde \mu_{0,t},\tilde \mu_{1,t})=\left((\Id-\mathbb E_{\mu_0}[X])_{\sharp}\mu_{0,t},(\Id-\mathbb E_{\mu_1}[X])_{\sharp}\mu_{1,t}\right)$. We compute each limit on the right hand side separately.

    First, let $(\varphi_{0}^{(\nu_0,\nu_1)},\varphi_{1}^{(\nu_0,\nu_1)})$ be some choice of EOT potentials for $\mathsf{OT}_{\mathbf{A},\varepsilon}(\nu_0,\nu_1)$ from \cref{prop:entropicGWPotentials} for any choice of $4$-sub-Weibull distributions $(\nu_0,\nu_1)$ and observe that  
    \begin{equation}
    \label{eq:entropicGateauxCentredUpper1}
    \begin{aligned}
    \mathsf{OT}_{\mathbf{A},\varepsilon}(\tilde\mu_{0,t},\tilde\mu_{1,t})-\mathsf{OT}_{\mathbf{A},\varepsilon}(\tilde\mu_{0,t},\bar\mu_{1})
       &\leq \int \varphi_{0}^{(\tilde \mu_{0,t},\tilde \mu_{1,t})}d\tilde \mu_{0,t}+\int \varphi_{1}^{(\tilde \mu_{0,t},\tilde \mu_{1,t})}d\tilde \mu_{1,t}-\int \varphi_{0}^{(\tilde \mu_{0,t},\tilde \mu_{1,t})}d\tilde \mu_{0,t}
       \\
       &-\int \varphi_{1}^{(\tilde \mu_{0,t},\tilde \mu_{1,t})}d\bar \mu_{1}+\varepsilon \int e^{\frac{\varphi_{0}^{(\tilde \mu_{0,t},\tilde \mu_{1,t})}\oplus\varphi_{1}^{(\tilde \mu_{0,t},\tilde \mu_{1,t})}-c_{\mathbf{A}}}\varepsilon} d\tilde \mu_{0,t}\otimes \bar \mu_1-\varepsilon
       \\
       &=t\int \bar\varphi_1^{(\tilde  \mu_{0,t},\tilde  \mu_{1,t})}d(\nu_1-\mu_1)
    \end{aligned}
    \end{equation}
    where we have used the fact that $\int e^{\frac{\varphi_{0}^{(\tilde \mu_{0,t},\tilde \mu_{1,t})}\oplus\varphi_{1}^{(\tilde \mu_{0,t},\tilde \mu_{1,t})}-c_{\mathbf{A}}}\varepsilon}d\tilde \mu_{0,t}\equiv 1$ on $\mathbb R^{d_1}$ (as the EOT potentials solve the relevant Shr\"odinger systems) and we recall the notation $(\bar g_0,\bar g_1)=(g_0(\cdot-\mathbb E_{\mu_0}[X]), g_1(\cdot-\mathbb E_{\mu_1}[X]))$. Similarly, 
    \begin{equation}
    \label{eq:entropicGateauxCentredUpper2}
    \begin{aligned}
    \mathsf{OT}_{\mathbf{A},\varepsilon}(\tilde\mu_{0,t},\tilde\mu_{1,t})-\mathsf{OT}_{\mathbf{A},\varepsilon}(\tilde\mu_{0,t},\bar\mu_{1})
       &\geq t\int \bar\varphi_1^{(\tilde  \mu_{0,t},\bar \mu_{1})}d(\nu_1-\mu_1),
       \\
       \mathsf{OT}_{\mathbf{A},\varepsilon}(\tilde\mu_{0,t},\bar\mu_{1})-\mathsf{OT}_{\mathbf{A},\varepsilon}(\bar\mu_{0},\bar\mu_{1})
       &\leq t\int \bar\varphi_0^{(\tilde  \mu_{0,t},\bar \mu_{1})}d(\nu_0-\mu_0),
       \\
       \mathsf{OT}_{\mathbf{A},\varepsilon}(\tilde\mu_{0,t},\bar\mu_{1})-\mathsf{OT}_{\mathbf{A},\varepsilon}(\bar\mu_{0},\bar\mu_{1})
       &\geq t\int \bar\varphi_0^{(\bar  \mu_{0},\bar \mu_{1})}d(\nu_0-\mu_0),
    \end{aligned}
    \end{equation}
    whereby 
    \begin{equation}
    \label{eq:entropicOTUpper}
    \begin{aligned}
        \left|\mathsf{OT}_{\mathbf{A},\varepsilon}(\tilde\mu_{0,t},\tilde\mu_{1,t})-\mathsf{OT}_{\mathbf{A},\varepsilon}(\bar\mu_{0},\bar\mu_{1})\right|\leq t\left( \left|\int \bar\varphi_0^{(\tilde  \mu_{0,t},\bar \mu_{1})}d(\nu_0-\mu_0)+\int \bar\varphi_1^{(\tilde  \mu_{0,t},\tilde  \mu_{1,t})}d(\nu_1-\mu_1)\right|\right.&
        \\
        &\hspace{-19em}\bigvee\left.\left|\int \bar\varphi_0^{(\bar  \mu_{0},\bar \mu_{1})}d(\nu_0-\mu_0)+\int \bar\varphi_1^{(\tilde  \mu_{0,t},\bar \mu_{1})}d(\nu_1-\mu_1)\right| \right)
    \end{aligned}
    \end{equation}
        It suffices, therefore to show that both terms in the absolute values converge to $\int \bar\varphi_0^{(\bar  \mu_{0},\bar \mu_{1})}d(\nu_0-\mu_0)+\int \bar\varphi_1^{(\bar  \mu_{0},\bar \mu_{1})}d(\nu_1-\mu_1)$ as $t\downarrow 0$.  

        To this end,  fix an arbitrary sequence $t_n\downarrow 0$ and a subsequence $t_{n'}\downarrow 0$. By the Arz{\`e}la-Ascoli theorem, there is a further subsequence $t_{n''}\downarrow 0$ along which $(\varphi_0^{(\tilde  \mu_{0,t_{n''}},\tilde \mu_{1,t_{n''}})},\varphi_1^{(\tilde  \mu_{0,t_{n''}},\tilde \mu_{1,t_{n''}})})\to (\psi_0,\psi_1)$  uniformly on compact sets (in particular pointwise) where $( \psi_0,\psi_1)$ is a pair of continuous functions. Given that the EOT potentials satisfy the estimates from \cref{prop:entropicGWPotentials} and that $\mu_0,\mu_1,\nu_0,\nu_1$ are $4$-sub-Weibull, we may apply the dominated convergence theorem to obtain that, for every $(x,y)\in\mathbb R^{d_0}\times \mathbb R^{d_1}$, 
        \[
        \begin{aligned} 
            e^{-\frac{\varphi_0^{(\tilde  \mu_{0,t_{n''}},\tilde \mu_{1,t_{n''}})}(x)}\varepsilon}=\int
             e^{\frac{\varphi_1^{(\tilde  \mu_{0,t_{n''}},\tilde \mu_{1,t_{n''}})}-c_{\mathbf{A}}(x,\cdot)}\varepsilon}d \tilde \mu_{1,t_{n''}}&\to  \int e^{\frac{\psi_1-c_{\mathbf{A}}(x,\cdot)}\varepsilon}d \bar \mu_{1}=e^{-\frac{\psi_0(x)}\varepsilon},\\  
            e^{-\frac{\varphi_1^{(\tilde  \mu_{0,t_{n''}},\tilde \mu_{1,t_{n''}})}(y)}\varepsilon}=\int e^{\frac{\varphi_0^{(\tilde  \mu_{0,t_{n''}},\tilde \mu_{1,t_{n''}})}-c_{\mathbf{A}}(\cdot,y)}\varepsilon}d \tilde \mu_{0,t_{n''}}&\to  \int e^{\frac{\psi_0-c_{\mathbf{A}}(\cdot,y)}\varepsilon}d \bar \mu_{0}=e^{-\frac{\psi_1(y)}\varepsilon},
        \end{aligned}
        \]
        so that $(\psi_0,\psi_1)$ solve the Schr\"odinger system for $\mathsf{OT}_{\mathbf{A},\varepsilon}(\bar \mu_0,\bar \mu_1)$ on $\mathbb R^{d_0}\times \mathbb R^{d_1}$. Conclude from \cref{prop:entropicGWPotentials} that $(\psi_0,\psi_1)$ coincides with $(\varphi_0^{(\bar  \mu_{0},\bar \mu_{1})},\varphi_1^{(\bar  \mu_{0},\bar \mu_{1})})$ up to additive constants whence $\int \bar\varphi_1^{(\tilde \mu_{0,t_{n''}},\tilde \mu_{1,t_{n''}})}d(\nu_1-\mu_1)\to \int \bar \varphi_1^{(\bar \mu_0,\bar \mu_1)}d(\nu_1-\mu_1)$. By an analogous argument it holds that $\int \bar\varphi_0^{(\tilde  \mu_{0,t_{n'''}},\bar \mu_{1})}d(\nu_0-\mu_0)\to\int \bar\varphi_0^{(\bar  \mu_{0},\bar \mu_{1})}d(\nu_0-\mu_0), 
        \int \bar\varphi_1^{(\tilde  \mu_{0,t_{n'''}},\bar \mu_{1})}d(\nu_1-\mu_1)\to \int \bar\varphi_1^{(\bar  \mu_{0},\bar \mu_{1})}d(\nu_1-\mu_1)$ up to passing to a further subsequence $t_{n'''}$. Conclude from equations \eqref{eq:entropicGateauxCentredUpper1} and \eqref{eq:entropicGateauxCentredUpper2} that 
        \[
            \lim_{n'''\to \infty}\frac{\mathsf{OT}_{\mathbf{A},\varepsilon}(\tilde \mu_{0,t_{n'''}},\tilde \mu_{1,t_{n'''}})-\mathsf{OT}_{\mathbf{A},\varepsilon}(\bar \mu_{0},\bar \mu_{1})}{t_{n'''}} = \int \bar\varphi_0^{(\bar \mu_0,\bar \mu_1)}d(\nu_0-\mu_0)+ \int \bar\varphi_1^{(\bar \mu_0,\bar \mu_1)}d(\nu_1-\mu_1). 
        \]
        As the original choice of subsequence was arbitrary and the limit is independent of this choice of subsequence, conclude that 
       \begin{equation}
       \label{eq:entropicOTLimit1}
       \lim_{t\downarrow 0}     
            \frac{\mathsf{OT}_{\mathbf{A},\varepsilon}(\tilde \mu_{0,t},\tilde \mu_{1,t})-\mathsf{OT}_{\mathbf{A},\varepsilon}(\bar \mu_{0},\bar \mu_{1})}{t} =\int \bar\varphi_0^{(\bar \mu_0,\bar \mu_1)}d(\nu_0-\mu_0)+ \int \bar\varphi_1^{(\bar \mu_0,\bar \mu_1)}d(\nu_1-\mu_1).
       \end{equation}

        On the other hand, adopting the notation $\tau^t$ from the proof of \cref{lem:semidiscreteOTGateaux} and letting $\bar \pi_t,\tilde \pi_t$ be optimal for $\mathsf{OT}_{\mathbf{A},\varepsilon}(\bar \mu_{0,t},\bar \mu_{1,t})$ and $\mathsf{OT}_{\mathbf{A},\varepsilon}(\tilde \mu_{0,t},\tilde \mu_{1,t})$ respectively, we have that 
        \begin{equation}
        \label{eq:entropicOTLower}
        \begin{aligned}
            \mathsf{OT}_{\mathbf{A},\varepsilon}(\bar \mu_{0,t},\bar \mu_{1,t})-\mathsf{OT}_{\mathbf{A},\varepsilon}(\tilde \mu_{0,t},\tilde \mu_{1,t})&\leq \int c_{\mathbf{A}}\circ \tau^{-t}-c_{\mathbf{A}} d\tilde \pi_t,
            \\ 
            \mathsf{OT}_{\mathbf{A},\varepsilon}(\bar \mu_{0,t},\bar \mu_{1,t})-\mathsf{OT}_{\mathbf{A},\varepsilon}(\tilde \mu_{0,t},\tilde \mu_{1,t})&\geq \int c_{\mathbf{A}}-c_{\mathbf{A}}\circ \tau^{t} d\bar \pi_t,
        \end{aligned}
        \end{equation}
        as $\mathsf{D}_{\mathrm{KL}}$ is translation invariant. 
        Using the characterization of the Radon-Nikodym derivative of the EOT plan (see \cref{sec:EOT}), we have
       \[
       \begin{aligned}
            &\int c_{\mathbf{A}}\circ \tau^{-t}-c_{\mathbf{A}} d\tilde \pi_t\\&=\int (c_{\mathbf{A}}\circ \tau^{-t}-c_{\mathbf{A}}) \frac{d\tilde \pi_t}{d\tilde \mu_{0,t}\otimes\tilde \mu_{1,t} }d\tilde \mu_{0,t}\otimes \tilde \mu_{1,t}
            \\
            &=\int \left((c_{\mathbf{A}}\circ \tau^{-t}-c_{\mathbf{A}}) e^{\frac{\tilde \varphi_{0,t}\oplus\tilde \varphi_{1,t}-c_{\mathbf{A}}}\varepsilon}\right)(\cdot-\mathbb E_{\mu_0}[X],\cdot-\mathbb E_{\mu_1}[X])d\left( \mu_{0}+t(\nu_0-\mu_0)\right)\otimes \left( \mu_{1}+t(\nu_1-\mu_1)\right),
        \end{aligned}
        \]
        where $(\tilde \varphi_{0,t},\tilde \varphi_{1,t})$ is a pair of EOT potentials for $\mathsf{OT}_{\mathbf{A},\varepsilon}(\tilde \mu_{0,t},\tilde \mu_{1,t})$. Arguing as in the first part of the proof, if we fix a sequence $t_n\downarrow 0$ and a subsequence $t_{n'}\downarrow 0$, there exists a further subsequence $t_{n''}\downarrow 0$ along which $\tilde \varphi_{0,t_{n''}}\oplus\tilde \varphi_{1,t_{n''}}\to \varphi_{0}\oplus\varphi_{1}$ pointwise on $\mathbb R^{d_0}\times \mathbb R^{d_1}$ where $(\varphi_{0},\varphi_{1})$ is a pair of EOT potentials for $\mathsf{OT}_{\mathbf{A},\varepsilon}(\bar\mu_{0},\bar \mu_{1})$ so that $e^{\frac{\tilde \varphi_{0,t_{n''}}\oplus\tilde \varphi_{1,t_{n''}}-c_{\mathbf{A}}}\varepsilon}\to e^{\frac{ \varphi_{0}\oplus\varphi_{1}-c_{\mathbf{A}}}\varepsilon} $ pointwise and we highlight that 
 $e^{\frac{ \varphi_{0}\oplus\varphi_{1}-c_{\mathbf{A}}}\varepsilon}=\frac{d\pi^{\mathbf{A}}}{d\bar\mu_0\otimes \bar \mu_1}$ $\bar\mu_0\otimes \bar \mu_1$-a.e. where 
 $\pi^{\mathbf{A}}$ is the EOT plan for $\mathsf{OT}_{\mathbf{A},\varepsilon}(\bar\mu_{0},\bar \mu_{1})$. From \eqref{eq:costTranslationExpansion}, 
        we have that
       \[
       \begin{aligned}
        t^{-1}\left(c_{{\mathbf{A}}}\circ \tau^{-t}(x,y)-c_{{\mathbf{A}}}(x,y)\right)
        &\to 8\langle x,\mathbb E_{\nu_0}[X]-\mathbb E_{\mu_0}[X]\rangle\|y\|^2+8\langle y,\mathbb E_{\nu_1}[X]-\mathbb E_{\mu_1}[X]\rangle\|x\|^2
        \\
        &+32\left(\mathbb E_{\nu_0}[X]-\mathbb E_{\mu_0}[X]\right)^{\intercal}{\mathbf{A}}y+32x^{\intercal}{\mathbf{A}}\left(\mathbb E_{\nu_1}[X]-\mathbb E_{\mu_1}[X]\right)
       \end{aligned}
       \]
       pointwise and, from the potential estimates, $t^{-1}\left(c_{{\mathbf{A}}}\circ \tau^{-t}(x,y)-c_{{\mathbf{A}}}(x,y)\right)e^{\frac{\tilde \varphi_{0,t}\oplus\tilde \varphi_{1,t}-c_{\mathbf{A}}}{\varepsilon}}$ is equiintegrable (in $t\in[0,1]$, $\mathbf{A}\in B_{\mathrm{F}}(M)$) with respect to any product of $4$-sub-Weibull distributions. Letting $h$ denote the limit in the above display, the dominated convergence theorem yields
       \[
       \begin{aligned}
            t_{n''}^{-1}\int (c_{{\mathbf{A}}}\circ \tau^{-t_{n''}}-c_{{\mathbf{A}}})d\tilde \pi_{t_{n''}}&\to\int \left(h\frac{d \pi^{\mathbf{A}}}{d\bar\mu_0\otimes \bar \mu_1}\right)(\cdot-\mathbb E_{\mu_0}[X],\cdot-\mathbb E_{\mu_1}[X])d\mu_0\otimes \mu_1
            \\
            &=\int hd\pi^{\mathbf{A}}
            \\
            &=\int 8\langle x,\mathbb E_{\nu_0}[X]-\mathbb E_{\mu_0}[X]\rangle\|y\|^2+8\langle y,\mathbb E_{\nu_1}[X]-\mathbb E_{\mu_1}[X]\rangle\|x\|^2d\pi^{\mathbf{A}}(x,y), 
       \end{aligned}
       \]
      as the marginals of $\pi^{\mathbf{A}}$ are mean-zero. Since the limit is independent of the choice of original subsequence, $t^{-1}\int (c_{{\mathbf{A}}}\circ \tau^{-t}-c_{{\mathbf{A}}})d\tilde \pi_{t}$ converges to the same limit. 
       The same arguments apply to $\int c_{\mathbf{A}}-c_{\mathbf{A}}\circ \tau^td\bar\pi_t$ so that 
       \begin{equation}
            \label{eq:entropicOTLimit2}
            \begin{aligned}
            \frac{\mathsf{OT}_{\mathbf{A},1}(\bar \mu_{0,t},\bar\mu_{1,t})-\mathsf{OT}_{\mathbf{A},1}(\tilde \mu_{0,t},\tilde\mu_{1,t})}{t}&\to \int 8\langle x,\mathbb E_{\nu_0}[X]-\mathbb E_{\mu_0}[X]\rangle\|y\|^2\\&\hspace{7em}+8\langle y,\mathbb E_{\nu_1}[X]-\mathbb E_{\mu_1}[X]\rangle\|x\|^2d\pi^{\mathbf{A}}(x,y).
            \end{aligned}
       \end{equation}
        Combining \eqref{eq:entropicOTLimit1} and \eqref{eq:entropicOTLimit2} proves the claim. 
\end{proof}

\begin{lemma}
    \label{lem:entropicS2Gateaux}
If $(\mu_0,\mu_1)\in\mathcal P_{4,\sigma}(\mathbb R^{d_0})\times \mathcal P_{4,\sigma}(\mathbb R^{d_1})$ for some $\sigma>0$ and $\mathcal A=\argmin_{B_{\mathrm{F}}(M)}\Phi_{(\bar \mu_0,\bar\mu_1)}$, then 
   \[
    \begin{aligned}
        &\lim_{t\downarrow 0 }\frac{\mathsf S_{2,\varepsilon}(\bar\mu_{0,t},\bar\mu_{1,t})-\mathsf S_{2,\varepsilon}(\bar \mu_0,\bar \mu_1)}{t}= \inf_{\mathbf{A}\in\mathcal A}\left\{
             \int g_{0,\pi^{\mathbf{A}}}+\bar{\varphi}_0^{\mathbf{A}}d(\nu_0-\mu_0)+\int g_{1,\pi^{\mathbf{A}}}+\bar{\varphi}_1^{\mathbf{A}}d(\nu_1-\mu_1)
            \right\},
    \end{aligned}
    \]   
    where 
    \[
        \begin{aligned}
           g_{0,\pi}&= 8\mathbb E_{(X,Y)\sim \pi}[\|Y\|^2X]^{\intercal}(\cdot),
            \\
            g_{1,\pi}&=8\mathbb E_{(X,Y)\sim \pi}[\|X\|^2Y]^{\intercal}(\cdot),
        \end{aligned}
    \]
    $\bar\varphi_i=\varphi_i(\cdot-\mathbb E_{\mu_i}[X])$ for $i=0,1$, and, for $\mathbf{A}\in \mathcal A$, $\pi^{\mathbf{A}}$ is the unique OT plan for $\OT_{\mathbf{A}}(\bar \mu_0,\bar \mu_1)$, and $(\varphi_0^{\mathbf{A}},\varphi_1^{\mathbf{A}})$ is any choice of  EOT potentials from \cref{prop:entropicGWPotentials}.
\end{lemma}
\begin{proof}
    Let $\mathbf{A}^{\star}\in\mathcal A$ be arbitrary. By \cref{lem:entropicOTGateaux},
    \[
    \begin{aligned}
        \mathsf S_{2,\varepsilon}(\bar\mu_{0,t},\bar\mu_{1,t})-\mathsf S_{2,\varepsilon}(\bar\mu_{0},\bar\mu_{1})&\leq \mathsf{OT}_{\mathbf{A}^{\star},\varepsilon} (\bar\mu_{0,t},\bar\mu_{1,t})-\mathsf{OT}_{\mathbf{A}^{\star},\varepsilon}(\bar\mu_{0},\bar\mu_{1}) \\&= \int g_{0,\pi^{\mathbf{A}^{\star}}}+\bar{\varphi}_0^{\mathbf{A}^{\star}}d(\nu_0-\mu_0)+\int g_{1,\pi^{\mathbf{A}^{\star}}}+\bar{\varphi}_1^{\mathbf{A}^{\star}}d(\nu_1-\mu_1)+o(t).
    \end{aligned}
    \]
    It follows that
    \begin{equation}
    \label{eq:entropicD2Upper}
        \limsup_{t\downarrow 0} \frac{\mathsf S_{2,\varepsilon}(\bar\mu_{0,t},\bar\mu_{1,t})-\mathsf S_{2,\varepsilon}(\bar\mu_{0},\bar\mu_{1})}{t}\leq \inf_{\mathbf{A}\in\mathcal A}\left\{\int g_{0,\pi^{\mathbf{A} }}+\bar{\varphi}_0^{\mathbf{A}}d(\nu_0-\mu_0)+\int g_{1,\pi^{\mathbf{A}}}+\bar{\varphi}_1^{\mathbf{A}}(\nu_1-\mu_1)\right\}.
    \end{equation}
    
    To prove the complementary inequality, let $\mathbf{A}_t$ be optimal for $\mathsf S_{2,\varepsilon}(\bar\mu_{0,t},\bar\mu_{1,t})$ for every $t\in [0,1]$. By the estimates for the EOT potentials, equations \eqref{eq:entropicOTUpper} and \eqref{eq:entropicOTLower}, and the $4$-sub-Weibull assumption, $\mathbf{A}\in B_{\mathrm{F}}(M)\mapsto 32\|\mathbf{A}\|_{\mathrm{F}}^2+\mathsf{OT}_{\mathbf{A},\varepsilon}(\bar \mu_{0,t},\bar \mu_{1,t})$ converges uniformly to $\mathbf{A}\in B_{\mathrm{F}}(M)\mapsto 32\|\mathbf{A}\|_{\mathrm{F}}^2+\mathsf{OT}_{\mathbf{A},\varepsilon}(\bar \mu_{0},\bar \mu_{1})$ as $t\downarrow 0$. Proceeding as in the proof of \cref{lem:semidiscreteGateauxD2} we obtain that any cluster point of $(\mathbf{A}_t)_{t\downarrow 0}$ is an element of $\mathcal A$.

   Now, 
    fix a sequence $t_n\downarrow 0$ and an arbitrary subsequence $t_{n'}\downarrow 0$. By the Bolzano-Weierstrass theorem, $(\mathbf{A}_{t_{n'}})_{n'\in\mathbb N}$ admits a cluster point $\mathbf{A}^{\star}\in\mathcal A$. Observe that
    \[
    \begin{aligned}
        \frac{\mathsf S_{2,\varepsilon}(\bar\mu_{0,t_{n'}},\bar\mu_{1,t_{n'}})-\mathsf S_{2,\varepsilon}(\bar\mu_{0},\bar\mu_{1})}{t_{n'}}&\geq \frac{\mathsf{OT}_{\mathbf{A}_{t_{n'}},\varepsilon} (\bar\mu_{0,t_{n'}},\bar\mu_{1,t_{n'}})-\mathsf{OT}_{\mathbf{A}_{t_{n'}},\varepsilon}(\bar\mu_{0},\bar\mu_{1})}{t_{n'}} \\&\geq  \int \bar{\varphi}_0^{\mathbf{A}_{t_{n'}},(\bar\mu_0,\bar\mu_1)}d(\nu_0-\mu_0)+\int \bar{\varphi}_1^{\mathbf{A}_{t_{n'}},(\tilde \mu_{0,t_{n'}},\bar \mu_1)}d(\nu_1-\mu_1) 
        \\&+t_{n'}^{-1}\int c_{\mathbf{A}_{t_{n'}}}-c_{\mathbf{A}_{t_{n'}}}\circ \tau^{t_{n'}}d \bar\pi^{\mathbf{A}_{t_{n'}}}.
    \end{aligned}
    \]
    due to 
 \eqref{eq:entropicGateauxCentredUpper2} and \eqref{eq:entropicOTLower}.
It remains to show that the term above converges to $\int g_{0,\pi^{\mathbf{A}^{\star}}}d(\nu_0-\mu_0)+\int g_{1,\pi^{\mathbf{A}^{\star}}}d(\nu_1-\mu_1)$ and that $(\bar{\varphi}_0^{\mathbf{A}_{t_{n'}},(\bar\mu_0,\bar\mu_1)}, \bar{\varphi}_1^{\mathbf{A}_{t_{n'}},(\tilde \mu_{0,t_{n'}},\bar \mu_1)})$ converges to $(\bar{\varphi}_0^{\mathbf{A}^{\star},(\bar\mu_0,\bar\mu_1)}, \bar{\varphi}_1^{\mathbf{A}^{\star},(\bar\mu_0,\bar\mu_1)})$ pointwise (up to additive constants).  

Given \cref{prop:entropicGWPotentials}, the Arzel{\`a}-Ascoli theorem implies that there exists a subsequence $t_{n''}$ along which $(\varphi_0^{\mathbf{A}_{t_{n''}},(\bar\mu_{0,t_{n''}},\bar\mu_{1,t_{n''}})},\varphi_1^{\mathbf{A}_{t_{n''}},(\bar\mu_{0,t_{n''}},\bar\mu_{1,t_{n''}})})$ converges to a pair $(\varphi_0,\varphi_1)$ of continuous functions uniformly on compact sets (hence pointwise); similar implications hold for the other pairs of EOT potentials. Given that $c_{\mathbf{A}_{n''}}$ converges pointwise to $c_{\mathbf{A}^{\star}}$ and that $\mu_0,\mu_1,\nu_0,\nu_1$ are $4$-sub-Weibull it follows from the dominated convergence theorem that 
\[
\begin{aligned}
    e^{-\frac{\varphi_0^{\mathbf{A}_{t_{n''}},(\bar\mu_{0,t_{n''}},\bar\mu_{1,t_{n''}})}(x)}{\varepsilon}}=\int e^{\frac{\varphi_1^{\mathbf{A}_{t_{n''}},(\bar\mu_{0,t_{n''}},\bar\mu_{1,t_{n''}})}-c_{\mathbf{A}_{t_{n''}}}(x,\cdot)}{\varepsilon}}d\bar\mu_{1,t_{n''}}\to \int e^{\frac{\varphi_1-c_{\mathbf{A}^{\star}}(x,\cdot)}{\varepsilon}}d\bar\mu_1 =e^{-\frac{\varphi_0(x)}{\varepsilon}},
\\
e^{-\frac{\varphi_1^{\mathbf{A}_{t_{n''}},(\bar\mu_{0,t_{n''}},\bar\mu_{1,t_{n''}})}(y)}{\varepsilon}}=\int e^{\frac{\varphi_0^{\mathbf{A}_{t_{n''}},(\bar\mu_{0,t_{n''}},\bar\mu_{1,t_{n''}})}-c_{\mathbf{A}_{t_{n''}}}(\cdot,y)}{\varepsilon}}d\bar\mu_{0,t_{n''}}\to \int e^{\frac{\varphi_0-c_{\mathbf{A}^{\star}}(\cdot,y)}{\varepsilon}}d\bar\mu_0 =e^{-\frac{\varphi_1(y)}{\varepsilon}},
\end{aligned}
\]
so that $(\varphi_0,\varphi_1)$ is an optimal pair for $\mathsf{OT}_{\mathbf{A}^{\star}_{\varepsilon}}(\bar \mu_0,\bar\mu_1)$ and similarly for the other EOT potentials. Conclude that $ e^{\frac{\varphi_0^{\mathbf{A}_{t_{n''}},(\bar\mu_{0,t_{n''}},\bar\mu_{1,t_{n''}})}\oplus\varphi_1^{\mathbf{A}_{t_{n''}},(\bar\mu_{0,t_{n''}},\bar\mu_{1,t_{n''}})}-c_{\mathbf{A}_{t_{n''}}}}{\varepsilon}}\to e^{\frac{\varphi_0\oplus\varphi_1-c_{\mathbf{A}_{\star}}}{\varepsilon}}$ pointwise on $\mathbb R^{d_0}\times \mathbb R^{d_1}$ and we recall that $e^{\frac{\varphi_0\oplus\varphi_1-c_{\mathbf{A}_{\star}}}{\varepsilon}}=\frac{d\pi^{\mathbf{A}}}{d\bar\mu_0\otimes \bar \mu_1}$ $\bar\mu_0\otimes \bar \mu_1$-a.e. so that, arguing as in the proof of \cref{lem:entropicOTGateaux}, 
\[
    t_{n'}^{-1}\int c_{\mathbf{A}_{t_{n'}}}-c_{\mathbf{A}_{t_{n'}}}\circ \tau^{t_{n'}}d \bar\pi^{\mathbf{A}_{t_{n'}}}\to\int 8\langle x,\mathbb E_{\nu_0}[X]-\mathbb E_{\mu_0}[X]\rangle\|y\|^2+8\langle y,\mathbb E_{\nu_1}[X]-\mathbb E_{\mu_1}[X]\rangle\|x\|^2 d\pi^{\mathbf{A}^{\star}}(x,y). 
\]
Conclude that 
\[
    \begin{aligned}
        \frac{\mathsf S_{2,\varepsilon}(\bar\mu_{0,t_{n''}},\bar\mu_{1,t_{n''}})-\mathsf S_{2,\varepsilon}(\bar\mu_{0},\bar\mu_{1})}{t_{n''}}&\to   \int \bar{\varphi}_0^{\mathbf{A}^{\star},(\bar\mu_0,\bar\mu_1)}d(\nu_0-\mu_0)+\int \bar{\varphi}_1^{\mathbf{A}^{\star},(\bar\mu_0,\bar\mu_1)}d(\nu_1-\mu_1) 
        \\&+\int 8\langle x,\mathbb E_{\nu_0}[X]-\mathbb E_{\mu_0}[X]\rangle\|y\|^2+8\langle y,\mathbb E_{\nu_1}[X]-\mathbb E_{\mu_1}[X]\rangle\|x\|^2 d\pi^{\mathbf{A}^{\star}}(x,y),
        \\
        &\geq \inf_{\mathbf{A}\in\mathcal A}\left\{\int g_{0,\pi^{\mathbf{A}}}+\bar{\varphi}_0^{\mathbf{A}}d(\nu_0-\mu_0)+\int g_{1,\pi^{\mathbf{A}}}+\bar{\varphi}_1^{\mathbf{A}}(\nu_1-\mu_1)\right\}.
    \end{aligned}
    \]
    Given that this final lower bound is independent of the original choice of subsequence, 
    \[
        \liminf_{t\downarrow 0}\frac{\mathsf S_{2,\varepsilon}(\bar\mu_{0,t},\bar\mu_{1,t})-\mathsf S_{2,\varepsilon}(\bar\mu_{0},\bar\mu_{1})}{t}\geq\inf_{\mathbf{A}\in\mathcal A}\left\{\int g_{0,\pi^{\mathbf{A}}}+\bar{\varphi}_0^{\mathbf{A}}d(\nu_0-\mu_0)+\int g_{1,\pi^{\mathbf{A}}}+\bar{\varphi}_1^{\mathbf{A}}(\nu_1-\mu_1)\right\} 
    \]
    which, in combination with \eqref{eq:entropicD2Upper} proves the claim.
\end{proof}

\begin{lemma}
    \label{lem:entropicS2Lipschitz}
    For any pairs $(\nu_0, \nu_1),(\rho_0, \rho_1)\in\mathcal P_{4,\sigma}(\mathbb R^{d_0})\times \mathcal P_{4,\sigma}(\mathbb R^{d_1})$, $\left|\mathsf S_{2,\varepsilon}(\bar\nu_0,\bar\nu_1)-\mathsf S_{2,\varepsilon}(\bar\rho_0,\bar\rho_1)\right|\leq C\|\nu_0\otimes \nu_1-\rho_0\otimes\rho_1\|_{\infty,\mathcal F^{\oplus}}$ for some constant $C>0$ which is independent of $(\nu_0, \nu_1),(\rho_0, \rho_1)$.
\end{lemma}

\begin{proof}
      Observe that $|\mathsf S_{2,\varepsilon}(\bar \nu_0,\bar \nu_1)-\mathsf S_{2,\varepsilon}(\bar \rho_0,\bar \rho_1)|\leq \sup_{\mathbf{A}\in B_{\mathrm{F}}(M)}\left|\mathsf{OT}_{\mathbf{A},\varepsilon}(\bar \nu_0,\bar \nu_1)-\mathsf{OT}_{\mathbf{A},\varepsilon}(\bar \rho_0,\bar \rho_1)\right|$ and 
    \begin{equation}
    \label{eq:entropicOTLipschitzUpperLower}    
       \begin{aligned}
            \mathsf{OT}_{\mathbf{A},\varepsilon}(\bar \nu_0,\bar \nu_1)-\mathsf{OT}_{\mathbf{A},\varepsilon}(\bar \rho_0,\bar \rho_1)&= \mathsf{OT}_{\mathbf{A},\varepsilon}(\bar \nu_0,\bar \nu_1)- \mathsf{OT}_{\mathbf{A},\varepsilon}(\bar \nu_0,\bar \rho_1)+\mathsf{OT}_{\mathbf{A},\varepsilon}(\bar \nu_0,\bar \rho_1)-\mathsf{OT}_{\mathbf{A},\varepsilon}(\bar \rho_0,\bar \rho_1)
            \\
            &
            \leq \int \varphi_1^{\mathbf{A},(\bar \nu_0,\bar \nu_1)} d(\bar \nu_1-\bar \rho_1)+\int \varphi_0^{\mathbf{A},(\bar \nu_0,\bar \rho_1)} d(\bar \nu_0-\bar \rho_0) 
        \\
        \mathsf{OT}_{\mathbf{A},\varepsilon}(\bar \nu_0,\bar \nu_1)-\mathsf{OT}_{\mathbf{A},\varepsilon}(\bar \rho_0,\bar \rho_1)&\geq \int \varphi_0^{\mathbf{A},(\bar \nu_0,\bar \rho_1)} d(\bar \nu_0-\bar \rho_0)+\int \varphi_1^{\mathbf{A},(\bar \rho_0,\bar \rho_1)} d(\bar \nu_1-\bar \rho_1).
        \end{aligned}
    \end{equation}
    By a Taylor expansion, for any $x\in\mathbb R^{d_1}$, 
    \[
    \begin{aligned}
        \varphi_1^{\mathbf{A},(\bar \nu_0,\bar \nu_1)}(x-\mathbb E_{\nu_1}[X])&=\varphi_1^{\mathbf{A},(\bar \nu_0,\bar \nu_1)}(x-\mathbb E_{\rho_1}[X])
        \\
        &+\sum_{i=1}^{d_1}(\mathbb E_{\rho_1}[X]-\mathbb E_{\nu_1}[X])_i\partial_{i} \varphi_1^{\mathbf{A},(\bar \nu_0,\bar \nu_1)}(x-(1-c)\mathbb E_{\rho_1}[X]-c\mathbb E_{\nu_1}[X]),
    \end{aligned}
    \]
    for some $c\in(0,1)$. It follows from the uniform bounds on the $4$-th moments of $4$-sub-Weibull distributions established in \cref{lem:4suWeibullAssn} and the estimates on the derivatives on the EOT potentials (see \cref{prop:entropicGWPotentials}) that 
    \[
    \begin{aligned}
        \left|\varphi_1^{\mathbf{A},(\bar \nu_0,\bar \nu_1)}(x-\mathbb E_{\nu_1}[X])-\varphi_1^{\mathbf{A},(\bar \nu_0,\bar \nu_1)}(x-\mathbb E_{\rho_1}[X])\right|&
        \\
        &\hspace{-15em}= \left|\sum_{i=1}^{d_1}(\mathbb E_{\rho_1}[X]-\mathbb E_{\nu_1}[X])_i\partial_{i} \varphi_1^{\mathbf{A},(\bar \nu_0,\bar \nu_1)}(x-(1-c)\mathbb E_{\rho_1}[X]-c\mathbb E_{\nu_1}[X])\right|  
        \\
        &\hspace{-15em}\leq \sum_{i=1}^{d_1}\left|(\mathbb E_{\rho_1}[X]-\mathbb E_{\nu_1}[X])_i\right| K_1(1+\|x-(1-c)\mathbb E_{\rho_1}[X]-c\mathbb E_{\nu_1}[X]\|^{3})
        \\
        &\hspace{-15em}\leq  C_1(1+\|x\|^{3})\max_{i=1}^{d_1}\left|(\mathbb E_{\rho_1}[X]-\mathbb E_{\nu_1}[X])_i\right|,
    \end{aligned}
    \]where $K_1$ is the constant from \cref{prop:entropicGWPotentials} corresponding to $k=1$ and 
    $C_1$ is a constant depending only on $\sigma,d_0,d_1$, and $\varepsilon$. It is easy to see that, for $i\in[N]$, $(\mathbb E_{\rho_1}[X]-\mathbb E_{\nu_1}[X])_i=\int x_id(\rho_1-\nu_1)(x)\leq \|\nu_0\otimes \nu_1-\rho_0\otimes \rho_1\|_{\infty,\mathcal F^{\oplus}}$ given that $x_i\in\mathcal F_1$ and that $0\in\mathcal F_0$. 

    In sum, 
    \[
    \begin{aligned}
        \int \varphi_1^{\mathbf{A},(\bar \nu_0,\bar \nu_1)} d(\bar \nu_1-\bar \rho_1)&= \underbrace{\int \varphi_1^{\mathbf{A},(\bar \nu_0,\bar \nu_1)}(\cdot-\mathbb E_{\nu_1}[X])-\varphi_1^{\mathbf{A},(\bar \nu_0,\bar \nu_1)}(\cdot-\mathbb E_{\rho_1}[X]) d\nu_1}_{\leq C_1(1+M_3(\nu_1))\|\nu_0\otimes \nu_1-\rho_0\otimes \rho_1\|_{\infty,\mathcal F^{\oplus}}}
        \\
        &+\int \varphi_1^{\mathbf{A},(\bar \nu_0,\bar \nu_1)}(\cdot-\mathbb E_{\rho_1}[X]) d(\nu_1-\rho_1).
    \end{aligned}
    \]
    Analogous bounds hold for each of the integrals in \eqref{eq:entropicOTLipschitzUpperLower} whereby 
    \[
    \begin{aligned}
       &\left|\mathsf{OT}_{\mathbf{A},\varepsilon}(\bar \nu_0,\bar \nu_1)-\mathsf{OT}_{\mathbf{A},\varepsilon}(\bar \rho_0,\bar \rho_1)\right|\\
       &\leq \left(\left|\int \varphi_1^{\mathbf{A},(\bar \nu_0,\bar \nu_1)}(\cdot-\mathbb E_{\rho_1}[X]) d(\nu_1-\rho_1)+\int \varphi_0^{\mathbf{A},(\bar \nu_0,\bar \rho_1)}(\cdot-\mathbb E_{\rho_0}[X]) d(\nu_0-\rho_0)\right|\right.
       \\&\bigvee \left.\left|\int \varphi_0^{\mathbf{A},(\bar \nu_0,\bar \rho_1)}(\cdot-\mathbb E_{\rho_0}[X]) d(\nu_0-\rho_0)+\int \varphi_0^{\mathbf{A},(\bar \rho_0,\bar \rho_1)}(\cdot-\mathbb E_{\rho_1}[X]) d(\nu_1-\rho_1)\right|\right)
       \\
       &+C_2\|\nu_0\otimes \nu_1-\rho_0\otimes \rho_1\|_{\infty,\mathcal F^{\oplus}},
    \end{aligned}
    \]
    where $C_2$ depends only on $\sigma,d_0,d_1,\varepsilon$.
   As $ \varphi_0^{\mathbf{A},(\bar \nu_0,\bar \rho_1)}(\cdot-\mathbb E_{\rho_0}[X])\oplus \varphi_1^{\mathbf{A},(\bar \nu_0,\bar \nu_1)}(\cdot-\mathbb E_{\rho_1}[X])$ and $
        \varphi_0^{\mathbf{A},(\bar \nu_0,\bar \rho_1)}(\cdot-\mathbb E_{\rho_0}[X]) \oplus \varphi_0^{\mathbf{A},(\bar \rho_0,\bar \rho_1)}(\cdot-\mathbb E_{\rho_1}[X])$ are elements of $\mathcal F^{\oplus}$ up to scaling by a constant that does not depend on $\nu_0,\nu_1,\rho_0,\rho_1$, it follows that $\left|\mathsf{OT}_{\mathbf{A},\varepsilon}(\bar \nu_0,\bar \nu_1)-\mathsf{OT}_{\mathbf{A},\varepsilon}(\bar \rho_0,\bar \rho_1)\right|\leq C_3\|\nu_0\otimes \nu_1-\rho_0\otimes \rho_1\|_{\infty,\mathcal F^{\oplus}}$ for some $C_3$ which is independent of $\nu_0,\nu_1,\rho_0,\rho_1$.
\end{proof}

\begin{proof}[Proof of \cref{thm:entropicGWStability}]
    The claimed G{\^a}teaux differentiability follows from Lemmas \ref{lem:entropicGateauxD1} and \ref{lem:entropicS2Gateaux} whereas Lipschitz continuity follows from 
 Lemmas \ref{lem:entropicS1Lipschitz} and \ref{lem:entropicS2Lipschitz} by applying the variational form \eqref{eq:GWVariational}.
\end{proof}

\subsubsection{Proof of \texorpdfstring{\cref{thm:entropicGWLimitDistribution}}{Theorem 9}}

Given \cref{thm:entropicGWStability}, the proof of 
 the first statement in \cref{thm:entropicGWLimitDistribution} follows directly from \cref{prop:unified}. 
Under the assumption that $\varepsilon>16\sqrt{M_4(\bar\mu_0)M_4(\bar\mu_1)}$, Theorem 6 in \cite{rioux2023entropic} asserts that $\Phi_{(\bar\mu_0,\bar\mu_1)}$ is strictly convex so that it admits at most one minimizer, $\mathbf A^{\star}$. Recall from the discussion following \eqref{eq:Objective} that $\mathcal A$ is always non-empty so that $\mathcal A=\{\mathbf A^{\star}\}$ and the limit distribution simplifies to $\chi_{\mu_0\otimes \mu_1}((f_0+g_{0,\pi}+\bar\varphi_0^{\mathbf A^{\star}})\oplus(f_1+g_{1,\pi}+\bar\varphi_1^{\mathbf A^{\star}}))$ 

As for the statement regarding the empirical measures, it 
 will follow upon showing that $\mathcal F_0$ and $\mathcal F_1$ are, respectively, $\mu_0$- and $\mu_1$-Donsker and that the empirical measures are $4$-sub-Weibull with some parameter with probability approaching one. As noted in the proof of \cref{thm:discreteGWLimitDistribution}, Donskerness of the classes and the independence assumption imply that  $\sqrt n(\hat \mu_{0,n}-\mu_0)\otimes(\hat \mu_{1,n}-\mu_1)\stackrel{d}{\to}G_{\mu_0\otimes \mu_1}$ in $\ell^{\infty}(\mathcal F^{\oplus})$ where $G_{\mu_0\otimes \mu_1}(f_0\oplus f_1)=G_{\mu_0}(f_0)+G_{\mu_1}(f_1)$ and $G_{\mu_0}$, $G_{\mu_1}$ are, respectively, tight $\mu_0$- and $\mu_1$-Brownian bridge processes in $\ell^{\infty}(\mathcal F_0)$ and $\ell^{\infty}(\mathcal F_1)$. 

Donskerness of $\mathcal F_0$ and $\mathcal F_1$ with respect to any $4$-sub-Weibull distribution follows from a direct adaptation of Lemma E.24 in \cite{goldfeld24statistical} (see also Lemma 8 in \cite{rioux2022smooth}) and is thus omitted.

Fixing $i\in\{0,1\}$ and some $\tilde \sigma>\sigma$, the law of large numbers yields
\[
    \int e^{\frac{\|\cdot\|^4}{2\tilde \sigma^2}}d \hat \mu_{i,n}\to \int e^{\frac{\|\cdot\|^4}{2\tilde \sigma^2}}d \mu_{i}\leq 2^{\frac{\sigma^2}{\tilde \sigma^2}}<2
\]
a.s.
so that $\hat \mu_{i,n}$ is $4$-sub-Weibull with parameter $\tilde \sigma^2$ with probability approaching $1$.  

The proof of bootstrap consistency follows the same lines as the proof of Theorem 7 in \cite{goldfeld24statistical} and is thus omitted for brevity.  
\qed

\subsection{Proofs for \texorpdfstring{\cref{sec:graphApplication}}{Section}}

\subsubsection{Proof of \texorpdfstring{\cref{prop:permContinuous}}{Proposition 9}}
\label{proof:prop:permContinuous}
   If $\mathsf D(\mu_0,\mu_1)=0$, there exists an isometry $T:\supp(\mu_0)\to \supp(\mu_1)$ for which $T_{\sharp}\mu_0=\mu_1$.
     For any $i\in[N]$, let $T(-e_i)=y$ and observe that, by construction,  
     \[\frac{1}{(N-1)N(N+1)}\sum_{\substack{i,j=1\\i\neq j}}^N\eta_{\rho_{0,ij}}(-e_i)+\frac{1}{N+1}=\mu_0(T^{-1}(\{y\}))=\mu_1(y),
     \]
     where   
     \[
        \mu_1(y)=\frac{1}{(N-1)N(N+1)}\sum_{\substack{i,j=1\\i\neq j}}^N\eta_{\rho_{1,ij}}(y)+\frac{1}{N+1}\sum_{i=1}^N\delta_{-e_i}(y).
     \]
     Note that $\sum_{\substack{i,j=1\\i\neq j}}^N\eta_{\rho_{1,ij}}(y)\leq N(N-1)$ with equality if and only if $\eta_{\rho_{1,ij}}(y)=1$ for every $i,j\in[N]$ with $i<j$. 
     Thus, if $y \not \in (-e_i)_{i=1}^N$, it holds that 
     $
        \mu_1(y) \leq \frac{1}{N+1} 
     $ with equality if and only if 
     \[
       1= \eta_{\rho_{1,ij}}(y) = \int \mathbbm 1_{\{y\}}(-e_i+te_j)d\rho_{ij}(t) 
     \]
     for every $i,j\in[N],i<j$ i.e. $y=-e_i+se_j$ for some $s\in[0,\infty)$ and $\rho_{ij}(\{s\})=1$, evidently this equality cannot hold simultaneously for all such $i,j$. Consequently, $y \in (-e_i)_{i=1}^N$ so that, for every $i\in[N]$, $T(-e_i)=-e_{\sigma'(i)}$ for some permutation $\sigma'$ on $N$ elements. 
    
    By extending the isometry to an isometric self-map on $\mathbb R^N$, we have by the Mazur-Ulam theorem that $T(x)=Ux+b$ for some orthogonal matrix $U\in\mathbb R^{N\times N}$ and $b\in\mathbb R^N$. For any $i\in[N]$, 
   \[
        1=\|T(-e_i)\|^2= \|-Ue_i\|^2+2b^{\intercal}(U(-e_i))+\|b\|^2=\|e_i\|^2+2b^{\intercal}(T(-e_i)-b)+\|b\|^2=1-\|b\|^2-2b^{\intercal}e_{\sigma'(i)}.
   \]
   As the above equality holds for every $i\in[N]$ and $\sigma'$ is a permutation, it follows that $\|b\|^2=-2b_i$ for every $i\in[N]$ i.e. $b=-\frac 2 { N}(1,\dots, 1)$ or $b=(0,\dots,0)$. Let $t\in[0,\infty)$ and $i\neq j\in[N]$ be arbitrary and observe that 
   \[T(-e_i+te_j)=U(-e_i+te_j)+b=-e_{\sigma'(i)}-(U(-e_j)+b)+b=-e_{\sigma'(i)}+te_{\sigma'(j)}+b.\]
   Since $T_{\sharp}\mu_0=\mu_1$, for any Borel set $I\subset [0,\infty)$, 
   \[
       \mu_1(\cup_{t\in I}\{-e_{\sigma'(i)}+te_{\sigma'(j)}+b\})= \mu_0(\cup_{t\in I}T^{-1}(\{-e_{\sigma'(i)}+te_{\sigma'(j)}+b\}))
        =\mu_0(\cup_{t\in I}\{-e_{i}+te_{j}\}).
   \]
   If $b=-\frac{2}{N}(1,\dots,1)$, $\cup_{t\in I}\left\{-e_{\sigma'(i)}+te_{\sigma'(j)}+b\right\}$ has measure $0$ with respect to $\mu_1$ whereas $\mu_0(\cup_{t\in I}\left\{-e_i+t_{e_j}\right\})\neq 0$ for some set $I$ (e.g. $0\in I$). Conclude that $b=0$ and, for every measurable subset, $I$, of $[0,\infty)$ and every $i,j\in[N]$ with $i< j$, 
  \[
    \frac{\rho_{0,ij}(I)}{N(N+1)(N-1)}+\frac{1}{N+1}\delta_{0}(I)=\mu_1(\cup_{t\in I}\{-e_{\sigma'(i)}+te_{\sigma'(j)}\})= \frac{\rho_{1,\sigma'(i)\sigma'(j)}(I)}{N(N+1)(N-1)}+\frac{1}{N+1}\delta_{0}(I).  
  \]
  It follows that $\rho_{0,ij}=\rho_{1,\sigma'(i)\sigma'(j)}$ as measures on $\mathbb R$ or, equivalently, $\rho_{1,ij}=\rho_{0,\sigma(i)\sigma(j)}$ where $\sigma$ is the inverse of $\sigma'$  for every $i,j\in[N]$ with $i\neq j$
   
    Conversely, if there exists a permutation $\sigma$ for which $\rho_{1,ij}=\rho_{0,\sigma(i)\sigma(j)}$ for every $i,j\in[N]$ with $i\neq j$, consider the linear isometry $T:(x_1,\dots,x_N)\in\mathbb R^N\mapsto (x_{\sigma'(1)},\dots,x_{\sigma'(N)})\in\mathbb R^N$, where $\sigma'$ is the inverse of $\sigma$. For any Borel set $A\subset\mathbb R^N$,
    \[
        T_{\sharp}\mu_0(A)=\frac{1}{N(N+1)(N-1)}\sum_{\substack{i,j=1\\i\neq j}}^N\eta_{\rho_{0,ij}}(T^{-1}(A))+\frac{1}{N+1}\sum_{i=1}^N\delta_{-e_i}(T^{-1}(A)).
    \]
    It is clear that $\sum_{i=1}^N\delta_{-e_i}(T^{-1}(A))=\sum_{i=1}^N\delta_{-e_{\sigma'(i)}}(A)$, whereas
    \[
    \begin{aligned}
        \eta_{\rho_{0,ij}}(T^{-1}(A))=\int \mathbbm 1_{T^{-1}(A)}(-e_i+te_j)d\rho_{0,ij}(t)&=\int \mathbbm 1_{A}(-e_{\sigma'(i)}+te_{\sigma'(j)})d\rho_{0,ij}(t)
        \\&=\int \mathbbm 1_{A}(-e_{\sigma'(i)}+te_{\sigma'(j)})d\rho_{1,\sigma'(i)\sigma'(j)}(t)
        \\&=\eta_{\rho_{1,\sigma'(i)\sigma'(j)}}(A),
   \end{aligned} 
    \]
    the previous two displays imply that 
 $T_{\sharp}\mu_0=\mu_1$ whereby $\mathsf D(\mu_0,\mu_1)=0$.
\qed

\subsubsection{Proof of \texorpdfstring{\cref{thm:statsContinuousWeights}}{Theorem 12}} 
\label{proof:thm:statsContinuousWeights}
 For a set $A\subset \mathbb R^{N}$, let $\mathcal F(A,L)$ denote the set of convex $L$-Lipschitz functions on $A$.
 For an arbitrary $f\in \mathcal F([-1,K]^{N},1)$  
    \[
    \begin{aligned}
        \mu_{0,n}(f) &= \frac{1}{(N-1)N(N+1)} \sum_{\substack{i,j=1\\i\neq j}}^N \eta_{\hat \rho_{0,ij,n}}(f)+\frac{1}{N+1}\sum_{i=1}^Nf(-e_i),
        \\
        &=\frac{1}{(N-1)N(N+1)} \sum_{\substack{i,j=1\\i\neq j}}^N \frac{1}{n}\sum_{k=1}^nf(-e_i+\mathbf M_{ij}^{G_{0,k}}e_j)+\frac{1}{N+1}\sum_{i=1}^Nf(-e_i),
    \end{aligned} 
    \]
    whereas 
    \[
    \begin{aligned}
        \mu_{0}(f) &= \frac{1}{(N-1)N(N+1)} \sum_{\substack{i,j=1\\i\neq j}}^N \int f(-e_i+te_j)d\rho_{0,ij}(t)+\frac{1}{N+1}\sum_{i=1}^Nf(-e_i),
    \end{aligned} 
    \]
    thus 
    \begin{equation}
    \label{eq:empiricalProcessDiscrete}
    \begin{aligned}
        (\mu_{0,n}-\mu_0)(f)&=\frac{1}{(N-1)N(N+1)}\sum_{\substack{i,j=1\\i\neq j}}^N\left(\frac{1}{n}\sum_{k=1}^nf(-e_i+\mathbf M_{ij}^{G_{0,k}}e_j)- \int f(-e_i+te_j)d\rho_{0,ij}(t)\right)
        \\
&=\frac{1}{(N-1)N(N+1)}\sum_{\substack{i,j=1\\i\neq j}}^N(\hat \rho_{0,ij,n}-\rho_{0,ij})\left(f(-e_i+(\cdot)e_j)\right).
    \end{aligned} 
    \end{equation}
    In sum, 
    \[
    \begin{aligned}
\mathbb E\left[\sup_{f\in \mathcal F([-1,K]^N,1)}|(\mu_{0,n}-\mu_0)(f)|\right]&
\\
&\hspace{-4em}\leq\frac{1}{(N-1)N(N+1)}\sum_{\substack{i,j=1\\i\neq j}}^N\mathbb E\left[\sup_{f\in \mathcal F([-1,K]^N,1)}|(\hat \rho_{0,ij,n}-\rho_{0,ij})\left(f(-e_i+(\cdot)e_j)\right)|\right]
\end{aligned}
    \]  
     Since $t\in[0,K]\mapsto f(-e_i+te_j)$ is $1$-Lipschitz and convex\footnote{Indeed, for any $t,s\in[0,K]$, $|f(-e_i+te_j)-f(-e_i+se_j)|\leq \|te_j-se_j\|=L|t-s|$ and, for $\lambda \in[0,1]$, $f(-e_i+(t\lambda +(1-\lambda)s)e_j)=f(\lambda(-e_i+te_j)+(1-\lambda)(-e_i+se_j))\leq  \lambda f(-e_i+te_j)+(1-\lambda)f(-e_i+se_j)$. } the supremum on the right hand side can be replaced by a supremum over $\mathcal F([0,K],1)$, which admits an $L^{\infty}$ metric entropy scaling as 
     $N(\mathcal F([0,K],1),\epsilon,L^{\infty})\lesssim \epsilon^{-\frac{1}{2}}$ (see  Theorem 1 in \cite{sen2012covering}). Evidently, the same result if $\mu_{0,n}-\mu_0$ is replaced by $\mu_{1,n}-\mu_1$. The remainder of the proof for the expected rate of convergence follows by the same argument as the proof of Theorem 3 in \cite{zhang2024gromov} with only minor modifications. 

    The limit distribution results follow from the proofs of  \cref{thm:discreteGWStability} and  \cref{cor:discreteGWLimitDistributionNull} upon characterizing the weak limit of $\sqrt n(\mu_{0,n}-\mu_0)$; the arguments evidently also apply to $\sqrt n(\mu_{1,n}-\mu_1)$. We provide a general argument which applies to compactly supported weight distributions and specialize to the finitely discrete case at the end, as the stability results derived herein can only be applied in this case.%

    To this end, it was previously noted that $N(\mathcal F([0,K],1),\epsilon,L^{\infty})\lesssim \epsilon^{-\frac{1}{2}}$ so that, by Theorem 2.5.2 in \cite{vanderVaart1996}, $\mathcal F([0,K],1)$ is $\rho_{0,ij}$-Donsker for every $i,j\in[N],i<j$. As the samples from different edges are  independent, Example 1.4.6 in \cite{van1996weak} and Lemma 3.2.4 \cite{dudley2014uniform} imply that 
    \begin{equation}
    \label{eq:weakConvergenceProduct}
\sqrt n( (\hat \rho_{0,12,n},\dots, \hat \rho_{0,N(N-1),n})-(\rho_{0,12},\dots, \rho_{0,N(N-1)}))\stackrel{d}{\to} (G_{\rho_{0,12}},\dots,G_{\rho_{0,N(N-1)}})
   \text{ in $\prod_{\substack{i,j=1\\i<j}}^N\ell^{\infty}(\mathcal F([0,K],1))$},  
    \end{equation}
    where $G_{\rho_{0,ij}}$ is a $\rho_{0,ij}$-Brownian bridge process in $\ell^{\infty}(\mathcal F([0,K],1))$ for $i,j\in[N],i<j$ and all processes are independent.
   As the map 
    \[  
    \begin{aligned}
    &\Upsilon:(\eta_{12},\dots,\eta_{N(N-1)})\in \prod_{\substack{i,j=1\\i<j}}^N \ell^{\infty}\left(\textstyle{\mathcal F([0,K],1)}\right)
    \\
    &
    \mapsto \left(f\in \mathcal F([-1,K]^N,1/2)\mapsto\sum_{\substack{i,j=1\\i<j}}^N\eta_{ij}\left(f(-e_i+(\cdot)e_j)+f(-e_j+(\cdot)e_i)\right)\right)\in \ell^{\infty}(\mathcal F([-1,K]^N,1/2)), 
    \end{aligned}
    \]
    is continuous it follows from \eqref{eq:weakConvergenceProduct}, the representation of $(\mu_{0,n}-\mu_0)$ furnished in \eqref{eq:empiricalProcessDiscrete}, and the continuous mapping theorem that 
    \[
       \sqrt n(\mu_{0,n}-\mu_0)\stackrel{d}{\to} \frac{1}{(N-1)N(N+1)}\Upsilon\left(G_{\rho_{0,12}},\dots,G_{\rho_{0,N(N-1)}}\right) \text{ in } \ell^{\infty}(\mathcal F([-1,K]^N,1/2)).
    \]
    We highlight that, for any $f\in \ell^{\infty}(\mathcal F([-1,K]^N,1/2))$,
    \begin{equation}\label{eq:contiuouslyWeightedLimt}
        \Upsilon\left(G_{\rho_{0,12}},\dots,G_{\rho_{0,N(N-1)}}\right)(f)=N\left(0,\frac{\sum_{\substack{i,j=1\\i<j}}^N\Var_{\rho_{0,ij}}(f(-e_i+(\cdot)e_j)+f(-e_j+(\cdot)e_i))}{(N-1)^2N^2(N+1)^2}\right)
    \end{equation} in distribution
    since, for any $f\in\mathcal F([0,K],1)$, $G_{\rho_{0,ij}}(f)=N(0,\Var_{\rho_{0,ij}}(f))$ in distribution $(i,j\in[N],i<j)$. 

    In the case that each of the $\rho_{0,ij}$ are finitely discrete, we let $(w_{ij}^{(k)})_{k=1}^{K_{ij}}$ denote the set of support points and let $(p_{ij}^{(k)})_{k=1}^{K_{ij}}$ be the corresponding probabilities $p_{ij}^{(k)}=\rho_{0,ij}(w_{ij}^{(k)})$ for $k\in [K_{ij}]$ and any $i,j\in[N]$ with $i<j$. Then, for  any bounded measurable function $g$ on $\mathbb R$, 
    \[
        \Var_{\rho_{0,ij}}(g) = \begin{pmatrix}
            g(w_{ij}^{(1)})\\\vdots\\ g(w_{ij}^{(K_{ij})}) 
        \end{pmatrix}^{\intercal}
        \begin{pmatrix}
            p_{ij}^{(1)}(1-p_{ij}^{(1)})&- p_{ij}^{(1)}p_{ij}^{(2)}&\dots &- p_{ij}^{(1)}p_{ij}^{(K_{ij})} \\ 
            - p_{ij}^{(2)}p_{ij}^{(1)}&p_{ij}^{(2)}(1-p_{ij}^{(2)})&\dots &- p_{ij}^{(2)}p_{ij}^{(K_{ij})}
            \\
            \vdots&\vdots&\ddots&\vdots
            \\
            - p_{ij}^{(K_{ij})}p_{ij}^{(1)}& - p_{ij}^{(K_{ij})}p_{ij}^{(2)}&\dots &p_{ij}^{(K_{ij})}(1-p_{ij}^{(K_{ij})})
        \end{pmatrix}
        \begin{pmatrix}
            g(w_{ij}^{(1)})\\\vdots\\ g(w_{ij}^{(K_{ij})}) 
        \end{pmatrix} 
    \]
    so that, if $B_{ij}$ denotes the above multinomial covariance matrix, the variance in \eqref{eq:contiuouslyWeightedLimt} can be written in terms of a block diagonal matrix 
    \begin{equation}
    \label{eq:covarianceStruct}
       \check f^{\intercal}\bsigma_Z\check f\coloneqq  \frac{1}{(N-1)^2N^2(N+1)^2}\check f^{\intercal} \begin{pNiceMatrix} 
        \Block[borders={bottom,right}]{1-1}{B_{12}}&&&\\
           &\Block[borders={bottom,right,left,top}]{1-1}{B_{13}}&&\\
           &&\ddots&\\
           &&&\Block[borders={top,left}]{1-1}{B_{N(N-1)}},
        \end{pNiceMatrix}
        \check f,
    \end{equation}
    where $
\check f=\left(f_{12},\dots,f_{N(N-1)}\right),$ and 
\[f_{ij}=\left(f(-e_i+w_{ij}^{(1)}e_j)+f(-e_j+w_{ij}^{(1)}e_i),\dots, f(-e_i+w_{ij}^{(K_{ij})}e_j)+f(-e_j+w_{ij}^{(K_{ij})}e_i)\right),
\]
for $i,j\in[N]$ with $i<j$. The limit theorem then follows from \cref{thm:discreteGWLimitDistribution}.

It remains to show that \cref{assn:simplifiedForm} holds in this setting. The expression for $\chi_{\mu_0}(f_0)$ is given in \eqref{eq:contiuouslyWeightedLimt} for $f_0\in\mathcal f([-1,K]^N,1/2)$, an analogous expression holds for $\chi_{\mu_1}(f_1)$ for $f_1$ in the same space. Now, assume that $g_1(\cdot-\mathbb E_{\mu_1}[X])\in\mathcal f([-1,K]^N,1/2)$ and let $T$ be a Gromov-Monge map between $\bar\mu_0$ and $\bar \mu_1$. In particular, $T = S(\cdot+\mathbb E_{\mu_0}[X])-\mathbb E_{\mu_1}[X]$ where $S$ corresponds to a permutation of axes according to some permutation $\sigma'$ as follows from \cref{prop:permContinuous}. Then, letting $\sigma$ denote the inverse of $\sigma$,
\[
\begin{aligned}
&\chi_{\mu_1}(g_1(\cdot-\mathbb E_{\mu_1}[X])) 
\\&=  N\left(0,\frac{\sum_{\substack{i,j=1\\i<j}}^N\Var_{\rho_{1,ij}}(g_1(-e_i+(\cdot)e_j-\mathbb E_{\mu_1}[X])+g_1(-e_j+(\cdot)e_i-\mathbb E_{\mu_1}[X]))}{(N-1)^2N^2(N+1)^2}\right),
\\
&=N\left(0,\frac{\sum_{\substack{i,j=1\\i<j}}^N\Var_{\rho_{0,\sigma(i)\sigma(j)}}(g_1(-e_i+(\cdot)e_j-\mathbb E_{\mu_1}[X])+g_1(-e_j+(\cdot)e_i-\mathbb E_{\mu_1}[X]))}{(N-1)^2N^2(N+1)^2}\right).
\end{aligned}
\]
For any $t\in\mathbb R$, $-e_i+te_j-\mathbb E_{\mu_1}[X]=T(-e_{\sigma(i)}+t\sigma(j)-\mathbb E_{\mu_0}[X])$ so that 
\[
\begin{aligned}
&\Var_{\rho_{0,\sigma(i)\sigma(j)}}(g_1(-e_i+(\cdot)e_j-\mathbb E_{\mu_1}[X])+g_1(-e_j+(\cdot)e_i-\mathbb E_{\mu_1}[X])) 
\\&=\Var_{\rho_{0,\sigma(i)\sigma(j)}}(g_1(T(-e_{\sigma(i)}+(\cdot)e_{\sigma(j)}-\mathbb E_{\mu_0}[X]))+g_1(T(-e_{\sigma(j)}+(\cdot)e_{\sigma(i)}-\mathbb E_{\mu_0}[X]))),
\end{aligned}
\]
and 
\[
\begin{aligned}
&\chi_{\mu_1}(g_1(\cdot-\mathbb E_{\mu_1}[X])) 
\\
&=N\left(0,\frac{\sum_{\substack{i,j=1\\i<j}}^N\Var_{\rho_{0,ij}}(g_1(T(-e_{i}+(\cdot)e_{j}-\mathbb E_{\mu_0}[X]))+g_1(T(-e_{j}+(\cdot)e_{i}-\mathbb E_{\mu_0}[X])))}{(N-1)^2N^2(N+1)^2}\right),
\\
&=\chi_{\mu_0}((g_1\circ T)(\cdot-\mathbb E_{\mu_0}[X])),
\end{aligned}
\]
as desired.

Letting $\Sigma_Z$ be the covariance matrix in \eqref{eq:covarianceStruct}, and $\Sigma_n$ be the same covariance matrix with the edge distributions $\rho_{0,ij}$ replaced by empirical estimates thereof, it is easy to see that, conditionally on $G_{0,1},\dots, G_{0,n}$ and $G_{1,1},\dots, G_{1,n}$ $\mathbb E_{\mu_{0,n}}[X]\to \mathbb E_{\mu_0}[X],\Sigma_{\mu_{0,n}}\to \Sigma_{\mu_0}[X],$ and $\Sigma_n\to \Sigma_Z$ given almost every realization of the data. Moreover, it is easy to see that $\Sigma_Z\mathbbm 1_N=\Sigma_n\mathbbm 1_N=0$ with probability $1$ so that the required orthogonality condition holds. The fact that the proposed test is of asymptotic level $\alpha$ then follows form   \cref{thm:directEstimatorGeneralized}. Under the alternative, it is easy to see that the test statistic diverges to $+\infty$ so that the test is asymptotically consistent.
\qed

\section{Variational Formulation of Entropic Gromov-Wasserstein Distances}
\label{sec:entropicGWVar}

The GW variational form \eqref{eq:GWVariational} adapts to the entropic case, for any $\varepsilon\geq 0$, as follows (see Theorem 1 in \cite{zhang2024gromov}),
\begin{equation}
    \label{eq:entropicGWVariational}
    \mathsf D_{\varepsilon}(\mu_0,\mu_1)^2=
    \mathsf D_{\varepsilon}(\bar\mu_0,\bar\mu_1)^2=\mathsf S_{1}(\bar\mu_0,\bar\mu_1)+\mathsf S_{2,\varepsilon}(\bar\mu_0,\bar\mu_1),
\end{equation}
where $\mathsf S_1(\mu_0,\mu_1)$ is defined as in \eqref{eq:GWVariational}, while $\mathsf S_{2,\varepsilon}(\mu_0,\mu_1)$ is obtained by replacing $\OT_{\mathbf{A}}(\mu_0,\mu_1)$ in $\mathsf{S}_2(\mu_0,\mu_1)$ with $\OT_{\mathbf{A},\varepsilon}(\mu_0,\mu_1)\coloneqq \OT_{c_{\mathbf{A}},\varepsilon}(\mu_0,\mu_1)$, under the same $c_{\mathbf{A}}(x,y)=-4\|x\|^2\|y\|^2-32x^{\intercal}\mathbf{A}y$ cost function. Again, $\mathsf S_1$ is a constant whereas $\mathsf S_{2,\varepsilon}$ is a minimization problem with objective function, 
\begin{equation}
    \label{eq:entropicObjective}
    \Phi_{(\mu_0,\mu_1),\varepsilon}:\mathbf{A}\in\mathbb R^{d_0\times d_1}\mapsto 32\|\mathbf{A}\|^2_{\mathrm F}+\OT_{\mathbf{A},\varepsilon}(\mu_0,\mu_1),
\end{equation}
and, if $\pi^{\star}$ is optimal for \eqref{eq:entropicGWPrimal}, then $\mathbf{A}^{\star}=\frac{1}{2}\int xy^{\intercal}d\pi^{\star}(x,y)$ is optimal for \eqref{eq:entropicObjective} provided that $\mu_0,\mu_1$ are centered (see the proof of Theorem 1 in \cite{zhang2024gromov}). The implications of \cref{thm:VariationalMinimizers} carry over to the entropic case as follows. 

\begin{theorem}[On minimizers of \eqref{eq:entropicObjective}]
    \label{thm:entropicVariationalMinimizers}
    Fix $\varepsilon>0$. If $(\mu_0,\mu_1)\in\mathcal P_{4,\sigma}(\mathbb R^{d_0})\times \mathcal P_{4,\sigma}(\mathbb R^{d_1})$ are $4$-sub-Weibull with parameter $\sigma^2>0$, then 
    \begin{enumerate}
        \item $\Phi_{(\mu_0,\mu_1),\varepsilon}$ is locally Lipschitz continuous and coercive. Moreover, 
            $\Phi_{(\mu_0,\mu_1),\varepsilon}$ is Fr{\'e}chet differentiable at $\mathbf{A}\in\mathbb R^{d_0\times d_1}$ with $\left(D\Phi_{(\mu_0,\mu_1),\varepsilon}\right)_{[\mathbf{A}]}(\mathbf{B})=64\langle\mathbf{A}-\frac{1}{2}\int xy^{\intercal}d\pi_{\mathbf{A}}(x,y),\mathbf{B}\rangle_{\mathrm F}$, where $\pi_{\mathbf{A}}$ is the unique optimal coupling for $\mathsf{OT}_{\mathbf{A},\varepsilon}(\mu_0,\mu_1)$. 
        \item  If $\mathbf{A}^{\star}$ minimizes \eqref{eq:Objective}, then $2\mathbf{A}^{\star}=\int xy^{\intercal}d\pi^{\star}(x,y)\in {B_{\mathrm{F}}(\sqrt{M_2(\mu_0)M_2(\mu_1)})}$ for the EOT plan $\pi^{\star}$ for $\OT_{\mathbf{A}^{\star},\varepsilon}(\mu_0,\mu_1)$. If $\mu_0,\mu_1$ are centered, then $\pi^{\star}$ solves \eqref{eq:GWPrimal}.  
    \end{enumerate} 
\end{theorem}

Compared with \cref{thm:VariationalMinimizers}, \cref{thm:entropicVariationalMinimizers} holds under the weaker assumption of $4$-sub-Weibull measures with the advantage that $\Phi_{(\mu_0,\mu_1),\varepsilon}$ is guaranteed to be continuously differentiable. These improvements are obtained by leveraging the strong structural properties of solutions to the EOT problem, see \cref{sec:proof:thm:entropicVariationalMinimizers} for details.

\subsection{Proof of \texorpdfstring{\cref{thm:entropicVariationalMinimizers}}{Theorem 2}}
\label{sec:proof:thm:entropicVariationalMinimizers}

\begin{lemma}
    \label{lem:entropicFrechetDerivativeObjective}
    $\Phi_{(\mu_0,\mu_1),\varepsilon}$ is Fr{\'e}chet differentiable at $\mathbf{A}\in\mathbb R^{d_0\times d_1}$ with $\left(D\Phi_{(\mu_0,\mu_1),\varepsilon}\right)_{[\mathbf{A}]}(\mathbf{B})=64\langle\mathbf{A}-\frac{1}{2}\int xy^{\intercal}d\pi_{\mathbf{A}}(x,y),\mathbf{B}\rangle_{\mathrm F}$, where $\pi_{\mathbf{A}}$ is the unique optimal coupling for $\mathsf{OT}_{\mathbf{A},\varepsilon}(\mu_0,\mu_1)$.
\end{lemma}
\begin{proof}
    As noted in the proof of \cref{lem:FrechetDerivativeObjective}, $\|\cdot\|_{\mathrm{F}}^2$ is Fr{\'e}chet differentiable at $\mathbf{A}$ with derivative $2\langle \mathbf{A},\cdot\rangle_{\mathrm F}$.
    Now, for any $\mathbf{H}\in\mathbb R^{d_0\times d_1}$, we have that 
    \[
           \OT_{\mathbf{A}+\mathbf{H}}(\mu_0,\mu_1)-\OT_{\mathbf{A}}(\mu_0,\mu_1)\leq \int c_{\mathbf{A}+\mathbf{H}}d\pi_{\mathbf{A}}-\int c_{\mathbf{A}}d\pi_{\mathbf{A}}=-32\int x^{\intercal}\mathbf{H}yd\pi_{\mathbf{A}}(x,y), 
    \]
    for the unique EOT plan $\pi_{\mathbf{A}}$ for $\OT_{\mathbf{A},\varepsilon}(\mu_0,\mu_1)$.
    Similarly, 
    \[
        \OT_{\mathbf{A}+\mathbf{H}}(\mu_0,\mu_1)-\OT_{\mathbf{A}}(\mu_0,\mu_1)\geq -32\int x^{\intercal}\mathbf{H}yd\pi_{\mathbf{A}+\mathbf{H}}(x,y),
    \]
    for the optimal coupling $\pi_{\mathbf{A}+\mathbf{H}}$ for $\OT_{\mathbf{A}+\mathbf{H},\varepsilon}(\mu_0,\mu_1)$.
    Proceeding as in the proof of \cref{lem:FrechetDerivativeObjective}, it suffices to show that 
    \[
    \left\|\int xy^{\intercal}d\pi_{\mathbf{A}+\mathbf{H}}(x,y)-\int xy^{\intercal}d\pi_{\mathbf{A}}(x,y)\right\|_{\mathrm{F}}\to 0
    \]
    as $\mathbf{H}\downarrow 0$.

    To this end, consider an arbitrary sequence $\mathbf{H}_n$ converging to $0$ and let $(\varphi_{0,n},\varphi_{1,n})$ be EOT potentials for $\mathsf{OT}_{\mathbf{A}+\mathbf{H}_{n},\varepsilon}(\mu_0,\mu_1)$. By Lemma 4 in \cite{zhang2024gromov}, we may choose the EOT potentials as to satisfy the estimates 
            \[   \varphi_{i,n}\leq  K(1+\|\cdot\|^{2}),\quad -\varphi_{i,n} \leq  K(1+\|\cdot\|^{4}), \quad, i\in\{0,1\},
    \]
    for all $n$ sufficiently large, where $K$ is a constant that depends only on $\|\mathbf A\|_{\infty},\sigma,d_0,d_1$, $\varepsilon$ (though the cited result holds for $\varepsilon=1$, the proof of \cref{prop:entropicGWPotentials} shows how to extend the result to arbitrary $\varepsilon>0$). Applying the Arzel\`a-Ascoli theorem, for any subsequence $n'$ of $n$ there exists a further subsequence $n''$ along which $(\varphi_{0,n},\varphi_{1,n})$ converges to a pair of continuous functions $(\varphi_0,\varphi_1)$ uniformly on compact sets (in particular pointwise). Note also that $c_{\mathbf A+\mathbf H_{n''}}\to c_{\mathbf A}$ pointwise. It follows from the $4$-sub-Weibull condition and the dominated convergence theorem that  
\[
        \begin{aligned} 
            e^{-\frac{\varphi_{0,n''}(x)}\varepsilon}=\int
             e^{\frac{\varphi_{1,n''}-c_{\mathbf{A}+\mathbf{H}_{n''}}(x,\cdot)}\varepsilon}d \mu_{1}&\to  \int e^{\frac{\varphi_1-c_{\mathbf{A}}(x,\cdot)}\varepsilon}d  \mu_{1}=e^{-\frac{\varphi_0(x)}\varepsilon},
             \\  
            e^{-\frac{\varphi_{1,n''}(y)}\varepsilon}=\int
             e^{\frac{\varphi_{0,n''}-c_{\mathbf{A}+\mathbf{H}_{n''}}(\cdot,y)}\varepsilon}d \mu_{0}&\to  \int e^{\frac{\varphi_0-c_{\mathbf{A}}(\cdot,y)}\varepsilon}d  \mu_{0}=e^{-\frac{\varphi_1(y)}\varepsilon},
        \end{aligned}
        \]
        for $\mu_0\otimes \mu_1$-a.e. $(x,y)$
        so that $(\varphi_0,\varphi_1)$ solve the Schr\"odinger system for $\mathsf{OT}_{\mathbf{A},\varepsilon}( \mu_0, \mu_1)$. Given the characterization of the optimal coupling in terms of the EOT potentials from \cref{sec:EOT}, we obtain that 
        \[
            \begin{aligned}
                \int xy^{\intercal}d\pi_{\mathbf{A}+\mathbf{H}_{n''}}(x,y)&=\int xy^{\intercal}e^{\frac{\varphi_{0,n''}(x)+\varphi_{1,n''}(y)-c_{\mathbf{A}+\mathbf H_{n''}}(x,y)}\varepsilon}d\mu_0\otimes \mu_1(x,y),
                \\
                &\to \int xy^{\intercal}e^{\frac{\varphi_{0}(x)+\varphi_{1}(y)-c_{\mathbf{A}}(x,y)}\varepsilon}d\mu_0\otimes \mu_1(x,y),
                \\
                &= \int xy^{\intercal}d\pi_{\mathbf{A}}(x,y),
            \end{aligned}
        \]
        by applying the dominated convergence theorem; recalling the EOT potential estimates and the sub-Weibull assumption. Given that the limit is independent of the choice of original subsequence, we have that $\int xy^{\intercal}d\pi_{\mathbf{A}+\mathbf{H}}(x,y)\to \int xy^{\intercal}d\pi_{\mathbf A}(x,y)$ as $\mathbf H\to 0$ proving the claim.   
\end{proof}

\begin{proof}[Proof of \cref{thm:entropicVariationalMinimizers}]

Given \cref{lem:entropicFrechetDerivativeObjective},  the remainder of the proof of  \cref{thm:entropicVariationalMinimizers} (1) follows the same lines as the proof of     \cref{prop:Lipschitz}, noting that the proof of that result requires only that $\mu_0$ and $\mu_1$ have finite fourth moments. 
  
Given the coercivity of $\Phi_{(\mu_0,\mu_1),\varepsilon}$ and the fact that it is continuously differentiable, its minimizers must be stationary points.  
\end{proof}

\section{Convergence of minimizers of \texorpdfstring{$\Phi$}{Phi}}
\label{app:gammaConvergence} 
Let $\mathcal X_0$ and $\mathcal X_1$ be compact subsets of $\mathbb R^{d_0}$ and $\mathbb R^{d_1}$. We prove the following general result. 

\begin{proposition}
\label{prop:convergenceMinimizers}
    Let  $(\nu_{0,n})_{n\in\mathbb N}\subset \mathcal P(\mathcal X_0)$ and $ (\nu_{1,n})_{n\in\mathbb N}\subset \mathcal P(\mathcal X_1)$ converge weakly to $\nu_0$ and $\nu_1$ respectively. Then, if $\mathbf{A}_n\in \argmin_{\mathbb R^{d_0\times d_1}}\Phi_{(\nu_{0,n},\nu_{1,n})}$, any limit point of $(\mathbf A_n)_{n\in\mathbb N}$ minimizes $\Phi_{(\nu_0,\nu_1)}$.  
\end{proposition}
\begin{proof}
    By \cref{thm:VariationalMinimizers}, the minimizers of all relevant problems are contained in $B_{\mathrm{F}}(M)$ for $M=\frac 12 \|\mathcal X_0\|_{\infty}\|\mathcal X_1\|_{\infty}$. Consider the functions 
    \[
    \begin{aligned}
        F_n(\mathbf A) &= \begin{cases}
                \Phi_{(\nu_{0,n},\nu_{1,n})}(\mathbf A),&\text{if } \mathbf A \in B_{\mathrm{F}}(M),
                \\
               +\infty,&\text{otherwise}, 
        \end{cases}
    \\ 
        F(\mathbf A) &= \begin{cases}
                \Phi_{(\nu_{0},\nu_{1})}(\mathbf A),&\text{if } \mathbf A \in B_{\mathrm{F}}(M),
                \\
               +\infty,&\text{otherwise}, 
        \end{cases}
    \end{aligned} 
    \]
    and note that, upon restriction to $B_{\mathrm{F}}(M)$, they are Lipschitz continuous with a shared Lipschitz constant (see the proof of \cref{prop:Lipschitz}). As $\mathsf{OT}_{\mathbf A}(\nu_{0,n},\nu_{1,n})$ converges to $\mathsf{OT}_{\mathbf A}(\nu_{0},\nu_{1})$ for any $\mathbf A\in B_{\mathrm{F}}(M)$ (see Theorem 5.20 in \cite{villani2008optimal}), it follows from Proposition 7.15 in \cite{rockafellar2009variational} and Corollary 1.8.6 in \cite{bogachev2020real} that $F_n$ epi-converges to $F$ relative to $B_{\mathrm F}(M)$. Explicitly, this means that $\liminf_{n\in\mathbb N}F_n(\mathbf A_n)\geq F(\mathbf A)$ for every sequence $B_{\mathrm{F}}(M)\supset\mathbf A_n\to \mathbf A$ and $\limsup_{n\in\mathbb N}F_n(\mathbf A_n')\leq F(\mathbf A)$ for some sequence $B_{\mathrm{F}}(M)\supset\mathbf A_n'\to \mathbf A$, for any choice of $\mathbf A\in B_{\mathrm F}(M)$. As $B_{\mathrm F}(M)$ is closed, both of these conditions extend verbatim with $\mathbb R^{d_0\times d_1}$ in place of $B_{\mathrm{F}}(M)$. With this, we may apply Theorem 7.33 in \cite{rockafellar2009variational} to obtain that any limit point of a sequence $\mathbf A_n\in \argmin_{\mathbb R^{d_0\times d_1}}F_n$ minimizes $F$ as desired. 
\end{proof}

As each of the empirical measures $\hat\rho_{k,ij,n}$ converge weakly to $\hat\rho_{k,ij}$  with probability $1$ for $k\in\{0,1\}$ and $i,j\in[N]$, $i<j$ (see e.g. Theorem 3 in \cite{varadarajan1958convergence}), it readily follows that $\mu_{0,n}$ and $\mu_{1,n}$ converge weakly to $\mu_0$ and $\mu_1$ with probability $1$.  

\section{Additional experiments}
\label{sec:extraExperiments}
\subsection{Initialization of the local solver for distributions on binary graphs}
\label{sec:warmstart}

This experiment illustrates the 
value obtained at each of the matrices $(\mathbf A_{T})_{T\in\mathcal T}$ when applying the exhaustive search procedure described in \cref{sec:computation}. 

First, we fix isomorphic distributions, $\nu_0,\nu_1$, on the set of all binary graphs with $7$ vertices, see Figure \ref{fig:warmstartPopulation} for the precise edge probabilities and the choice of permutation. Next, we sample $n\in\{10,25,50,500,5000\}$ graphs according to these distributions to construct the empirical estimators $\mu_{0,n}$ and $\mu_{1,n}$. Finally, we run \cref{alg:subgradient} at each matrix $(\mathbf A_{T})_{T\in\mathcal T}$ and store the resulting value, denoted $\Lambda_T$, for each of of the $5040$ permutations in $\mathcal T$ and for each of the pairs $(\mu_{0,n},\mu_{1,n})$ as well as population distributions  $(\mu_0,\mu_1)$.  These values are compiled in Figure \ref{fig:warmstart}.  
\begin{figure}[!htb]
    \centering
    \includegraphics[width=0.7\textwidth]{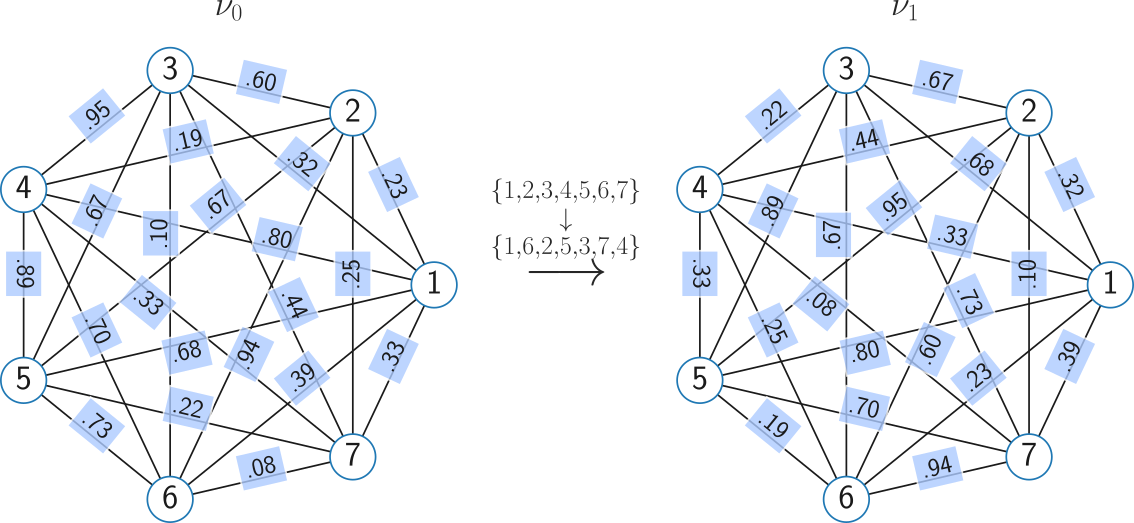}
    \caption{Underlying distributions for this experiment. First, edge probabilities for $\nu_0$ were sampled from the uniform distribution on $[0,1]$, then a random permutation was chosen to generate $\nu_1$ from $\nu_0$ as depicted. Edge probabilities are presented with two significant figures for ease of reading.}
    \label{fig:warmstartPopulation}
\end{figure}

\begin{figure}[!htb]
    \centering
    \includegraphics[width=0.8\textwidth]{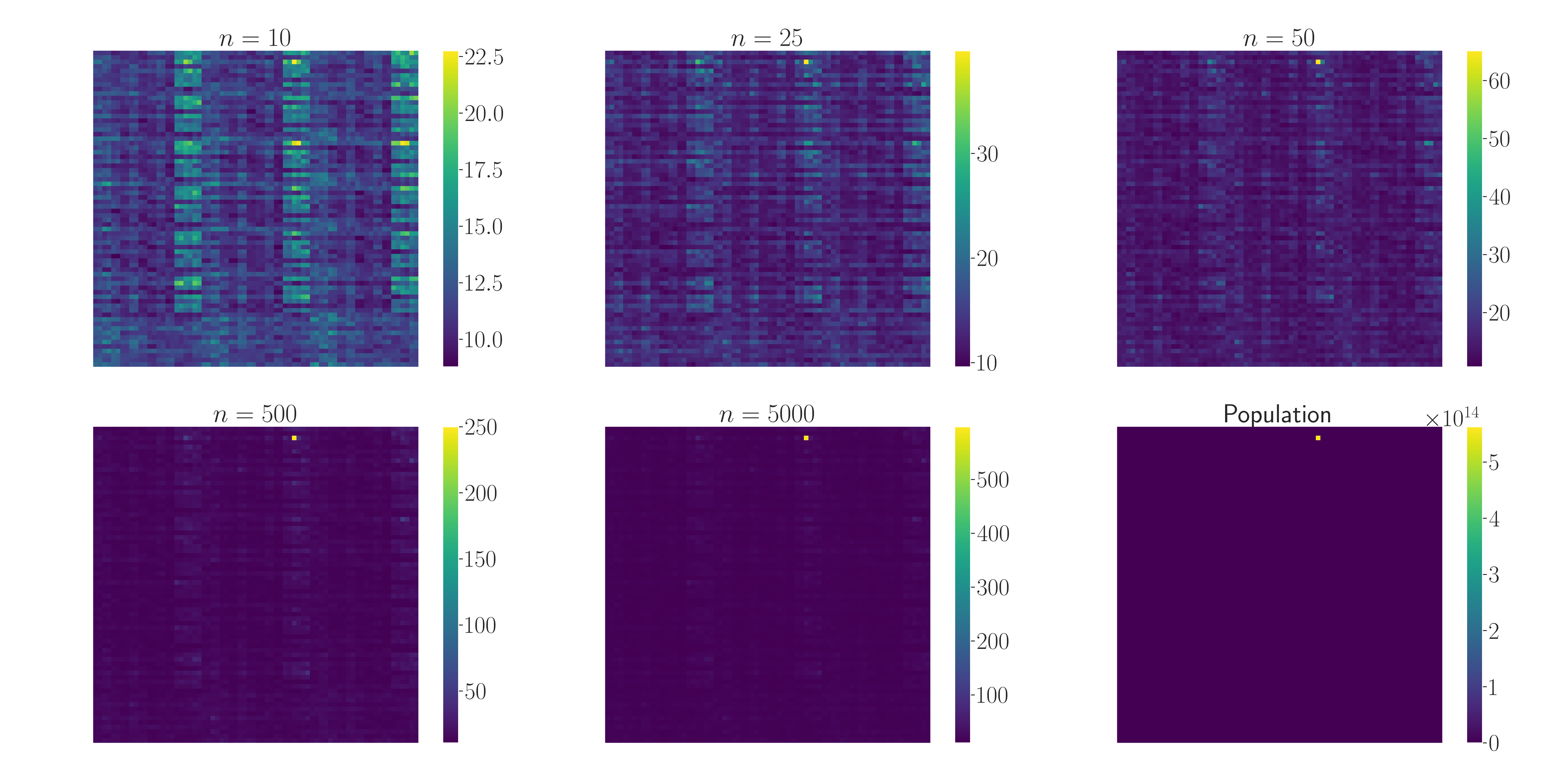}
    \caption{Each image corresponds to performing the described experiment with different numbers of samples or on the true population distributions $\nu_0$ and $\nu_1$ as illustrated by the title. Each pixel of the $70\times 72$ images above is the inverse of the value obtained starting from one of the $5040$ possible initializations. Apart from the experiment with $n=10$, the images contain a unique yellow pixel, indicating that the associated permutation results in a value significantly closer to the true value of $0$ than the other permutations and, indeed, corresponds to the permutation relating $\nu_0$ and $\nu_1$.}
    \label{fig:warmstart}
\end{figure}

We highlight that, for these experiments, $\argmin_{T\in\mathcal T}\Lambda_T=\{\bar T\}$, where $\bar T$ coincides with the true permutation, $T^{\star}$, even when only $10$ samples are used. On the other hand, the relaxed graph matching approacj, \cref{alg:relaxedGraph} fails to recover $T^{\star}$ for $n=10$, but succeeds otherwise.

The ratio $\frac{\min_{T\in\mathcal T\backslash \{\bar T\}}\Lambda_T}{\Lambda_{\bar T}}$ was computed for each experiment and is of the order $1.05,1.30,3.63,4.49,9.61,$ and $8.47\times 10^{12}$ for $n=10,25,50,500,5000$, and the exact distributions respectively. This illustrates that there is a substantial gap between the smallest and next smallest

(i.e. $\Lambda_{\bar T}$ is significantly smaller than $\Lambda_{T}$ for every $T\neq \bar T$). Finally, as $\nu_0$ and $\nu_1$ are isomorphic, $\mathsf D(\mu_{0,n},\mu_{1,n})$ converges to $0$ as $n\to \infty$, this is reflected by the values obtained for $\Lambda_{\bar T}$ as $n$ varies.

\subsubsection{Comparison between warm-start methods}
\label{sec:comparisonWarmStart}

As evidenced by Figure \ref{fig:type12errorBinary}, using relaxed graph matching to approximate the GW distance between distributions on graphs appears to result in inflated values when the number of samples is low. It is thus of interest to explore if the exhaustive search approach fares better. To this end, we repeat the procedure described in \cref{sec:testingBinary} with binary graphs on $6$ nodes so that both methods can be implemented. The distributions for this experiment are displayed in Figure \ref{fig:comparisonGraphs}.

\begin{figure}[!htb]
    \centering
    \includegraphics[width=0.7\textwidth]{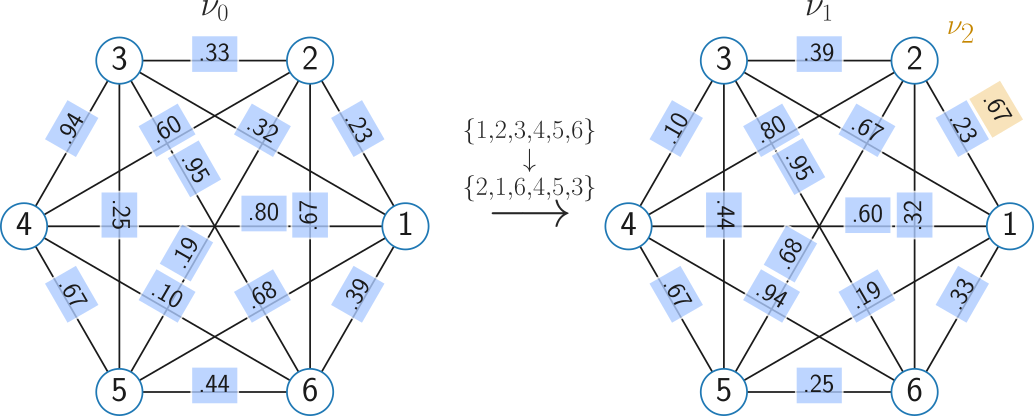}
    \caption{Distributions for \cref{sec:comparisonWarmStart}. $\nu_0$ is generated by   sampling $(p_{0,ij})_{\substack{i,j\in[N]\\i<j}}$ from the uniform distribution on $[0,1]$. $\nu_1$ is obtained from $\nu_0$ by a randomly chosen permutation. The modified probability for $\nu_2$ is highlighted in orange. Edge probabilities are presented with two significant figures for ease of reading.}
    \label{fig:comparisonGraphs}
\end{figure}

\begin{figure}[!htb]
    \centering
    \includegraphics[width=\textwidth]{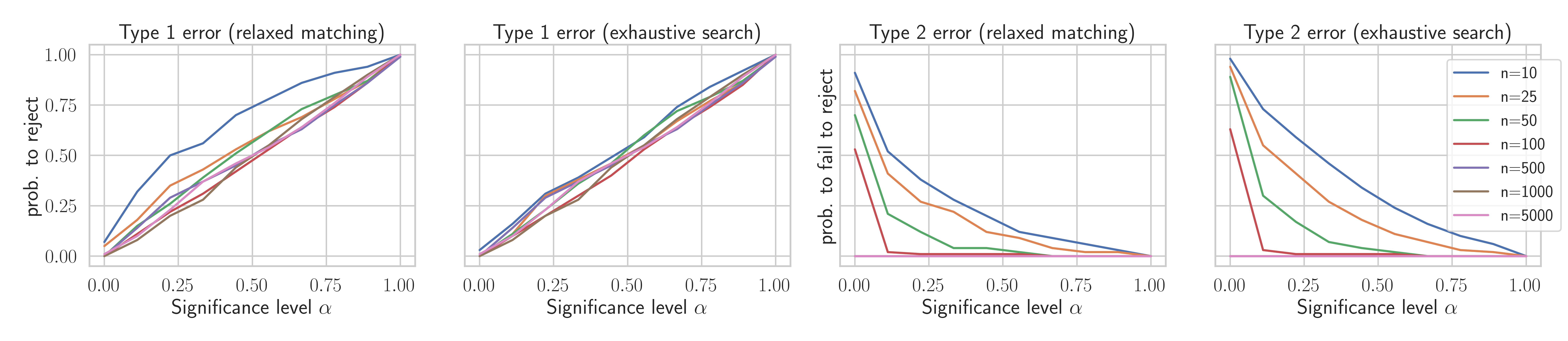}
    \caption{Comparison of the type 1 and type 2 error between the warm-start method based on relaxed graph matching and the exhaustive search for the experiment described in \cref{sec:comparisonWarmStart}.}
    \label{fig:type12comparison}
\end{figure}

It can be seen from Figure \ref{fig:type12comparison} that, as expected, the exhaustive search results in a better approximation of the GW distance than the relaxed graph matching approach. This is particularly evident when $n=10$ or $25$ as reflected by the behaviors of the type $1$ and type $2$ errors.

\section{The delta method}
\label{sec:deltaMethod}
Our approach to proving limit distributions for GW distances is based on the functional delta method \cite{shapiro1991asymptotic,dumbgen1993nondifferentiable,romisch2004,fang2019}. To introduce this approach, we treat probability measures $(\nu_0,\nu_1)\in\mathfrak P\subset \mathcal P(\mathbb R^{d_0})\times \mathcal P(\mathbb R^{d_1})$ as functionals on some function space $\mathcal F^{\oplus}=\mathcal F_0\oplus \mathcal F_1\coloneqq \{f_0\oplus f_1:f_0\in\mathcal F_0,f_1\in\mathcal F_1\}$, where $\mathcal F_i$ is a class of Borel measurable functions on $\mathbb R^{d_i}$ for $i\in\{0,1\}$. Here, we require $\mathfrak P$ and $\mathcal F^{\oplus}$ to be compatible in the sense that 
\[%
\|\nu_0\otimes \nu_1\|_{\infty,\mathcal F^{\oplus}}\coloneqq\sup_{(f_0, f_1)\in\mathcal F_0\times \mathcal F_1}|\nu_0\otimes \nu_1(f_0\oplus f_1)|<\infty,\quad \forall (\nu_0,\nu_1)\in\mathfrak P,
\]
whereby $\nu_0\otimes \nu_1$ can be identified with an element of $\ell^{\infty}(\mathcal F^{\oplus})$. 

With these preliminaries, we may specialize Proposition 1 in \cite{goldfeld24statistical} as follows.

\begin{proposition}[Proposition 1 in \cite{goldfeld24statistical}]
\label{prop:unified}
   For $i\in\{0,1\}$, let $\mathcal F_i$ be a class of Borel measurable functions on $\mathbb R^{d_i}$ with a finite envelope $F_i$ and set $\mathcal F^{\oplus}=\mathcal F_0\oplus \mathcal F_1$. Fix $(\mu_0,\mu_1)\in\mathcal P(\mathbb R^{d_0})\times \mathcal P(\mathbb R^{d_1})$ and let $\mathsf{L}$ be a functional on a subset $\mathfrak P$ of $\mathcal P(\mathbb R^{d_0})\times \mathcal P(\mathbb R^{d_1})$, where $\mathfrak P$ can be identified with a {convex subset of $\ell^{\infty}(\mathcal F^{\oplus})$}   and $\int F_0 d\nu_0+\int F_1 d\nu_1<\infty$ for every $(\nu_0,\nu_1)\in\mathfrak P$. Suppose, furthermore, that
   \begin{enumerate}[label=(\alph*)]
       \item $(\mu_{0,n},\mu_{1,n}):\Omega\to \mathfrak P$, $n\in\NN$, are pairs of random probability measures with values in $\mathfrak P$ such that $\sqrt n(\mu_{0,n}-\mu_0)\otimes \sqrt n(\mu_{1,n}-\mu_1)\stackrel{d}{\to} G_{\mu_0\otimes \mu_1}$ in $\ell^{\infty}(\mathcal F^{\oplus})$,  where $G_{\mu_0\otimes \mu_1}$ is a tight random variable in $\ell^{\infty}(\mathcal F^{\oplus})$.
       \item  there exists a finite constant $C$ for which $|\mathsf{L}(\nu_0,\nu_1)-\mathsf{L}(\rho_0,\rho_1)|\leq C\|\nu_0\otimes \nu_1-\rho_0\otimes \rho_1\|_{\infty,\mathcal F^{\oplus}}$ for any pairs $(\nu_0,\nu_1),(\rho_0,\rho_1)\in\mathfrak P$.
       \item for any pair $(\nu_0,\nu_1)\in\mathfrak P$, the limit $\lim_{t\downarrow 0}t^{-1}\big(\mathsf{L}(\mu_0+t(\nu_0-\mu_0),\mu_1+t(\nu_1-\mu_1))-\mathsf{L}(\mu_0,\mu_1)\big)\eqqcolon \mathsf{L}'_{(\mu_0,\mu_1)}(\nu_0-\mu_0,\nu_1-\mu_1)$ exists.
   \end{enumerate}
   Then, %
   $\sqrt n\big(\mathsf{L}(\mu_{0,n},\mu_{1,n})-\mathsf{L}(\mu_{0},\mu_{1})\big)\stackrel{d}{\to}\mathsf{L}'_{(\mu_0,\mu_1)}(G_{\mu_0\otimes \mu_1})$.
\end{proposition}

\end{document}